  \documentclass[reqno, 11 pt]{amsart}
 \usepackage[utf8]{inputenc}

\usepackage{amssymb}   
\usepackage{amsmath}  
\usepackage{euscript} 
\usepackage{amsthm} 
\usepackage{dsfont}  
\usepackage{amsfonts}     
\usepackage{marginnote}
\usepackage{latexsym} 
\usepackage{amsopn} %per definire nuovi operatori "operator name"
\usepackage{geometry}
\usepackage{mathrsfs}
\usepackage{caption} 
\usepackage{aurical}  
\usepackage[english]{babel}
\usepackage[utf8]{inputenc}
\usepackage{graphicx}
\usepackage[dvipsnames]{xcolor}
\usepackage{multicol}
\usepackage{nomencl}
\makeglossary
\makenomenclature  
\usepackage{refcount}
\usepackage{float}
\usepackage{hyperref}
\usepackage{times}
\usepackage[titletoc]{appendix}
\usepackage{accents}

\DeclareFontFamily{U}{mathx}{}
\DeclareFontShape{U}{mathx}{m}{n}{<-> mathx10}{}
\DeclareSymbolFont{mathx}{U}{mathx}{m}{n}
\DeclareMathAccent{\widehat}{0}{mathx}{"70}
\DeclareMathAccent{\widecheck}{0}{mathx}{"71}

\newlength{\dhatheight}
\newcommand{\doublehat}[1]{%
	\settoheight{\dhatheight}{\ensuremath{\widehat{#1}}}%
	\addtolength{\dhatheight}{-0.35ex}%
	\widehat{\vphantom{\rule{1pt}{\dhatheight}}%
		\smash{\widehat{#1}}}}
\newcommand{\comment}[1]{}

\newcommand{\pa}{\partial}

\newcommand{\al}{\alpha}

\newcommand{\g}{\gamma}
\newcommand{\eps}{\varepsilon}
\newcommand{\vphi}{{\varphi}}
\newcommand{\s}{{\sigma}}
\newcommand{\oo}{{\omega}}
\newcommand{\ad}{\operatorname{ad}} 
\newcommand{\id}{\operatorname{Id}}
\newcommand{\Id}{\mathbb{I}}

\newcommand{\op}{{\mathrm{Op}}}

\providecommand{\vect}[2]{{\bigl[\begin{smallmatrix}#1\\#2\end{smallmatrix}\bigr]}} \providecommand{\sign}{\mathrm{sgn}\,}  
\providecommand{\sm}[4]{{\bigl(\begin{smallmatrix}#1&#2\\#3&#4\end{smallmatrix}\bigr)}}

\providecommand{\mat}[4]{{\begin{pmatrix}#1&#2\\#3&#4\end{pmatrix}}}
\newtheorem{thm}{Theorem}[section]
\newtheorem*{thm*}{Theorem}
\newtheorem{prop}[thm]{Proposition}
\newtheorem{lemma}[thm]{Lemma}

\newtheorem*{cor*}{Corollary}
\newtheorem{rmk}[thm]{Remark}
\newtheorem{ex}{Example}

\newtheorem{defn}[thm]{Definition}

\numberwithin{equation}{section}
\newcommand{\ii}{{\rm i}}

\newcommand{\gso}{{s_*}}
\newcommand{\so}{{{s}_0}}
\newcommand{\su}{{{ s}_1}}

\newcommand{\om}{{\omega}}

\newcommand{\x}{\xi}

\newcommand{\ov}{\overline}

\newcommand{\C}{{\mathbb C}}

\newcommand{\N}{{\mathbb N}}

\newcommand{\R}{{\mathbb R}}

\newcommand{\T}{{\mathbb T}}
\newcommand{\Z}{{\mathbb Z}}

\newcommand{\cA}{{\mathcal A}}
\newcommand{\cB}{{\mathcal B}}
\newcommand{\cC}{{\mathcal C}}
\newcommand{\cD}{{\mathcal D}}

\newcommand{\cG}{{\mathcal G}}
\newcommand{\cH}{{\mathcal H}}

\newcommand{\cL}{{\mathcal L}}
\newcommand{\cM}{{\mathcal M}}

\newcommand{\cO}{{\mathcal O}}
\newcommand{\calO}{\mathcal{O}}
\newcommand{\cP}{{\mathcal P}}
\newcommand{\cQ}{{\mathcal Q}}
\newcommand{\cR}{{\mathcal R}}
\newcommand{\cS}{{\mathcal S}}
\newcommand{\cT}{{\mathcal T}}
\newcommand{\cU}{{\mathcal U}}

\newcommand{\cY}{{\mathcal Y}}

\newcommand{\fa}{{\mathfrak{a}}}
\newcommand{\fc}{{\mathfrak{c}}}

\newcommand{\fD}{{\mathfrak{D}}}
\newcommand{\fd}{{\mathfrak{d}}}

\newcommand{\fr}{{\mathfrak{r}}}
\newcommand{\fS}{{\mathfrak{S}}}

\newcommand{\fU}{{\mathfrak{U}}}

\newcommand{\ta}{{\mathtt{a}}}
\newcommand{\tb}{{\mathtt{b}}}
\newcommand{\tc}{{\mathtt{c}}}
\newcommand{\td}{{\mathtt{d}}}

\newcommand{\tf}{{\mathtt{f}}}
\newcommand{\tg}{{\mathtt{g}}}

\newcommand{\tm}{{\mathtt{m}}}

\newcommand{\tr}{{\mathtt{r}}}

\renewcommand{\tt}{{\mathtt{t}}}

\newcommand{\tw}{{\mathtt{w}}}

\newcommand{\tA}{{\mathtt{A}}}
\newcommand{\tB}{{\mathtt{B}}}
\newcommand{\tC}{{\mathtt{C}}}
\newcommand{\tD}{{\mathtt{D}}}

\newcommand{\tP}{{\mathtt{P}}}
\newcommand{\tQ}{{\mathtt{Q}}}

\newcommand{\tT}{{\mathtt{T}}}

\newcommand{\tU}{{\mathtt{U}}}

\usepackage{bm}

\renewcommand{\d}{\td}

\newcommand{\dhh}{\td_{\vphi_h}}

\newcommand{\0}{{(0)}}
\newcommand{\2}{{(2)}}
\newcommand{\1}{{(1)}}

\newcommand{\im}{{\rm i}}
\newcommand{\jap}[1]{\langle #1 \rangle}
  
\newcommand{\und}[1]{\underline{#1}}

\newcommand{\uno}{{\bf 1}}

\newcommand{\diag}{\mathop{\mathrm{diag}}}

\newcommand{\bnorm}[1]{{|\mkern-6mu |\mkern-6mu | \,  #1 \,  |\mkern-6mu |\mkern-6mu |}  }

\newcommand{\nnorm}[1]{{\left\vert\kern-0.25ex\left\vert\kern-0.25ex\left\vert #1 
    \right\vert\kern-0.25ex\right\vert\kern-0.25ex\right\vert}}
\usepackage{framed,enumitem}

\makeatletter

\def\l@subsection{\@tocline{2}{0pt}{2.5pc}{5pc}{}}
\def\l@subsubsection{\@tocline{3}{0pt}{4.5pc}{5pc}{}}
\renewcommand\tocchapter[3]{%
  \indentlabel{\@ifnotempty{#2}{\ignorespaces#2.\quad}}#3%
}
\begin{document} 
 
\title[Reducibility of 
quasi-linear wave equations on the torus]{Reducibility of % and stability of
 Klein-Gordon equations \\% on the torus
with % unbounded
  maximal order perturbations}
\date{}
\author{Massimiliano Berti}
\address{\scriptsize{SISSA, Via Bonomea 265
34136, Trieste Italia}}
\email{berti@sissa.it}

\author{Roberto Feola}
\address{\scriptsize{Dipartimento di Matematica e Fisica, Universit\`a degli Studi RomaTre, 
Largo San Leonardo Murialdo 1, 00144, Roma Italia}}
\email{roberto.feola@uniroma3.it}

\author{Michela Procesi}
\address{\scriptsize{Dipartimento di Matematica e Fisica, Universit\`a degli Studi RomaTre, 
Largo San Leonardo Murialdo 1, 00144, Roma Italia}}
\email{michela.procesi@uniroma3.it}

\author{Shulamit Terracina} 
\address{\scriptsize{Dipartimento di Matematica, Universit\`a degli Studi di Milano, Via Saldini 50, I-20133, Milano Italia}}
\email{shulamit.terracina@unimi.it}



%\maketitle
   
\begin{abstract}   
We prove that all the solutions of  
a quasi-periodically forced 
 linear Klein-Gordon equation 
$\psi_{tt}-\psi_{xx}+\mathtt{m}\psi+\mathcal{Q}(\omega t)\psi=0  $
where  $ \mathcal{Q} (\omega t) :=
a^{\2}(\oo t, x) \partial_{xx} +  a^{\1}(\oo t, x)\partial_x  + a^{\0}(\oo t, x)  $ is a 
differential operator of {\it order $ 2 $},  parity preserving and reversible,  
are {\it almost periodic} in time and  uniformly bounded  for all times, 
provided that the coefficients $ a^{(2) }, a^{(1) }, a^{(0) } $ are small enough and the 
forcing frequency  $\omega\in \R^{\nu}$ 
belongs to a Borel set of asymptotically full measure. 
This result is obtained by reducing 
the Klein-Gordon equation 
to a diagonal constant coefficient 
system 
with purely imaginary eigenvalues.
The main difficulty
is 
 the presence in the perturbation $ \mathcal{Q} (\omega t) $  
 of the second order  differential 
operator $ a^{\2}(\omega t, x)\partial_{xx} $.
In suitable coordinates 
the  
Klein-Gordon equation  is the composition of two  
 backward/forward quasi-periodic in time perturbed 
 transport equations with non-constant coefficients, up to 
 lower order
 pseudo-differential remainders.  A key idea is to straighten 
this %a quasi-periodic
 first order pseudo-differential operator 
with bi-characteristics
 through a novel quantitative Egorov analysis.  
 \end{abstract}

\maketitle

\setcounter{tocdepth}{1}
\tableofcontents

\section{Introduction and main results}

We consider a linear 
Klein-Gordon equation 
with periodic boundary conditions $ x \in \T := \R / ( 2 \pi \Z) $   perturbed by 
a time {\it quasi-periodic} 
differential operator of {\it maximal order two}, of the form
\begin{equation}\label{NLW}
\psi_{tt}-\psi_{xx}+\mathtt{m} \psi + a^{\2}(\oo t, x)\psi_{xx} +  a^{\1}(\oo t, x)\psi_x 
+ a^{\0}(\oo t, x)\psi =0\, ,
\end{equation} 
where  the coefficients 
$$
a^{(i)}  :\T^{\nu}_\varphi\times \T_x 
\to \R \, , \quad  (\varphi, x)  \mapsto a^{(i)} (\varphi, x) \, , 
\quad  i = 0,1,2 \, , 
$$ 
are $ {\mathcal C}^{\infty}$  functions,  the frequency 
vector $\omega \in \R^\nu $,  $\nu\in  \N$, 
is  diophantine and 
 belongs to the compact set 
$ \Lambda:=[-1/2,1/2]^{\nu} $, and 
the  mass $\mathtt{m} > 0 $. 

 \smallskip

The goal of this paper 
is to prove a {\it reducibility} result for the equation 
\eqref{NLW} --namely conjugate it
 to a constant coefficient diagonal 
 system with purely imaginary eigenvalues--,  
 assuming 
the parity and reversibility properties  \eqref{oddness}-\eqref{revers} 
and suitable 
smallness conditions for the coefficients $ a^{(i)} $. As a corollary 
we deduce
 an upper bound on the Sobolev norms 
  of the solutions of \eqref{NLW},  uniformly for any time $t $  in $ \R $. 
 Actually we prove that all the solutions of \eqref{NLW} are {\it almost periodic} in time. 
This is the content of the {\it perpetual stability} Theorem \ref{thm:stab}   
and of the reducibility Theorem \ref{main:thm}, from which it is deduced. 
Both these results hold if the frequency vector 
$ \omega $ belongs to  a
subset 
of the diophantine vectors  in $ \Lambda $    
having
 asymptotically full measure,  
that we characterize 
in Theorem \ref{thm:cantorset}.

\smallskip

Partial differential equations (PDEs) as 
\eqref{NLW}  arise, for instance, from the linearization of quasi-linear  or fully nonlinear 
 Klein-Gordon equations --commonly 
 used in mechanics, relativity or elasticity theory--  at a quasi-periodic function. 
Actually, a major 
motivation for the development of reducibility theorems 
comes from KAM theory for nonlinear PDEs. 
Indeed 
it is well known that 
a key step of a Nash-Moser iterative scheme % fundamental ingredient 
is 
 to prove the invertibility of the 
linearized operators % obtained
 at a quasi-periodic approximate solution % $ u(\omega t, x)$  
and to provide
 suitable 
tame estimates of its inverse on Sobolev spaces  $ H^s $. % without 
% variable quasi-periodic coefficients  
Such estimates are readily obtained
provided one is able to conjugate % reduce
 the linearized operator 
%with  quasi-periodic coefficients % argument and % 
to  a constant coefficients diagonal %Fourier multiplier 
system, 
via a 
change of variables  which satisfies 
tame estimates in Sobolev spaces as \eqref{stimemappa}. 
The key point of these estimates is to  
require
 only a ``loss of  $ \mu $ 
 derivatives" for $ \| a^{(i)}  \|_{s+\mu} $ 
{\it independently} of $ s $ large. 
This makes Theorem \ref{main:thm} applicable 
 to nonlinear KAM theory 
 (unlike reducibility results which require a non controlled loss of  
derivatives on the coefficients as \cite{Mo,BGMR,BLM}). 
The proof of the tame estimate \eqref{stimemappa}  
 is the reason for 
a lot of technical work 
in the present paper, see comment ii) after Theorem \ref{main:thm}.

\smallskip

 The idea of reducibility has a long history. 
 After the seminal work of Floquet 
 for time periodic ordinary differential equations, 
 it has been % used 
 effectively  % generalized  
 extended 
  for quasi-periodic perturbations of 
  constant coefficients  linear ODEs  via KAM techniques 
  (see for instance 
 \cite{Eli} and references therein) and, more recently,  generalized for 
 linear partial differential equations. 
 
 In the infinite dimensional  PDE context % phases spaces 
 the {\it order} 
 of the perturbative operator  $ P $ 
 versus the order of the unperturbed %linear 
one  $ L $
 plays a key %dominant 
 role. We mention
 for instance that the  seminal paper \cite{Comb87}  
 was able to deal with 
 time periodic smoothing perturbations of $ 1d $-Schr\"odinger operators, namely if 
 $ ord (P) \ll ord (L) $. 
This result was later   
 improved in  \cite{DuS} for bounded perturbations and 
 in \cite{Bambusi-Graffi} for % stronger 
 time quasi-periodic  unbounded 
 perturbations with $ ord (P) < ord (L)-1 $, 
 by means of the Kuksin-lemma idea employed 
 in KAM theory for semilinear perturbations of KdV in \cite{K2}. 
 Actually reducibility results and nonlinear KAM theory have a parallel development.
The possibility to deal with  maximal order perturbations 
where $ ord (P) = ord (L) $  requires the use 
 of pseudo-differential and Fourier integral operator techniques. 
 This idea has been 
 effectively introduced 
in the context of KAM theory for $ 1d $ quasi-linear and fully nonlinear KdV equations
in \cite{BBM14, BBM16} and extended in \cite{BM20, BBHM}
for the water waves equations, see also \cite{BFM1,BFM2,FGww}.
These 
techniques have been also applied 
 for $ 1d$ Schr\"odinger  equations in \cite{FP, FGP1},
for quasi-linear perturbations of the 
DP equation in \cite{FGP19,FGP1},
 for transport equations in \cite{ FGMPmoser,BLM}, 
for quasi-linear perturbations  of large finite gap solutions of KdV in \cite{BKM},  
for the vortex patch equations
of  Euler and $ \alpha$-SQG % equations
 in 
\cite{BerHMasmo,HHM21,GsIP}, and for $ 3d$ Euler in \cite{BM3d}.
Reducibility results for the Schr\"odinger operator on $ \R $ with 
unbounded confining potential up to the maximal order 
have been proved in \cite{Bam1,Bam2,BM}
thanks to pseudo-differential techniques, 
and in \cite{BGMR} in higher dimension.

So far 
there are no reducibility results  concerning quasi-linear $ 1d$ 
wave or Klein-Gordon  equations as \eqref{NLW} with $ a_2 (\varphi , x ) \neq 0 $,
a fortiori neither KAM theorems for nonlinearities with $ 2 $ space derivatives. 
This is the gap that the present paper aims to fill. 
 For linear wave or Klein-Gordon equations 
 there are only some reducibility 
 results  \cite{SLX,Mo, FM,LF1} for bounded perturbations, 
the KAM results  \cite{Ku, W1, Po3,  BoK, Po2, CY} for semilinear nonlinearities, 
and those 
for derivative wave equations 
  \cite{
Berti-Biasco-Procesi-Ham-DNLW,
Berti-Biasco-Procesi-rev-DNLW} 
(with one space derivative in the nonlinearity).
The main source of difficulties lies in 
straightening a first order quasi-periodic pseudo-differential 
operator with $ 2$ characteristic directions 
(due to the term $ a^{\2}(\oo t, x)\pa_{xx} $ in \eqref{NLW})
by conjugation with suitable Fourier integral  operators 
 through a novel quantitative Egorov analysis. 
 We discuss some of these ideas at the end of the introduction.

\smallskip

Before stating precisely our results we 
record some literature about PDEs in 
higher  space dimension. 
All the results concern semilinear perturbations.
The reducibility approach requires 
the verification of the so called 
``second order Melnikov" non-resonance conditions, which concern 
lower bounds for the difference of the eigenvalues of the linearized operators. 
For $ 1d$ Schr\"odinger and wave  equations \cite{Ku, Ku1,Po3, W1, BoK, KP, Po2, CY}
 this is  possible.  
On the other hand for 
PDEs in higher space dimension the eigenvalues are much less separated.
Reducibility and KAM results for Sch\"rodinger 
equations on 
$ \T^d $, $ d \geq 2 $,  have been nevertheless obtained in
\cite{EK2} and \cite{GY, EK, GYX, PP, EGK, GrP} for semilinear perturbations, and 
for Klein-Gordon  equations on the sphere in 
\cite{GrePatu,  FGNzoll,FGsphere}.
On the other hand, for wave equations on $ \T^d $ 
the second order Melnikov 
non-resonance conditions are violated, see \cite{E17}. 
Existence of quasi-periodic solutions 
for semilinear wave equations in $ \T^d $  % higher space dimension
has been nevertheless proved in 
 \cite{ B5, BW1, BB12, BCP, CM, Berti-Bolle_book}
with  ``multiscale" techniques, not based on 
 reducibility arguments,  
 stemming from \cite{CW,Bo1, B-Gafa}. 
It is 
not clear 
if this approach also provides  stability of the   solutions of the 
linearized equation at the quasi-periodic solutions, as 
a reducibility result does. 
If the frequency vector satisfies only a diophantine condition 
(and not also second order Melnikov non-resonance conditions) 
non-uniform upper bounds as $t^\eps$ for the growth of the 
 Sobolev norms  have been obtained
 for linear Schr\"odinger  equations  
 in \cite{Blin,D,BGMR2}, see also \cite{Mo18} for the half wave. 

\smallskip

Let us now state precisely our results. First of all, in order to 
guarantee the reducibility of \eqref{NLW} to a diagonal system
with purely imaginary eigenvalues --this means avoiding
  ``dissipative/exploding" dynamical effects--
some hypotheses are in order. 
In this work we assume that  the coefficients $a^{(i)}(\vphi,x)$ in \eqref{NLW} satisfy the following conditions
\begin{align}
	&a^{(i)}(\vphi,-x)=a^{(i)}(\vphi,x)\,,\; i=0,2\,,
	\qquad  \quad
	a^{(1)}(\vphi,-x)=-a^{(1)}(\vphi,x)\,, \label{oddness} \\
	& a^{(i)}(-\vphi,x)=a^{(i)}(\vphi,x)\,,\; i=0,1,2\,,
	\label{revers}
\end{align}
which endow \eqref{NLW} with a {\it parity} preserving,  {\it reversible} 
structure, as we shall explain below. 

Our perpetual stability result is the following, 
where we denote by $H^{s}(\mathbb{T}^d,\R)$ the  Sobolev spaces
of real-valued periodic functions of $ d $ variables.

\begin{thm}{\bf (Sobolev stability).}\label{thm:stab}
Assume \eqref{oddness}, \eqref{revers} and 
fix any $ \mathtt{m} >0  $. Let  $\nu\in\N$ and
fix  $\so>(\nu+7)/2$.
There is $\bar{s}:=\bar s(\nu)$ and for any $s>\so$ there exists  $ \mathfrak{d}_0(s) > 0 $  such that, 
assuming the smallness condition 
\begin{equation}\label{condsaggia}
 \mathfrak d:=\max_{i=0,1,2}\|a^{(i)}\|_{H^{\bar{s}}(\T^{\nu+1},\R)}\leq\mathfrak{d}_{0}(s)\,, 
\end{equation}
then there exists a Borel set of frequencies $\mathcal{C}_{\infty} \subset \Lambda$ 
of asymptotically full measure as $\mathfrak{d}\to 0$,
 such that for any $\om\in \mathcal{C}_{\infty}$ and  for any initial condition 
$(\psi_0,v_0)\in H^{s+1}(\mathbb{T},\R)\times H^{s}(\mathbb{T},\R) $,
 the unique global in time solution $ \psi(t,\cdot)$ in $H^{s+1}(\mathbb{T},\R)$ of 
the Cauchy problem
\begin{equation}\label{cauchyProb}
\left\{\begin{aligned}
&\psi_{tt}-\psi_{xx}+\mathtt{m} \psi + a^{\2}(\omega t, x)\psi_{xx} 
+  a^{\1}(\omega t, x)\psi_x 
+ a^{\0}(\omega t, x)\psi =0\,, 
\\&
\psi(0,x)=\psi_0(x)\,,
\\&
\pa_{t}\psi(0,x)=v_0(x)\,,
\end{aligned}\right.
\end{equation}
is almost-periodic in time  and  satisfies, for some $C(s)>0$,  
\begin{equation}\label{stimaTempoinfinito}
\sup_{t\in\R}\big(\|\psi(t,\cdot)\|_{H^{s+1}(\T,\R)}
+\|(\pa_{t}\psi)(t,\cdot)\|_{H^{s}(\T,\R)}
\big)\leq C(s)
\big(\|\psi_0\|_{H^{s+1}(\T,\R)}
+\|v_0\|_{H^{s}(\T,\R)}\big) \, .
\end{equation}

\end{thm}

Theorem \ref{thm:stab}
 is deduced as a corollary of the 
reducibility KAM Theorem \ref{main:thm}  for the Klein Gordon
equation \eqref{NLW}.
Before stating it, we make 
the following comments:
\begin{enumerate}
\item With respect to the previous 
 literature 
regarding Klein-Gordon  equations, 
the main {\it novelty} of our results 
is  the presence in \eqref{NLW} of the second order perturbative 
operator $ a^{\2}(\omega t, x)\partial_{xx} $.
 This  is a major 
 source of difficulties and 
requires  
a completely  different strategy.
 \item 
Other   properties on the coefficients of \eqref{NLW}, different
than the parity and reversibility assumptions 
\eqref{oddness}-\eqref{revers}, for example the requirement  that 
\eqref{NLW}  has a Hamiltonian structure, could 
imply a perpetual stability result as \eqref{stimaTempoinfinito}.
% avoiding the emergence of  non zero Lyapunov exponents. 
For definiteness  
we  decided to consider in this paper 
the case \eqref{oddness}-\eqref{revers} which contains all the difficulties of the problem. 
\item If $ \mathtt{m} < 0 $ clearly 
Theorem \ref{thm:stab} does not hold for 
the presence of % appereance 
finitely many ``hyperbolic" directions: consider 
for instance \eqref{NLW} with all the $ a^{(i)} = 0 $. 
If $ \mathtt{m} = 0 $ 
the stability/instability of \eqref{NLW}  depends on the coefficients $ a^{(i)}$: consider for instance 
$ \psi_{tt} - \psi_{xx} + \epsilon \psi = 0 $  
with  $ \epsilon  $ positive or negative. 
\end{enumerate}

In order to state Theorem \ref{main:thm},   it is first 
convenient to rewrite \eqref{NLW}
as a first order linear system
\begin{equation}\label{patX}
\partial_t \vect{\psi}{v} = X(\oo t) \vect{\psi}{v} \, , 
\end{equation}
where $ X(\oo t) $ is the second order quasi-periodic linear operator 
\begin{equation}\label{firstorder}
X(\oo t) := \begin{pmatrix}
0 & 1 \\
-\tD_{\mathtt{m}}^{2} 
- a^{\2}(\omega t, x)\pa_{xx} -  a^{\1}(\omega t, x)\pa_x - a^{\0}(\omega t, x) & 0 
\end{pmatrix} 
\end{equation}
and $\tD_{\mathtt{m}} := \sqrt{-\partial_{xx} + \mathtt m }$ is the Fourier multiplier 
\begin{equation}\label{TDM}
\tD_{\mathtt{m}}e^{\ii j\cdot x}=\tD_{\mathtt{m}}(j)e^{\ii j\cdot x}\,,
\qquad \tD_{\mathtt{m}}(j):=\sqrt{j^{2}+\mathtt{m}}\,,
\qquad \forall\, j\in \mathbb{Z}\,.
\end{equation} 
By assumption \eqref{oddness}, 
the linear operator $ X(\oo t) $ in  \eqref{firstorder} 
commutes with the  involution
 \begin{equation}\label{invoPP}
({\mathcal P} \vect{\psi}{v})(x) := \vect{\psi(-x)}{v(-x)}\, , \quad {\mathcal P}^2 = \Id:=\sm{1}{0}{0}{1} \, , 
 \end{equation}
 namely
 \begin{equation}\label{invoPP2}
 X(\oo t)  \circ {\mathcal P} = {\mathcal P} \circ  X(\oo t) \,,\quad \forall t\in \R \, , 
 \end{equation}
  and therefore the 
 subspaces of \emph{odd/even} functions in $x $ 
are invariant under  the flow of \eqref{patX}.
We say that $ X(\omega t) $ is \textit{parity preserving}. 
Moreover assumption \eqref{revers} implies the reversibility property
\begin{equation}\label{invorev2}
E\circ X(\oo t)=-X(-\oo t)\circ E
\end{equation}
where  $E$
is  the involution
\begin{equation}\label{involutionReale}
E\vect{\psi}{v}=\vect{\psi}{-v}\,,
\qquad E:=\sm{1}{0}{0}{-1}\,,
\qquad E^{2}=\Id\,.
\end{equation}
We say that  $ X(\omega t) $ is {\it reversible}. 
We remark that this is tantamount to  the Moser reversibility property 
$ (E\circ\mathcal{P})\circ  X (\oo t) =-X(-\omega t) \circ (E\circ\mathcal{P})  $ introduced in \cite{Mo67}.  

 We  study system 
 \eqref{patX}
 on the phase spaces
 \begin{equation}\label{spazioODD2}
 \mathcal{X}_{\R}^{s}:=H^{s+1}(\mathbb{T},\R)\times H^{s}(\mathbb{T},\R)\,,
 \end{equation}
 which is the direct sum of the invariant subspaces
 \begin{equation}\label{spazioODD}
 \mathcal{X}_{odd,\R}^{s}:=\mathcal{X}_{\R}^{s}\cap \big\{ 
 \vect{\psi(-x)}{v(-x)}=-\vect{\psi(x)}{v(x)}\big\}\,, \qquad 
 \mathcal{X}_{even,\R}^{s}:=\mathcal{X}_{\R}^{s}\cap \big\{ 
 \vect{\psi(-x)}{v(-x)}=\vect{\psi(x)}{v(x)}\big\}\,.
 \end{equation}
It  is convenient to introduce the complex variables
\begin{equation}\label{complexVar}
\begin{aligned}
&\qquad \vect{u}{\bar{u}}:=\mathcal{C}\vect{\psi}{v} \qquad 
\Leftrightarrow \qquad \vect{\psi}{v}
=\mathcal{C}^{-1}\vect{u}{\bar{u}}
\\
&
\mathcal{C}:=\tfrac{1}{\sqrt{2}}\left(\begin{matrix}
\tD_{\mathtt{m}} & -\ii
\\ 
\tD_{\mathtt{m}} & \ii
\end{matrix}
\right)\,,
\qquad 
\mathcal{C}^{-1}:=\tfrac{1}{\sqrt{2}}
\left(\begin{matrix}
\tD_{\mathtt{m}}^{-1} & \tD_{\mathtt{m}}^{-1}
\\ 
\ii & -\ii
\end{matrix}
\right)\, , 
\end{aligned}
\end{equation} 
in which 
system \eqref{patX} becomes the first order complex system 
\begin{equation}
\label{coordcomp}
\pa_{t}U=\ii E\mathfrak{D}(\omega t)U\,,\quad U :=\vect{u}{\bar{u}} \, , 
\end{equation}
where $\mathfrak{D}(\omega t) $ is the matrix of complex operators 
\begin{equation}\label{firstorderComplex}
\mathfrak{D}(\omega t) :=
\mathtt{D}_{\mathtt{m}}\Id 
+\uno
\Big(
\tfrac{1}{2}a^{(2)}(\omega t,x)\pa_{xx}+
\tfrac{1}{2}a^{(1)}(\omega t,x)\pa_{x}+ \tfrac{1}{2}a^{(0)}(\omega t,x)
\Big){\mathtt{D}_{\mathtt{m}}^{-1}}
\end{equation}
and
\begin{equation}\label{matE}
\uno := \sm{1}{1}{1}{1}\,.
\end{equation}
The main advantage of introducing the variable $ U $ 
is that in \eqref{firstorderComplex} 
the constant coefficients part $\mathtt{D}_{\mathtt{m}}\Id $ is diagonal. We also note that the map 
$ \mathcal{C} $ in \eqref{complexVar} is an isomorphism 
\begin{equation}\label{lambolambo}
\mathcal{C} \,:\, \mathcal{X}_{\R}^{s}\to 
\mathcal{H}^{s}\,,
\qquad
\mathcal{C} \,:\, \mathcal{X}_{p,\R}^{s}\to 
\mathcal{H}_{p}^{s}\,,\quad p\in\{odd, even\}\, , 
\end{equation}
between the spaces \eqref{spazioODD} and 
the complex spaces 
\begin{equation}\label{spazioODDComplessi}
\begin{aligned} 
\mathcal{H}^{s}&:= \left\{U=\vect{u^+}{u^-}\in H^{s}(\T;\C^{2})\;:\quad  \overline{u^+}=u^- \right\} 
\\
\mathcal{H}_{odd}^{s}&:=\mathcal{H}^{s}\cap \big\{U(-x)=-U(x)\big\}\,, \qquad
\mathcal{H}_{even}^{s}:=\mathcal{H}^{s}\cap \big\{U(-x)=U(x)\big\}\, \, .
\end{aligned}
\end{equation}
In the  complex coordinates \eqref{complexVar},
the involution $ \mathcal{P} $ in \eqref{invoPP} 
reads $\mathcal{C}\circ\mathcal{P}\circ\mathcal{C}^{-1}=\mathcal{P}$, 
and  the involution $E$ in \eqref{involutionReale} 
becomes
\begin{equation}\label{involutionCompl}
S:=\sm{0}{1}{1}{0}=\mathcal{C}\circ E\circ\mathcal{C}^{-1} \, , \quad 
S\vect{u}{\bar{u}}=\vect{\bar{u}}{{u}} \,  .
\end{equation}

\begin{thm}{\bf (Reducibility).}\label{main:thm}
Assume \eqref{oddness}-\eqref{revers} 
and fix any $ \mathtt{m} > 0 $. Let  $\nu\in\N$ and fix $\so>(\nu+7)/2$.
There exists $\mu>0$ and for any $\su \geq \so $,  
there exists $\delta_0(\su)>0$ such that 
for any $\gamma\in(0,\tfrac{1}{2})$, assuming the smallness condition
\begin{equation}\label{smallCondCoeff}
\gamma^{-7/2}\|a^{(i)}\|_{H^{\so+\mu}(\T^{\nu+1},\R)}\leq \delta_{0}(\su)\,,\qquad i=0,1,2\,,
\end{equation}
then there is 
a Borel set $\mathcal{G}_{\infty}(\gamma)\subseteq\Lambda$ 
with asymptotically full Lebesgue measure, i.e.
\begin{equation}\label{measureestimateGGG}
|\Lambda\setminus\mathcal{G}_{\infty}(\g)|\to 0\quad \text{as}\quad \gamma\to0\,,
\end{equation}
such that the following holds. For any $\omega\in \cG_{\infty}(\g)$ 
there exists
 a bounded invertible map  $\mathfrak{F}(\omega t) $, 
 depending quasi-periodically 
in time, 
$ \mathfrak{F}(\varphi) : H^{s}(\T,\C^2)\to H^{s}(\T,\C^2) $ for any
$ \so\leq s\leq \su $, 
satisfying 
\begin{enumerate}
	\item{\bf Real-to-real}: $\mathfrak{F}(\varphi): \mathcal{H}^{s}\to \mathcal{H}^{s}$; 
\item {\bf Parity preserving}: $\mathfrak{F}(\varphi): \mathcal{H}_{odd}^{s}\to \mathcal{H}_{odd}^{s}$ and 
$
\mathfrak{F}(\varphi): \mathcal{H}_{even}^{s}\to \mathcal{H}_{even}^{s}$; 
\item {\bf Reversibility preserving}:  
 ${S}\circ \mathfrak{F}(\varphi)=\mathfrak{F}(-\varphi)\circ {S}$;
\item{\bf Tameness}: the operator $ \mathfrak{F}(\varphi) $, as well 
as its inverse $ \mathfrak{F}(\varphi)^{-1}  $,  satisfies, for any 
$\so\leq s\leq \su $,
for any $h\in \mathcal{H}^{s}$, for some $C(s)>0$,   the tame estimates 
\begin{align}
& \qquad \| \mathfrak{F}(\varphi ) h\|_{H^{s}(\T,\C^2)}\leq
C(s) 
\|h\|_{H^{s}(\T,\C^2)}+ \frac{C(s) }{\gamma^{7/2}}\sup_{i=0,1,2} \|a^{(i)}\|_{H^{s+\mu}(\T^{\nu+1},\R)} {\|h\|_{H^{\so}(\T,\C^2)}} , \label{stimemappa} %\\
%& 
%\qquad \| \mathfrak{F}(\varphi ) h\|_{H^{s}(\T^{\nu+1},\C^2)}\leq
%C(s) 
%\|h\|_{H^{s}(\T^{\nu+1},\C^2)}+ \frac{C(s) }{\gamma^{7/2}}\sup_{i=0,1,2} \|a^{(i)}\|_{H^{s+\mu}(\T^{\nu+1},\R)} {\|h\|_{H^{\so}(\T^{\nu+1},\C^2)}} \, ; \notag 
\end{align}
In addition,  for any $w\in H^{s}(\T^{\nu+1},\C^2)$, 
\begin{equation}\label{nuovatamediff}
\| \mathfrak{F}(\varphi ) w\|_{H^{s}(\T^{\nu+1},\C^2)}\leq
C(s) 
\|w\|_{H^{s}(\T^{\nu+1},\C^2)}
+ \frac{C(s) }{\gamma^{7/2}}\sup_{i=0,1,2} \|a^{(i)}\|_{H^{s+\mu}(\T^{\nu+1},\R)} {\|w\|_{H^{\so}(\T^{\nu+1},\C^2)}}.
\end{equation}
\item {\bf Reducibility}: 
there exist Lipschitz 
 functions $\mathfrak{c}\,,\;\mathfrak{r}_{j}^{\s j} : \Lambda \to\R\,,$ $j\in \Z$,
satisfying the parity properties  
\begin{equation}\label{lipfinal}
\mathfrak{r}_{j}^{\s j}=\mathfrak{r}_{-j}^{-\s j}\,,\qquad \forall 
j\in\N_{0}\,,\s=\pm\,,
\end{equation}
and the estimates 
\begin{align}\label{pressione}
& \qquad \sup_{\omega\in\Lambda}|\mathfrak{c}(\omega)|
+
\gamma\sup_{\omega_1\neq\omega_2}
\frac{|\mathfrak{c}(\omega_1)-\mathfrak{c}(\omega_2)|}{|\omega_1-\omega_2|}
\leq C  \sup_{i=0,1,2} \|a^{(i)}\|_{H^{\so+\mu}(\T^{\nu+1},\R)}
\\
&\sup_{j\in\N_{0}, \s=\pm} \left(\sup_{\oo\in\Lambda} |\fr^{\s j}_{j}(\oo)| + \g^{3/2} \sup_{\oo_{1}\neq \oo_{2}} \frac{|\fr^{\s j}_{j}(\oo_{1})-\fr^{\s j}_{j}(\oo_{2})|}{|\oo_{1}-\oo_{2}|}\right) \leq 
\frac{C}{ \g^{2}} \sup_{i=0,1,2}\|a^{(i)}\|_{H^{\so+\mu}(\T^{\nu+1}, \R)}\,, \notag 
\end{align}
such that 
 $U=\vect{u}{\bar{u}}$ solves \eqref{coordcomp}
 if and only if  $ Z :=\vect{z}{\bar{z}}:=\mathfrak{F}(\omega t)\vect{u}{\bar{u}} $
  solves the system 
 \begin{equation}\label{firstorderBis}
\pa_{t}Z= 
\left(\begin{matrix}
\ii \mathfrak{D}_{\infty}^{+}z \\
-\ii \mathfrak{D}_{\infty}^{+}\bar{z}
\end{matrix}\right) 
\end{equation}
where, in the Fourier representation $z=\sum_{j\in\Z}e^{\ii jx}z_j$, one has 
\begin{equation}\label{Dinfinito}
\mathfrak{D}_{\infty}^{+}z := ((1+\mathfrak{c})\sqrt{\mathtt{m}}+\mathfrak{r}_{0}^{0})z_0+ \sum_{j\in \Z\setminus\{0\}} e^{\ii j x}\Big(\big((1+\mathfrak{c})\tD_{\tm}(j)+ \frac{\mathfrak{r}_{j}^{j}}{\langle j\rangle} \big)z_{j}+
\frac{\mathfrak{r}_{j}^{-j}}{\langle j\rangle}z_{-j}  \Big)\,.
\end{equation}
\end{enumerate}
All the constants are uniform for $ \mathtt{m} > 0 $ in a compact set.  
\end{thm}

Let us make some comments. 
\begin{enumerate}
\item[i)] 
Item 5 is the reducibility result: it means that $\mathfrak{F}(\oo t)$ conjugates the quasi-periodic  variable coefficients 
 system \eqref{coordcomp} to
the  constant in time $2\times 2$ block diagonal system \eqref{firstorderBis}-\eqref{Dinfinito},
which is actually  diagonal % system 
in the real basis $ \{ 1 $, $ (\cos(jx),\sin(jx))_{j\in \N} \} $,  see item iii) below.  
\item[ii)] We remark that  the tame 
estimate  \eqref{stimemappa}  controls the $H^s$ norm of $\mathfrak{F} h$  in terms of the  $H^s$ norm of $h$ and of the $H^{s+\mu}$ norm of the coefficients $ a^{(i)}$, where 
the ``loss of derivatives"   $ \mu $ 
 is {\it independent} of $ s $. 
This makes \eqref{stimemappa} applicable to 
construct quasi-periodic solutions of nonlinear Klein-Gordon equations
via a Nash-Moser-KAM iterative scheme, see e.g. 
\cite{BB12,Berti-Bolle_book,BCP,BM20,BBHM, FP,BFM1,BFM2,BKM,BerHMasmo}. 
On the other hand an estimate as \eqref{stimemappa} 
with a  loss $ \mu (s) $  depending on $ s $ (as is done in most results on reducibility, see for instance \cite{Mo,BGMR,BLM})  would be much easier to deduce,
but it would not be sufficient 
for KAM applications to nonlinear PDEs.
We remark that according to \cite{LZ} a Nash-Moser scheme may work 
only under a tame estimate with $\mu(s)< 2 s $. 
\end{enumerate}

The Cantor set $\mathcal{G}_{\infty} (\gamma) $
in Theorem \ref{main:thm} is characterized 
 in terms of the eigenvalues $\lambda_{j,\pm}^{(\infty)}$ of the operator
$ \mathfrak{D}_{\infty}^+ $ as follows. 

\begin{thm}{\bf (Cantor set $\mathcal{G}_{\infty} (\gamma) $).}\label{thm:cantorset}
Under the assumptions of Theorem \ref{main:thm}
setting for $\omega\in \Lambda$
\begin{equation}\label{finaleigenv}
	\begin{aligned}
\lambda_{0,\pm}^{(\infty)} :=\lambda_{0,\pm}^{(\infty)}(\omega)&=
(1+\mathfrak{c})\sqrt{\mathtt{m}}+\mathfrak{r}_{0}^{0} \, ,  \\	
\lambda_{j,\pm}^{(\infty)} :=\lambda_{j,\pm}^{(\infty)}(\omega)&=
(1+\mathfrak{c})\tD_{\mathtt{m}}(j)
+\frac{\mathfrak{r}_{j}^{j}\pm \mathfrak{r}_{j}^{-j}}{\langle j\rangle}\,,
\qquad \forall\,j\in\N\,, 
\end{aligned}
\end{equation}
and defining the sets 
\begin{equation}\label{calOzero}
\Lambda_0:=\left\{\omega\in \Lambda\,:\, 
|\omega\cdot\ell|\geq 2\gamma |\ell|^{-\nu}\,, \ \ell\in \Z^{\nu}\setminus\{0\}
\right\}
\end{equation}
\begin{equation}\label{calOinfty1}
\Lambda_1:=\left\{\omega\in \Lambda\,:\,  
|\omega\cdot\ell-(1+\mathfrak{c})j|\geq 2\gamma\langle\ell\rangle^{-\tau}\,,
\;\; j\in\N\,,\; \ell\in \Z^{\nu}
\right\}
\end{equation}
\begin{equation}\label{calOinfty2}
\begin{aligned}
\Lambda_{2}^{+}:=&\left\{\omega\in \Lambda\,:\, 
|\omega\cdot\ell+\lambda_{j,\eta}^{(\infty)}+\lambda_{k,\eta}^{(\infty)}|\geq 
\frac{2\gamma}{\langle\ell\rangle^{\tau}}\,,
\;\; j,k\in\N_{0}\,,\; \ell\in \Z^{\nu}\,,\eta\in\{\pm\}
\right\}
\\
\Lambda_{2}^{-}:=&\Big\{\omega\in \Lambda\,:\, 
|\omega\cdot\ell+\lambda_{j,\eta}^{(\infty)}-\lambda_{k,\eta}^{(\infty)}|\geq 
\frac{2\gamma^{3/2}}{\langle\ell\rangle^{\tau}}\,,\;\eta\in\{\pm\}
\\&\qquad \qquad\qquad \qquad\qquad\qquad  \qquad
\;\; j,k\in\N_{0}\,,\; \ell\in \Z^{\nu}\,,
\;\;(\ell, j,k)\neq(0, j, j)
\Big\}
\end{aligned}
\end{equation}
for $\tau>2\nu+4$ and $\gamma\in(0,\tfrac{1}{2})$,
one has the inclusion
\begin{equation}\label{includoAutovalfinali}
\Lambda_{0}\cap\Lambda_{1}\cap\Lambda_{2}^{+}
\cap\Lambda_{2}^{-}\subseteq\mathcal{G}_{\infty}(\gamma) 
\quad \text{and}\quad
|(\Lambda_{0}\cap\Lambda_{1}\cap\Lambda_{2}^{+}
\cap\Lambda_{2}^{-})^c|\leq C \gamma
\end{equation}
for some $ C := C(\tau, \nu, {\mathtt m})$.
\end{thm}
Let us make some further comments.
\begin{itemize}
 \item[iii)] 
The  property  \eqref{lipfinal} is a consequence of the parity assumption 
\eqref{oddness} and it implies 
that  the operator $\mathfrak{D}_{\infty}^{+}$ in \eqref{Dinfinito}  acts diagonally 
on the subspaces of odd/even functions defined in \eqref{spazioODDComplessi}. 
Actually $\mathfrak{D}_{\infty}^{+}$ written in the basis $(1,\cos(jx),\sin(jx))_{j\in \N}$ is diagonal with  eigenvalues $ \lambda_{j,\pm}^{(\infty)}$ as in \eqref{finaleigenv}
and
all the solutions of \eqref{firstorderBis} have the form
\begin{equation}\label{zevodd}
\quad z(t)= z_0(0)e^{\ii \lambda_{0,+}^{(\infty)} t}+ \sum_{j\in \N} 
w^{(\rm{ev})}_j(0) e^{\ii \lambda_{j,+}^{(\infty)} t} \cos(jx) + w^{(\rm{odd})}_j(0) e^{\ii \lambda_{j,-}^{(\infty)} t} \sin(jx)\,.
\end{equation}
\item[iv)] 
The  reality of the 
$ \mathfrak{r}_{j}^{\s j} $ is a consequence of the  reversibility condition \eqref{revers}. 
Without assuming  \eqref{revers}  we would obtain a reducibility result  for the operator $ \ii E\mathfrak{D}(\oo t)  $ to a constant coefficient
operator 
with eigenvalues with possibly a non zero real part. 
\item[v)] All the solutions of \eqref{coordcomp} are almost-periodic in time. Indeed
by \eqref{zevodd}  
all the solutions of \eqref{firstorderBis}
are almost-periodic in time, with the Sobolev norm bounded uniformly for 
any $ t $ in $ \R $. Recalling that $\mathfrak F(\omega t)$ is quasi-periodic in time and
\eqref{stimemappa}, we deduce that  $U(t)=(\mathfrak F(\omega t))^{-1}Z(t)$ is almost periodic and
 $\|u(t)\|_{H^s(\T,\C)} \le C(s) \|u(0)\|_{H^s(\T,\C)}$ for any $t\in \R$. 
\end{itemize}

We now show how Theorem \ref{main:thm}  implies the stability Theorem \ref{thm:stab}.
\begin{proof}[Proof of the stability result.]
	The global well posedness of the the Cauchy problem 
	\eqref{cauchyProb} with  initial conditions 
	$(\psi_0,v_0)\in  \mathcal{X}_{\R}^{s}$, with  $s > 5/2 $,
	follows by classical results for quasi-linear hyperbolic equations
	(see \cite{Kato-loc}) with smooth coefficients.
We deduce the stability bound \eqref{stimaTempoinfinito}  as follows. We write \eqref{cauchyProb}  as the  first order complex system \eqref{coordcomp}.
Fix $\so>(\nu+7)/2$.  
Set  $ \bar{s} :=\so+\mu  $
where $ \mu $  is defined  by Theorem \ref{main:thm}. 
For any  $s > s_0 $ we apply  Theorem \ref{main:thm} with 
$\su \equiv s + \tfrac{\nu}{2} + 1 $. Letting $\gamma:=\mathfrak{d}^{1/4}$, where 
$\mathfrak{d} $ is defined in \eqref{condsaggia},  
the smallness condition \eqref{smallCondCoeff} is implied by \eqref{condsaggia}
taking $ \mathfrak{d}_0(s) \le (\delta_0(s))^8$.
Then we define 
% the Borel set  
$\mathcal{C}_{\infty} := \mathcal{G}_{\infty}(\mathfrak{d}^{1/4})$, whose 
Lebesgue measure 
tends to one as
$\mathfrak{d}\to 0$ thanks to \eqref{measureestimateGGG}.
 For any $\omega\in \mathcal{C}_{\infty}$,
all the solutions of \eqref{coordcomp} are 
 almost-periodic in time, see  remark v) after Theorem \ref{thm:cantorset}. Recalling  \eqref{lambolambo}, we deduce  \eqref{stimaTempoinfinito}.
\end{proof}

\paragraph{\bf Ideas of the proof.} 
In order to prove Theorem \ref{main:thm} we look for a time quasi-periodic 
transformation of the phase space 
$ \mathfrak{F}(\omega t): \mathcal{H}^{s}
\to \mathcal{H}^{s} $, $  Z:=\mathfrak{F}(\omega t)U $, 
such that 
\begin{equation}\label{complexZZ2}
\mathfrak{F}(\omega t)\circ
\ii E\mathfrak{D}(\omega t)\circ \mathfrak{F}^{-1}(\omega t)
+(\pa_t\mathfrak{F}(\omega t))\circ \mathfrak{F}^{-1}(\omega t)
=\ii E \sm{\mathfrak{D}^{+}_{\infty}}{0}{0}{\mathfrak{D}^{+}_{\infty}}
\end{equation}
where the operator $\mathfrak{D}^{+}_{\infty}$ is defined  in \eqref{Dinfinito}.
This problem can be restated
as the  diagonalization of a linear operator acting 
on 
periodic functions of $ (\varphi, x) $. %  in time and space.  
Explicitly a  time quasi-periodic family of linear operators 
 $\mathfrak{A}(\omega t)$  acting on 
$\mathcal{H}^{s}$ (possibly unbounded) defines
a (possibly unbounded)  linear operator acting
on  $ H^{s}(\mathbb{T}^{\nu+1},\C^2)$
by setting 
\begin{equation}
	\label{toppa}
(\mathfrak{A}U)(\vphi,x):=\mathfrak{A}(\vphi) U(\vphi,x)\, ,  
\quad \forall U(\vphi,x)  \in H^{s}(\mathbb{T}^{\nu+1},\C^2) \, . 
\end{equation}
With this notation we associate to the dynamical system \eqref{coordcomp}
%$\pa_{t}U=\ii E\mathfrak{D}(\omega t)U$
the linear operator 
\begin{equation}\label{ellePhiPhi}
\mathcal{L}:=
\omega\cdot\pa_{\vphi}-\ii E\mathfrak{D}\,.
\end{equation}
Suppose now  that there exists a 
family of operators $\mathfrak F(\omega t)$ acting on $\cH^s$, 
%map 
%$\mathfrak F$ acting on $H^{s}(\mathbb{T}^{\nu+1},\C^2)$
which 
block-diagonalizes $\cL$, 
%it is associated  to a 
and such that 
\begin{equation}\label{pirl}
\mathfrak{F}\circ \cL \circ\mathfrak{F}^{-1}- \oo\cdot\pa_{\vphi} = \mathfrak{F}\circ \ii E\mathfrak{D}\circ\mathfrak{F}^{-1}
+(\omega\cdot\pa_{\vphi}\mathfrak{F})\circ
\mathfrak{F}^{-1}
\end{equation}
 commutes with $\oo\cdot\pa_{\vphi}$ and $\tD_{\mathtt{m}}$. As a consequence 
\eqref{pirl} must have the  $2\times2$ block diagonal form 
\begin{equation}\label{ellediag}
- \ii E \sm{\mathfrak{D}^{+}_{\infty}}{0}{0}{\mathfrak{D}^{+}_{\infty}}\,,
\end{equation}
which means that $\mathfrak{F}(\oo t)$ reduces the system \eqref{coordcomp} 
 to the system 
\eqref{firstorderBis}.

We perform the diagonalization of the unbounded  operator $ \cL $  in two steps:
\\
(i) we first %  ``regularization'' of the operator $\cL$, namely to 
transform $ \cL $ into   a diagonal operator plus a sufficiently smoothing in space remainder (see Sections \ref{sec:diagBlock}--\ref{sec:ordine11});
\\		
(ii) we then complete the KAM diagonalization of the remainders via a KAM iterative scheme
(see Sections \ref{sec:kam}--\ref{sec:measure}).

Let us now briefly describe the ideas. %  concentrating on the main novelties.
The linear operator  $\cL$ is a $2\times 2$ matrix 
whose entries are  
pseudo-differential operators acting on $H^s(\T^{\nu+1},\C)$ of the form  
(see \eqref{L-omega})
 \[
 \mathcal{L}=\omega\cdot\pa_{\vphi}-\ii E\sm{1+b_1}{b_1}{b_1}{1+b_1}\tD_{\mathtt{m}}+  {\bf  O}(\fd\pa_{x}^{0}) 
 \]
 where $b_1=b_1(\vphi,x)$ is a smooth real valued function of size $\fd$ (in some appropriate Sobolev norm), $\tD_{\mathtt{m}}$ 
 is the Fourier multiplier of order $1$
 in \eqref{TDM}, and where by 
 ${\bf O}(\fd\pa_{x}^{-k})$, $k\geq0$, we denote a $2\times 2$ matrix 
 whose entries 
 operators in $O(\fd\pa_{x}^{-k})$, namely they are 
 $k$-smoothing   in 
 space  
 and of size $\fd$ (in some norm). 
  
 The first step is to  apply a change of variables which makes the off-diagonal operators very smoothing (in space) thus decoupling the equations for $u$ and $\bar u$.  
 Such change of variable is constructed iteratively 
 in  Section \ref{sec:diagBlock} obtaining a conjugate operator of  the form 
 (see formula \eqref{elle3}) 
 \begin{align}\label{calL3}
 {\mathcal{L}_{3}}=\omega\cdot\partial_{\varphi}
 -
\ii \sm{\lambda}{0}{0}{-\lambda} \, \mathtt{D}_{\tm}
  +
\sm{O({\fd}\pa_{x}^{0})}{0}{0}{O({\fd}\pa_{x}^{0})}+ {\bf O}({\fd}\pa_{x}^{-\rho})\,, \quad 
\rho\gg1 \, , 
\end{align}
 where $\lambda=\lambda(\vphi,x)$ is a 
 smooth real valued function with $\lambda-1 $ of size $ \fd $.
The operator $\cL_3$ is still pseudo-differential. Let us concentrate on  the highest order component of $\cL_3$.
Denoting by $\Pi_{\pm}$ the {\it Szeg\"o} projectors onto positive and negative space Fourier modes, we have the expansion 
$$
\ii \tD_\tm=\ii |\pa_x| + O(\pa_x^{-1})= \Pi_+ \pa_x - \Pi_- \pa_x  +O(\pa_x^{-1}) \, , 
$$
% Let us concentrate on  the highest order component of $\cL_3$  whose 
and hence  the action of $\cL_3$  on  $u$ is (up to lower order pseudo-differential 
operators)
% (at the highest order) 
\begin{equation}
	\label{nabla}
\omega\cdot\pa_{\vphi}- \lambda (\Pi_+ \pa_x - \Pi_- \pa_x )= (\omega\cdot\pa_{\vphi}-  \lambda  \pa_x) \Pi_+ + (\omega\cdot\pa_{\vphi}+ \lambda  \pa_x) \Pi_- 
\end{equation}
namely  two distinct transport  operators  for positive and negative Fourier modes. The action on $\bar u$ is just the complex conjugate.
We  reduce  \eqref{nabla} to constant coefficients % of \eqref{nabla}
%  transport operator
by conjugation via a map of the form 
 \begin{equation}
	 \label{calpa}
L:= \cC_{\alpha_{+}}\Pi_+ +  \cC_{\alpha_{-}}\Pi_-\,,
\end{equation} 
 where $ (\cC_{\alpha}u)(\vphi,x):=u(\vphi,x+\alpha(\vphi,x))$ is the composition 
 operator  induced by a diffeomorphism of the torus.  
We look for two different periodic functions $\al_+, \al_-$
which rectify the transport operators along positive and negative Fourier modes,
corresponding to the {\it two characteristic directions of the wave equation}.  In Theorem \ref{IncredibleConjugate} we construct $\alpha_\pm$ so that \eqref{calpa} conjugates \eqref{nabla} to a constant coefficient bi-characteristic operator 
$$
\omega\cdot\pa_{\vphi}- \ii (1+\fc)|\pa_x| + O(\fd  \pa_x^{-\rho})
$$
 where $\fc$ is a constant and the remainder $O(\fd  \pa_x^{-\rho})$ is no longer pseudo-differential but it satisfies ``tame estimates".  
 
A novelty of the paper is the way we control the tame action of a linear operator
via the norm  $ \bnorm{\cdot}_s $ that we introduce in  Section \ref{sec:modulotame}.
A related notion was proposed 
in  \cite[Def. 2.18]{BM20} and \cite[Def 2.6]{FGP1}  via the concept of  `modulo tame constants".
The novelty of our approach is that we introduce a Banach algebra (of couples of operators) 
which has the same flexibility of modulo tame constants but in the more structured setting of Banach algebras.
To an operator $ A $ we associate a couple $\mathtt A = (M,R) $ so that 
$ A = M + R $ where $ M : H^p \to H^p $ is {\it majorant-bounded} for any $ p $
in a fixed range $[s_*,s_1]$  and $ R $ is {\it smoothing} operator, namely given 
$ s \in [s_*,s_1] $, it results $ H^{s_*} \to H^{s} $.  
Then we define 
$$
\bnorm{\tA}_s := \sup_{\gso\le p\le \su}|M|_{p,p} + |R|_{\gso,s}\
$$
where $ | \ |_{s,s'} $ is the majorant norm in Def. \ref{majOp}.
A natural decomposition 
of a linear operator $A$ is to set 
$\mathtt A=(M,A-M)$ where  $M$  is the  \emph{Bony} part 
provided in Definition \ref{def:bonydeco}.

%A parallel construction has been implemented in 
 % \cite{Mo18}, which deals with a non-uniform control of the Sobolev norms  of the solutions of a half-wave equation, 
%  but in our setting we need quantitative tame estimates on the norm of the remainder. 
In Proposition \ref{prop:ridord1} we conjugate the whole 
$\cL_3$ in \eqref{calL3} via the map
$ \Theta_1:=\sm{L}{0}{0}{\bar{L}} $ where $ L $ is defined in 
 \eqref{calpa}, obtaining
\begin{equation}\label{L4forma}
\mathcal{L}_{4}:= \Theta_1 \cL_3 \Theta_1^{-1}=\omega\cdot\pa_{\vphi}-\ii \sm{1+\fc}{0}{0}{-(1+\fc)}\tD_{\mathtt{m}} +
\sm{O(\fd\pa_{x}^{0})}{0}{0}{O(\fd\pa_{x}^{0})}+ {\bf O}(\fd\pa_{x}^{-1})\, , 
\end{equation}
where the diagonal terms $O(\fd\pa_{x}^{0})$ are pseudo differential operators, while $ {\bf O}(\fd\pa_{x}^{-1})$ is no longer pseudo-differential.
%but it admits good  tame estimates in a sufficiently strong norm which allows to perform a KAM diagonalization scheme.
% We also underline that % the  coupling between  $u,\bar{u}$. 
% Since $\bar{L}\neq L$ the 
The off diagonal term of $\cL_4$ is no longer  pseudo-differential, 
having the  form $L A\bar L^{-1}$ where $A=O(\fd \pa_x^{-\rho})$ is a pseudo-differential operator, which is not a % This is not a 
conjugation. %  and hence the pseudo-differential structure of $A$ is not preserved.
However $L A\bar L^{-1}$ is a tame operator provided 
that $\rho$ is large enough.  
%with  norm  in a sufficiently strong norm 
% which allows to perform a KAM diagonalization scheme.

 Before applying a KAM scheme we need a further regularization step, which reduces to constant coefficients the diagonal  terms of order zero.
 This is the content of Proposition \ref{prop:ridord0}.

 Finally   the norm 
$  \bnorm{{\bf O}(\fd\pa_{x}^{-1})}_s $ of the remainder ${\bf O}(\fd \pa_x^{-1})$
 in \eqref{L4forma} 
 is  bounded 
 together with a sufficient number of commutators with $\pa_x, \pa_\vphi $.
  This is enough to implement the convergent KAM scheme of Section \ref{sec:kam}.
  In Section \ref{sec:measure} we prove Theorems \ref{main:thm} and \ref{thm:cantorset}.

\smallskip

As we said before the main difficulty  lies in the 
quantitative Egorov estimates for rectifying the quasi-periodic 
pseudo-differential operator
$ \omega \cdot \pa_\vphi 
- \ii \lambda (\vphi, x)  \tD_\tm $ which has two  characteristic directions.
%  not covered by the Egorov analysis in \cite{FGP19}.
Proposition  \ref{prop:ridord1}  is in no way a direct application of Theorem \ref{IncredibleConjugate} which deals with the leading term of $ \cL_3 $. 
% There are two subtle questions to be dealt with. 
The main difficulty is that, since in $\cL_3$ appears $\tD_{\tm}$ and not only $|\pa_x|$
we need to apply an Egorov theorem directly to $\tD_\tm = |\pa_x| + O(\pa_x^{-1}) $
 and only a posteriori recognize that the conjugated operator
 has the form $\cL_4$ in \eqref{L4forma}.

Classical pseudo-differential Egorov theory does not imply any tame control on conjugated symbols and remainders.
 The main outcome of our novel analysis is the estimate \eqref{estranged4} of 
the leading term $ O(\fd\pa_{x}^{0}) $  in \eqref{L4forma}. 
Rough estimates would give $ \| c^{(4)} \|_s \lesssim \max_{i=0,1,2} \| a_i \|_{ s + \mu (s) } $
and thus would allow to obtain an estimate, which unlike \eqref{stimemappa}, is not tame. 
The way such a strong tame quantitative estimate is obtained is  detailed 
in Section \ref{sec:egomichela}  (Theorem \ref{thm:egorov})
 and implemented % for other remainders  
 in  Section \ref{sec:ordine11}. 
In the water waves works \cite{BBHM,BM20,BFM1,BFM2} this problem 
is bypassed % was overcome 
implementing the
Egorov strategy  at the level of differential operators (which seems not possible here)
and in \cite{BKM} by expanding the  symbols  in homogeneous components.
The DP equation  \cite{FGP19}, although 
% even though it is a first-order system, shares some common features with the  Klein-Gordon. Indeed the leading term is of the form 
the dispersion relation $\partial_x + 3(1-\partial_{xx})^{-1}\partial_x $ is not 
differential,
has a very special structure which  allows to use Egorov techniques 
similar to those of the differential case.
The Egorov analysis in \cite{FGP19,FGww} would provide an estimate as
 $ \| c^{(4)} \|_s \lesssim 1 + \max_{i=0,1,2} \| a_i \|_{ s + \mu  } $, namely 
  ``tame but not small". 
%  not covered by the Egorov analysis in \cite{FGP19}.

In this work we have to apply and Egorov strategy on the first  order 
system \eqref{firstorderComplex} whose dispersion
 relation $ \sqrt{ - \pa_{xx} + \tm } $ is strongly pseudo-differential.
We do not see any mechanism to work directly 
on the  second order differential Klein-Gordon 
equation \eqref{NLW}. 

%Some parallel ideas for quantitative Egorov theory were developed in \cite{FGP19} and applied in \cite{FGww}. 
%For the Degasperis-Procesi equation  \cite{FGP19}
% even though it is a first-order system, shares some common features with the  Klein-Gordon. Indeed the leading term is of the form 
%whose dispersion relation is $\partial_x + 3(1-\partial_{xx})^{-1}\partial_x $,
%In fact the dispersion relation of the 
%i.e. a transport plus a pseudo-differential operator of order minus one. 
%the very special structure of this operator allows to use Egorov techniques 
%similar to those of the differential case.
% bypass most of the sharp  Egorov analysis implemented here. 

%\red{Some parallel ideas for quantitative Egorov theory were developed in \cite{FGP19} and applied in \cite{FGww}. 
%In fact the Degasperis-Procesi equation in \cite{FGP19}, even though it is a first-order system, shares some common features with the  Klein-Gordon. Indeed the leading term is of the form $\partial_x + 3(1-\partial_{xx})^{-1}\partial_x$, i.e. a transport plus a pseudo-differential operator of order minus one.  However, the very special structure of this operator allows to bypass most of the sharp  Egorov analysis implemented here.}

 \medskip

 \noindent 
 {\bf Notation:} 
 we denote by $ \N := \{ 1, 2, \ldots \} $  the natural numbers and  $\N_0:=\N\cup\{0\}$.
 We recall some important matrices  already introduced
\[
\Id:=\sm{1}{0}{0}{1}\,, \;E:=\sm{1}{0}{0}{-1}\,, \;\uno := \sm{1}{1}{1}{1}\,,\; S:=\sm{0}{1}{1}{0}\,,\; \fU:= \sm{1}{1}{1}{-1}\,.
\]
 We fix once and for all  the mass $ \mathtt{m} >0 $ and the constants
 \begin{equation}\label{costanti}
 	\nu\in \N\,,\quad     \gso > (\nu+5)/2\,,\quad \so :=\gso +1\,, %\quad  \su>  \so\,, 
	\quad\tau>2\nu+4\,.
	%\;\;\gamma\in(0,\tfrac{1}{2})\,.
 \end{equation}
Here  $\nu$ is number of frequencies, 
 the choice of $\gso $ ensures that  the Sobolev space $H^s$ has algebra and interpolation property for any $s\geq \gso$.
 We also introduce the parameters 
 \begin{equation}\label{costantiGAMMA}
 \su>  \so\,,\qquad \gamma\in(0,\tfrac{1}{2})\,.
 \end{equation}
 The parameters $\so, \su$ represent respectively \emph{low/high} Sobolev norms
 for which we obtain reducibility and tame estimates. 
The parameters $\tau$ and $\gamma$ denote respectively the Diophantine exponent and 
 the Diophantine constant, cfr. \eqref{calOinfty2}.

Along the paper we keep track of the dependence on parameters such as $s_1$, $s\in (s_*,s_1)$ and $\gamma$.
The notation $a \lesssim_{s, p, m} b$ means that $a \le C(s, p, m) b$ for some constant $C(s, p,m ) > 0$ 
depending on the Sobolev index $s$ and the constants $p, m $. 
We omit to write the dependence $\lesssim_{\gso, \nu, \tau,\mathtt{m}}$ 
with respect to $\gso, \nu, \tau, \mathtt{m}$, 
because these constants have been fixed in \eqref{costanti}. 

\smallskip

\noindent
\thanks{ {\em Acknowledgements.} 
We thank L. Biasco, P. G\'erard and
 S. Kuksin for useful comments.
The authors have been  supported by the  project 
PRIN 2020XBFL ``Hamiltonian and dispersive PDEs" and INdAM.
% of the Italian Ministry of Education and Research (MIUR). 
S.T. acknowledges the support of the project ERC STARTING GRANT 2021, `Hamiltonian Dynamics, Normal Forms and Water Waves" (HamDyWWa), Project Number: 101039762. 
Views and opinions expressed are however those of the authors only and do not necessarily reflect those of the European Union or the
European Research Council. Neither the European Union nor the granting authority
can be held responsible for them.   }  

\part{A quantitative Egorov result and  a straightening theorem for bi-charcteristics}
	%{A quantitative Egorov result and  a straightening theorem 
%	for a quasi-periodic 
%	  bi-characteristic operator}

\section{Functional Setting}\label{sec:funcset}

In this section we introduce the basic  function spaces and linear operators we use. 
\\[1mm]
\noindent
{\bf Function spaces.} Given $\nu\in \N$  we  consider periodic functions 
$ u : \T^\nu\times\T \to \C $, 
$ (\vphi,x) \mapsto u(\vphi, x) $, 
that we  Fourier expand  as
\begin{equation}\label{Fou1}
	\begin{aligned}
		& u(\vphi, x) = \sum_{j\in\Z} u_j(\vphi)e^{\ii jx} 
		= \sum_{\ell \in \Z^\nu, j\in\Z} u_{\ell, j}e^{\ii(\ell\cdot \vphi + jx)}\,, \\
		& u_{\ell, j} :=  u_{\ell, j} :=\frac{1}{{(2\pi)^{{\nu+1}}}}\int_{\T^{\nu+1}}
		u(\vphi,x)e^{-\ii (\ell\cdot\vphi+jx)}d\vphi dx  \, . 
	\end{aligned}
\end{equation}
For any  $s \geq 0 $ we define  the scale of Sobolev spaces
\begin{equation}\label{sobspace}
\begin{aligned}
H^{s}  :=H^{s}(\T^{\nu+1}) & :=H^{s}(\T^{\nu+1},\C) \\
& := 
\Big\{ u(\varphi, x)\in L^2(\T^{\nu+1},\C) : 
\|u\|_s^2 := \sum_{\ell\in \Z^\nu,j\in\Z} \jap{\ell, j}^{2s} | u_{\ell, j}|^2<\infty\Big\}
\end{aligned}
\end{equation}
where 
$\jap{\ell, j}:=\max\{1, |\ell|, |j|\}$ and $|\ell|:= \sum_{i=1}^\nu |\ell_i|$.
For $ s < 0 $ the Sobolev spaces $ H^s $ are defined by duality. 
Sobolev spaces $H^{s}(\T^{\nu+1},\C^2) $ are defined analogously. 
Given $u\in  H^s,$ define $\underline u\in  H^s$
with Fourier coefficients 
$\underline u_{\ell,j} := |u_{\ell,j}| $.  Note that
\begin{equation}\label{maju}
\|\underline u\|_s=\|u\|_s \, . 
\end{equation}
We consider also families of Sobolev functions 
$u:\cO\to H^{s}$, $ \omega \mapsto u(\omega) $, which are Lipschitz with respect to 
a parameter  
$\omega \in \cO \subset \R^\nu $ and,  
for $ \gamma \in (0,\tfrac{1}{2}) $ 
we introduce the weighted Lipschitz norm
\begin{equation}\label{sobolevpesata}
\|u\|_s^{\g, \cO}:= \|u\|_s^{\rm sup}+ \g \|u\|^{\rm lip}_{s-1} := \sup_{\omega\in\calO}\|u(\omega)\|_s+\gamma\sup_{\substack{\omega_1\neq\omega_2 \\
\omega_1, \omega_2 \in \cO}}
\frac{\|u(\omega_1)-u(\omega_2)\|_{s-1}}{|\omega_1-\omega_2|}\,.
\end{equation}
For a scalar function $ f : \calO\to \R $ independent of $ (\varphi, x)  $,  
the norms $ \|f \|_s^{\g, \cO},  \|f \|_s^{\rm sup},  \|f \|_s^{\rm lip}  $  is independent of $ s $ and 
we denote it simply 
$ |f|^{\g,\calO} := \|f \|_s^{\g, \cO} , |f|^{\rm sup} := \|f \|_s^{\rm sup}  , |f|^{\rm lip} := \|f \|_s^{\rm lip}   $. 

The norm \eqref{sobolevpesata} controls the 
Lipschitz variation of a function with a weaker norm that the function itself.
This is convenient since  the torus diffeomorphism 
$u\mapsto \cC_{\alpha}u:= u(\cdot+\al(\cdot))$ is a continuous operator w.r.t. this norm, see Lemma \ref{bastalapasta}.
	The norm \eqref{sobolevpesata} satisfies 
 the tame  and interpolation estimates for the product of functions.
It is 
 proved in Lemma 6.1 of \cite{BBM14}, or  \cite{BFM2}[Lemma A.1],  that 
  for any $s\ge \so$, $u,v\in H^s$ 
\begin{equation}\label{tameProduct}
\|u v\|_s^{\g, \cO}\lesssim_{s}\|u \|_s^{\g, \cO}\| v\|_{\so}^{\g, \cO}+
\|u \|_{\so}^{\g, \cO}\| v\|_{s}^{\g, \cO} \,. 
\end{equation}
For any $a_0,b_0\geq 0$, $p,q>0$ 
	and any $u\in H^{a_0+p+q}$, $v\in H^{b_0+p+q}$, one has (see e.g. Lemma 2.2 of \cite{BM20} or \cite{BFM2}[Lemma A.1])
	\begin{equation}\label{interpolotutto}
		\|u\|_{a_0+p}^{\g, \cO}\|v\|_{b_0+q}^{\g, \cO}\leq\|u\|_{a_0+p+q}^{\g, \cO}\|v\|_{b_0}^{\g, \cO}+\|u\|_{a_0}^{\g, \cO}\|v\|_{b_0+p+q}^{\g, \cO}\,.
	\end{equation}
	The same interpolation estimates holds for Sobolev functions with values in $\C^{d}$, $d\geq1$.
\\[1mm]
{\bf Linear operators.}
We identify a linear operator $\mathcal{A} $ in $  \mathcal{L}(L^{2}(\T^{\nu+1},\C))$
with its matrix 
\begin{equation}\label{opgeneral}
\mathcal{A} :=
\big(\mathcal{A}_{j,\ell}^{j',\ell'}\big)_{\substack{j,j'\in \Z \\ \ell,\ell'\in\Z^{\nu}}}\,,
\qquad
\mathcal{A}_{j,\ell}^{j',\ell'}:=
\frac{1}{{(2\pi)^{{\nu+1}}}}\int_{\T^{\nu+1}}\mathcal{A}[e^{\ii (\ell'\cdot\vphi+ j'x)}]
\cdot e^{-\ii (\ell\cdot\vphi + jx)}d\vphi dx\, , 
\end{equation}
in the exponential basis.  
In this paper we  deal with the subclass of T\"oplitz in time 
operators.

\begin{defn}{\bf (T\"oplitz in time operators).}\label{topos}
An operator $\mathcal{A}$ as in \eqref{opgeneral} 
is \emph{T\"oplitz in time} if its matrix entries have the form 
\begin{equation}\label{piazzadelsole}
\mathcal{A}_{(j,\ell)}^{(j',\ell')}=A_j^{j'}(\ell-\ell')\, . 
\end{equation}
We denote by 
 $\mathcal{L}^{\mathtt{T}}(H^s, H^{s'}) $ a linear operator  
bounded between $ H^s $ and $ H^{s'} $
with matrix entries \eqref{piazzadelsole}.
\end{defn}
In view of formul\ae\, \eqref{opgeneral} and \eqref{piazzadelsole}
we have a one-to-one correspondence 
between the T\"oplitz in time operators in 
$\mathcal{L}^{\mathtt{T}}(L^{2}(\T^{\nu+1},\C))$
and 
 $\varphi-$dependent families of operators 
\begin{equation}\label{topop1}
A :  \;\T^\nu \to \;\cL(L^2(\T, \C)) \, , \quad \vphi\;\;\mapsto\, A(\vphi) \, ,  
\end{equation}
that we regard acting as 
%$A\in \cL(L^2(\T^{\nu+1}, \C))$  (denote simply by $ A $ as well) defined by
\begin{equation}\label{linearoperator}
(Au)(\varphi, x)= (A(\varphi)u(\varphi, \cdot))(x) = 
\sum_{\ell\in\Z^\nu, j\in\Z} 
\Big(\sum_{\ell'\in\Z^\nu, j'\in\Z} A_j^{j'}(\ell-\ell') 
{u}_{\ell', j'}\Big) e^{\im(\ell\cdot\varphi+j x)}\, . 
\end{equation}
We shall identify the operator $A$ 
with the matrix $( A_j^{j'}(\ell))_{\ell\in\Z^\nu, j,j'\in\Z}$. 

We represent
a linear operator  $\mathcal{T}$ acting on 
$L^2(\T^{\nu+1}, \C)\times L^2(\T^{\nu+1}, \C)\simeq L^2(\T^{\nu+1}, \C^2) $
by the $2\times 2$ matrix of operators 
\begin{equation}\label{divano}
\cT:=(\cT_\sigma^{\sigma'})_{\sigma, \sigma'}:=\begin{pmatrix}
\cT_+^{+} && \cT_+^{-}\\
\cT_-^{+} && \cT_-^{-}
\end{pmatrix}\,,\;\;\;  \sigma, \sigma'\in\{\pm\} \,.
\end{equation}
Given two periodic functions
\begin{equation}\label{usigma}
u^{\s}(\vphi,x)=
\sum_{j\in\Z} u_{j}^{\s}(\vphi)e^{\s\ii jx} =
\sum_{\ell\in \Z^\nu, j\in\Z} u^\sigma_{\ell, j}e^{\ii(\ell\cdot \vphi + \sigma jx)} \, , 
% \in L^{2}(\T^{\nu+1},\C)\,,
\quad \s\in\{\pm\} \, , 
\end{equation}
the action of an operator $\cT$ of the form \eqref{divano} is 
\begin{equation}\label{matricepm}
\cT \begin{bmatrix}
u^+ \\ u^- 
\end{bmatrix} = \begin{pmatrix}
\cT_+^{+}u^+ + \cT_+^{-}u^-\\
\cT_-^{+}u^+ + \cT_-^{-}u^-
\end{pmatrix}\,.
\end{equation}
We denote by 
$\cL^{\mathtt{T}}(H^{s},H^{s'})\otimes \mathcal{M}_2(\C)$
 matrices of operators as in \eqref{divano}
whose entries are T\"oplitz in time  operators in 
$\cL^{\mathtt{T}}(H^{s},H^{s'})$.
In view of  \eqref{linearoperator}-\eqref{matricepm} 
we  identify an operator $\cT \in \cL^{\mathtt{T}}(H^{s},H^{s'})\otimes \mathcal{M}_2(\C) $ 
as in \eqref{divano}
with the matrix 
\begin{equation}\label{matrrep2}
(\cT_{\sigma, j}^{\sigma', j'}(\ell))_{\ell\in\Z^\nu, j,j'\in\Z, \sigma,\sigma'=\pm}
\end{equation} 
where
\begin{equation}\label{matrrep}
\cT_{\sigma, j}^{\sigma', j'}(\ell):=\frac{1}{(2\pi)^{\nu+1}}
\int_{\T^{\nu+1}}  \cT_{\sigma}^{\sigma'} 
[e^{\im \sigma' j'x} ]  
e^{- \im (\sigma jx + \ell\cdot\varphi)}  \, d\varphi dx \,.
\end{equation}

\begin{defn}
{\bf (Operatorial norm).} Let $s,s'\in \R$. We denote the operatorial  norm of 
a linear operator $A \in \cL^{\mathtt{T}}(H^{s},H^{s'})$ 
   by
\begin{equation}\label{operatorNorma}
\|A \|_{s,s'}:=
\sup_{\|u\|_{s}\leq1}\|A u\|_{s'}=
\sup_{\|u\|_{s}\leq1}\Big(
\sum_{\ell\in\Z^{\nu},j\in\Z}\langle \ell,j\rangle^{2s'}\Big|
\sum_{\ell'\in\Z^{\nu},j'\in\Z}A _{j}^{j'}(\ell-\ell')u_{\ell',j'}
\Big|^{2}
\Big)^{1/2} 
\end{equation}
and, for any 
$T\in \cL^{\mathtt{T}}(H^{s},H^{s'})\otimes \mathcal{M}_2(\C) $,  
we define $\|T\|_{s,s'}:=\max_{\s,\s'\in\{\pm\}}\{\|T_{\s}^{\s'}\|_{s,s'}\}$.
\end{defn}

We now introduce the notion of majorant operator. 

\begin{defn}{\bf (Majorant operator and norm).} \label{majOp}
Given a T\"opliz in time operator 
$ A = ( A_j^{j'}(\ell))_{\ell\in\Z^\nu, j,j'\in\Z} $ we associate its \emph{majorant} operator  
\begin{equation}\label{defmaj}
\und{A}:=\big(\und{A}_{j}^{j'}(\ell)\big)_{\ell\in\Z^{\nu},j,j'\in\Z}\,,
\qquad 
\und{A}_{j}^{j'}(\ell):=|A_{j}^{j'}(\ell)|\,,\qquad \forall\, \ell\in\Z^{\nu},j,j'\in\Z\, .  
\end{equation}
For any $ s, s' \in \R $  we define the \emph{majorant}  operatorial norm of $ A $, 
\begin{equation}\label{majass}
 |A|_{s,s'} :=\|\underline A \|_{s,s'} \, , 
\end{equation}
and  the space of T\"oplitz in time bounded majorant linear operators 
\begin{equation}\label{Mmajo}
\mathcal{M}^{\mathtt{T}}( H^s, H^{s'}):=
\big\{ A \in\mathcal L^{\mathtt{T}}( H^s, H^{s'})\ \  \text{ s.t.} \ \  
 |A|_{s,s'}  <\infty  \big\}\, . 
\end{equation}
Given a matrix of T\"oplitz in time operators 
$ T := (\cT_{\sigma, j}^{\sigma', j'}(\ell))_{\ell\in\Z^\nu, j,j'\in\Z, \sigma,\sigma'\in \{\pm\}}  $ as in \eqref{matrrep2} we define the \emph{majorant} matrix $\und{T} 
:= (\und{\cT}_{\sigma, j}^{\sigma', j'}(\ell))_{\ell\in\Z^\nu, j,j'\in\Z, \sigma,\sigma'\in \{\pm\}} $, taking the majorant of each component,
 the \emph{majorant} norm $ |T|_{s,s'} :=\|\underline T \|_{s,s'} $,
and we denote by $\mathcal{M}^{\mathtt{T}}( H^s, H^{s'})\otimes \mathcal{M}_2(\C)$ 
the subspace of T\"oplitz in time 
bounded majorant  linear operators in 
$\cL^{\mathtt{T}} ( H^s, H^{s'})\otimes \mathcal{M}_2(\C)$.
\end{defn}

The majorant operatorial norm $ | A |_{s,s'} $ is stronger than the operatorial norm, i.e.  
 $ \| A \|_{s,s'}\leq | A |_{s,s'} $. 

We introduce in  $\mathcal{M}^{\mathtt{T}}( H^s, H^{s'})$ 
the partial ordering relation 
\begin{equation} \label{rmk:propmajNorm}
\begin{aligned}
A \preceq B
\quad \Leftrightarrow \quad & 
\underline A \preceq \underline B \quad \Leftrightarrow\quad
 |A_{j}^{j'}(\ell)|\leq |B_{j}^{j'}(\ell)|\,,
 \quad 
 \forall\, \ell\in\Z^{\nu},j,j'\in\Z\, , \\
 \text{so that} \quad  & A \preceq B \quad \Longrightarrow \quad  
|A|_{s,s'}\leq |B|_{s,s'}\, . 
 \end{aligned}
\end{equation}
The spaces $\cL^{\mathtt{T}}( H^s, H^{s'})$ and 
$\mathcal{M}^{\mathtt{T}}( H^s, H^{s'})$
are  Banach algebras with respect to the composition of operators, namely
\begin{equation}\label{ABss}
\|AB\|_{s,s'}\leq \| A \|_{s_1,s'} 
\| B \|_{s,s_1} \, , \quad |AB|_{s,s'}\leq |A|_{s_1,s'} |B|_{s,s_1} \, , 
\end{equation}
where the second estimate  follows  by 
$ \underline {AB} \preceq  \underline A\ \underline B $.

We use the following notation
\begin{equation}\label{lip.variation}
		\Delta_{12}A :=  \frac{A(\omega_{1}) - A(\omega_{2})}{|\omega_{1}-\omega_{2}|}
		\end{equation}
		and note that $\Delta_{1 2}$ satisfies the Leibnitz rule $\Delta_{12} (A\circ B) =\Delta_{12}A \circ B(\oo_{2})+  A(\oo_{1})\circ \Delta_{12}B $.   

\begin{defn} {\bf (Weighted Lipschitz norms of operators)}\label{def:cuoreditenebra}
Given a  family of operators 
$ \oo\mapsto A(\oo)$ in $\mathcal{M}^{\mathtt{T}}( H^s, H^{s'})$, 
resp. in $\mathcal{L}^{\mathtt{T}}(H^{s},H^{s'}) $, 
Lipschitz in $ \omega \in \cO $,  
we define, for $\gamma \in (0,\tfrac{1}{2}) $,   
\begin{align}
& |A|_{s,s'}^{\gamma,\cO}
:= \sup_{\oo\in\cO}|A(\omega)|_{s,s'} 
+\gamma \sup_{\substack{\oo_1\neq \oo_2 \\
\omega_1, \omega_2 \in\cO}}  |\Delta_{12}A |_{s,s'}  \, , \label{A1sbss'} \\
& \| A \|_{s,s'}^{\gamma,\cO}
:= \sup_{\oo\in\cO}\| A(\omega)\|_{s,s'} 
+\gamma \sup_{\substack{\oo_1\neq \oo_2 \\
\omega_1, \omega_2 \in\cO}}
\|\Delta_{12}A \|_{s,s'}   \, . \label{A2ss'}
\end{align}
For a matrix of operators $T\in \cM^{\mathtt{T}}(H^{s},H^{s'})\otimes \mathcal{M}_2(\C)$ 
(resp. $\mathcal{L}^{\mathtt{T}}(H^{s},H^{s'})\otimes \mathcal{M}_2(\C)$)
we set 
$|T|^{\gamma,\cO}_{s,s'}
:=\max_{\s,\s'\in\{\pm\}}\{|T_{\s}^{\s'}|^{\gamma,\cO}_{s,s'}\}$ 
(resp. 
$\|T\|^{\gamma,\cO}_{s,s'}
:=\max_{\s,\s'\in\{\pm\}}\{\|T_{\s}^{\s'}\|^{\gamma,\cO}_{s,s'}\}$).
\end{defn}

From \eqref{ABss}  
we deduce
	that  also $\|\cdot \|^{\g,\cO}$ is  an algebra w.r.t products
and that 
\[
\| AB \|_{s,s'}^{\gamma,\cO} \le \|A\|_{s_1,s'}^{\gamma,\cO}\| B\|^{\gamma,\cO}_{s,s_1}\,,
\quad 
\| Au\|_{s'}^{\g,\cO} \le ( \| A \|_{s,s'}^{\gamma,\cO} +  \| A \|_{s-1,s'-1}^{\gamma,\cO})\|u\|^{\g,\cO}_s\,.
\]
The same holds for the majorant norm $|\cdot|^{\g, \cO}_{s, s'}$.
	
We need some further definitions.
\begin{defn} \label{defn:japphi} {\bf (Commutators and projections).}
Given a T\"oplitz in time linear operator 
$ A := (A_{j}^{j'}(\ell))_{j, j'\in\Z, \ell\in\Z^{\nu}} $ 
we  define the \emph{commutator} operators 
$ \td_{\vphi_h}A $, $ h=1,\ldots,\nu $,
and $\td_x A $ as the operators
whose matrix entries are
\begin{equation} \label{def:commuDD}
(\td_{\vphi_h}A)_{j}^{j'}(\ell) := \ii \ell_{h} A_{j}^{j'}(\ell) \, ,
 \qquad 
(\td_x A)_{j}^{j'}(\ell) :=\ii (j-j')A_{j}^{j'}(\ell) \, . 
\end{equation}
We also define the operator $ \langle\td_\vphi\rangle A $ 
 whose matrix entries are
$$
 (\langle\td_\vphi\rangle A)_{j}^{j'}(\ell):=\langle\ell\rangle A_{j}^{j'}(\ell) 
 \qquad \text{and} \qquad 
 \td_{\vphi}^p A := \prod_{h=1}^{\nu} \td_{\vphi_h}^{p_h} A \, , \quad
  \forall  p\in \N_0^{\nu} \, . 
$$
The same definitions extend to matrices of T\"oplitz in time linear operators componentwise.

Given a matrix of T\"oplitz in time linear operators 
$ T := ( T_{\s,j}^{\s',j'}(\ell))_{j, j'\in\Z, \ell\in\Z^{\nu}, \sigma, \sigma' = \pm} $,  
we define the \emph{projections} $\Pi_{N}, \Pi_{N}^{\perp}$, $ N \in \N $,  as
\begin{equation}\label{def:proj}
(\Pi_{N}T)_{\s,j}^{\s',k}(\ell):=
\left\{\begin{aligned}
& T_{\s,j}^{\s',k}(\ell)\;\;\; {\rm if}\;\; |\ell|\leq N
\\
&0\qquad \;\;\;\quad {\rm otherwise} \, , 
\end{aligned}\right. \qquad \Pi_{N}^{\perp}:=\id-\Pi_{N}\,.
\end{equation}
Given two T\"oplitz in time  operators $A$ and $B$ 
we define the \emph{adjoint action} of $A$ on $B$ as the commutator
\begin{equation}\label{adjdef}
{\rm ad}_{A}[B]:=\big[A,B\big]\equiv A\circ B-B\circ A\,.
\end{equation}
\end{defn}

Note that the operators  in  \eqref{def:commuDD} are the adjoint actions of $\partial_{\vphi_h}$ on $A$, namely
$\td_{\vphi_h} A=[\pa_{\vphi_h}, A]=\pa_{\vphi_h}A$ 
 and 
$\td_{x}A=[\pa_{x}, A]$. 
The operator $\td_{\vphi_h}$ satisfies the Leibniz rule
$ \td_{\vphi_h}(AB)=\td_{\vphi_h}(A)B+A\td_{\vphi_h}(B) $. 
In addition 
\begin{equation}\label{rmk:commuES}
\begin{aligned}
& \qquad \qquad \quad  \und{M}\preceq \langle\td_{\vphi_h}\rangle \und{M} \, , \quad
\und{\langle \td_\vphi\rangle M}\preceq  \und{M}
+ \sum_{h=1}^{\nu}\und{\td_{\vphi_h} M} \\
& \und{\langle \td_\vphi\rangle^\tb M}\preceq_{\tb,\nu}  
\und{M}+ \sup_{\substack{\beta\in \N_0^{\nu} \\
|\beta|\leq \tb}}
\und{\prod_{h=1}^{\nu}\td_{\vphi_h}^{\beta_h}M}
\preceq_{\tb,\nu}  \
\und{M}+\sum_{h=1}^{\nu} \und{\td_{\vphi_h}^{\tb}M} \, . 
\end{aligned}
\end{equation}
Finally note that $\Pi_{N}T, \Pi_{N}^{\perp}T \preceq T$.

\smallskip 

An operator
which is bounded, together with sufficiently many commutators, 
 in the operatorial norm \eqref{operatorNorma}
is also bounded in the majorant operatorial norm of Definition \ref{majOp}.

\begin{lemma}\label{maggio}
Let $ A (\omega )$ be a family of operators  in $ \mathcal{L}^\mathtt T( H^s, H^{s})$, $ s \geq 0 $ 
(see Def. \ref{topos}), Lipschitz in $ \omega \in \cO $, 
such that  
$ \td_{\vphi_h}^\beta A $, $  \td_{x}A $,
$ \td_x\td_{\vphi_h}^\beta A $, $1\leq h\leq \nu $,  where 
$\beta:=[\nu/2]+1$  belong to $ \mathcal{L}^\mathtt T( H^s, H^{s}) $.
Then 
$A(\omega) $ is a majorant bounded  operator in  $ \mathcal{M}^\mathtt T( H^s, H^{s})$
and
\begin{equation}\label{laverna}
\begin{aligned}
|A|^{\gamma,\mathcal{O}}_{s,s}
 &\lesssim_{s}
 \|A\|^{\gamma,\mathcal{O}}_{s,s}
+\|\td_{x}A\|^{\gamma,\mathcal{O}}_{s,s} + 
\max_{1\leq h\leq \nu} \{ \|\dhh^\beta A\|^{\gamma,\mathcal{O}}_{s,s},
\|\td_x\dhh^\beta A\|^{\gamma,\mathcal{O}}_{s,s} \} \,.
 \end{aligned}
\end{equation}
\end{lemma}

\begin{proof}
By \eqref{sobspace}, \eqref{linearoperator}  
and Cauchy-Schwarz inequality we have, since $ \beta > \nu / 2 $, 
\begin{align}
\| \underline{A}  u \|_s^2
&\le 
\sum_{\ell\in\mathbb{Z}^{\nu}, j \in \mathbb{Z}} \langle \ell, j \rangle^{2 s} \Big( \sum_{\ell'\in\mathbb{Z}^{\nu}, j' \in \mathbb{Z}}  \lvert A_j^{j'}(\ell-\ell') \rvert 
\lvert u_{\ell',j'} \rvert \Big)^2 \notag
\\
%&\le 
%\sum_{\ell\in\mathbb{Z}^{\nu}, j \in \mathbb{Z}} 
%\langle \ell, j \rangle^{2 s} 
%\Big( \sum_{\ell'\in\mathbb{Z}^{\nu}, j' \in \mathbb{Z}} 
%\frac{\langle \ell-\ell' \rangle^{\beta} \jap{j-j'}}{\langle \ell-\ell' \rangle^{\beta} \jap{j-j'}}  
%\lvert A_j^{j'}(\ell-\ell') \rvert  \lvert u_{\ell' j'} \rvert \Big)^2 \notag \\
&\lesssim 
\sum_{\ell\in\mathbb{Z}^{\nu}, j \in \mathbb{Z}} 
\langle \ell, j \rangle^{2 s} 
%( \sum_{\ell'\in\mathbb{Z}^{\nu} , j' \in \mathbb{Z}}  
\Big( \sum_{\ell'\in\mathbb{Z}^{\nu} , j' \in \mathbb{Z}}  
\langle \ell-\ell' \rangle^{2 \beta} \jap{j-j'}^2  \lvert A_j^{j'}(\ell-\ell')\rvert^2 
\lvert u_{\ell',j'} \rvert^2 \Big) \notag  \\
& =
\sum_{\substack{\ell'\in\mathbb{Z}^{\nu}, j' \in \mathbb{Z}}}
 \lvert u_{\ell',j'} \rvert^2 
\Big( \sum_{\substack{\ell\in\mathbb{Z}^{\nu},j \in \mathbb{Z}}}  
\langle \ell, j \rangle^{2 s}\, \, 
\langle \ell-\ell' \rangle^{2 \beta} \jap{j-j'}^2  \lvert A_j^{j'}(\ell-\ell')\vert^2\Big)  \label{causc}
\, . 
\end{align}
Now 
\begin{equation}\label{lancomm}
\langle \ell-\ell' \rangle^{2 \beta}  \jap{j-j'}^2 
\lesssim_\beta 1 + |j-j'|^2 + \sum_{h=1}^{\nu}  |\ell_h-\ell_h'|^{2\beta} +
|j-j'|^2 \sum_{h=1}^{\nu} |\ell_h-\ell_h'|^{2\beta}  \, . 
\end{equation}
For any linear operator $ A $, 
specializing 
$ \| A h \|_s^2 \leq \| A \|_{s,s}^2 \| h \|_s^2 $ for any 
$ h = e^{\ii (\ell' \cdot \varphi + j' x )}$, we get,  
recalling \eqref{sobspace}, 
\begin{equation}\label{normasobop}
\sum_{\ell\in\mathbb{Z}^{\nu} , j \in \mathbb{Z}} 
\langle \ell, j \rangle^{2 s} |A_j^{j'}(\ell-\ell')|^2
\le  \| A\|_{s,s}^2 \langle \ell',j'\rangle^{2s} \, . 
\end{equation}
Therefore, by \eqref{lancomm}, \eqref{normasobop} 
and recalling the definition of  $ \td_{x}A $  and $ \td_{\vphi_h}^\beta $in 
\eqref{def:commuDD} 
we get 
 \[
 \begin{aligned}
& \sum_{\ell\in\mathbb{Z}^{\nu} , j \in \mathbb{Z}} \langle \ell, j \rangle^{2 s}
\langle \ell-\ell' \rangle^{2 \beta} \jap{j-j'}^2  \lvert A_j^{j'}(\ell-\ell')\vert^2
\lesssim_\beta \\
& \big(  \| A\|_{s,s}^2 + 
\|  \td_{x}A \|_{s,s}^2 
+ \max_{h=1, \ldots,\nu} \{ \| \td_{\vphi_h}^\beta A\|^2, 
  \| \td_{x}\td_{\vphi_h}^\beta A\| ^2 \} \big) 
\langle \ell',j'\rangle^{2s} \, .
\end{aligned}
\]
Inserting the last bound in \eqref{causc} we get 
\[
\lVert \underline{A}  u \rVert_s^2
\lesssim_\beta  
(\|  A \|_{s,s}^2 + \|  \td_x A \|_{s,s}^2 +\max_{h=1, \ldots, \nu} 
\{ \|  \td_{\vphi_h}^\beta A \|_{s,s}^2,
\| \td_{\vphi_h}^{\beta}\td_x A \|_{s,s}^2 \} ) \|u\|_s^2 \, .
\]
The estimate for the Lipschitz variation follows as well recalling \eqref{A2ss'}.
This proves  \eqref{laverna}, recall the notation \eqref{A1sbss'}, \eqref{majass}. 
\end{proof}

\begin{rmk}\label{rem:varphi}
With similar arguments for $ b > \tfrac{\nu}{2} $ one has that 
$ \sup_{\varphi \in {\mathbb T}^\nu} \| \underline A(\varphi) \|_{{\cL}(H^{s}_x,H^{s'}_x)} 
\lesssim_b   |  \langle \td_{\vphi} \rangle^b A  |_{s,s'}  $, where $A(\vphi)$ is defined in \eqref{topop1}.
\end{rmk}

\noindent
{\bf Algebraic properties of operators.}\label{sec:revparity}
Let us define  
\begin{align}
\cU&:=\big\{
(u^{+},u^{-})\in L^{2}(\mathbb{T}^{\nu+1},\C^{2})\; :\; 
\overline{u^{-}}=u^{+} 
\big\}\,,
\label{real}
\\
{\bf X}&:= (X\times X) \cap \cU\,,
\quad 
X:=\{u\in L^{2}(\mathbb{T}^{\nu+1},\C) \;:\; 
u(\vphi,x)=-\overline{u(-\vphi,x)} \}\,,
\label{rev1}
\\
{\bf Y}&:= (Y\times Y) \cap \cU\,,
\quad 
Y:=\{u\in L^{2}(\mathbb{T}^{\nu+1},\C) \;:\; 
u(\vphi,x)=\overline{u(-\vphi,x)} \}\,,\label{rev2}
\\
%{\bf O}&:=O\times O\cap \cU\,,
%\quad 
O&:=\{u\in L^{2}(\mathbb{T}^{\nu+1},\C) \;:\; 
u(\vphi,x)=-u(\vphi,-x) \}\,, \label{odd1}
\\
%{\bf P}&:=P\times P\cap \cU\,,
%\quad 
P&:=\{u\in L^{2}(\mathbb{T}^{\nu+1},\C) \;:\; 
u(\vphi,x)=u(\vphi,-x) \}\,.\label{odd2}
\end{align}

\begin{defn}{\bf (Reversible, reversibility/parity preserving operators).}\label{giornatasolare}
Let $s,s'\in\R$.
\\[1mm] 
$(i)$ 
A linear operator 
$B = B(\varphi) \in \mathcal{L}^{\mathtt{T}}(H^{s},H^{s'})$ is 
\begin{itemize}
\item[$\bullet$] \emph{reversible} if and only if $B : X\to Y$ and $B : Y\to X$,
\item[$\bullet$]  \emph{reversibility preserving} if and only if $B: X\to X$ and $B : Y\to Y$,
\item[$\bullet$] \emph{parity preserving} if and only if $B : O\to O$ and $B : P\to P$.
\end{itemize}
\noindent
$(ii)$ A matrix of  % T\"oplitz in time 
linear operators   
\begin{equation}\label{bolena}
A :=\big( A_{\s}^{\s'}\big)_{\s,\s'\in \{\pm\}}:=
\left(\begin{matrix} A_{+}^{+} & A_{+}^{-}
\\
A_{-}^{+}  & A_{-}^{-}
\end{matrix}
\right)\in\mathcal{L}^{\mathtt{T}}(H^{s},H^{s'})\otimes \mathcal{M}_2(\C) 
\end{equation}
is
\begin{itemize}
\item[$\bullet$] \emph{real-to-real} if and only 
maps the real subspace $\mathcal{U}$ defined in \eqref{real} in itself,
\item[$\bullet$] \emph{reversible} if and only if  it is real-to-real and
$A : {\bf X}\to {\bf Y}$ and $A : {\bf Y}\to {\bf X}$,
\item[$\bullet$] \emph{reversibility preserving} if and only if  it is real-to-real and
$A : {\bf X}\to {\bf X}$ and $A : {\bf Y}\to {\bf Y}$,
\item[$\bullet$] \emph{parity preserving}
if and only if 
%it is real-to-real and $A : {\bf O}\to {\bf O}$ and $A : {\bf P}\to {\bf P}$.
$A : O\times O\to O\times O$ and $A : {P}\times P\to {P}\times P$.
\end{itemize}
\end{defn}

We also define the conjugate of a linear operator. 

\begin{defn}{\bf (Conjugate operator).}\label{def:conjoperator}
Let $B\in \mathcal{L}^{\mathtt{T}}(H^{s}, H^{s})$ for some $ s, s' \in \R $. % (recall Def. \ref{topos}).
We define the \emph{conjugate} operator $\overline{B}$ as
\begin{equation}\label{coniugato}
\overline{B}[h]:=\overline{B[\overline{h}]}\,.
\end{equation}
For a matrix of T\"oplitz in time linear operators    $A(\vphi)$ as in \eqref{bolena} we define 
$\overline{A}(\vphi):=(\overline{A_{\s}^{\s'}}(\vphi))_{\s,\s'\in\{\pm\}}$.
\end{defn}

An operator $B$ maps real valued functions to real valued functions
if and only if $B=\overline{B}$. 
Passing to the Fourier representation, the matrix entries of 
the conjugate operator $\overline{B}$ are 
\begin{equation}\label{coeffrealtoreal}
\big(\overline{B}\big)_{j}^{j'}(\ell)=\overline{B_{-j}^{-j'}(-\ell)}\,,
\quad 
\forall\, \ell\in\Z^{\nu}\,,\; j,j'\in\Z\,.
\end{equation}

The following lemma, which directly follows  from  Def.  
\ref{giornatasolare} 
and \eqref{real}-\eqref{odd1}, 
 provides alternative characterizations 
of these algebraic properties.

\begin{lemma}\label{equidefalgebra}
Let $ s, s' \in \R $. 

\noindent
$(i)$ A linear operator 
$ B $ in $  \mathcal{L}^{\mathtt{T}}(H^{s},H^{s'})$ is 
\begin{align}
{\rm reversible} \quad &\Leftrightarrow \quad  
B(-\vphi)=-\overline{B}(\vphi)\, , \quad \forall \vphi\in\T^{\nu} \,  , 
\label{coccodrillo1}
\\
{\rm  reversibility \; preserving}\quad &\Leftrightarrow \quad  
B(-\vphi)=\overline{B}(\vphi)\, , \quad \forall \vphi\in\T^{\nu} \,  , \label{coccodrillo2}\\
{\rm  parity \; preserving}\quad &\Leftrightarrow \quad  
B_{j}^{k}(\vphi)=B_{-j}^{-k}(\vphi)\,,\;\;\;
\;  \forall \vphi\in\T^{\nu} \, ,  j,k\in\Z\,.\label{coccodrillo3}
\end{align}
($ii$)
{ $ B $ is parity  preserving
if and only if $ \cP B(\varphi) = B(\varphi) \cP  $ where $ \cP $ is the involution
$ h(x) \mapsto h(-x) $. 
}

\noindent
$(iii)$
A matrix of operators 
\begin{equation}\label{AmatRP}
A = A(\vphi)=(A_{\s}^{\s'})_{\s,\s'\in\{\pm\}}
\in \mathcal{L}^{\mathtt{T}}(H^{s},H^{s'})\otimes \mathcal{M}_2(\C)
\end{equation}
is % as in Definition \ref{giornatasolare} is 
\begin{align}
{\rm real\,to\, real}\quad &\Leftrightarrow 
\quad  
\overline{A}(\vphi)S=S A(\vphi)\,, \qquad S =\sm{0}{1}{1}{0}\, , 
\quad \forall \vphi\in\T^{\nu} \, ,  
\label{bolena1}
\\
{\rm reversible}\quad &\Leftrightarrow 
\quad  
A(-\vphi)=- S A(\vphi)S\stackrel{\eqref{bolena1}}{=}-\overline{A}(\vphi) \, , 
\quad \forall \vphi\in\T^{\nu}  
\, ,  \label{bolena2}
\\
{\rm reversibility \; pres.}\quad &\Leftrightarrow 
\quad  
A(-\vphi)= S A(\vphi)S\stackrel{\eqref{bolena1}}{=}\overline{A}(\vphi)\,, \label{bolena3}
\\
{\rm parity \; pres.} \quad & \Leftrightarrow \quad  A_{\s,j}^{\s',k}(\vphi)=A_{\s,-j}^{\s',-k}(\vphi)\,,\;
 \; \forall \vphi\in\T^{\nu} \, , j,k\in\Z\,,\;\s,\s'\in\{\pm\}.\label{bolena4}
\end{align}
($iv$) 
The matrix of operators  $A $  in \eqref{AmatRP} is 
\begin{align}
& {\rm real}-{\rm to}-{\rm real} \; \Leftrightarrow \quad % \eqref{bolena1} \quad 
A_{\s,j}^{\s',k}(\ell)=\overline{A_{-\s,-j}^{-\s',-k}(-\ell)}\, , 
\label{bolena1bis}
\\
& {\rm reversible} \; \Leftrightarrow \quad
A_{\s,j}^{\s', k}(\ell)=- A_{-\s,j}^{-\s', k}(-\ell)\stackrel{\eqref{bolena1bis}}{=}- \overline{A_{\s,-j}^{\s',- k}(\ell)}\,,
\label{bolena2bis}
\\
& {\rm revers. pres.} \; \Leftrightarrow\quad 
A_{\s,j}^{\s', k}(\ell) =
A_{-\s,j}^{-\s', k}(-\ell)\stackrel{\eqref{bolena1bis}}{=} \overline{A_{\s,-j}^{\s',- k}(\ell)}\,.
\label{bolena3bis}
\end{align}
($v$) \label{rmk:revVSrevpres} 
The matrix of operators $A(\vphi)$ is real-to-real and reversibility preserving
 if and only if 
 $\ii EA(\vphi)$
is real-to-real and reversible.
\\[1mm]
$(vi)$ Given real-to-real, parity preserving matrix operators $A_{1}(\vphi),A_{2}(\vphi), A_3(\vphi)\in \mathcal{L}^{\mathtt{T}}(H^{s},H^{s'})\otimes \mathcal{M}_2(\C)$
such that $A_1(\vphi), A_2(\vphi)$ are reversibility preserving and $A_3(\vphi)$  is reversible, then
$A_1(\vphi)\circ A_2(\vphi)$ is real-to-real, parity preserving and reversibility preserving, 
while 
$A_1(\vphi)\circ A_3(\vphi)$ and $A_3(\vphi)\circ A_1(\vphi)$ are real-to-real, parity preserving and reversible.
% i.e. satisfies \eqref{bolena1} and  \eqref{bolena2}.
\end{lemma}

\subsection{$2\times 2$ block decomposition}\label{sec:matrici22}
We now reorganize the matrix 
$ 
\big(A_{\s,j}^{\s',k}(\ell)\big)_{\substack{\s,\s'\in \{\pm\},
\ell\in\Z^{\nu}, j,k\in\Z}}$ defined in \eqref{matrrep2},\eqref{matrrep} which 
represents a linear operator 
\begin{equation}\label{pioggia}
A=A(\vphi)=\big(A_{\s}^{\s'}(\vphi)\big)_{\s,\s'\in \{\pm\}}\in 
\mathcal{L}^{\mathtt{T}}(H^{s},H^{s'})\otimes \mathcal{M}_2(\C) 
\end{equation}
in Fourier basis in (at most) $2\times 2$ blocks 
as follows.
For any $\s,\s'\in \{\pm\}$, $\ell\in \Z^{\nu}$, 
we define the $p\times q$ matrices where $p,q=1,2$,
\begin{equation}\label{supermatrice}
\begin{aligned}
A_{\s, \,\vec{\jmath}}^{\s',\,\vec{k}}(\ell)&:=\left(
\begin{matrix}
A_{\s,j}^{\s',k}(\ell) & A_{\s,j}^{\s',-k}(\ell)
\vspace{0.2em}\\
A_{\s,-j}^{\s',k}(\ell) & A_{\s,-j}^{\s',-k}(\ell)
\end{matrix}
\right)\,,\qquad j,k\in\N\,,
\\
A_{\s, \,\vec{0}}^{\s',\,\vec{k}}(\ell)&:=\left(
\begin{matrix}
A_{\s,0}^{\s',k}(\ell) & A_{\s,0}^{\s',-k}(\ell)
\end{matrix}
\right)\,,\qquad j=0\,,\; k\in\N\,,
\\
A_{\s, \,\vec{\jmath}}^{\s',\,\vec{0}}(\ell)&:=\left(
\begin{matrix}
A_{\s,j}^{\s',0}(\ell) & A_{\s,-j}^{\s',0}(\ell)
\end{matrix}
\right)^{T}\,,\qquad j\in\N\,,\; k=0\,,
\\
A_{\s, \,\vec{0}}^{\s',\,\vec{0}}(\ell)&:=\left(
\begin{matrix}
A_{\s,0}^{\s',0}(\ell)
\end{matrix}
\right)\,,\qquad j=k=0\,.
\end{aligned}
\end{equation}
Using the notation \eqref{supermatrice}
the operator $ A $ in \eqref{pioggia} is thus identified with the matrices 
\begin{equation}\label{supermatrice2}
A(\vphi) = 
\big(A_{\s,j}^{\s',k}(\ell)\big)_{\substack{\s,\s'\in \{\pm\},\\
\ell\in\Z^{\nu}, j,k\in\Z}}
\stackrel{\eqref{supermatrice}}{=}
\Big(
A_{\s, \,\vec{\jmath}}^{\s',\,\vec{k}}(\ell)
\Big)_{\substack{\s,\s'\in \{\pm\} \\
\ell\in \Z^{\nu}, j,k\in\N_{0}}}\,.
\end{equation}
The following lemma is used to estimate the solution of the homological equation in Lemma \ref{lem:homoeq}.

\begin{lemma}\label{rmk:maggiorantevsSupblocchi}
Let $ A $ be an operator 
in $ \mathcal{L}^{\mathtt{T}}(H^{s},H^{s'})\otimes \mathcal{M}_2(\C)  $
as in \eqref{pioggia}.  
Then for all $\omega\in \calO$, the  majorant operator norm $|A|_{s,s'}$ 
(see Def. \ref{def:cuoreditenebra}) satisfies 
\begin{equation}\label{Ass'eq}
\tfrac{1}{4}|\widecheck{A}|_{s,s'}\leq |A|_{s,s'}\leq  |\widecheck{A}|_{s,s'}\,,
\end{equation}
where $ \widecheck{A} $ is the matrix as in \eqref{supermatrice2} with entries  (we denote 
$\|\cdot\|_{\infty}$ the sup norm of matrices) 
$$
\widecheck{A}_{\s,\vec{\jmath}}^{\s'\vec{k}}(\ell) :=\|{A}_{\s,\vec{\jmath}}^{\s'\vec{k}}(\ell)\|_{\infty}\sm{1}{1}{1}{1} \,,\;\;j,k \in \N \,, \quad 
\widecheck{A}_{\s,\vec{0}}^{\s'\vec{k}}(\ell) :=\|{A}_{\s,\vec{0}}^{\s'\vec{k}}(\ell)\|_{\infty} (1 \ 1) \, , \ 
j=0, k\in \N \, , 
$$
and similarly for the $2\times1$, $1\times1$ cases.  
\end{lemma}

\begin{proof}
The upper bound \eqref{Ass'eq} follows 
directly because $A \preceq \widecheck A $, see \eqref{rmk:propmajNorm}. 
Let us  prove the lower bound. We claim that 
\begin{equation}\label{scimmia}
|\widecheck{A}|_{s,s'}= \sup_{\substack{	\| u \|_s \leq 1 \\ u^\sigma_j= u^\sigma_{-j}\ge 0\,\forall j\in \N_0,
\sigma \in \{ \pm \}}} \| \und{\widecheck A} u \|_{s'}\leq  4  \sup_{\substack{	\| u \|_s \leq 1 \\ u^\sigma_j= u^\sigma_{-j}\ge 0\,\forall j\in \N_0,
\sigma \in \{ \pm \}}} \| \und{A} u \|_{s'}\leq 4|{A}|_{s,s'}\, . 
\end{equation}
The last inequality in \eqref{scimmia} directly follows recalling the definitions \eqref{majass} and 
 \eqref{operatorNorma}.
Let us prove the first equality in \eqref{scimmia}.
First note that $\|\und{\widecheck{A}}u\|_{s'}\leq \|\und{\widecheck{A}}\,\und{u}\|_{s'}$
and that $ \|\underline u\|_s=\|u\|_s $ by \eqref{maju}. 
Moreover, for any $ j \in \N_0 $, 
 for any $ u = \vect{u^+}{u^-} $ 
 (recall \eqref{usigma}) one has
\begin{align}
(\und{\widecheck{A}} \,\und{u})^\s_j + (\und{\widecheck{A}}\, \und{u})^\s_{-j}
&= 2  \sum_{j'\in \N,\s'\in \{\pm\}} \|A_{\s,\vec\jmath}^{\s',\vec\jmath'}\|_\infty (|u_{j'}^{\s'}|+|u_{-j'}^{\s'}|)
+
2  \sum_{\s'\in \{\pm\}} \|A_{\s,\vec\jmath}^{\s',\vec 0}\|_\infty |u_{0}^{\s'}| \notag 
\\&
= 4  \sum_{j'\in \N,\s'\in \{\pm\}} \|A_{\s,\vec\jmath}^{\s',\vec\jmath'}\|_\infty z_{j'}^{\s'}
+
2  \sum_{\s'\in \{\pm\}} \|A_{\s,\vec\jmath}^{\s',\vec 0}\|_\infty |u_{0}^{\s'}| \notag 
\\&
= (\und{\widecheck{A}} \, z)^\s_j + (\und{\widecheck{A}}\, z)^\s_{-j} \label{ultipas}
\end{align}
where $z_{j'}^{\s'}:=(|u_{j'}^{\s'}| +|u_{-j'}^{\s'}|)/2$ for any $ j' \in \N_0 $. 
By \eqref{ultipas} we deduce that 
$\|\und{\widecheck{A}}u\|_{s'}\leq \|\und{\widecheck{A}}\,\und{u}\|_{s'} 
\leq \|\und{\widecheck{A}} z\|_{s'}$. Furthermore $ \|z\|_{s} \leq \|u\|_{s} $
and $z_{j'}^{\s'}=z_{-j'}^{\s'}\geq 0$ for any $ j' \in \N_0$, $ \sigma' \in \{ \pm \} $. 
This implies   the first equality in \eqref{scimmia}.

Finally we prove the intermediate inequality in \eqref{scimmia}. For any 
$ u = \vect{u^+}{u^-} $   such that $ u_j^\sigma = u_{-j}^\sigma \ge 0$ for 
any  $j\in \N_0$, $\s\in\{\pm \}$ we have 
\begin{align*}
(\und{A} u)^{\s}_j + (\und{A} u)^{\s}_{-j} &
=   \sum_{j'\in \N, \s'\in \{\pm\}}( |A_{\s, j}^{\s' ,j'}|+  |A_{\s,- j}^{\s' ,j'}|) u_{j'}^{\s'}
+ (|A_{\s, j}^{\s' , - j'}|+  |A_{\s, -j}^{\s' ,-j'}|)  u_{-j'}^{\s'} 
\\&
+   \sum_{ \s'\in \{\pm\}}( |A_{\s, j}^{\s' ,0}|+  |A_{\s,- j}^{\s' ,0}|) u_{0}^{\s'}
\\
&= \frac12  
\sum_{j'\in \N,\s'\in\{\pm\}} \sum_{\eta,\eta'\in \{\pm\}}
|A_{\s,\eta j}^{\s' ,\eta 'j'}|(u_{j'}^{\s'}+
u_{-j'}^{\s'})
+   \sum_{ \s'\in \{\pm\}}( |A_{\s, j}^{\s' ,0}|+  |A_{\s,- j}^{\s' ,0}|) u_{0}^{\s'}
\\ & \ge  \frac12  \sum_{j'\in \N_0,\s'\in\{\pm\}} \|A_{\s,\vec\jmath}^{\s',\vec\jmath'}\|_\infty (u_{j'}^{\s'}
+ u_{-j'}^{\s'})= \frac14\Big( (\und{\widecheck{A}} u)^{\s}_j 
+ (\und{\widecheck{A}} u)^{\s}_{-j}\Big)\,,
\end{align*}
and thus $\|\und{A} u\|_{s'} \ge \frac14 \|\und{\widecheck A} u\|_{s'}$.
This proves \eqref{scimmia}. 
\end{proof}

In view of Lemma \ref{equidefalgebra}-($iv$) and the notation \eqref{supermatrice}
the following lemma holds.

\begin{rmk}\label{pioggia2matrici}
A matrix of operators  $A(\vphi) \in \mathcal{L}^{\mathtt{T}}(H^{s},H^{s'})\otimes \mathcal{M}_2(\C) $
as in \eqref{AmatRP} is\footnote{Here the $2\times2$-matrix $S=\sm{0}{1}{1}{0}$ acts on the 
$ 2 \times 2 $ block
$A_{\s,\vec{\jmath}}^{\s',\vec{k}}({\ell})$, for fixed $\s,\s'\in\{\pm\}$ differently from
\eqref{bolena1}-\eqref{bolena3} where $S$ acts on blocks $(A_{\s}^{\s'})_{\s,\s'\in\{\pm\}}$.
} if $j,k\ne 0$
\begin{align}
{\rm real}-{\rm to}-{\rm real} \; \Leftrightarrow\quad\eqref{bolena1}\quad &\Leftrightarrow\quad
A_{\s,\vec{\jmath}}^{\s', \vec{k}}(\ell)= S\overline{A_{-\s,\vec{\jmath}}^{-\s',\vec{k}}(-\ell)}S\, , 
\label{bolena1tris}
\\
{\rm reversible} \; \Leftrightarrow\quad\eqref{bolena2}\quad &\Leftrightarrow\quad
A_{\s,\vec{\jmath}}^{\s', \vec{k}}(\ell)=-{A_{-\s,\vec{\jmath}}^{-\s',\vec{k}}(-\ell)}
\stackrel{\eqref{bolena1tris}}{=} -S\overline{A_{\s,\vec{\jmath}}^{\s',\vec{k}}(\ell)}S \, , 
\label{bolena2tris}
\\
{\rm revers. \, pres.} \; \Leftrightarrow\quad
\eqref{bolena3}\quad &\Leftrightarrow\quad
A_{\s,\vec{\jmath}}^{\s', \vec{k}}(\ell)={A_{-\s,\vec{\jmath}}^{-\s',\vec{k}}(-\ell)}
\stackrel{\eqref{bolena1tris}}{=} S\overline{A_{\s,\vec{\jmath}}^{\s',\vec{k}}(\ell)}S 	\, , 
\label{bolena3tris}
\\
{\rm parity \ pres.} \; \Leftrightarrow\quad\eqref{bolena4} \quad &\Leftrightarrow\quad
A_{\s,\vec{\jmath}}^{\s', \vec{k}}(\ell)=SA_{\s,\vec{\jmath}}^{\s', \vec{k}}(\ell)S\,.\label{bolena4tris}
\end{align}
Note  that \eqref{bolena4tris} together with  \eqref{bolena2tris} (resp.\eqref{bolena3tris}) implies that 
$A_{\s,\vec{\jmath}}^{\s', \vec{k}}(\ell)$ has purely immaginary (resp. real) valued entries.
\end{rmk}
\begin{lemma}\label{algebra2}
A  parity preserving matrix $A=A(\vphi)$  in the block form \eqref{supermatrice} 
satisfies 
	\begin{equation}	\label{algebra}
	\begin{aligned}
	& A_{\s,\vec{\jmath}}^{\s', \vec{k}}(\ell)=  A_{\s,j}^{\s', k}(\ell)\Id + A_{\s,j}^{\s', -k}(\ell) S\,, \quad  j,k \in \N \, , \\
		& A_{\s,\vec{0}}^{\s', \vec{k}}(\ell) =  A_{\s,0}^{\s', k}(\ell)(1 \ 1) \,,\quad k\ne 0\,,\quad A^{\s',\vec{0}}_{\s, \vec{\jmath}}(\ell) =  A^{\s',0}_{\s, j}(\ell)(1 \ 1)^T \,,\quad j\ne 0 \, . 
			\end{aligned}
	\end{equation}
These matrices are simultaneously diagonalizable 
by conjugation with the matrix $\mathfrak{U}:=\sm{1}{1}{1}{-1}$, namely 
\begin{equation}\label{algebratilde}
\begin{aligned}
\widetilde{A}_{\s,\vec{\jmath}}^{\s', \vec{k}}(\ell)
&:=\mathfrak{U}^{-1}A_{\s,\vec{\jmath}}^{\s', \vec{k}}(\ell)\mathfrak{U}
\\&=\left(
\begin{matrix}
	A_{\s,j}^{\s',k}(\ell)+A_{\s,j}^{\s',-k}(\ell) & 0 \\ 0 & 	A_{\s,j}^{\s',k}(\ell)-A_{\s,j}^{\s',-k}(\ell) 
\end{matrix}
\right)= A_{\s,j}^{\s',k}(\ell) \Id + A_{\s,j}^{\s',-k}(\ell) E \,,
\end{aligned}
\end{equation}
for any $\s,\s'\in \{\pm\}$, $j,k\in \N $ while 
%namely 
\begin{equation}\label{contilde}
\begin{aligned}
%&	\widetilde{A}_{\s,\vec{\jmath}}^{\s', \vec{k}}(\ell):=\mathfrak{U}^{-1}A_{\s,\vec{\jmath}}^{\s', \vec{k}}(\ell)\mathfrak{U}\,, \ \  \forall \;\s,\s'\in \{\pm\}, \; j,k\in \N \, , \ \text{are } \ 2 \times 2 \
%\text{diagonal matrices}  \ 
%\text{and}  
%	\\
&\widetilde{A}_{\s,\vec{\jmath}}^{\s',\vec{0}}(\ell)
:=
\mathfrak{U}^{-1}A_{\s,\vec{\jmath}}^{\s',\vec{0}}(\ell)\,, \quad \widetilde{A}_{\s,\vec{0}}^{\s',\vec{k}}(\ell)
:= {A}_{\s,\vec{0}}^{\s',\vec{k}}(\ell) \fU \, , \ \text{ \ are proportional to} \ (1 \ 0) \, .  
\end{aligned}
\end{equation}
We also  define the scalar 
$  \widetilde{A}_{\s,\vec{0}}^{\s',\vec{0}}(\ell) := {A}_{\s,\vec{0}}^{\s',\vec{0}}(\ell) $. 
%This is  
%not surprising since by definition $A$ preserves even/odd functions in $x$ and the matrix $\fU$ goes from the exponential basis $\{e^{\ii j x}\}_{j\in \Z}$ to the  even/odd one $\{1, \cos(jx),\sin(jx)\}_{j\in \N}$.
%Note that for  $j,k\in \N$ 
%\begin{equation}\label{algebratilde}
%\mathfrak{U}^{-1}\,A_{\s,\vec{\jmath}}^{\s',\vec{k}}(\ell)\,\mathfrak{U}=\left(
%\begin{matrix}
%	A_{\s,j}^{\s',k}(\ell)+A_{\s,j}^{\s',-k}(\ell) & 0 \\ 0 & 	A_{\s,j}^{\s',k}(\ell)-A_{\s,j}^{\s',-k}(\ell) 
%\end{matrix}
%\right)= A_{\s,j}^{\s',k}(\ell) \Id + A_{\s,j}^{\s',-k}(\ell) E \,.
%\end{equation}
Moreover 
	\begin{align}
		%A(\vphi) \; {\rm satisfies} 
		%\eqref{bolena4tris}+	\eqref{bolena1tris}\quad 
		A \;\;{\rm is\, real-to-real}\;\;
		&\Leftrightarrow\quad
		\widetilde{A}_{\s,\eta j}^{\s', \eta k}(\ell)= \overline{\widetilde{A}_{-\s,\eta j}^{-\s',\eta k}(-\ell)}\,;
		\label{bolena1quatuor}
		\\
		%\eqref{bolena4tris}+	\eqref{bolena2tris}\quad 
		A \;\;{\rm is\, reversible}\;\;
		&\Leftrightarrow\quad
		\widetilde{A}_{\s,\eta j}^{\s', \eta k}(\ell)= -\overline{\widetilde{A}_{\s,\eta j}^{\s',\eta k}(\ell)}
		\stackrel{\eqref{bolena1quatuor}}{=}
		-\widetilde{A}_{-\s,\eta j}^{-\s',\eta k}(-\ell)
		%- A_{-\s,j}^{-\s', k}(-\ell)\,;
		\label{bolena2quatuor}
		\\
		%	\eqref{bolena4tris}+	\eqref{bolena3tris}\quad 
			A \;\;{\rm is\, reversibility \;preserving}\;\;
			&\Leftrightarrow\quad
		\widetilde{A}_{\s,\eta j}^{\s', \eta k}(\ell)= \overline{\widetilde{A}_{\s,\eta j}^{\s',\eta k}(\ell)}
		\stackrel{\eqref{bolena1quatuor}}{=}
		\widetilde{A}_{-\s,\eta j}^{-\s',\eta k}(-\ell)\,.
		%- A_{-\s,j}^{-\s', k}(-\ell)\,;
		\label{bolena3quatuor}
	\end{align}
\end{lemma}

\begin{proof}
Formula \eqref{algebra} follows by \eqref{bolena4} and \eqref{supermatrice},
while \eqref{algebratilde} follows noting 
noting that 
\[
\mathfrak{U}^{-1}S\,\mathfrak{U}= \sm{1}{0}{0}{-1}=: E\,,\quad (1 \ 1) \fU = (2 \ 0)\,,\quad \fU^{-1}(1 \ 1)^T = (1 \ 0)^T  \, .
\]
Conditions \eqref{bolena1quatuor}-\eqref{bolena3quatuor}
follow by  Lemma \ref{equidefalgebra}.
\end{proof}
\smallskip
\noindent
{\bf Normal forms.}
The following matrices of operators will enter in the KAM reducibility scheme. 

\begin{defn}{\bf (Normal form).}\label{def:normalform}
Given a matrix of T\"opliz in time linear operators $A$ as in \eqref{supermatrice2} % \eqref{pioggia}
we define the operator $[A]$  as 
\[
[A]_{\s,\vec{\jmath}}^{\s',\vec{k}}(\ell):=
\left\{
\begin{aligned}
&A_{\s,\vec{\jmath}}^{\s,\vec{\jmath}}\,(0)\,,\qquad \ell=0\,,\; k= j\,,\; \s=\s'\,,
\\
&0\qquad \qquad \quad {\rm otherwise}\,.
\end{aligned}\right.
\]
We say that $A$ is in \emph{normal form} if and only if $A=[A]$.
\end{defn}

By  Remark \ref{pioggia2matrici} we deduce the following. 
\begin{lemma}\label{rmk:algebra2}
The normal form $ [A] $ of a 
real-to-real, reversibility
and parity preserving  matrix $A$ as in \eqref{supermatrice2} 
%it is  according to Def. \ref{giornatasolare}.
has the form 
\begin{equation}\label{pioggiamilano1}
[A]:=\left(\begin{matrix} 
[A]_{+}^{+}(0) & 0
\\ 0 & [A]_{-}^{-}(0) 
\end{matrix}\right)\,,
\qquad [A]_{-}^{-}(0)= [A]_{+}^{+}(0)={\rm diag}_{j\in\N_{0}}A_{+,\vec{\jmath}}^{+,\vec{\jmath}}(0)\,,
\end{equation}
where (recall the matrix $S$ in \eqref{bolena1}) 
\begin{equation}\label{pioggiamilano2}
\begin{aligned}
&A_{+,\vec{\jmath}}^{+,\vec{\jmath}}\,(0)
\stackrel{\eqref{bolena4}}{=}
\left(
\begin{matrix}
A_{+,j}^{+,j}(0) & A_{+,j}^{+,-j}(0)
\vspace{0.2em}\\
A_{+,j}^{+,-j}(0) & A_{+,j}^{+,j}(0)
\end{matrix}
\right)
=A_{+,j}^{+,j}(0)\Id+A_{+,j}^{+,-j}(0)S
\,,\;\; j\in\N\,,
\\
&A_{+,\vec{0}}^{+,\vec{0}}\,(0) =\Big(A_{+,0}^{+,0}\,(0)\Big)\,, \qquad \qquad \qquad
A_{+,j}^{+,j}(0) \, , \ A_{+,j}^{+,-j}(0) \in \R ,   \ \forall \, j \in \N_0  \, . 
\end{aligned}
\end{equation}
% and all the coefficients $A_{+,j}^{+,j}(0)$, $A_{+,j}^{+,-j}(0)$, $ j \in \N_0 $, are real valued.
\end{lemma}

The set of real-to-real, reversibility and parity preserving matrices 
 in normal form  
is a commutative algebra.

\section{Pseudo-differential operators}

We introduce pseudo-differential operators
acting on periodic functions 
following \cite{BM20}, \cite{FGP1}.
The main  novelties are in  Lemma
\ref{fearofthedark} where we  provide a sharp estimate for  the 
symbol of the inverse of pseudo-differential operators. These estimates play an important role in Section 8.
% see for instance the sharp bound \eqref{simboloCC33} in Proposition \ref{blockTotale}. 

\begin{defn}\label{simbolo}
{\bf (Symbols and  Pseudo-differential operators)}
Let $m\in\R$. 
We say that a $ C^\infty $ function 
$ a : \mathbb T \times \R \to \C $, $ (x, \xi) \mapsto a(x, \xi )$,  is a symbol of order $\le m$ 
if, for any $ \al, \beta\in \N_0 $, 
there exists a constant $C_{\al,\beta}$ such that
\begin{equation}\label{symbolpseudo}
|\pa_x^\alpha \partial_\xi^\beta a( x, \xi)|\le C_{\al, \beta} \jap{\xi}^{m-\beta} \, , \quad 
\forall (x, \xi) \in \mathbb T \times \R \, .  
\end{equation}
We denote $\Gamma^m$ the class of symbols $a( x, \xi)$ of order $ m $
and  $\Gamma^{-\infty}:=\cap_{m\in \R}\Gamma^{m}$.

For any  $ u (x) =\sum_{j\in\Z} u_je^{\im j x}$ we  define
the associated pseudo-differential operator 
\begin{equation}\label{actionapseudo}
\op (a(x,\xi))u:= a( x, D )u :=\sum_{j\in\Z} a( x, j) u_j e^{\im j x} \, . 
\end{equation}
If the symbol $ a = a( \xi) $ is independent of $ x $ we say that $ a(D)$ is
a Fourier multiplier operator. 
\end{defn}

For any $m\in\R\setminus\{0\}$ we set 
\begin{equation}\label{notazioneNota}
|D|^m:= \op(\chi(\xi)|\xi|^m)
\end{equation}
where $\chi\in C^\infty (\R, \R)$ is an even 
and positive cut-off function such that
\begin{equation}\label{cutoff}
\chi(\xi)=\begin{cases}
0 \quad {\rm if} \quad |\xi|\le \frac{1}{2}\\
1 \quad {\rm if} \quad |\xi|\ge \frac{2}{3}
\end{cases}\,, \quad \partial_\xi\chi(\xi)>0 
\quad \forall\, \xi\in\Big(\frac{1}{2},\frac{2}{3}\Big)\, .
\end{equation}
We also define the Fourier multiplier operator 
\begin{equation}\label{langlerangle}
\langle D\rangle :=\op(\langle\x\rangle)\,,\quad \langle \x\rangle:=\sqrt{|\x|^{2}+1}\,,\;\;\;\x\in\R\,.
\end{equation}
Given a symbol $a\in \Gamma^m$ the conjugate of the pseudo-differential operator
$ \op(a(x,\x)) $ according to  Definition \ref{def:conjoperator}) is 
\begin{equation}\label{conjsimb}
\overline{\op(a(x,\x))} = \op(\overline{a(x,-\x)}) \, . 
\end{equation}
Along the paper we consider 
families of symbols depending on parameters $\vphi\in\T^\nu$  and on $\omega\in \R^\nu$. More explicitly we consider maps
$(\om,\vphi)\mapsto  a(\om;\vphi,x,\x)\in \Gamma^{m} $ which are 
$ C^{\infty} $ in  $ \varphi $ and Lipschitz  in $\om\in \cO$ for some compact set $\cO\subset \R^\nu$. For simplicity we  often do not write explicitly the dependence on $\om$.
We shall call $a$ simply a symbol, 
 and consider the associated pseudo-differential operator $ \op(a ) $
acting on a function
 $u(\vphi,x)\in H^s(\T^{\nu+1}, \mathbb{C})$   in the natural way as
 $ \op(a (\varphi) )  u(\varphi, \cdot ) $.  It is  convenient to consider the 
 Fourier coefficients of the symbols $a(\vphi,x, \xi)$ both in $ \varphi $ and 
 $x\in \T$:
 \begin{equation}\label{aFourierphix}
 a(\vphi,x,\xi)= \sum_{j \in \Z} \widehat a (\vphi, j ,\xi) e^{\im j x} = \sum_{\ell\in \Z^\nu,j \in \Z} \doublehat a (\ell, j,\xi) e^{\im (\ell\cdot \vphi +j x )} \, . 
 \end{equation}
In view of \eqref{actionapseudo}  the matrix which represents  % the pseudo-differential operator 
 the action of $ A = \op(a) $ in the exponential basis is 
\begin{equation}\label{actionpseudo}
 A_{j}^{k}(\vphi):=\widehat{a}(\vphi,j-k,k) \, . 
 \end{equation}
 The following norm 
controls the regularity in $( \varphi, x)\in \T^{\nu+1}$, the Lipschitz variation in $\om$ 
and the decay in $\xi$ 
of a symbol $a(\omega; \varphi, x, \xi)$, 
together with its derivatives 
$\partial_\xi^\beta a $ for any $ 0\le \beta \le p$, 
in the Sobolev norm $\|\cdot\|_s$.

\begin{defn}\label{norma pesata} {\bf (Symbols and its norms)}
We denote by $ S^m $ the set of symbols 
$ a(\omega; \varphi, x, \xi ) \in \Gamma^m $ which are Lipschitz in $ \omega \in \cO $ and,  
for  any $ p \in \N_0 $, $ s\ge {\so}$ , 
we define the weighted norm, for $\gamma\in (0,\tfrac{1}{2}) $,  
\begin{equation}\label{normaSymboloGamma}
\|a\|_{m, s, p}^{\gamma, \cO} := 
\sup_{{\oo\in\cO}} \|a(\oo)\|_{m, s, p} 
+ \g \sup_{\substack{\oo_1\neq\oo_2\\ 
\omega_1, \omega_2 \in\cO}} \frac{\|a(\oo_1; \vphi, x, \xi) 
-a(\oo_2; \vphi, x, \xi)\|_{m, s-1, p} }{|\oo_1-\oo_2|} 
\end{equation}
where 
\begin{equation}\label{normaSymbolo}
\|a(\omega; \cdot) \|_{m,s,p} := 
\max_{0\le \beta\le p} \sup_{\xi\in\R} 
\|\partial_\xi^\beta a(\omega;  \cdot, \cdot, \xi) \|_s \jap{\xi}^{-m+\beta} \, . 
\end{equation}
We denote $ S^{-\infty} := \cap_{m <  0} S^m $. 
\end{defn}
The norm 
$\|\cdot\|^{\gamma,\cO}_{m, s, p}$ in \eqref{normaSymboloGamma}
is \emph{non-decreasing} in the indexes $s,p$ and \emph{non-increasing} in  $m$: 
\begin{equation}\label{zanzara}
\begin{aligned}
\forall\, m\le m' \quad &\Rightarrow \quad 
\|\cdot\|^{\gamma,\cO}_{m',s,p}\le \|\cdot\|^{\gamma,\cO}_{m, s, p}\,,
\\
\forall s\le s',\; p\leq p' \quad  &\Rightarrow 
\quad
\|\cdot\|^{\gamma,\cO}_{m, s, p}\le \|\cdot\|^{\gamma,\cO}_{m,s',p}\,,\;\;\; 
\|\cdot\|^{\gamma,\cO}_{m, s, p}\le \|\cdot\|^{\gamma,\cO}_{m,s,p'}\,.
\end{aligned}
\end{equation}
Note that the weighted Sobolev norm 
of a symbol $a (\omega; \varphi, x ) $  independent of $\xi $ is equal to 
$ \|a\|_{0,s,p}^{\g, \cO}=\|a\|_s^{\g, \cO} $ 
the weighted Sobolev norm of a function defined   in \eqref{sobolevpesata}. 

\smallskip

The norm of a $2\times 2$ matrix of symbols in  $S^m $ 
\begin{equation}\label{matricidisimboli}
A:=A(\omega; \varphi, x, \xi):= \sm{a}{b}{c}{d}
\in S^m\otimes \cM_2(\C)
\end{equation} 
is 
\begin{equation}\label{simbolimatrici}
	\|A\|_{m,s, p}^{\gamma, \cO}:=\max\{ \|f\|_{m, s, p}^{\gamma, \cO}, f=a,b,c,d\} \, . 
\end{equation}
The associated matrix of pseudo-differential operators is 
\begin{equation}\label{matrixsy}
\op\big(A\big):=\left(\begin{matrix}
\op(a) & \op(b)
 \\
 \op(c) & \op(d)
 \end{matrix}\right)\,.
\end{equation}

\begin{lemma}{\bf (Action on Sobolev spaces.)}\label{sobaction}
Let  $m\ge0$ and $a\in S^m$. 
Then for any $s\geq \so$
one has $\op(a) \in \cM^{\mathtt T}(H^{s+m},H^s)$ 
and, for any $u\in H^{s+m}(\T^{\nu+1},\C)$, 
\[
\|\und{\op(a)}u\|^{\g, \cO}_s\le 
C(\so)\|a\|^{\g, \cO}_{m, \so, 0}\|u\|^{\gamma,\cO}_{s+m}
+ C(s) \|a\|^{\g, \cO}_{m, s, 0} \|u\|^{\gamma,\cO}_{\so+m} \, . 
\]
\end{lemma}
\begin{proof}
	The proof is like  the one in  \cite[Lemma 2.21]{BM20}, which is given for $m=0$.
\end{proof}

\noindent {\bf Composition rules for symbols.}\
Given symbols $a\in S^m$, $b\in S^{m'}$ 
we define 
\begin{equation}\label{mareostia}
	\begin{aligned}
	a\#b (\oo; \vphi, x, \xi)&:=\
	\sum_{j\in\Z} a(\oo; \varphi, x, \xi+j)\widehat{b}(\oo; \varphi, j, \xi) e^{\im j x}\\ &= \sum_{j, j'\in\Z} 
	\widehat{a}( \oo; \varphi, j'-j, \xi+j)
	\widehat{b}(\oo; \varphi,  j, \xi)e^{\im  j' x}\,,
	\end{aligned}
\end{equation}
so that $	\op(a\#b) = \op(a)\circ\op(b) $. 
The symbol $a\# b$ has the following asymptotic expansion:
for any $N\geq1$  (for simplicity of notation we drop the dependence on $\oo,\vphi$)
\begin{equation}\label{composizione troncata}
(a\# b)(x,\x)=\sum_{n=0}^{N-1}\frac{1}{n! \mathrm{i}^{n}}\pa_{\x}^{n}a(x,\x)\pa_{x}^{n}b(x,\x)
+r_{N}(x,\x)\,,
\end{equation}
where the symbol $r_{N}\in S^{m+m'-N}$ has the form
\begin{equation}\label{remainder composizione}
r_{N}(x,\x) := \frac{1}{(N-1)!{\mathrm{i}}^{N}}\int_{0}^{1}(1-\tau)^{N}\sum_{j\in\mathbb{Z}}
(\pa_{\x}^{N}a)(x,\x+\tau j)\widehat{\pa_{x}^{N}b}(j,\x)e^{\mathrm{i} jx}d \tau\,.
\end{equation}
We also define, for $0\leq n\leq N-1$,
\begin{equation}\label{cancellittiEspliciti}
\begin{aligned}
a\#_{n} b&:=\frac{1}{n! {\rm i}^{n}} (\partial_{\xi}^n a)(\partial_x^n b)\in S^{m+m'-n}\,, 
\\ a\#_{< N} b&:=\sum_{n=0}^{N-1} a\#_n b\in S^{m+m'}\,, \qquad a\#_{\geq N} b:=r_N:=r_{N, ab}\in S^{m+m'-N}\,.
\end{aligned}
\end{equation}
We define the composition $\#$ of matrices of symbols so that  $\op(A\#B)=\op(A)\circ \op(B)$,  as
\begin{equation}\label{resinprog}
\mat{A_+^+}{A_+^-}{A_-^+}{A_-^-}\#\mat{B_+^+}{B_+^-}{B_-^+}{B_-^-} 
= 
\mat{A_+^+\#B_+^++A_+^-\#B_-^+}{A_+^+\#B_+^-+A_+^-\#B_-^-}{A_-^+\#B_+^++A_-^-\#B_-^+}{A_-^+\#B_+^-+ A_-^-\#B_-^-}\,,
\end{equation}
same for $\#_n, \#_{\ge n}$.

The norm $\|\cdot\|_{m, s, p}^{\g, \cO}$  is closed 
under the  $ \# $ operation on symbols and satisfies 'tame' bounds w.r.t the parameters $m,s$.  
Regarding the third parameter $p$ (which controls the derivatives w.r.t. the variable $\xi$) 
the composition $ \# $ does not satisfy  tame bounds.
However if we restrict our 
attention to $\#_n$ or $\#_{\le n}$ then we have tame bounds also w.r.t. $p$, as stated in 
\eqref{stimacancellettoesplicitoAlgrammo}. 
% This property is made quantitative in the following lemma. 

\begin{lemma}{\bf (Composition).}\label{stima composizione}
Let $a\in S^m$, $b\in S^{m'}$, 
$m, m'\in\R$.  
\\[1mm]
$(i)$ The symbol $a\#b$ in \eqref{mareostia} belongs to 
$S^{m+m'}$ and, for any $p\in\N_0$, $ s \ge \so $, 
\begin{equation}\label{stimasharp}
\begin{aligned}
\|a\#b\|_{m+m', s, p}^{\g, \cO} \lesssim_{m,s, p} 
&\|a\|_{m, s, p}^{\g, \cO} \|b\|_{m', \so+p+|m|, p}^{\g, \cO} 
+ \|a\|_{m, \so, p}^{\g, \cO} \|b\|_{m', s+p+|m|, p}^{\g, \cO}\,.
\end{aligned}
\end{equation}

\noindent
$(ii)$  The symbol $a\#_{n} b$ in  \eqref{cancellittiEspliciti} satisfies
\begin{equation}\label{stimacancellettoesplicitoAlgrammo}
\|a\#_{n} b\|_{m+m'-n,s,p}^{\gamma,\calO}\lesssim_{m,s,p}
\sum_{\substack{\beta_1,\beta_2\in\N_0 \\ \beta_1+\beta_2=p}}
\|a\|_{m,s,\beta_1+n}^{\gamma,\calO}
\|b\|_{m',\so+n,\beta_2}^{\gamma,\calO}
+\|a\|_{m,\so,\beta_1+n}^{\gamma,\calO}
\|b\|_{m',s+n,\beta_2}^{\gamma,\calO}\,.
\end{equation}
\noindent
$(iii)$ For any integer $N\ge 1$  the symbol $a\#b$ admits the  expansion 
\eqref{composizione troncata} and the symbol 
 $r_N := r_{N, a b} := a \#_{\geq N} b  \in S^{m+m'-N}$ satisfies
\begin{equation}\label{restocancellettoNp}
	\begin{aligned}
		\|r_N\|^{\g, \cO}_{m+m'-N, s, p} 
	&	\lesssim_{m, N,s, p} \\
	\|a\|^{\g, \cO}_{m, s, N+p} &\|b\|^{\g, \cO}_{m', \so+2N+p+|m|, p} 
		+\|a\|^{\g, \cO}_{m, \so, N+p} \|b\|^{\g, \cO}_{m', s+2N+p+|m|, p} \,.
	\end{aligned}
\end{equation}

\noindent
$(iv)$ If the  symbol $b(\x)$ is a Fourier multiplier in $S^{m'}$ 
then 
 $\|a\#b\|_{m+m', s, p}^{\g, \cO}\lesssim_{m, p} \|a\|_{m, s, p}^{\g, \cO}$.
 
 \noindent
 $(v)$ The estimates above hold verbatim for $A\in S^m\otimes \cM_2(\C)$,  $B\in S^{m'}\otimes \cM_2(\C)$.
\end{lemma}
\begin{proof}
Items $(i)$, $(iii)$ are proved like  in \cite[Lemma 2.13]{BM20}. Item $(iv)$ trivially follows by item $(i)$. 
Item $(ii)$ % and $(v)$ 
follows  
as in  \cite[Lemma 2.13]{BM20}. 
\end{proof}

Given symbols $ a \in S^m $, $ b \in S^{m'}$ it is natural to define  the symbol of the commutator 
(recall \eqref{adjdef}) 
$$
a\star b : =a\#b-b\#a
$$
so that
\begin{equation}\label{adjdefpseudo}
[\op(a) ,\op(b)]  = {\rm ad}_{\op(a)}[\op(b)]  = \op(a\star b)   \,.
\end{equation} 
By \eqref{composizione troncata}, \eqref{remainder composizione}, 
\eqref{cancellittiEspliciti}
one deduces the expansion  for any $N\geq2$
\begin{equation}\label{espstar}
a\star b =a\#b-b\#a = -\ii \{a,b\}+
\sum_{\beta=2}^{N-1} 
(a\#_{\beta}b-b\#_{\beta}a) + \tr_N
\end{equation}
where 
\begin{equation}\label{espstar2}
\{a, b\}:= \partial_\xi a\partial_x b - \partial_x a\partial_\xi b\,,
\qquad
 \tr_N :=r_{N, ab}-r_{N, ba}\,.
\end{equation}
The same definition holds for matrices of operators $A,B\in S^m\otimes \cM_2(\C)$ as in \eqref{resinprog}. The following lemma follows like \cite[Lemma 2.15]{BM20}. 

\begin{lemma}{\bf (Commutators).}\label{lemma:Commutator}
Let $a\in S^m$, $b\in S^{m'}$, 
$m, m'\in\R$.
Then 

\noindent
$(i)$
The symbol $a\star b$ in \eqref{espstar} belongs to $S^{m+m'-1}$ and, 
for any $p\in\N_0, s\ge \so $, 
\begin{equation}\label{giggi}
\begin{aligned}
\|a\star b\|^{\g, \cO}_{m+m'-1, s, p} \lesssim_{m, m', s, p} 
&\|a\|^{\g, \cO}_{m, s+2+|m'|+p, p+1} \|b\|^{\g, \cO}_{m', \so+2+|m|+p, p+1} 
\\&
+ \|a\|^{\g, \cO}_{m, \so+2+|m'|+p, p+1} \|b\|^{\g, \cO}_{m', s+2+|m|+p, p+1}  \,.
\end{aligned}
\end{equation}

\noindent
$(ii)$ The Poisson bracket $\{a,b\}$ in \eqref{espstar2} is a symbol in $S^{m+m'-1}$ satisfying,
for any $p\in\N_0, s\ge \so$, 
\begin{equation}
\begin{aligned}
\|\{a,b\}\|^{\g, \cO}_{m+m'-1, s, p} 
\lesssim_{s,p} 
&\|a\|^{\g, \cO}_{m, s+1, p+1} \|b\|^{\g, \cO}_{m', \so+1, p+1} 
+ \|a\|^{\g, \cO}_{m, \so+1, p+1} \|b\|^{\g, \cO}_{m', s+1, p+1}  \,.
\end{aligned}
\end{equation}

\noindent
$(iii)$ For any  $N\geq2$ the remainder $\tr_{N}$ in \eqref{espstar2}
belongs to $S^{m+m'-N}$ and, for any $p\in\N_0, s\ge \so$, 
\begin{align*}
\|\tr_N\|^{\g, \cO}_{m+m'-N, s, p} 
&\lesssim_{m, m',s, N, p} 
\|a\|^{\g, \cO}_{m, s+2N+|m'|+p, p+N} 
\|b\|^{\g, \cO}_{m', \so+2N+|m|+p, p+N} 
\\&\qquad
+ \|a\|^{\g, \cO}_{m, \so+2N+|m'|+p, p+N} 
\|b\|^{\g, \cO}_{m', s+2N+|m|+p, p+N} \,.
\end{align*}
\end{lemma}
We now discuss the pseudo-differential symbol associated to a % totally convergent 
series of pseudo-differential operators. 
Given a symbol $a\in S^{m}$,    we define 
\begin{equation}\label{simboloiteratoKK}
	a^{\# 0} :=1\,,\quad a^{\# 1} :=a\,,\quad a^{\# k} := a\# a^{\# k-1} \,,\quad \forall k\geq 1 \, , 
\end{equation}
so that $\op(a)^k= \op(a^{\# k} )$.
%In Lemma \ref{potenza pseudo}, for $m\le 0$, we give some preliminary bound  on $a^{\# k}$, which ensure that   a totally convergent series of pseudo-differential symbols is a symbol, provided that some appropriate smallness conditions hold. See Lemma \ref{ecponential pseudo diff} for an application to the exponential.} 

%Note however that,  in that Lemma, the estimates on the $\|\ \|_{m,s,p}$ norm are not tame in $p$ and that the smallness condition depends on $p$.  Unfortunately this type of estimate, used in our reduction procedure, would lead to non-tame bounds on the change of variables,  in Theorem \ref{main:thm}, which in turn would lead to a very strong requirement in the smallness condition \eqref{smallCondCoeff}, where $\mu$  would depend on $s_1$. 

\begin{lemma}\label{potenza pseudo}
	Let $a\in S^{m}$ with $m\leq 0$. 
	For any $s\geq \so$ and $p\in \N_0 $, 
	there is a constant $\tC(m,s,p)\geq1$ such that for any $k\ge 1$   
	\begin{equation}\label{californication}
		\|a^{\# k}\|^{\gamma,\mathcal{O}}_{m, s, p}\le %C(s,\alpha) 
		(\tC(m, s,p)\|a\|^{\gamma,\mathcal{O}}_{m,\so+p,p})^{k-1}
		\|a\|^{\gamma,\mathcal{O}}_{m,s+p,p}\, . 
	\end{equation}
	The same notations and bounds hold for $A\in S^m\otimes \cM_2(\C)$.
\end{lemma}

\begin{proof}
	By induction, using \eqref{stimasharp}, it results that for any $ k \geq 0 $
	\begin{equation}\label{zanzare}
		\begin{aligned}
			\|a^{\# k+1}\|^{\gamma,\mathcal{O}}_{0, s_0, p}&\le   %C(s,\alpha) 
			(\tC(s_0,p))^k(\|a\|^{\gamma,\mathcal{O}}_{0,\so+p,p})^{k+1}\,,
			\\
			\|a^{\# k+1}\|^{\gamma,\mathcal{O}}_{0, s, p}&\le   %C(s,\alpha) 
			(\tC(s,p)\|a\|^{\gamma,\mathcal{O}}_{0,\so+p,p})^k
			\|a\|^{\gamma,\mathcal{O}}_{0,s+p,p}\,.
		\end{aligned}
	\end{equation}
	Let us now prove \eqref{californication}. For $ m \leq 0 $ we 
	\begin{equation*}
		\begin{aligned}
			& \|a^{\#k}\# a \|_{m, s, p}^{\g, \cO} \stackrel{\eqref{stimasharp}}{\lesssim_{m,s,p}} 
			\|a^{\# k}\|_{0, s, p}^{\g, \cO} \|a\|_{m, \so+p, p}^{\g, \cO} 
			+ \|a^{\# k}\|_{0, \so, p}^{\g, \cO} \|a\|_{m, s+p, p}^{\g, \cO}
			\\
			&  \stackrel{\eqref{zanzare},\eqref{zanzara}}{\lesssim_{m,s,p}}	 %C(s,\alpha) 
			(\tC(s,p)\|a\|^{\gamma,\mathcal{O}}_{0,\so+p,p})^{k-1}
			\|a\|^{\gamma,\mathcal{O}}_{0,s+p,p} \|a\|_{m, \so+p, p}^{\g, \cO} 
			+ (\tC(\so,p))^{k-1}(\|a\|^{\gamma,\mathcal{O}}_{0,\so+p,p})^{k}\|a\|_{m, s+p, p}^{\g, \cO} \\
		\end{aligned}
	\end{equation*}
	and \eqref{californication}  follows. % provided  $\tC(m,s,p)\ge \max\{ c(m,s,p), \tC(s,p)\}$.
\end{proof}

\begin{lemma} {\bf (Exponential map)} \label{ecponential pseudo diff}
Let  $a $ be a symbol in $ S^0 $. 
Then 
$ e^{\op(a)}=\op(1+\Phi) $ 
% \quad  e^{-\op(a)}=\op(1+\Psi)\,,\quad 
where $ \Phi $ is a symbol in  $ S^0  $ 
satisfying the following: 
for any $ s \geq \so $, $p\in\N_0$, there is a constant  $ \delta(s, p) >0 $ such that
	if 
	$  \| a \|_{0, \so+p, p}^{\g, \cO}\le \delta (s,p) $ then 
	\begin{equation}\label{flusso ord 0}
		\| \Phi   \|_{0,s,p}^{\g, \cO} % 	\| \Psi \|_{0,s,p}^{\g, \cO}
		 \lesssim_{s, p}  \| a \|_{0,s+p, p}^{\g, \cO} \, . 
	\end{equation}
	The same estimates hold for $ ({\rm Id} + \op(a))^{-1}$ with $a\in S^{0}$ as well as
	for matrix  operators  of the form $e^{\op(A)}, (\Id + \op(A))^{-1}$ where
	$A\in S^0\otimes \cM_2(\C)$.
\end{lemma}

\begin{proof}
Since
$ \exp(\op(a)) = \id + \sum_{k\geq 1}
\frac{\op(a)^k}{k!} $, 
the estimate \eqref{flusso ord 0} follows by Lemma \ref{potenza pseudo} 
and the smallness condition on $a$. 
The last assertion follows by a Neumann argument. 
\end{proof}

The following lemma will be crucial for  the symmetrization at lower orders of Proposition  \ref{blockTotale}. 
%We refine this result in Lemma \ref{fearofthedark}, 
We give sharp estimates 
for the symbol of the inverse of a close to identity 
pseudo-differential operator. %  of order $m\le -1$. 
We prove that for each $ \rho, p_* \in \N$, 
with the smallness condition \eqref{smalleffino} depending only on $ \rho , p_* $,
such symbol is the sum of two terms 
$ g_{<\rho} $ and  $g_{\geq\rho} $. In \eqref{pearljam1Bis} 
we bound  the norm $\| g_{<\rho} \|_{m,s,p} $  for any $ p $, 
whereas  we estimate $ \| g_{\geq\rho} \|_{-\rho,s,p} $  
only for $ p \leq p_* $. 
% the second is in $S^{-\rho}$ 

\begin{lemma}\label{fearofthedark}
Let $a$ be a matrix of symbols in $ S^{m}\otimes\mathcal{M}_2(\mathbb{C})$ with $m\leq -1 $ 
satisfying, for any $s\geq \so$, $p\geq0$,
\begin{equation}\label{ridethelightning}
\|a\|_{m,s,p}^{\gamma,\calO}\lesssim_{m,s,p}\|f\|_{s+p}^{\gamma,\calO}\, , 
\end{equation}
where $f $ is a function in  $ C^{\infty}(\T^{\nu+1};\C^{d}) $, $d\geq1$. 
Then, for any $\rho\in \N$ and   $ p_* \in \N_0 $, there exist 
$\widetilde{\s} :=\widetilde{\s}(m,\rho,p_*) >0 $  and 
$\delta:=\delta(m,s,\rho,p_*) > 0 $ such that, if 
\begin{equation}\label{smalleffino}
	\|f\|_{\so+\widetilde{\s}}^{\gamma,\calO}\leq\delta\,,
\end{equation}
then 
there exist matrices of symbols $g_{<\rho}\in S^{2m}\otimes\mathcal{M}_2(\mathbb{C})$ and 
$g_{\geq\rho}\in S^{-\rho}\otimes\mathcal{M}_2(\mathbb{C})$ such that, 
\[
(\Id-\op(a))^{-1}=
\sum_{k\geq 0} (\op(a))^{k} =
 \Id+\op(a)+\op(g_{<\rho})+\op(g_{\geq\rho}) 
\]
and
\begin{align}
\|g_{<\rho}\|^{\gamma,\calO}_{2m,s,p}&
\lesssim_{m,s,p}\|f\|_{s+\widetilde \s+p}^{\gamma,\calO}\,,\quad \forall\,p\geq0\,,
\label{pearljam1Bis}
\\
\|g_{\geq\rho}\|_{-\rho,s,p}^{\gamma,\calO}&\lesssim_{m,s,\rho,p_*}
\|f\|_{s+\widetilde{\s}}^{\gamma,\calO}\,,\quad \forall \, 0\leq p\leq p_{*}\label{pearljam1011bis} \, . 
\end{align}
\end{lemma}

\begin{proof}
	The key element of the proof is the following result.
\\[1mm]	
{\bf Claim.}{\it %\label{potenza pseudosharpsharp}
		\ Let $ a \in S^{m}$, $m\leq 0$. 
		For any $\rho\in \N $, for any $ k \in \N $,  there exist symbols 
		$a_{<\rho}^{(k)}\in S^{m} $ and $a_{\geq\rho}^{(k)}\in S^{-\rho}$ such that 
		\begin{equation}\label{pearljam}
			a^{\# k}=a_{<\rho}^{(k)}+a_{\geq \rho}^{(k)} \, ,  
			\qquad   a_{<\rho}^{(1)} := a \,,\quad a_{\ge \rho}^{(1)} :=0 \, , 
		\end{equation}
		and the following estimates hold. 
		There exist $\s_k :=\s_k(m,\rho)>0$ (non decreasing  in $k$) such that 
		for any  $s\geq \so$ and any $p\in \N_0 $
		there are constants $\tC_1(k):=\tC_{1}(m,s,\rho,p,k)>0$ and  $\tC_2(k) := \tC_{2}(s,\rho,p,k)>0$ such that 
		\begin{align}
			\|a_{<\rho}^{(k)}\|^{\gamma,\mathcal{O}}_{m, s, p}&\leq \tC_1
			\sum_{\sum_{i=1}^{k}\beta_{i}=p}
			\Big(\prod_{i=1}^{k-1}\|a\|^{\gamma,\mathcal{O}}_{m,\so+\s_k,\beta_i+\s_k}\Big)
			\|a\|^{\gamma,\mathcal{O}}_{m,s+\s_k,\beta_{k}+\s_k}\,,
			\label{trigun0}
			\\
			\|a_{\geq\rho}^{(k)}\|^{\gamma,\mathcal{O}}_{-\rho, s, p}&\leq \tC_2
			( \|a\|^{\gamma,\mathcal{O}}_{0,\so+\s_k+p,p+\s_k})^{k-1}
			\|a\|^{\gamma,\mathcal{O}}_{0,s+\s_k+p,p+\s_k}\,.\label{trigun00}
		\end{align}
		The same notations and bounds hold for $A\in S^m\otimes \cM_2(\C)$.
}
	\\[1mm]
		{\sc Proof of the claim.} We reason by induction on $k\geq 1$. For $k=1$ the estimates \eqref{trigun0} and \eqref{trigun00}  follow by definition,  for convenience we take $\s_1=\rho$ and $\tC_1(1)= \tC_2(1)=1$.
\\[1mm]
Then we consider 	$k\geq2$. We assume inductively that 
	\eqref{trigun0}-\eqref{trigun00} holds up to $k$ and we prove them for $k+1$.
	Recalling \eqref{simboloiteratoKK} and the inductive hypothesis we write
	\begin{equation}\label{trigunNebbia}
	\begin{aligned}
		a^{\#k+1}&=a^{\#k}\#a=
		%\big(a_{<\rho}^{(k)}+a_{\geq\rho}^{(k)}\big)\#a
		a_{<\rho}^{(k)}\#a+a_{\geq\rho}^{(k)}\#a
		%\\&
		\stackrel{\eqref{cancellittiEspliciti}}{=}
		\underbrace{a_{<\rho}^{(k)}\#_{<\rho}a}_{=:a_{<\rho}^{(k+1)}}+
		\underbrace{a_{<\rho}^{(k)}\#_{\geq\rho}a+a_{\geq\rho}^{(k)}\#a}_{=:a_{\geq\rho}^{(k+1)}}\,.
	\end{aligned}
	\end{equation}
	We first estimate the symbol $a_{<\rho}^{(k+1)}$. Recalling \eqref{zanzara} and using \eqref{stimacancellettoesplicitoAlgrammo} with $m=m'$ we obtain
		\begin{align}\label{stimasharpnn}
	\|a_{<\rho}^{(k)}\#_{<\rho}a\|_{m,s,p}^{\gamma,\calO}&\le \|a_{<\rho}^{(k)}\#_{<\rho}a\|_{2m,s,p}^{\gamma,\calO} \\ &
	\stackrel{\eqref{stimacancellettoesplicitoAlgrammo}}{\lesssim_{m,s,p,\rho}}
	\sum_{\substack{\beta_1,\beta_2\in\N_0 \\ \beta_1+\beta_2=p}}
	\|a_{<\rho}^{(k)}\|_{m,s,\beta_1+\rho}^{\gamma,\calO}
	\|a\|_{m,\so+\rho,\beta_2}^{\gamma,\calO}
	+\|a_{<\rho}^{(k)}\|_{m,\so,\beta_1+\rho}^{\gamma,\calO}
	\|a\|_{m,s+\rho,\beta_2}^{\gamma,\calO} \notag 
	\\&
	\stackrel{\eqref{trigun0}}{\lesssim_{m,s,p,\rho,k}} 
	\sum_{\substack{ \beta_1+\beta_2=p}}
	\sum_{\sum_{i=1}^{k}\beta_{i}'=\beta_1+\rho}\Big[
	\Big(\prod_{i=1}^{k-1}\|a\|^{\gamma,\mathcal{O}}_{m,\so+\s_k,\beta_i'+\s_k}\Big)
	\|a\|^{\gamma,\mathcal{O}}_{m,s+\s_k,\beta_{k}'+\s_k}
	\|a\|_{m,\so+\rho,\beta_2}^{\gamma,\calO} \notag 
	\\&\qquad\qquad\qquad\qquad\quad+%\sum_{\substack{\beta_1+\beta_2=p}}
	%\sum_{\sum_{i=1}^{k+1}\beta_{i}'=\beta_1+\rho}
	\Big(\prod_{i=1}^{k}\|a\|^{\gamma,\mathcal{O}}_{m,\so+\s_k,\beta_i'+\s_k}\Big)
	\|a\|_{m,s+\rho,\beta_2}^{\gamma,\calO}\Big] \notag 
	\\&\quad
	\le \tC_1(k+1)
	\sum_{\substack{ \sum_{i=1}^{k+1}\beta_i=p}}
	\Big(\prod_{i=1}^{k}
	\|a\|_{m,\so+\s_{k}+\rho,\beta_i+\s_k+\rho}^{\gamma,\calO}\Big)
	\|a\|_{m,s+\s_{k}+\rho,\beta_{k+1}+\s_{k}+\rho}^{\gamma,\calO} \notag 
	\end{align}
	%provided that we take $\tC_1 > 2 C$ and 
	for some  $\tC_1(k+1):= \tC_1(m,s,\rho,p,k+1)$. This proves \eqref{trigun0} by taking $\s_{k+1}\geq\s_{k}+\rho$.  
	We now bound  the second summand in 
	\eqref{trigunNebbia}.
	First of all, using \eqref{restocancellettoNp} with $ N\rightsquigarrow\rho$ and $m,m'\rightsquigarrow 0$ we get 
	\[
	\begin{aligned}
	\|a_{<\rho}^{(k)}\#_{\geq\rho}a\|_{-\rho,s,p}^{\gamma,\calO}
	&%\stackrel{\eqref{restocancellettoNp}\sim N\rightsquigarrow\rho\,m,m'\rightsquigarrow 0}
	{\lesssim_{s,\rho,p}}
	\|a_{<\rho}^{(k)}\|^{\g, \cO}_{0, s, \rho+p} \|a\|^{\g, \cO}_{0, \so+2\rho+p, p} 
	%\\&\qquad\qquad
	+
	\|a_{<\rho}^{(k)}\|^{\g, \cO}_{0, \so, \rho+p} \|a\|^{\g, \cO}_{0, s+2\rho+p, p}
	\\&\stackrel{\eqref{trigun0}}{\lesssim_{s,\rho,p,k}} %C(s,\rho,p)(\tC_1(0,s,\rho,p,k))\times\\&\times
	\sum_{\sum_{i=1}^{k}\beta_{i}=p+\rho}\Big[
	\Big(\prod_{i=1}^{k-1}\|a\|^{\gamma,\mathcal{O}}_{0,\so+\s_k,\beta_i+\s_k}\Big)
	\|a\|^{\gamma,\mathcal{O}}_{0,s+\s_k,\beta_{k}+\s_k}
	\|a\|^{\g, \cO}_{0, \so+2\rho+p, p} 
	\\&\qquad\quad+
	%\sum_{\sum_{i=1}^{k+1}\beta_{i}=p+\rho}
	\Big(\prod_{i=1}^{k}\|a\|^{\gamma,\mathcal{O}}_{0,\so+\s_k,\beta_i+\s_k}\Big)
	\|a\|^{\g, \cO}_{0, s+2\rho+p, p} \Big]
	\\&\lesssim_{s,\rho,p,k}(\|a\|_{0,\so+{\s}_{k}+p+\rho,p+{\s}_{k}+\rho}^{\gamma,\calO})^{k}
	\|a\|_{m,s+{\s}_{k}+p+\rho,{\s}_{k}+\rho}^{\gamma,\calO}\,,
	\end{aligned}
	\]
	recalling that $\s_{k}+\rho\ge 2\rho$. Similarly,  using \eqref{stimasharp} with $ m,m'\rightsquigarrow -\rho, 0$, we have  
	\[
	\begin{aligned}
	&\|a_{\geq\rho}^{(k)}\#a\|_{-\rho,s,p}^{\gamma,\calO}%\stackrel{\eqref{stimasharp} m,m'\rightsquigarrow -\rho, 0}
	{\lesssim_{\rho,s,p}}
	\|a_{\geq\rho}^{(k)}\|_{-\rho, s, p}^{\g, \cO} \|a\|_{0, \so+p+\rho, p}^{\g, \cO} 
	+ \|a_{\geq\rho}^{(k)}\|_{-\rho, \so, p}^{\g, \cO} \|a\|_{0, s+p+\rho, p}^{\g, \cO}
	\\&
	\stackrel{\eqref{trigun00}}{\lesssim_{s,\rho,p,k}} \Big[
	(\|a\|^{\gamma,\mathcal{O}}_{0,\so+\s_k+p,p+\s_k})^{k-1}
	\|a\|^{\gamma,\mathcal{O}}_{0,s+\s_k+p,p+\s_k}
	\|a\|_{0, \so+p+\rho, p}^{\g, \cO} 
	\\ &+
	(\|a\|^{\gamma,\mathcal{O}}_{0,\so+\s_k+p,p+\s_k})^{k}
	\|a\|_{0, s+p+\rho, p}^{\g, \cO}\Big]
	\lesssim_{s,\rho,p,k} ( \|a\|^{\gamma,\mathcal{O}}_{0,\so+{\s}_{k}+p,p+{\s}_{k}})^{k} 
	\|a\|_{0, s+p+{\s}_{k},p+{\s}_{k}}^{\g, \cO}\,.
	\end{aligned}
	\]
	By the discussion above we deduce 
	\eqref{trigun00} with $k\rightsquigarrow k+1$
	taking $\s_{k+1}\geq \s_{k}+\rho$. The proof of the claim is complete.  \rule{2mm}{2mm}

\smallskip
We are now ready to prove Lemma \ref{fearofthedark}. We recall that, following Proposition 5.4 of \cite{AG_book}, one has that if $\|a\|_{m,s_0,0}<\delta$  then 
$(\op(1+a))^{-1} =\op(b)$ with $b\in S^0$. 
For $ \rho = 1,2 $ we define 
	 \[
 g_{<\rho} :=  0 \, , \quad 
 	 g_{\ge \rho} :=  a \# a \#  b=   a^{\#2}\#  b\,.
	 \]
For $\rho\ge 3$ we have  
\begin{equation}\label{pearljamtris}
\begin{aligned}
&	(\Id-\op(a))^{-1}=\Id +\op\Big( \sum_{k=1}^{\rho-1}a^{\#k}+ a^{\#\rho}\#b\Big)
\\
&
\stackrel{\eqref{pearljam}}=\Id +\op(a)+\op\Big( \underbrace{\sum_{k=2}^{\rho-1}a^{(k)}_{< \rho}}_{=: g_{<\rho}} +\underbrace{\sum_{k=2}^{\rho-1}a^{(k)}_{\ge  \rho} + a^{\#\rho}\#b}_{=: g_{\geq \rho}} \Big)\,.
\end{aligned}
\end{equation}
Note that  $g_{\geq \rho}\in S^{-\rho}$  since $m\le -1$ and $b\in S^0$.
For any  $ 2 \leq k \leq \rho - 1  $ we estimate  the term 
$\|a^{(k)}_{<\rho}\|_{2m,s,p}^{\gamma,\calO}= \|a^{(k-1)}_{<\rho}\#_{<\rho} a\|_{2m,s,p}^{\gamma,\calO}$   using  \eqref{stimasharpnn} with $k\rightsquigarrow k-1$, obtaining, for any $ \,p\geq 0 $, 
\begin{equation}\label{pearljam2001}
\begin{aligned}
\|g_{<\rho}\|_{2m,s,p}^{\gamma,\calO}
\leq
\sum_{k=2}^{\rho-1}\tC_1(k)
\sum_{\sum_{i=1}^{k}\beta_{i}=p}
\Big(\prod_{i=1}^{k-1}\|a\|^{\gamma,\mathcal{O}}_{m,\so+\s,\beta_i+\s}\Big)
\|a\|^{\gamma,\mathcal{O}}_{m,s+\s,\beta_{k}+\s} 
\end{aligned}
\end{equation}
with  $\s :=  \s_{\rho-1}$ 
where $\s_{k}$ is the non decreasing sequence in \eqref{trigun0}. Bounding $\tC_1(k)\le \tC(\rho)$ for any $k=1,\dots,\rho-1$ and using \eqref{ridethelightning} we obtain, for any $p\geq0$,  
\begin{equation}\label{pearljam1}
\begin{aligned}
\|g_{<\rho}\|_{2m,s,p}^{\gamma,\calO}
%&\lesssim_{m,s,\rho,p}
%\sum_{k=2}^{\rho-1}
%\sum_{\sum_{i=1}^{k}\beta_{i}=p}
%\Big(\prod_{i=1}^{k-1}\|a\|^{\gamma,\mathcal{O}}_{m,\so+\s,\beta_i+\s}\Big)
%\|a\|^{\gamma,\mathcal{O}}_{m,s+\s,\beta_{k}+\s}\,,%\quad \forall\,p\geq0\,,
%\\&
\stackrel{\eqref{ridethelightning}}{\lesssim_{m,s,\rho,p}}
\sum_{k=2}^{\rho-1}
\sum_{\sum_{i=1}^{k}\beta_{i}=p}
\Big(\prod_{i=1}^{k-1}\|f\|^{\gamma,\mathcal{O}}_{\so+2\s+\beta_i}\Big)
\|f\|^{\gamma,\mathcal{O}}_{s+2\s+\beta_{k}}\, . 
\end{aligned}
\end{equation}
The interpolation estimate \eqref{interpolotutto} 
implies that 
$
\|f\|^{\gamma,\mathcal{O}}_{\so+2\s+\beta_i}\|f\|^{\gamma,\mathcal{O}}_{s+2\s+\beta_{k}}
%&
%\leq
%\|f\|_{s+2\s+\beta_{i}+\beta_{k}}^{\gamma,\calO}\|f\|_{\so+2\s}^{\gamma,\calO}
%+\|f\|_{\so+2\s}^{\gamma,\calO}\|f\|_{s+2\s+\beta_i+\beta_k}^{\gamma,\calO}
%\\&
\lesssim\|f\|_{\so+2\s}^{\gamma,\calO}\|f\|_{s+2\s+\beta_i+\beta_k}^{\gamma,\calO}\,,
$
(we use $a_0\rightsquigarrow\so+2\s$, $b_0\rightsquigarrow\so+2\s$,
$p\rightsquigarrow\beta_i$ and $q\rightsquigarrow\beta_{k}+s-\so$) and 
therefore by \eqref{pearljam1} we get, for any $p\geq0$,
$$
\|g_{<\rho}\|^{\gamma,\calO}_{2m,s,p}\lesssim_{m,s,\rho,p}
\sum_{k=2}^{\rho-1}(\|f\|^{\gamma,\mathcal{O}}_{\so+2\s})^{k-1}
\|f\|^{\gamma,\mathcal{O}}_{s+2\s+p}
\lesssim_{m,s,\rho,p}\|f\|^{\gamma,\mathcal{O}}_{s+2\s+p}\,,
$$
using  \eqref{smalleffino}. This proves \eqref{pearljam1Bis} provided $\widetilde\s\ge 2\s$.

Now we prove \eqref{pearljam1011bis}. 
We estimate the first symbol in the definition of $ g_{\geq \rho} $ in 
\eqref{pearljamtris}.
By  \eqref{trigun00} we obtain
\[
\begin{aligned}
\Big\| \sum_{k=2}^{\rho-1}a^{(k)}_{\ge  \rho} \Big\|_{-\rho,s,p}^{\gamma,\calO}
&\stackrel{\eqref{trigun00}}{\leq}
\sum_{k=2}^{\rho-1}
\tC_2(k)(\|a\|^{\gamma,\mathcal{O}}_{0,\so+\s_k+p,p+\s_k})^{k-1}
\|a\|^{\gamma,\mathcal{O}}_{0,s+\s_k+p,p+\s_k}
\\&
\stackrel{\eqref{ridethelightning}}{\lesssim_{s,\rho,p}}
\sum_{k=2}^{\rho-1}
(\|f\|^{\gamma,\mathcal{O}}_{\so+2\s_k+2p})^{k-1}
\|f\|^{\gamma,\mathcal{O}}_{s+2\s_k+2p}
\lesssim_{s,\rho,p_*}
\|f\|^{\gamma,\mathcal{O}}_{s+ 2p_*+2\s}
\end{aligned}
\]
for  all $0\leq p\leq p_{*}$,   and
provided that $\widetilde\s\ge  2\s + 2p_*$ and  \eqref{smalleffino}. %We have proved  \eqref{pearljam1011bis}.
%and $\mu=\mu(m,\rho,p_*)> \s_{\rho-1}+ p_*$.%arge w.r.t. $\s_{k}$, $k=1,\ldots,\rho-2$.

We now estimate the symbol $ a^{\# \rho} $. 
By applying iteratively the estimate 
\eqref{stimasharp} in Lemma \ref{stima composizione}
for $0\leq p\leq p_{*}$, using $m\le -1$ %and using \eqref{pearljamblack} with $\s\geq\rho|m|$ and $\mu\geq \rho p_{*}$
we obtain for $\rho\ge 2$

\begin{equation}\label{pearljam11}
\begin{aligned}
\|a^{\# \rho}\|_{-\rho,s,p}^{\gamma,\calO}&=		\|a^{\# \rho-1}\#a\|_{-\rho,s,p}^{\gamma,\calO} \\ &\lesssim_{s,p,\rho}
	\|a^{\# \rho-1}\|_{-\rho+1,s,p}^{\gamma,\calO} \|a\|_{-1,\so+p+\rho,p}^{\gamma,\calO} +  	\|a^{\# \rho-1}\|_{-\rho+1,\so,p}^{\gamma,\calO} \|a\|_{-1,s+p+\rho,p}^{\gamma,\calO}
	\\ &
	\lesssim_{s,p,\rho}
	\|a\|_{-1,s+p+\rho,p}^{\gamma,\calO} (\|a\|_{-1,\so+p+\rho,p}^{\gamma,\calO})^{\rho-1} 
	\\
	&\stackrel{\eqref{ridethelightning}}{\lesssim_{m,s,\rho,p}}
	\|f\|_{s+2p+\rho}^{\gamma,\calO} (\|f\|_{\so+2p+\rho}^{\gamma,\calO})^{\rho-1} \lesssim_{s,\rho, p}
	\|f\|_{s+2p+\rho}^{\gamma,\calO} 
\end{aligned}
\end{equation}
using the smallness condition \eqref{smalleffino}.

We now consider the symbol $b$,  if we only want to control the $\| \ \|_{0,s,p}$ norm for $p\le p_*$ we can use Lemma \ref{ecponential pseudo diff} to write $b= \sum_{k\geq0}a^{\#k }$.
By applying 
estimate \eqref{californication} in 
Lemma \ref{potenza pseudo} with $m=0$,
we have for any $0\leq p\leq p_{*}$
\begin{equation}\label{californication3}
\begin{aligned}
\|b\|_{0,s,p}^{\gamma,\calO}
&\leq \sum_{k\geq0}(\tC(s,p)\|a\|^{\gamma,\mathcal{O}}_{0,\so+p,p})^k
\|a\|^{\gamma,\mathcal{O}}_{0,s+p,p}\\
&\stackrel{\eqref{ridethelightning}}{\lesssim_{s,\rho,p}}\sum_{k\geq0}(\tC(s,p)\|f\|^{\gamma,\mathcal{O}}_{\so+2p})^k
\|f\|^{\gamma,\mathcal{O}}_{s+2p} 
\stackrel{\eqref{smalleffino}}{\lesssim_{s,p}}\|f\|^{\gamma,\mathcal{O}}_{s+2p}\,,
\end{aligned}
\end{equation}
by taking $\delta$ in \eqref{smalleffino} small enough w.r.t. $\tC(s,p)$.
By estimates \eqref{pearljam11}, \eqref{californication3}, 
combined with \eqref{stimasharp} (applied with $m\rightsquigarrow -\rho$, $m'\rightsquigarrow0$)
we deduce 
\[
\begin{aligned}
\|a^{\#\rho}\#b\|_{-\rho,s,p}^{\gamma,\calO}
&\lesssim_{s,p,\rho}
\|a^{\#\rho}\|_{-\rho, s, p}^{\g, \cO} \|b\|_{0, \so+p+\rho, p}^{\g, \cO} 
+ \|a^{\#\rho}\|_{-\rho, \so, p}^{\g, \cO} \|b\|_{0, s+p+\rho, p}^{\g, \cO}
\\&
\lesssim_{s,p,\rho}
\|f\|_{s+2p+\rho}^{\gamma,\calO} \|f\|^{\gamma,\mathcal{O}}_{s_{0}+3p+\rho}
+
\|f\|_{\so+2p+\rho}^{\gamma,\calO}\|f\|^{\gamma,\mathcal{O}}_{s+3p+\rho}
\,.
\end{aligned}
\]
In conclusion 
 we deduce the estimate 
 \begin{equation}\label{pearljam1011}
\|g_{\geq\rho}\|_{-\rho,s,p}^{\gamma,\calO}\lesssim_{s,\rho,p_*}
\|f\|^{\gamma,\calO}_{s+3\s+3p_{*}}\,,\quad \forall 0\leq p\leq p_{*} \, .
\end{equation}
Then estimate \eqref{pearljam1011bis} follows trivially by \eqref{pearljam1011}
 taking $\widetilde{\s}\geq 3\s+3p_{*}$. 
 This concludes the proof.
 \end{proof}

\paragraph{\bf Algebraic properties of pseudo-differential operators.}
We now describe 
how the algebraic properties introduced 
in Section \ref{sec:revparity} read on 
pseudo-differential operators.

\begin{defn}\label{giornataSupersolare} 
A symbol  $a=a(\vphi, x, \xi)\in S^m$ (resp. a matrix of symbols 
$A \in S^m\otimes \cM_2(\C)$), $m\in\R$, is \emph{reversible}, 
or  \emph{reversibility preserving}, or
\emph{parity preserving} if 
its associated pseudo-differential operator
$\op(a)$ (resp. the matrix of pseudo-differential operators $\op(A)$) is, respectively, a 
reversible, or reversibility preserving, or  parity preserving operator
according to item $(i)$  (resp. (ii)) of Definition  \ref{giornatasolare}. 
A matrix of symbols  $ A \in S^m\otimes \cM_2(\C)$ is \emph{real-to-real} if 
its associated pseudo-differential operator
$\op(A)$ is \emph{real-to-real}.
\end{defn}

Reversible,  reversibility preserving, parity preserving symbols 
and real-to-real matrices of symbols 
are caractherized as follows. 
\begin{lemma}\label{paritasuisimboli}
A symbol $a(\vphi,x,\x)\in S^{m}$, $m\in\R$, is 

\noindent
$(i)$  \emph{reversible} if and only if 
\begin{equation}\label{reverSimbolo}
a(-\vphi,x,\x)=-\overline{a(\vphi,x,-\x)}\,,
\end{equation}
equivalently if and only if
$ \doublehat{a}(\ell,j,\x)=-\overline{\doublehat{a}(\ell,-j,-\x)} $ (recall notation \eqref{aFourierphix}), 
for any $ \ell\in \T^{\nu}\,,\,j,\x\in \Z $.

\noindent
$(ii)$   \emph{reversibility preserving} if and only if 
\begin{equation}\label{reverpreserSimbolo}
a(-\vphi,x,\x)=\overline{a(\vphi,x,-\x)}\,,
\end{equation}
equivalently if and only if
$ \doublehat{a}(\ell,j,\x)=\overline{\doublehat{a}(\ell,-j,-\x)}
$, for any $  \ell\in \T^{\nu}$, $ j,\x\in \Z $. 

\noindent
$(iii)$  \emph{parity preserving} if and only if 
\begin{equation}\label{paritySimbolo}
a(\vphi,-x,-\x)=a(\vphi,x,\x)\,,
\end{equation}
equivalently if and only if
$ 
\doublehat{a}(\ell,-j,-\x)=\doublehat{a}(\ell,j,\x) $, 
for any $ \ell\in \T^{\nu}\,,\,j,\x\in \Z $. 

\noindent
$(iv)$ A matrix of symbols $A\in S^{m}\otimes \mathcal{M}_2(\C)$ as in \eqref{matricidisimboli}
is \emph{real-to-real} if and only if it has the form
\begin{equation}\label{divanocomodo}
A=A(\vphi,x,\x)=\sm{a(\vphi,x,\x)}{b(\vphi,x,\x)}{\overline{b(\vphi,x,-\x)}}{\overline{a(\vphi,x,-\x)}}\, . 
\end{equation}
Moreover $ A $ is reversible, resp. reversibility, parity preserving, if and only if 
$ a(\vphi,x,\x), b(\vphi,x,\x) $ satisfy \eqref{reverSimbolo}, resp. \eqref{reverpreserSimbolo}, 
\eqref{paritySimbolo}.
\end{lemma}

\begin{proof}
Items $(i)$-$(ii)$ follow by 
\eqref{coccodrillo1}-\eqref{coccodrillo2} and \eqref{conjsimb}.
Item $(iii)$  follows by \eqref{actionpseudo}, \eqref{coccodrillo3}.
For item $(iv)$
consider a matrix of symbols $A$
 as in \eqref{matricidisimboli}. By  
 \eqref{bolena1} 
 the operator 
$\op(A)$ in \eqref{matrixsy}
is real-to-real
if and only if 
$ {\op(a)}=\overline{\op(d)} $, 
$  \op(b)=\overline{{\op(c)}} $, 
which amounts, thanks to \eqref{conjsimb}, to 
$ a(\vphi,x,\x)=\overline{d(\vphi,x,-\x)} $ and $  b(\vphi,x,\x)=\overline{c(\vphi,x,-\x)} $, 
proving  \eqref{divanocomodo}.
\end{proof}

The class of reversibility, resp. parity,  preserving symbols is closed under composition,
as well as the class of real-to-real matrices of symbols.  The composition of a 
reversibility preserving and a reversible symbol is reversibility preserving. 

\begin{lemma}\label{paritasuisimboliBIS}
Let $ m,m'\in \R$ and  $n\in \N_0 $. 
\\[1mm]
$(i)$ Consider symbols $a\in S^{m}, b\in S^{m'}$. 
If $a$ is reversible and $b$ is  reversibility preserving then 
the symbols $a\#_{n}b, b\#_{n}a$ defined in \eqref{cancellittiEspliciti} are  reversible.
 If $a,b $ are reversibility preserving then $a\#_{n}b$ is  reversibility preserving.
If $a,b $ are parity preserving then $a\#_{n}b$ is  parity preserving. 

\noindent
$(ii)$ Let $A\in S^{m}\times\mathcal{M}_2(\C)$ and  $B\in S^{m'}\times\mathcal{M}_2(\C)$
be real-to-real matrices of symbols.
Then $A\#B$ and  $A\#_nB$ are real-to-real matrices of symbols.
\end{lemma}

\begin{proof}
Let us prove that if $a$ is reversible and $b$ is  reversibility preserving then 
$a\#_{n}b $ is  reversible. 
In view of \eqref{reverSimbolo} and \eqref{reverpreserSimbolo}
 we have, for any $n\in\N_0 $,
\[
\begin{aligned}
&(\pa_{\x}^na)(-\vphi,x,\x)=-(-1)^{n}\overline{(\pa_{\x}^na)(\vphi,x,-\x)}\,,
\qquad (\pa_{x}^na)(-\vphi,x,\x)=-\overline{(\pa_{x}^na)(\vphi,x,-\x)}\,,
\\&
(\pa_{\x}^nb)(-\vphi,x,\x)=(-1)^{n}\overline{(\pa_{\x}^n b)(\vphi,x,-\x)}\,,
\qquad (\pa_{x}^nb)(-\vphi,x,\x)=\overline{(\pa_{x}^nb)(\vphi,x,-\x)}\,.
\end{aligned}
\] 
Recalling \eqref{cancellittiEspliciti} and using the above formul\ae  we deduce that
\[
\overline{(a\#_{n}b)(\vphi,x,-\x)} =
 \overline{\frac{1}{n! \ii^{n}}(\partial_{\xi}^n a)(\vphi,x,-\x)(\partial_x^n b)(\vphi,x,-\x)}
 %\\&
 %=-\frac{(-1)^{n}}{n! \ii^{n}}
% (-1)^{n}(\partial_{\xi}^n a)(-\vphi,x,\x)
%(\partial_{x}^n b)(-\vphi,x,\x) 
=-(a\#_{n}b)(-\vphi,x,\x) 
\]
proving that the symbol $(a\#_{n}b)(\vphi,x,\x)$ is reversible.
The other claims follow similarly.
%The claim for the symbol $(b\#_{n}a)(\vphi,x,\x)$ and 
 %the case $2$ follow similarly.
% One can show that $(a\#_{n}b)(\vphi,x,\x)$ is parity preserving using 
 % \eqref{paritySimbolo}.
Item $(ii)$ follows similarly
recalling \eqref{resinprog}, \eqref{divanocomodo}, \eqref{mareostia}, \eqref{cancellittiEspliciti}.
\end{proof}

\section{Bony-smoothing couples}\label{sec:modulotame}

In Definition \ref{Es} we associate to a  majorant bounded operator $A$ 
a convenient  ``Bony-smoothing couple''  $(M,R)\in E_s$ so that $A=M+R$.  
A main point is that  Bony-smoothing couples have a  Banach algebra structure with respect to a norm 
which ``weights'' differently the 
``Bony" and the ``smoothing" component, see Definition \ref{Es}. 
This norm guarantees tame estimates % boundedness
 for the action of $A$ on Sobolev spaces (Lemma \ref{rmk:tametresbarre}).
 The important Proposition \ref{prop:immersionepseudo} shows 
how to  decompose  
a pseudo-differential operator
 via a Bony-smoothing couple.

\smallskip

For $s\geq \gso$ (cfr. \eqref{costanti}) we consider the Banach space
\begin{equation}\label{prestige2}
\cM^{\mathtt{T}}_{s}:=\bigcap_{\gso\le p\le s}\cM^{\mathtt{T}}( H^p, H^p)\ \ \ 
\text{ endowed with the norm }\ \ 
|\cdot|_{\cM^{\mathtt{T}}_{s}}:=\sup_{\gso\le p\le s}|\cdot|_{p,p} 
\end{equation}
where $|\cdot|_{p,p}$ is the majorant operator norm  in \eqref{majass}.
For $s'\geq s$ we have
$ \cM^{\mathtt{T}}_{s'}\subseteq
\cM^{\mathtt{T}}_{s} $ with 
$ |\cdot|_{\cM^{\mathtt{T}}_{s}}\leq 
|\cdot|_{\cM^{\mathtt{T}}_{s'}} $. 
Moreover, in view of \eqref{ABss}, any $\cM^{\mathtt{T}}_{s} $ 
is a Banach algebra: for any 
$ A, B  $ in $ \cM^{\mathtt{T}}_{s} $ we have 
\begin{equation}\label{tameprop1}
|A B|_{\cM^{\mathtt{T}}_{s}}
\leq 
|A|_{\cM^{\mathtt{T}}_{s}}
|B|_{\cM^{\mathtt{T}}_{s}}\,.
\end{equation}

\begin{defn}{\bf (Bony-smoothing couples).}\label{Es} 
Fix $\su>\so = s_* + 1 $. For any $\su\geq s\geq \so$
we define  
 the vector   space  $E_s :=E_{s,\gso,\su}$  of couples 
\begin{equation*}
\mathtt{A}=(M,R)\, , \quad M\in \cM^{\mathtt{T}}_{\su}\,,
\quad 
R\in \cM^{\mathtt{T}}( H^{\gso}, H^s) \, , 
\end{equation*}
with sum and multiplication by a scalar 
$$
(M_1,R_1)+(M_2,R_2) := (M_1+M_2,R_1+R_2) \, , \quad 
k (M,R):=(k M,k R) \, , 
$$
and  norm
\begin{equation}\label{normaincredibile}
\bnorm{\mathtt{A}}_s = \bnorm{(M,R)}_s := \bnorm{(M,R)}_{s,\gso,\su}
:=|M|_{\cM^{\mathtt{T}}_{\su}}+ |R|_{\gso,s}
:=\sup_{\gso\le p\le \su}|M|_{p,p} + |R|_{\gso,s}\,.
\end{equation}
Given a Lipschitz family of  $\tA=(M,R)$  in $E_s $ 
we define the weighted Lipschitz norm\footnote{An equivalent norm 
is $\sup_{\gso\leq p\leq \su}|M|^{\gamma,\cO}_{p,p}+|R|_{\gso,s}^{\gamma,\cO}$. } 
\begin{equation}\label{normagamma3}
\begin{aligned}
& 
\bnorm{\mathtt{A}}_{s}^{\gamma,\cO}:= \sup_{\omega} \bnorm{\mathtt{A}}_{s} + 
 \gamma \sup_{\omega_1\neq \omega_2 } \frac{ \bnorm{{\mathtt{A}}(\omega_1) 
 - {\mathtt{A}}(\omega_2)}_s}{|\omega_1-\omega_2|} \, . 
\end{aligned}
\end{equation}
Following the notation of Section \ref{sec:funcset}, we denote by 
$E_s\otimes \cM_2(\C)$ the space of $2\times 2$ matrices with entries in $E_s$, 
with the norms $\bnorm{\cdot}$ and $\bnorm{\cdot}^{\g, \cO}$ given by the max of the corresponding norms of the components.

The  space $E_{s}$ is a Banach algebra
with respect to the (associative but non commutative) product
\begin{equation}\label{A1A2}
	\mathtt{A}_1 \circ \mathtt{A}_2 = 
	(M_1,R_1) \circ (M_2,R_2):= (M_1M_2, M_1 R_2 + R_1 M_2+ R_1 R_2)\,.
\end{equation}
We also denote $\mathtt{A}_1\circ \mathtt{A}_2 $ simply as $\mathtt{A}_1\mathtt{A}_2 $. 

Finally  we define   the linear homomorphism $E_{s} \leftrightarrow \cM^{\mathtt{T}}_{s}$ as
\begin{equation}\label{pairMR}
\mathfrak{S} \; : \; E_{s} \to \cM^{\mathtt{T}}_{s} \, , \quad 
\quad (M,R)\mapsto 
\mathfrak{S}(M,R):=  M+R\, .
\end{equation}
The product $\circ$ and the homomorphism $\fS$ extend naturally to $E_s\otimes \cM_2(\C)$.
\end{defn}

The identity in $E_{s}$ is the couple $(\id, 0)$ which, with a slight abuse of notation, we  denote as $\id$.

\begin{lemma}\label{sergiu}
The map 
$\mathfrak{S}$ in \eqref{pairMR}  is  bounded, with  
	$ \| \mathfrak{S} \|_{\cL(E_{s} ,\cM^{\mathtt{T}}_{s})}=1 $, 
	for any $\gso\leq s\leq \su$. 
Furthermore  $\mathfrak{S}$ is surjective but not injective.
\end{lemma}

\begin{proof}
The first statement follows since 
\[
|\mathfrak{S}(M,R)|_{\cM^{\mathtt{T}}_{s}}
=|M+R|_{\cM^{\mathtt{T}}_{s}}=\sup_{\gso\leq p\leq s} |M+R|_{p,p}
\leq \bnorm{(M,R)}_{s,\gso,\su}\,.
\]
Regarding the second statement, 
given any operator $A\in\cM^{\mathtt{T}}_{ \su}$ 
one has 
$(A, 0)\in E_s$ and $\mathfrak{S} (A,0)= A$.
On the other hand, 
if $ A= \mathfrak{S}(M,R) = M+R $   then, taking  $ R' \in 
\cM^{\mathtt{T}}_{ \su}\cap \cM^{\mathtt{T}}( H^{\gso},  H^{s})\neq \emptyset $, 
we also have $ A= \mathfrak{S} ( \mathtt  A') $  
where $\mathtt A' := (M- R', R+ R') $. 
\end{proof}

The spaces $E_{s,\gso,\su}$ are scales of Banach spaces
\begin{equation}\label{monotoniaEs} 
	\begin{aligned}
		\gso\leq s'_*\,, \ \su'\leq \su\,, \  s'\leq s \quad
		&\Longrightarrow\quad
		E_{s,\gso,\su}\subseteq E_{s', s'_*,\su'}\,,\ \ \ 
		\bnorm{\cdot}^{\g, \cO}_{s', s'_*,\su'}\leq \bnorm{\cdot}^{\g, \cO}_{s,\gso,\su}\, . 
	\end{aligned}
\end{equation}
Moreover in $E_s$ we have a partial ordering relation
\begin{equation}\label{od}
	(M_1,R_1)\preceq (M_2,R_2) \;\Leftrightarrow\;  
	M_1\preceq M_2\,,\quad R_1\preceq R_2\,.
\end{equation}
The next lemma discusses the reality, parity and reversibility
properties of couples. 
\begin{lemma}\label{lemma:generalitaparity}
Assume that $A = \mathfrak{S}({\tA}) $ is matrix of T\"opliz operators in $  \cM^{\mathtt{T}}_{s}\otimes\mathcal{M}_{2}(\C) $, 
for some 
$\tA=(M,R)\in E_{s}\otimes\mathcal{M}_{2}(\C)$,  
real-to-real, parity and reversibility preserving (resp. reversible).
Then there exists $\tA'=(M',R')\in E_{s}\otimes\mathcal{M}_{2}(\C)$
and  a constant $c>0$  such that 
\begin{equation}\label{roses3}
A=\mathfrak{S}({\tA}')\quad \text{ and }\quad
\tfrac{1}{c}\,\bnorm{\tA}_{s}^{\gamma,\calO}\leq 
\bnorm{\tA'}_{s}^{\gamma,\calO}\leq c \,\bnorm{\tA}_{s}^{\gamma,\calO}
\end{equation}
and both  $M',R'$ are
real-to-real, parity and reversibility preserving (resp. reversible).
\end{lemma}

\begin{proof}
By assumption $A $
satisfies \eqref{bolena1}, \eqref{bolena4} and \eqref{bolena3} (resp. \eqref{bolena2}). 
We represent in Fourier the operators
$M$ and $R$  as in \eqref{supermatrice}-\eqref{supermatrice2}.
By assumption we know that $A$ (and hence $M+R$)
is \emph{parity preserving}. Therefore, recalling 
\eqref{bolena4}, we have
\begin{equation}\label{roses2}
M_{\s,j}^{\s',k}(\ell)-M_{\s,-j}^{\s',-k}(\ell) =
R_{\s,-j}^{\s',-k}(\ell)
-R_{\s,j}^{\s',k}(\ell)\,,
\quad
 j,k\in\Z\,,\ell\in\Z^{\nu}\,,\;\s,\s'\in\{\pm\}\,.
\end{equation}
Secondly we define 
\begin{equation}\label{roses}
(M_1')_{\s,j}^{\s',k}(\ell):=\frac{1}{2}\big(M_{\s,j}^{\s',k}(\ell)+M_{\s,-j}^{\s',-k}(\ell)\big)\,,
\qquad
(R_1')_{\s,j}^{\s',k}(\ell):=\frac{1}{2}\big(R_{\s,j}^{\s',k}(\ell)+R_{\s,-j}^{\s',-k}(\ell)\big)\,.
\end{equation}
The operators $M'_1,R'_1$ are parity preserving by construction, 
moreover
by definition, and by \eqref{roses2} we have
$M+R=M'+R'$.
Finally passing to the corresponding operators $\widehat{M}_{1}', \widehat{R}_{1}'$
according to Lemma \ref{rmk:maggiorantevsSupblocchi}
we deduce that $\tA'=(M',R')\in E_{s}\otimes\mathcal{M}_2(\C)$ and satisfies 
\eqref{roses3}.

The next step is to construct $M'_{2}, R_{2}'$ which are parity preserving and real-to-real,
and then, as last step, one has to construct $M'_{2}, R_{2}'$ which are parity preserving, real-to-real
and reversibility preserving (resp. reversible).
This is done reasoning exactly as in the first step.
\end{proof}

The norm \eqref{normaincredibile} controls the action of an operator in Sobolev spaces.
\begin{lemma}{\bf (Action).}\label{rmk:tametresbarre}
	Let $A= \mathfrak{S}(\tA)$ for some $\mathtt{A}\in E_{s,\gso,\su}$ with 
	$\so\leq s\leq \su$, then 
\begin{equation}\label{Auazione}
	\|A u\|^{\gamma,\mathcal{O}}_{s}\leq 	\|\und{A}u\|^{\gamma,\mathcal{O}}_{s} \le  \bnorm{\tA}^{\g, \cO}_{\gso}\|u\|^{\gamma,\cO}_s 
	+ \bnorm{\tA}^{\g, \cO}_s \|u\|^{\gamma,\cO}_{\so}\,.
\end{equation}
\end{lemma}

\begin{proof} 
	Recalling \eqref{normaincredibile} and 
	setting $\tA=(M,R)$  we have, for any $ s_* \leq s \leq s_1  $, 
	\[
	\|\und{A}u\|_{s} \le |M|_{s,s}\|u\|_s + |R|_{\gso,s} \|u\|_{\gso}  
	\le \sup_{\gso\le p\le \su} |M|_{p,p} \|u\|_s  
	+  % |R|_{\gso,s} 
	\bnorm{\tA}_{s}  \|u\|_{s_*}  \leq 
	 \bnorm{\tA}_{s_*} \|u\|_s  
	+  % |R|_{\gso,s} 
	\bnorm{\tA}_{s}  \|u\|_{s_*}  \, . 
	\]
	We recall the definition of the Lipschitz variation \eqref{lip.variation}
	and we define 
	 $ \Delta_{12} \tA := (\Delta_{12} M, \Delta_{12} R)$. 
	 For any  $ s_0 \leq s \leq s_1  $, 
	 since $ s_0 =  s_* + 1 $ we have $ s - 1 \in [s_*,s_1]$ and  the above estimate gives
$$
\begin{aligned}
	 \| \Delta_{12} (\und{A}u) \|_{s-1} 
	 & \le 
	\| ( \Delta_{12}  \und{A}) u \|_{s-1}  + \|  \und{A} \Delta_{12} u \|_{s-1} \\
	&
	\leq 
	 \| \und{( \Delta_{12}  A)} u \|_{s}  + \bnorm{\tA}_{s_*} 
	\| \Delta_{12} u\|_{s-1}  
	+  % |R|_{\gso,s} 
	\bnorm{\tA}_{s}  \| \Delta_{12} u\|_{s_*} 
	 \\
	& \le
	\bnorm{ \Delta_{12} \tA}_{s_*} \|  u\|_s  + 
		\bnorm{ \Delta_{12} \tA}_{s}  \|  u\|_{s_*}    + 
	\bnorm{\tA}_{s_*} \| \Delta_{12} u\|_{s-1}  +  	
	\bnorm{\tA}_{s}  \| \Delta_{12} u\|_{s_*}  
	 \\
	& \le
	\bnorm{ \Delta_{12} \tA}_{s_*} \|  u\|_s  + 
		\bnorm{ \Delta_{12} \tA}_{s}  \|  u\|_{s_0}    + 
	\bnorm{\tA}_{s_*} \| \Delta_{12} u\|_{s-1}  +  	
	\bnorm{\tA}_{s}  \| \Delta_{12} u\|_{s_0-1} \, . 
\end{aligned}
$$
Recalling \eqref{sobolevpesata} and \eqref{normagamma3},
taking the supremum 
on $\oo\in \calO$ (resp. $\oo_{1}\neq\oo_{2}\in\cO$) and summing the last estimates we deduce \eqref{Auazione}. 
\end{proof}

\begin{lemma}{\bf (Algebra and Tame properties).}\label{tretameestimate}
	If $\mathtt{A}_1,\mathtt{A}_2\in E_{s}$
	with $\gso\leq s\leq \su$ then 
	\begin{align}
		& \bnorm{\mathtt{A}_1 \mathtt{A}_2}^{\g, \cO}_{s}
		\leq
		\bnorm{\mathtt{A}_1}^{\g, \cO}_{s}\bnorm{\mathtt{A}_2}^{\g, \cO}_{s} \label{stima:algebrabassa}\,, \\ 
		& \label{stima:tame3b}
		%\qquad\Longrightarrow\qquad 
		\bnorm{\mathtt{A}_1 \mathtt{A}_2}^{\g, \cO}_s
		\leq
		\bnorm{\mathtt{A}_1}^{\g, \cO}_{\gso}\bnorm{\mathtt{A}_2}^{\g, \cO}_s + \bnorm{\mathtt{A}_1}^{\g, \cO}_{s}\bnorm{\mathtt{A}_2}^{\g, \cO}_{\gso} \,.
	\end{align}
\end{lemma}

\begin{proof} 
 Let us prove \eqref{stima:tame3b}. 
  We  first prove the estimate pointwise in $\omega$.  
	Setting $\mathtt{A}_1:=(M_1,R_1)$, 
	$\mathtt{A}_2:=(M_2,R_2)$, by \eqref{A1A2},  %\eqref{tameprop1} 
	\eqref{ABss}, \eqref{normaincredibile} we have,
	for any $ \gso \leq s\leq \su $, 
	\begin{equation}
		\label{vediamo}
	\begin{aligned}
		&\bnorm{(M_1,R_1) \circ (M_2,R_2)}_s
		=  \sup_{\gso\le p\le \su}|M_1M_2|_{p,p} 
		+| M_1 R_2 + R_1 M_2+ R_1 R_2|_{\gso,s}
		\\&
		\qquad \leq 
		 \sup_{\gso\le p\le \su} | M_1|_{p,p}|M_2|_{p,p} + 
		 |M_1|_{s,s} | R_2|_{\gso,s} + | R_1|_{\gso,s} |M_2|_{\gso,\gso} +
		| R_1|_{\gso,s} |R_2|_{\gso,\gso} 
		 \\&
		\qquad \le
		\sup_{\gso\le p\le \su} | M_1|_{p,p} (\sup_{\gso\le p\le \su} |M_2|_{p,p} + | R_2|_{\gso,s} ) +
		| R_1|_{\gso,s} (|M_2|_{\gso,\gso} + |R_2|_{\gso,\gso}  )
		\\&
		\qquad \le 
	  \sup_{\gso\le p\le \su} | M_1|_{p,p} \bnorm{\tA_2}_{s}  
	+   | R_1|_{\gso,s} \bnorm{\tA_2}_{\gso}  
	  \le  \bnorm{\mathtt{A}_1}_{\gso}\bnorm{\mathtt{A}_2}_s + \bnorm{\mathtt{A}_1}_{s}\bnorm{\mathtt{A}_2}_{\gso}\,.
	\end{aligned}
	\end{equation}
	Now, by Leibnitz rule and \eqref{vediamo}  one has
	\[
\begin{aligned}
		&\bnorm {\Delta_{12}(\tA_1\circ\tA_2)}_s \le  \bnorm {(\Delta_{12}\tA_1)\circ\tA_2}_s +  \bnorm {\tA_1\circ (\Delta_{12}\tA_2)}_s \\ &  \le \bnorm {\Delta_{12}\tA_1}_{\gso} \bnorm {\tA_2}_s + \bnorm {\Delta_{12}\tA_1}_{s} \bnorm {\tA_2}_{\gso} + 
		 \bnorm {\tA_1}_{\gso} \bnorm {\Delta_{12}\tA_2}_{s}   +
		  \bnorm {\tA_1}_s \bnorm {\Delta_{12}\tA_2}_{\gso}  \, . 
\end{aligned}
	\]
	Recalling \eqref{normagamma3} 
	taking the supremum on $\oo\in \calO$ (resp. $\oo_{1}\neq\oo_{2}\in\cO$) and summing the last estimates we deduce \eqref{stima:tame3b}.
The estimate 
\eqref{stima:algebrabassa} 
follows analogously.
\end{proof}

Recalling the Definition \ref{defn:japphi} of the commutator $ \langle\td_\vphi\rangle $ 
we define for $\mathtt{A}=(M,R)\in E_s $ the operator
\begin{equation}
\jap{\d_\vphi} \mathtt{A} := (\jap{\d_\vphi} M, \jap{\d_\vphi} R)\, . 
\end{equation}
Note that the operator  $\langle\td_\vphi\rangle$ is compatible with $\mathfrak{S}$, namely
if $A=\mathfrak{S}(\tA)$ then $ \langle\td_\vphi\rangle A=\mathfrak{S}(\langle\td_\vphi\rangle\tA) $. 

\begin{rmk}\label{rmk:sezphi}
In view of Lemma \ref{rmk:tametresbarre} 
and Remark \ref{rem:varphi}, for any $ \varphi \in {\mathbb T}^\nu $, 
	$$
	\|A (\varphi) u\|_{H^s_x}\leq 
	\bnorm{\langle  \td_{\vphi}\rangle^b \tA}_{s_*} 
	\|u\|_{H^s_x} 
	+ 
	\bnorm{\langle \td_{\vphi} \rangle^b \tA}_s
	 \|u\|_{H^{s_*}_x}\,. 
	$$
\end{rmk}

The operator $\jap{\td_{\vphi}}^\tb $, $ \mathtt{b} \geq 0 $, 
satisfies the Leibniz rule for differentiation
\begin{equation}\label{LeibnizdbABprece}
\jap{\td_{\vphi}}^\tb (\mathtt{A}\mathtt{B}) \preceq_\tb
(\jap{\td_{\vphi}}^\tb \mathtt{A}) \mathtt{B}+\mathtt{A}(\jap{\td_{\vphi}}^\tb \mathtt{B})
\end{equation}
recalling \eqref{rmk:propmajNorm} and since, for any $\ell,\ell',\ell_1\in\Z^{\nu}$, 
\[
\jap{\ell-\ell'}^\tb \le (\jap{\ell-\ell_1}+\jap{\ell_1-\ell'})^\tb 
\lesssim_{\mathtt{b}} \jap{\ell-\ell_1}^\tb+\jap{\ell_1-\ell'}^\tb\,.
\]

\begin{lemma}\label{tretameestimate3}
Let $ \mathtt{b} \geq 0 $.
If $\mathtt{A},\mathtt{B},\jap{\td_{\vphi}}^\tb \mathtt{A}, \jap{\td_{\vphi}}^\tb \mathtt{B} \in E_s$ with  $\gso\leq s\leq \su$,  then
 $\jap{\td_{\vphi}}^\tb (\mathtt{A} \mathtt{B})\in E_s$ and 
\begin{equation*}
\bnorm{\jap{\td_{\vphi}}^\tb (\mathtt{A} \mathtt{B})}^{\g, \cO}_{s}\lesssim_{\mathtt{b}}
\bnorm{\jap{\td_{\vphi}}^\tb \mathtt{A}}^{\g, \cO}_{s} \bnorm{\mathtt{B}}^{\g, \cO}_{\gso} 
+ \bnorm{\jap{\td_{\vphi}}^\tb \mathtt{A}}^{\g, \cO}_{\gso} \bnorm{\mathtt{B}}^{\g, \cO}_{s}
+ \bnorm{\jap{\td_{\vphi}}^\tb \mathtt{B}}^{\g, \cO}_{s} \bnorm{\mathtt{A}}^{\g, \cO}_{\gso} 
+ \bnorm{\jap{\td_{\vphi}}^\tb \mathtt{B}}^{\g, \cO}_{\gso} \bnorm{\mathtt{A}}^{\g, \cO}_{s}\,.
\end{equation*}
%The same bound holds also for the norm $\bnorm{\cdot}_{s}^{\gamma,\mathcal{O}}$. 
\end{lemma}

\begin{proof}
By \eqref{LeibnizdbABprece} and  Lemma \ref{tretameestimate}. 
\end{proof}

Applying iteratively 
Lemmata \ref{tretameestimate}, 
\ref{tretameestimate3} % and  \eqref{quo} 
we deduce that, 
for any $k\geq 2$, $\mathtt{b}\geq 1$ and $\gso\leq s\leq \su $, 
\begin{align}
\bnorm{\mathtt{A}^k}^{\g, \cO}_{\gso}&\le
(\bnorm{\mathtt{A}}^{\g, \cO}_{\gso})^{k}\,,
\qquad
\bnorm{\mathtt{A}^k}^{\g, \cO}_{s}\le
2^{k-2}\, k
(\bnorm{\mathtt{A}}^{\g, \cO}_{\gso})^{k-1} \bnorm{\mathtt{A}}^{\g, \cO}_{s}\label{quo}\,,\\
\bnorm{\jap{\td_{\vphi}}^\tb \mathtt{A}^{k}}^{\g, \cO}_{s}&\lesssim_{\tb} 
2^{k-1 }\,k\Big(\bnorm{\jap{\td_{\vphi}}^\tb \mathtt{A}}^{\g, \cO}_{s} (\bnorm{\mathtt{A}}^{\g, \cO}_{\gso})^{k-1}
+ \bnorm{\jap{\td_{\vphi}}^\tb \mathtt{A}}^{\g, \cO}_{\gso} \bnorm{A}^{\g, \cO}_{s}(\bnorm{\mathtt{A}}^{\g, \cO}_{\gso})^{k-2}\Big)\label{qua}\,.
\end{align}

\begin{lemma}{\bf (Invertibility).}\label{inv3sbarrette}
Let $\mathtt{A}\in E_{s}$ of the form $\mathtt{A}= \id + \mathtt{Q}$ with $\jap{\td_\vphi}^\tb \mathtt{Q}\in E_s$. There exists a constant $c(\tb)>0$ such that, 
if 
\begin{equation}\label{paperone}
c(\tb)\bnorm{\mathtt{Q}}^{\g, \cO}_{\gso}<1 \, , 
\end{equation}
then $\mathtt{A}$ is invertible and for $\su\geq s\ge \gso$
\begin{align}\label{ipiuq}
\bnorm{ \mathtt{A}^{-1}-\id}^{\g, \cO}_s  
&\le 
\bnorm{\mathtt{Q}}^{\g, \cO}_{s} (1+ \bnorm{\mathtt{Q}}^{\g, \cO}_{\gso}) 
\\ 
\bnorm{\jap{\td_\vphi}^\tb \mathtt{A}^{-1}-\id}^{\g, \cO}_s 
&\lesssim_{\tb} 
 \bnorm{\jap{\td_\vphi}^\tb \mathtt{Q}}^{\g, \cO}_{\gso}
 \bnorm{\mathtt{Q}}^{\g, \cO}_{s} (1+ \bnorm{\mathtt{Q}}^{\g, \cO}_{\gso})  
 +  \bnorm{\jap{\td_\vphi}^\tb \mathtt{Q}}^{\g, \cO}_{s}
 \bnorm{\mathtt{Q}}^{\g, \cO}_{\gso} (1+ \bnorm{\mathtt{Q}}^{\g, \cO}_{\gso})  \,. \label{ipiuq1}
\end{align}
Moreover, the estimates \eqref{ipiuq}-\eqref{ipiuq1}  hold for $\tA=e^{\tQ} - \id $. 
Finally, if $Q=\mathfrak{S}(\tQ)$ then $e^{Q}=\mathfrak{S}(e^{\tQ})$.
\end{lemma}

\begin{proof}
By a  Neumann series argument 
using  \eqref{quo}, \eqref{qua} and the smallness condition \eqref{paperone}. 
The last statement follows from the fact that $\mathfrak{S}$ is a linear homomorphism.
\end{proof}

For $\tA=(M,R)\in E_s$ and $n_1,n_2\in \Z$ we set
\begin{equation}
	\label{derivata}
	\jap{D}^{n_1}\tA \jap{D}^{n_2} := (\jap{D}^{n_1}M \jap{D}^{n_2}, \jap{D}^{n_1}R \jap{D}^{n_2}  )\,
\end{equation}
so that 
$ \fS\big(\jap{D}^{n_1}\tA \jap{D}^{n_2}\big)= \jap{D}^{n_1}\fS(\tA) \jap{D}^{n_2} $. 

Recalling  the definition \eqref{def:proj} of the projector $ \Pi_N $   
we define for $\mathtt{A}=(M,R) $ in $ E_s $ the operator
\begin{equation}\label{projtresbarrette}
  \Pi_N \mathtt{A} := (\Pi_N M, \Pi_N R)\,,\quad \Pi_N^\perp := \id - \Pi_N \, . 
\end{equation}
Note that the projector $\Pi_{N}$ is compatible with $\mathfrak{S}$, namely
if $A=\mathfrak{S}(\tA)$ then
$ \Pi_{N}A=\mathfrak{S}(\Pi_{N}\tA) $. 

\begin{lemma}{\bf (Smoothing).}\label{dito}
For $\mathtt{A}\in E_s$ we have, for any  $ N \in \N$ and any $\gso\leq s\leq \su$ %(see Def.  \ref{Es}) 
\begin{equation*}
\begin{aligned}
\bnorm{\Pi^\perp_N \mathtt{A}\jap{D}^m}^{\g, \cO}_s&\le N^{-\mathtt{b}}\bnorm{\jap{\d_{\vphi}}^\mathtt{b} \mathtt{A}\jap{D}^m}^{\g, \cO}_s\,,
\quad \mathtt{b},m\in \N_0\,.
%\\
%\bnorm{\langle \d_{\vphi}\rangle^{\tb}\Pi^\perp_N \mathtt{A}}^{\g, \cO}_s&\le \bnorm{\jap{\d_{\vphi}}^\mathtt{b} \mathtt{A}}^{\g, \cO}_s
\end{aligned}
\end{equation*}
\end{lemma}
\begin{proof}
It follows by  \eqref{def:proj} and  \eqref{normaincredibile}.
\end{proof}

For any  $\mathtt{A}, \mathtt{B} \in E_s $ we define 
the adjoint action of $\mathtt{A}$ on $\mathtt{B}$ as (recall \eqref{A1A2})
\[
{\rm ad}_{\mathtt{A}}[\mathtt{B}]:=[\mathtt{A},\mathtt{B}]:=\mathtt{A}\circ\mathtt{B}- \mathtt{B}\circ\mathtt{A}\,,
\]
and consequently
for  $k\geq 0$ we define 
\[
{\rm ad}_{\mathtt{A}}^{k}[\mathtt{B}]:=\mathtt{A}\circ{\rm ad}_{\mathtt{A}}^{k-1}[\mathtt{B}]
-{\rm ad}_{\mathtt{A}}^{k-1}[\mathtt{B}]\circ \mathtt{A} \,,
\qquad {\rm ad}_{\mathtt{A}}^{0}[\mathtt{B}]:=\mathtt{B}\,.
\] 
\begin{lemma}{\bf (Adjoint action).}\label{dito2}
For any $ \mathtt{A}, \mathtt{B} \in E_s  $ and $ k \geq 1, m\geq 0 $ it results for any $\gso\leq s\leq \su$
	\begin{equation}\label{conlaD}
	\begin{aligned}
			&\bnorm{{\rm ad}_{\mathtt{A}}^k[\mathtt{B}]\jap{D}^m}^{\g, \cO}_s  \le \\
			&
		2^k\Big((\bnorm{\mathtt{A}\jap D^{m}}^{\g, \cO}_{\gso})^k \bnorm{\mathtt{B}\jap{D}^m}^{\g, \cO}_s 
		+k(\bnorm{\mathtt{A}\jap{D}^m}^{\g, \cO}_{\gso})^{k-1}\bnorm{\mathtt{A}\jap{D}^m}^{\g, \cO}_s\bnorm{\mathtt{B}\jap{D}^m}^{\g, \cO}_{\gso}\Big) 
	\end{aligned}
	\end{equation}
and for all $\tb\in \N$
\begin{equation}\label{baba}
\begin{aligned}
\bnorm{ \jap{\d_\vphi}^\tb {\rm ad}_{\mathtt{A}}^k[\mathtt{B}] \langle D\rangle^m}^{\g, \cO}_s 
&\le
2^{k(\tb+1)} \,k
\bnorm{ \mathtt{A}\langle D\rangle^m}^{\g, \cO}_{\gso}
(\bnorm{\mathtt{A}\jap{D}^m}^{\g, \cO}_{\gso})^{k-1}\bnorm{\mathtt{B}\jap{D}^m}_{s}   	
\\&
+
2^{k(\tb+1)} \,k 
\bnorm{\jap{\d_\vphi}^\tb \mathtt{A}\jap{D}^m }^{\g, \cO}_{s}
(\bnorm{\mathtt{A}\jap{D}^m}^{\g, \cO}_{\gso})^{k-1}
\bnorm{\mathtt{B}\jap{D}^m}^{\g, \cO}_{\gso}\\
&+ 
2^{k(\tb+1)} \,k 
\bnorm{\mathtt{A}\jap{D}^m}^{\g, \cO}_{s}
(\bnorm{\mathtt{A}\jap{D}^m}^{\g, \cO}_{\gso})^{k-1}
\bnorm{\jap{\d_\vphi}^\tb  \mathtt{B}\jap{D}^m}^{\g, \cO}_{\gso}
\\&
+ 2^{k(\tb+1)} k(k-1) \bnorm{\mathtt{A}\jap{D}^m}^{\g, \cO}_{s}(\bnorm{\mathtt{A}\jap{D}^m}^{\g, \cO}_{\gso})^{k-2}
\bnorm{\jap{\d_\vphi}^\tb  \mathtt{A}\jap{D}^m}^{\g, \cO}_{\gso}\bnorm{\mathtt{B}\jap{D}^m}^{\g, \cO}_{\gso} 
\\ &+ 2^{k(\tb+1)}(\bnorm{\mathtt{A}\jap{D}^m}^{\g, \cO}_{\gso})^{k}\bnorm{\jap{\d_\vphi}^\tb  \mathtt{B}\jap{D}^m}^{\g, \cO}_{s}\,.
\end{aligned}
\end{equation}

Finally  if $A=\mathfrak{S}(\tA)$ and  $B=\mathfrak{S}(\tB)$ then we have
$ {\rm ad}_{A}^{k}[B]=\mathfrak{S}\big({\rm ad}_{\mathtt{A}}^{k}[\mathtt{B}]\big) $, $  k\geq 0 $.
\end{lemma}

\begin{proof}\label{app:tech1}
We set $\tC_k:= \ad_{\tA}^{k-1}[\tB]$ and note that (recall \eqref{rmk:propmajNorm}) 
\[
{\rm ad}_{\mathtt{A}}^k[\mathtt{B}]\jap{D}^m= \ad_\tA( \tC_k)  \jap{D}^m \preceq \tA \jap{D}^m \tC_k \jap{D}^m + \tC_k \jap{D}^m \tA \jap{D}^m\,,
\]
thus \eqref{conlaD} with $k=1$ follows by Lemma \ref{tretameestimate}, then the general bound follows by induction.
Regarding \eqref{baba} we note that, by the Leibniz rule \eqref{LeibnizdbABprece}
\[
\begin{aligned}
	\jap{\d_\vphi}^\tb{\rm ad}_{\mathtt{A}}^k[\mathtt{B}]\jap{D}  = 
 	\jap{\d_\vphi}^\tb \ad_\tA( \tC_k)  \jap{D}
	&  \preceq (\jap{\d_\vphi}^\tb \tA\jap{D} )\tC_k \jap{D} + \tA\jap{D}  (\jap{\d_\vphi}^\tb \tC_k \jap{D})
	\\
 & \ +	( \jap{\d_\vphi}^\tb\tC_k \jap{D}) \tA \jap{D} + \tC_k \jap{D} ( \jap{\d_\vphi}^\tb\tA \jap{D})\,,
\end{aligned}
\]
and then we argue as for  \eqref{conlaD}.
The last statement follows because $\mathfrak{S}$ is a linear homomorphism.
\end{proof}

\begin{rmk}
Lemmata \ref{sergiu}--\ref{dito2} hold verbatim on $E_s\otimes \cM_2(\C)$
 with possibly larger pure constants.
\end{rmk}

We now prove the important result that we can associate to 
a pseudo-differential operator
 a Bony-smoothing couple according to
Definition \ref{Es}.

\begin{prop}{\bf (Bony-smoothing couple of a Pseudo-differential operator).}\label{prop:immersionepseudo}
Let $\gso\le s\le \su$. 
Consider a symbol $a\in S^{m}$ for some $m\in \R$.  
For any $n_1,n_2\in \R$ with  $n_1+n_2=m$
there exists  $\mathtt{L}\in E_s $ 
such that 
$$
\mathfrak{S} (\mathtt{L}) =L \, , \qquad 
L:=\langle D\rangle^{-n_1}\op(a)\langle D\rangle^{-n_2} \, , 
$$ 
satisfying, for any $\mathtt{b}\in \N_0 $,
\begin{equation}\label{stimapseudoES}
\bnorm{\langle \td_{\vphi}\rangle^{\mathtt{b}}\mathtt{L}}_{s,\gso,\su}^{\gamma,\mathcal{O}}
\lesssim_{s,\su,n_1,\tb} \|a\|_{m,s+\gso+|n_1|+\tb+1,0}^{\gamma,\mathcal{O}}\,.
\end{equation}
The constants 
 in \eqref{stimapseudoES} are non-decreasing in the parameter $ |n_1|$.
\end{prop}
The rest of the section is devoted to the proof of Proposition \ref{prop:immersionepseudo}.

We first introduce operators with ``off-diagonal decay'', cfr. \cite{BBM14}.  
\begin{defn}{\bf (Decay norm).}\label{decay}
We define the $s$-\emph{decay norm} of a T\"opliz  in time operator 
 $A $ 
as
\begin{equation}\label{normadecay}
|A|_{s}^{\rm dec}:=\Big(
\sum_{p\in\Z^{\nu},h\in\Z}\langle p,h\rangle^{2s}
\sup_{\substack{ j-j'=h \\ \ell-\ell'=p} }|A_{j}^{j'}(\ell-\ell')|^2
\Big)^{1/2}
\end{equation}
and we define the weighted  norm of  a Lipschitz family of  T\"opliz
operators $ \omega \mapsto A(\omega) $,  % \in \mathcal{L}^{\mathtt{T}}(H^{s},H^{s}) $,
$\omega \in  \mathcal{O} $,  for any $ \gamma \in (0,\tfrac{1}{2}) $, as
\[
|A|_{s}^{{\rm dec},\gamma,\mathcal{O}}:=
\sup_{\omega\in\mathcal{O}}|A(\omega)|_{s}^{\rm dec}+\gamma
\sup_{\substack{ \omega_1\neq\omega_2 \\
\omega_1, \omega_2 \in \cO}}\frac{|A(\omega_1)
-A(\omega_2)|_{s}^{\rm dec}}{|\omega_1-\omega_2|}\,.
\]
\end{defn}

An operator $A$ with 
finite $s$-decay norm is majorant bounded from $H^s$ to $H^{s}$.

\begin{lemma}\label{lem:decay-majorant}
Let $s\geq \gso$. For any  T\"opliz in time operator 
 $A $ we have 
\begin{equation}\label{ammazza}
|A|^{\g, \cO}_{s,s}\leq C_2(s) |A|_s^{{\rm dec},\gamma,\mathcal{O}}\,.
\end{equation}
\end{lemma}

\begin{proof}
We  prove  \eqref{ammazza} for the sup norm. The estimate for the Lipschitz variation follows similarly.
Defining  the function
$ \ta := \sum_{p \in {\mathbb Z}^\nu,h\in {\mathbb Z}}\ta_{h}(p)e^{\ii p\cdot\vphi+\ii hx}  $ where $
\ta_{h}(p):=\sup_{\substack{ j-j'=h \\ \ell-\ell'=p} }|A_{j}^{j'}(\ell-\ell')| $, 
we have $\| \ta \|_{s}=|A|_{s}^{\rm dec}$, recalling \eqref{sobspace}. 
Given $u\in  H^s,$ defining $\underline u\in  H^s$
through $\underline u_{\ell,j} := |u_{\ell,j}| $, we have 
\[
\begin{aligned}
\| \und{A} u\|_s^2&\leq \sum_{\ell\in\Z^{\nu},j\in\Z} 
\langle \ell,j\rangle^{2s} 
\Big(\sum_{\ell'\in\Z^{\nu}, j'\in\Z} |\ta_{j-j'}(\ell-\ell')| |u_{\ell',j'}|\Big)^{2}
\\&=\| \ta\cdot\und{u}\|_{s}^{2}
\stackrel{ \eqref{tameProduct}}{\lesssim_{s}}\|\mathtt a\|_s^{2} \| \underline u\|_s^{2}
\lesssim_{s}( |A|_{s}^{\rm dec})^{2}\|u\|^{2}_{s} 
\end{aligned}
\]
by \eqref{maju}.  This implies $|A|_{s,s}\leq C(s) |A|_s^{\rm dec} $. 
\end{proof}

In order to prove that an 
operator $A$ with finite $s$-decay norm can be associated to a couple in
$E_{s}$ with good estimates on the norm, we first give the following definition.
\begin{defn}{\bf (Time/space Bony decomposition).}\label{def:bonydeco}
For any T\"opliz in time operator  $ L \equiv 
\big( L_{j}^{j'}(\ell-\ell')\big)_{j,j'\in\Z,\ell,\ell'\in\Z^{\nu}} $ 
we define the operator 
$ L^{B} :=\big((L^{B})_{j}^{j'}(\ell-\ell')\big)_{j,j'\in\Z,\ell,\ell'\in\Z^{\nu}} $ 
with matrix entries 
\[
(L^B)_{j}^{j'}(\ell-\ell'):= 
\left\{\begin{aligned}
& L_{j}^{j'}(\ell-\ell') \;\;\; \ \, \mbox{if } \;\;\; |\ell-\ell'| + |j-j'| < \tfrac{1}{2}(|\ell|+|j|) 
\\
&0 \quad \qquad  \qquad \mbox{otherwise}  
\end{aligned}\right.
\]
and $L^{U}:=L-L^{B}$.
The operator $L^{B}$ is called the \emph{Bony} part of $L$
and  $L^{U}$ the \emph{Ultraviolet} part of $L$.  
\end{defn}

The majorant operatorial norm of the Bony part $L^B $ is controlled for any $ s $ by the
majorant operatorial norm $ |L|_{\gso,\gso} $, whereas 
the ultraviolet part $L^{U}$ of an operator $ L $ with finite $s$-decay norm 
is actually a ``smoothing'' operator.

\begin{lemma}\label{lem:ilfreddo}
There exists $ C(\gso) > 0 $ such that, for any $s\geq \gso$ 
and T\"opliz in time operator $L$  we have
\begin{equation}\label{Katsendorff}
 |L^B|^{\g, \cO}_{s,s} \le 3^{s-\gso}|L|^{\g, \cO}_{\gso,\gso}\,,\qquad 	\qquad
|L^U|^{\g, \cO}_{\gso,s} 
\le 
4^s C(\gso)|L|_s^{{\rm dec},\gamma,\mathcal{O}}\, . 
\end{equation}
\end{lemma}

 \begin{proof} We prove the bounds in \eqref{Katsendorff} for the sup norm. The estimate for the Lipschitz variation follows verbatim.
For any $u\in  H^s$,  $ s \geq s_* $, 
noting than on the support of $ L^B $ one has $ \langle \ell, j \rangle 
 \leq 3 \langle \ell', j' \rangle $,   we have
\[
\begin{aligned}
\|\und{L^{B}} \, u\|_{s}^{2}
&\leq
\sum_{\ell\in\Z^{\nu},j\in\Z}\langle\ell,j\rangle^{2s}
\Big(
\sum_{|\ell-\ell'| + |j-j'| < \tfrac{1}{2}(|\ell|+|j|)}
|L_{j}^{j'}(\ell-\ell')||u_{\ell',j'}|
\Big)^{2}
\\&
\leq 3^{2(s-\gso)}
\sum_{\ell\in\Z^{\nu},j\in\Z}\langle\ell,j\rangle^{2\gso}
\Big(
\sum_{|\ell-\ell'| + |j-j'|< \tfrac{1}{2}(|\ell|+|j|)}
|L_{j}^{j'}(\ell-\ell')|\langle \ell',j'\rangle^{s-\gso}|u_{\ell',j'}|
\Big)^{2}
\\&
\leq 3^{2(s-\gso)}
\sum_{\ell\in\Z^{\nu},j\in\Z}\langle\ell,j\rangle^{2\gso}
\Big(
\sum_{\ell'\in\Z^{\nu},j'\in\Z}
|L_{j}^{j'}(\ell-\ell')|\langle \ell',j'\rangle^{s-\gso}|u_{\ell',j'}|
\Big)^{2}
\\&  = 3^{2(s-\gso)}\big\|
\und{L}{(\langle {\bf D}\rangle^{s-\gso} \underline u)} \big\|_{\gso}^{2}\leq 3^{2(s-\gso)}
|L|_{\gso,\gso}^2 \big\| \langle {\bf D}\rangle^{s-\gso} \underline u \big\|_{\gso}^2 =
3^{2(s-\gso)}
|L|_{\gso,\gso}^2 \| u \|_{s}^2
\end{aligned}
\]
proving the first estimate in \eqref{Katsendorff}.

Let us prove the second estimate in \eqref{Katsendorff}.
Given $u\in  H^{\gso}$, 
noting than on the support of $ L^U $ one has $ \langle \ell, j \rangle 
 \leq 4 \langle \ell-\ell', j-j' \rangle $,   we have,
using also the Cauchy-Schwarz inequality,
 \[
 \begin{aligned}
 \|\und{L^{U}}u\|_{s}^2 &\leq
 \sum_{\ell\in\Z^{\nu},j\in\Z}\langle\ell,j\rangle^{2s}
\Big(
\sum_{|\ell-\ell'|+|j-j'|\geq \tfrac{1}{2}(|\ell|+|j|)}
|L_{j}^{j'}(\ell-\ell')||u_{\ell',j'}|
\Big)^{2}
\\&
\leq 4^{2s}
\sum_{\ell\in\Z^{\nu},j\in\Z}
\Big(
\sum_{\ell'\in\Z^{\nu},j'\in\Z}
\langle\ell-\ell',j-j'\rangle^{s}
|L_{j}^{j'}(\ell-\ell')| \langle\ell',j'\rangle^{\gso}|u_{\ell',j'}|
\frac{1}{\langle\ell',j'\rangle^{\gso}} 
\Big)^{2} \\
&\leq 4^{2s} C(\gso)^2 \sum_{\ell',j'} 
\langle\ell',j'\rangle^{2\gso} |u_{\ell',j'}|^2
\sum_{\ell\in\Z^{\nu}, j\in\Z}
\langle\ell-\ell',j-j'\rangle^{2s}
|L_{j}^{j'}(\ell-\ell')|^2 \\
&\leq 4^{2s} C(\gso)^2 \sum_{\ell',j'} 
\langle\ell',j'\rangle^{2\gso} |u_{\ell',j'}|^2  (|L|_{s}^{\rm dec})^2
\leq 4^{2s} C(\gso)^2 \|u\|_{\gso}^{2}(|L|_{s}^{\rm dec})^2
 \end{aligned}
 \]
proving the second estimate in \eqref{Katsendorff}.
\end{proof}

As a corollary we get  the following result.

\begin{lemma}\label{lem:decay-3sbarre}
For any
  T\"opliz in time operator $L $ 
the couple $\mathtt{L}:=(L^B,L^U)$ introduced in Definition \ref{def:bonydeco}
 belongs to the space  
$ E_s$  introduced in Definition \ref{Es},
satisfies  $\mathfrak{S}(\mathtt{L})=L$ and
\begin{equation}\label{edgard}
\bnorm{\mathtt{L}}^{\g, \cO}_{s,\gso,\su}\leq c(s,\gso,\su) |L|_{s}^{{\rm dec},\gamma,\mathcal{O}}
\end{equation}
with $c(s,\gso,\su):= 3^{\su-\gso} C_{2}(\gso)+ 4^s C(\gso)$.
\end{lemma}

\begin{proof}
By Def. \ref{def:bonydeco} we have  
$\mathfrak{S}(\mathtt{L})=L$.
By 
 \eqref{Katsendorff}
we get
$ \bnorm{\mathtt{L}}_{s,\gso,\su}^{\gamma,\cO}
\leq 3^{\su-\gso} |L|_{\gso,\gso}^{\gamma,\cO} +
4^s C(\gso) 
|L|^{\rm dec,\gamma,\cO}_{s} $
and by \eqref{ammazza} we deduce \eqref{edgard}.
\end{proof}

Pseudo-differential operators with symbols in $S^{m}$
have a finite decay norm. 
\begin{lemma}\label{lem:pseudodecay}
Consider a symbol  $ a\in S^{m}$ for some $m\in\R$.
Then for any $n_1,n_2\in \R$ such that $n_1+n_2=m$ 
and any $ \mathtt b\in\N_0 $ 
\begin{equation}\label{trillo}
|\langle D\rangle^{-n_1}\jap{\td_{\vphi}}^{\mathtt b} \op(a) 
\langle D\rangle^{-n_2}|_{s}^{{\rm dec},\gamma,\mathcal{O}}
\lesssim_{s,n_1,\mathtt b}\|a\|^{\g, \cO}_{m,s+\gso+|n_1|+ \mathtt b+1,0}\,.
\end{equation}
\end{lemma}

\begin{proof}
In view of \eqref{aFourierphix} and 
recalling \eqref{normaSymbolo}, 
for any $\ell\in\Z^{\nu}$, $h\in\Z$, $\x\in\R$,  
%\[
%\doublehat{a}(\ell,h,\x):=\frac{1}{(2\pi)^{\nu+1}}\int_{\mathbb{T}^{\nu+1}}
%a(\vphi,x,\x)e^{-\ii\ell\cdot\vphi}e^{-\ii hx}d\vphi dx\, , 
%\]
 % for any $s\geq 0$,
\begin{equation}\label{chiarodiluna}
\langle \ell ,h\rangle^{s+s_*}
| \doublehat{a}(\ell,h,\x)|\lesssim_{s}\langle \x\rangle^{m}\|a\|_{m,s+s_*,0}\,.
\end{equation}
The matrix elements of 
$P:=\langle D\rangle^{-n_1}\jap{\td_{\vphi}}^{\mathtt b} \op(a)  
\langle D\rangle^{-n_2}$ are, 
for any $\ell,\ell'\in\Z^{\nu}$, $j,j'\in\Z $, 
\[
P_{j}^{j'}(\ell-\ell')=  \jap{\ell -\ell}^{\mathtt b}
\doublehat{a}(\ell-\ell',j-j',j')\langle j\rangle^{-n_1}
\langle j' \rangle^{- n_2 } \, .
\] 
Recalling that $   - n_2 = n_1 - m $ and using Peetre's inequality 
$ \langle j' \rangle^{n_1} 
\leq C \langle j \rangle^{n_1} \langle j - j' \rangle^{|n_1|} $, we have 
\[
\begin{aligned}
|P_{j}^{j'}(\ell-\ell')|&\lesssim\jap{\ell-\ell'}^{\mathtt b}
|\doublehat{a}(\ell-\ell',j-j',j')|\langle j'\rangle^{-m} 
\frac{\langle j' \rangle^{n_1}}{\langle j\rangle^{n_1}}
\\&
\lesssim\langle\ell-\ell',j-j'\rangle^{ \mathtt b+|n_1|}
|\doublehat{a}(\ell-\ell',j-j',j')|\langle j'\rangle^{-m} \,.
\end{aligned}
\]
Thus multiplying by $\langle \ell-\ell',j-j'\rangle^{s}$ and using estimate \eqref{chiarodiluna}
we get
\[
\langle \ell-\ell',j-j'\rangle^{s}|P_{j}^{j'}(\ell-\ell')|\lesssim_{s,n_1,\mathtt b}
\|a\|_{m,s+\gso+\mathtt b+|n_1|,0} \langle \ell-\ell',j-j'\rangle^{-s_*} 
\]
which, recalling \eqref{normadecay},  
implies the bound \eqref{trillo} for the sup norm.   The estimate for the Lipschitz norm follows
similarly recalling that by Definition \ref{norma pesata} 
$ \frac{\|a(\omega_1) -a(\omega_2)\|_{m,s,0}}{|\omega_1-\omega_2|} \le \|a\|^{\g,\cO}_{m,s+1,0} $. 
This concludes the proof.
\end{proof}

\begin{proof}[Conclusion of the proof of Proposition \ref{prop:immersionepseudo} ]
Apply
Lemma \ref{lem:pseudodecay} and  Lemma \ref{lem:decay-3sbarre}.
\end{proof}

\section{The torus diffeomorphism}\label{sec:Calpha}

In this section we study the connection between a  diffeomorphisms of the torus  and
Bony-smoothing couples. 
We consider a $ \tau $-dependent family of torus diffeomorphisms 
\begin{equation}\label{def:Calpha}
x \mapsto y=x+ \tau \alpha(\vphi,x)\,,\quad x\in \T\,,\;\;\; \vphi\in \T^{\nu}\,, \;\; \tau\in[0,1] \, , 
\end{equation}
where $ \alpha(\vphi,x) $ 
is a function in  $ C^{\infty}(\mathbb{T}^{\nu+1},\R) $,   
and we denote its inverse diffeomorphisms of $ \T $ by 
\begin{equation}\label{def:Calphainv}
y\mapsto x=y+\breve{\alpha}(\tau;\vphi,y) \,.
\end{equation}
Accordingly we consider  the associated composition operators
\begin{equation}\label{ignobel}
\mathcal{C}_{\alpha}^{\tau}h(\varphi, x):=h(\varphi, x+\tau\alpha(\varphi, x))\,, 
\quad 
(\mathcal{C}_{\alpha}^{\tau})^{-1}h(\varphi, y):=
h(\varphi, y+\breve{\alpha}(\tau; \varphi, y))\, , 
\end{equation}
and we  denote  $\mathcal{C}_{\alpha}:=\mathcal{C}_{\alpha}^{1}$.

The main result of the section is the following.
\begin{prop}{\bf (Bony-smoothing couple associated to a Torus diffeomorphism).}\label{inclusionetotale}
Let 
$\alpha(\vphi,x) := \alpha(\omega;\vphi,x)$ be a  family of functions 
in $C^{\infty}(\mathbb{T}^{\nu+1},\R)$,  
Lipschitz  in  $\omega\in \mathcal{O} \subset\R^{\nu}  $. 
Let $\su\ge s\geq \gso$ and fix $\mathtt{b}>0$.

For any $N_1,N_2\in \R$ such that $N_1+N_2>[\nu/2]+3+\tb$, 
there exists $\mu :=\mu( N_1,N_2,\tb) > 0 $ and $\delta:=\delta(s,N_1,N_2,\tb) > 0 $
such that if  (recall that $\so =  \gso +1$)
\begin{equation}\label{buttalapasta3}
\|\alpha\|^{\gamma, \cO }_{\so+\mu}\leq \delta  \, , 
\end{equation}
then
there exists a Bony-smoothing couple  $\tP^{\tau}\in E_{s}$, $ \tau \in [0,1] $
(see Def. \ref{Es}), 
such that 
\[
\mathfrak{S}(\tP^{\tau})=P^{\tau} \, , \quad 
P^{\tau}:=\langle D\rangle^{-N_1}\circ\big(\mathcal{C}_{\alpha}^{\tau}- \id \big)
\circ\langle D\rangle^{-N_2}\, , 
\]
satisfying 
\begin{align}
%\bnorm{\tP^{\tau}}^{\gamma,\calO}_{\gso}&\lesssim_{\su, N_{1}, N_{2}}
%\|\al\|_{\so+\mu}^{\gamma,\calO}\,,
%2^{\su+\gso+2}|\al|_{\gso+\mu}\,,
%\label{chiarodiluna1}
%\\
\bnorm{\jap{\td_{\vphi}}^{q} \tP^{\tau}}^{\gamma,\calO}_s 
&\lesssim_{s,\su,q, N_{1}, N_{2}} 
\|\al\|_{s+\mu}^{\gamma,\calO}\,,\quad q=0,\tb \, , 
\label{chiarodiluna2}
\end{align}
uniformly in $\tau\in[0,1]$. If $|N_1|, |N_2|\leq N$ then $\delta,\mu$ 
and the constant in \eqref{chiarodiluna2}
depend only on $N$.
\end{prop}

The rest of the section is devoted to the proof of  Proposition \ref{inclusionetotale}

We first recall the following tame 
estimates satisfied by $\mathcal{C}_{\alpha}^{\tau}$, provided 
by \cite{FGP19}[Lemmata B.7,B.8]. 
\begin{lemma}\label{bastalapasta}
There is $\s>0$ such that 
for any 
Lipschitz family 
$\alpha(\vphi,x)=\alpha(\omega;\vphi,x)$
in $C^{\infty}(\mathbb{T}^{\nu+1},\R)$  defined for $\omega\in \mathcal{O}\subset \R^{\nu}$ the following holds.
For any $s\geq \so$ there $c(s)>0$
such that 
if $c(s)\| \alpha\|^{\gamma, \cO}_{\so+\s}\leq 1$,
then
 for any $u=u(\oo) \in H^{s}$ Lipschitz in 
$\oo\in \mathcal{O}$,
one has
\begin{equation}\label{pasta7}
\begin{aligned}
&\sup_{\tau\in[0,1]} 
\lVert \mathcal{C}_{\alpha}^{\tau} u \rVert_s^{\gamma,  \cO }
\lesssim_{s}
\|u\|_{s}^{\gamma,  \cO }
+\| \alpha\|_{s+\s}^{\gamma,  \cO }\|u\|_{\so}^{\gamma,  \cO }\,,
\end{aligned}
\end{equation}
and for any $ m_1 + m_2 = 1 $, 
\begin{equation}\label{Amenouno}
\begin{aligned}
&\sup_{\tau\in[0,1]} 
\lVert  \langle D \rangle^{-m_1} (\mathcal{C}_{\alpha}^{\tau}-\Id)\langle D \rangle^{-m_2} u \rVert_s^{\gamma,  \cO }
\lesssim_{s,m_1,m_2}
\lVert \alpha \rVert_{\so+\s}^{\gamma,  \cO } 
\|u\|_{s}^{\gamma,  \cO }+\| \alpha\|_{s+\s}^{\gamma,  \cO }\|u\|^{\gamma,  \cO }_{\so}\,.
\end{aligned}
\end{equation}
Moreover, 
for any $q\in\mathbb{N}_0^{\nu}$, $i=0,1$ such that  $|q|+i>0$, for any 
$m_1,m_{2}\in \mathbb{R}$ such that $m_1+m_{2}=|q|+i$, 
for any  $s\geq \so$ there exist  constants
$\mu=\mu(|q|,m_1,m_2)$ and $\delta=\delta(m_1,s)$
such that if 
$ \|\alpha\|_{2\so+|m_1|+2}\leq \delta $, 
$ \|\alpha\|^{\gamma, \cO }_{\so+\mu}\leq1 $, 
then (recall \eqref{def:commuDD}) for any $w\in H^s $,
\begin{equation}\label{pasta1} 
\sup_{\tau\in[0,1]}
\|
\langle D\rangle^{-m_1}
\td_{\vphi}^{q}\td_{x}^{i}\mathcal{C}_{\alpha}^{\tau}
\langle D\rangle^{-m_2}
w\|^{\gamma, \cO }_{s}
\lesssim_{s, |q|,m_1,m_2}
\| w \|_{s}\|\alpha\|^{\gamma, \cO }_{\so+\mu}
+\|\alpha\|^{\gamma, \cO }_{s+\mu}\| w\|_{\so}\, . 
\end{equation}

\noindent
Finally, for $s\ge \gso$ we have that
\begin{equation}\label{panelle} 
\sup_{\tau\in[0,1]}
\|
\langle D\rangle^{-m_1}
\td_{\vphi}^{q}\td_{x}^{i}\mathcal{C}_{\alpha}^{\tau} 
\langle D\rangle^{-m_2}
\|^{\gamma, \cO }_{s,s}
\lesssim_{s, |q|,m_1,m_2}
\|\alpha\|^{\gamma, \cO }_{s+\mu}\, . 
\end{equation}

\noindent
The inverse $(\mathcal{C}_{\alpha}^{\tau})^{-1}$ 
% and the adjoint $(\mathcal{C}_{\alpha}^{\tau})^{*}$ 
satisfies  the same estimates  with possibly larger % constants
 $\s$ and $\mu$.
\end{lemma}

For notational convenience we give the proof for $ \tau = 1 $.
Recall that  $\mathcal{C}_{\alpha}:=\mathcal{C}_{\alpha}^{1}$.
We consider the \emph{Bony decomposition} $ P = P^B + P^U $ 
of the operator 
\begin{equation}\label{def:Pcalalpha}
P:=P^{1}:=\langle D\rangle^{-N_1}\circ\big(\mathcal{C}_{\alpha}-\Id\big)
\circ\langle D\rangle^{-N_2}\,,\qquad N_1+N_2=N\, , 
\end{equation}
 given in  Definition \ref{def:bonydeco}, and 
we define
$ \mathtt{P}:=(P^{B},P^{U}) $. 
We have to prove  that  $\mathtt{P}$ belongs to $E_{s}$
and satisfies  \eqref{chiarodiluna2}.

\begin{proof}[{\bf Proof of \eqref{chiarodiluna2}.}]
We write the composition operator (we do not explicitly state the dependence on $\oo$)
\begin{equation}\label{sacrorito2}
		\mathcal{C}_{\alpha}u=
		\sum_{j'\in\Z}\mathtt{t}_{\alpha}(\vphi,x,j')u_{j'} e^{\ii j'x}
		=\sum_{j\in\Z}\Big(\sum_{j'\in\Z} 
		\widehat {\mathtt{t}}_{\alpha}(\vphi,j-j',j')u_{j'}\Big)e^{\ii jx} 
\end{equation}
where  
$  \mathtt{t}_{\alpha}(\vphi,x,\xi):= e^{\ii \xi \al(\vphi,x)} $
has the space Fourier expansion 
\begin{equation}\label{sacrorito}
	\mathtt{t}_{\alpha}(\vphi,x,\xi) 
	=
	\sum_{k\in\Z}\widehat{\mathtt{t}}_{\alpha}(\vphi,k,\xi) e^{\ii k x} \, , 
	\quad \widehat {\mathtt{t}}_{\alpha}(\vphi,k,\xi) :=
\frac{1}{2\pi}
\int_{\T} e^{\ii\x\alpha(\vphi,x)} e^{-\ii k x}dx\,.
\end{equation}
We first prove the following lemma (the estimate 
\eqref{chiarodiluna10} will be used in Section \ref{sec:szego}).

\begin{lemma}\label{decayultravioletto}
There exists $\eta>0$ such that for any $s\ge\so$ there exist $\delta :=\delta(s)>0$ such that 
 for any $\alpha, f\in C^\infty(\T^{\nu+1},\R)$ 
	with $\|\al\|^{\g,\cO}_{\so+\eta}\leq \delta$, for any 
	$ \ell \in {\mathbb Z}^\nu $, $  k \in {\mathbb Z} $, $ \xi \in 
	{\mathbb \R} $, satisfying $|k|>|\xi|- \tfrac12  $, 
		\begin{equation}\label{chiarodiluna10}
		\begin{aligned}
			&	\big|\doublehat{f \mathtt{t}_{\alpha}}(\ell,k,\x)\big|^{\g,\cO}\lesssim_{s}
			\frac{1}{\langle \ell,k\rangle^{s}}
			\big(\|\alpha\|^{\g,\cO}_{s+\eta}\|f\|^{\g,\cO}_{\so}+
			\|f\|^{\g,\cO}_{s+1}\big)\\
		&	\big|\doublehat{(\mathtt{t}_{\alpha}-1)}(\ell,k,\x)\big|^{\g,\cO}\lesssim_{s}
			\frac{1}{\langle \ell,k\rangle^{s}}
		\|\alpha\|^{\g,\cO}_{s+\eta}\,.
		\end{aligned}
	\end{equation}
Moreover for  any $\ell, k $ such that $|\ell| + |k| \ge |\xi|/3$ and  
for any $s\geq \so$, 
\begin{equation}\label{chiarodiluna11}
	\big|\doublehat{(\mathtt{t}_{\alpha}-1)}(\ell,k,\x)\big|^{\g,\cO}\lesssim_{s}
	\frac{1}{\langle \ell,k\rangle^{s}}
	\|\alpha\|^{\g,\cO}_{s+\eta}\,.
\end{equation}
\end{lemma}

\begin{proof}
Recalling \eqref{sacrorito} we have,  for any  
$ \ell \in {\mathbb Z}^\nu $, $ k \in {\mathbb Z} $, $ \xi \in 
	{\mathbb \R} $, 
\begin{equation}\label{coeftalpha}
\doublehat{f\mathtt{t}_{\alpha}}(\ell,k,\x)=
 \frac{1}{(2\pi)^{\nu+1}}\int_{\T^{\nu+1}} 
 f (\vphi,x)e^{\ii \xi \alpha(\vphi,x) -\ii k x}
 e^{-\ii\ell\cdot\vphi}d\vphi dx \, .
\end{equation}
We first consider $k\ne 0$. Setting $  \eta:= \xi/k $ we have 
\begin{equation}
	\label{maratona}
\begin{aligned}
	& \doublehat{f\mathtt{t}_{\alpha}}(\ell,k,\x) =
 \frac{1}{(2\pi)^{\nu+1}}\int_{\T^{\nu+1}} 
 f (\vphi,x)e^{- \ii k( x- \eta \alpha(\vphi,x))}
 e^{-\ii\ell\cdot\vphi}d\vphi dx 
 \\
 & =
 \frac{1}{(2\pi)^{\nu+1}}\int_{\T^{\nu+1}} 
 f (\vphi,y+\breve{\alpha}(\vphi, y,\eta))e^{- \ii ky}
 e^{-\ii\ell\cdot\vphi} (1+\breve{\alpha}_y(\vphi, y,\eta))) d\vphi dy 
\end{aligned}
\end{equation}
where, since $\eta \alpha$ is small (note that $ |\eta| \leq 3/2 $ since
$ |\xi| < \frac12 + |k| $), the map  $y= x- \eta \alpha(\vphi,x)$ 
is a diffeomorphism of $\T$, see for instance \cite{FGMPmoser}, 
with inverse 
 $x= y+ \breve{\alpha}(\vphi,y,\eta)$, where $\breve\alpha$ satisfies the bound
\begin{equation}\label{stimacheckcheck}
\|  \breve{\alpha}\|_{s}^{\g,\cO}\lesssim_s  \|  {\alpha}\|_{s+\so}^{\g,\cO}\,.
\end{equation}
Denoting
$ g(\vphi,y):= f (\vphi,y+\breve{\alpha}(\vphi, y,\eta))
 (1+\breve{\alpha}_y(\vphi, y,\eta)) $ we have
 $ \doublehat{f\mathtt{t}_{\alpha}}(\ell,k,\x) = \doublehat{g} (\ell,k) $ and 
 $$
 | \doublehat{g} (\ell,k) |^{\gamma,\cO} \leq \frac{1}{\langle \ell, k\rangle^s }
\| g  \|^{\gamma,\cO}_{s+1} 
 $$
 the first bound in \eqref{chiarodiluna10}  for $ \doublehat{f\mathtt{t}_{\alpha}}(\ell,k,\x) $ follows by \eqref{tameProduct} and \eqref{pasta7}. 
 
If  $k=0$  we  have $|\xi|\le \tfrac12$ and by \eqref{coeftalpha} 
 \begin{equation}
 	\label{maratonafacile}
 	\begin{aligned}
 		| \doublehat{f\mathtt{t}_{\alpha}}(\ell,0,\x)|^{\g,\cO} 
 		&\lesssim	\frac{1}{\langle \ell,0\rangle^{s}} 
		\| 	f e^{\ii \xi \alpha}\|^{\g,\cO}_{s+1} 
		\lesssim_s \frac{1}{\langle \ell,0\rangle^{s}} (\| 	f\|^{\g,\cO}_{s+1} + \| 	f\|^{\g,\cO}_{\so}\|\alpha\|^{\g,\cO}_{s+1})\,.
 	\end{aligned}
 \end{equation}
Let us prove the second bound. By \eqref{coeftalpha} 
and the mean value theorem 
\begin{equation}\label{ta-1}
 \begin{aligned}
 	 \doublehat{(\mathtt{t}_{\alpha}-1)}(\ell,k,\x) 
 	& =
 	\frac{\ii}{(2\pi)^{\nu+1}}  \int_0^1 d\tau \int_{\T^{\nu+1}}  \alpha(\vphi,x) \xi e^{\ii \tau \xi \alpha(\vphi,x)} e^{-\ii k x}
 	e^{-\ii\ell\cdot\vphi}d\vphi dx  \, . 
\end{aligned}
\end{equation}
We now consider $k\ne 0$. Setting $\eta:= \tau\xi/k $ we have 
 \begin{equation}
 	\label{integroperparti}
 \begin{aligned}
 	 \doublehat{(\mathtt{t}_{\alpha}-1)}(\ell,k,\x)  
 	& =  
 	\frac{{\ii \xi}}{(2 \pi)^{\nu+1}} \int_0^1 d\tau 
	 \int_{\T^{\nu+1}} \alpha(\vphi,x)  e^{-\ii k(x - \eta \alpha(\vphi,x))}
 	e^{-\ii\ell\cdot\vphi}d\vphi dx  
 \end{aligned}
 \end{equation}
 and we estimate the integral in $ (\vphi,x)$, uniformly in $ \tau \in [0,1]$ and $ |\eta | \leq 3/2 $, 
proceeding as in \eqref{maratona} with $f=\alpha$. 
 We also note that $ |\xi| \leq \tfrac12 + |k| \leq 2 |k| $. 
 The case $k=0$ is dealt with just as in \eqref{maratonafacile}.
 \smallskip
 
For $\xi=0$ the bound \eqref{chiarodiluna11} is trivial, since the left hand side is zero. 
For $\xi\ne 0$ we note that, by $|k|+|\ell|\ge  |\xi|/3$ we deduce that at least one between $|k|,|\ell_1|,\dots,|\ell_\nu| \ge \frac1{3(\nu+1)}|\xi|$. If $|k|\ge \frac1{3(\nu+1)} |\xi|$, the proof follows directly by \eqref{integroperparti}, by using that $\eta\|\alpha\|_{\so}^{\g, \cO}$ is small. Otherwise if  -say-  
 $|\ell_1|\ge  \frac1{3(\nu+1)} |\xi| > |k|$ then, by \eqref{ta-1}, setting 
 $ \eta := \tau \xi / \ell_1 $, we have 
 $$
 	\begin{aligned}
 		\doublehat{(\mathtt{t}_{\alpha}-1)}(\ell,k,\x)  = 
 		\frac{\ii \xi }{(2 \pi)^{\nu+1}} \int_0^1 d\tau 
		 \int_{\T^{\nu+1}} \alpha(\vphi,x)  e^{-\ii \ell_1(\vphi_1 -\eta \alpha(\vphi,x))} 	 e^{-\ii\sum_{i=2}^{n}\ell_i\vphi_i}
 		e^{-\ii k x}d\vphi dx  
 	\end{aligned}
$$
and we estimate the integral in $ (\vphi,x)$, uniformly in $ \tau \in [0,1]$, 
$ |\eta | \leq 3 (\nu + 1 )$, 
proceeding as in \eqref{maratona} 
applying the invertible change of variables $\vartheta_1= \vphi_1 -\eta \alpha(\vphi,x)$, $\vartheta_i=\vphi_i$ for $i\ge 2$.
\end{proof}

We are now in position to prove \eqref{chiarodiluna2}. Let consider $ q = \tb $. 

Observing that $(\jap{\td_{\vphi}}^{\tb} P)^B=\jap{\td_{\vphi}}^{\tb} P^B$ and  
$(\jap{\td_{\vphi}}^{\tb} P)^U=\jap{\td_{\vphi}}^{\tb} P^U$ 
we get 
\begin{equation}\label{chiarodiluna16}
\begin{aligned}
\bnorm{\jap{\td_{\vphi}}^{\tb} \mathtt{P}}^{\g, \cO}_s &= 
\sup_{\gso \leq p \leq \su} |\jap{\td_{\vphi}}^{\tb} P^B|^{\g, \cO}_{p,p} 
+  |\jap{\td_{\vphi}}^{\tb} P^U|^{\g, \cO}_{\gso,s} 
\\&
\stackrel{\eqref{Katsendorff}}{\le}  
\sup_{\gso \leq p \leq \su} 3^{p-\gso} |\jap{\td_{\vphi}}^{\tb} P|^{\g, \cO}_{\gso, \gso} 
+  |\jap{\td_{\vphi}}^{\tb} P^U|^{\g, \cO}_{\gso,s}\,.
\end{aligned}
\end{equation}
Let us estimate  the first term.
We have by Lemma \ref{maggio} 
and Lemma \ref{bastalapasta},
\begin{equation}\label{stocastico}
\begin{aligned}
|\jap{\td_\vphi}^{\tb} & P|^{\g, \cO}_{\gso, \gso} 
\stackrel{\eqref{rmk:commuES}}{\lesssim_{\mathtt b}}
| P|^{\g, \cO}_{\gso, \gso} +\sum_{h=1}^{\nu}|\td_{\vphi_h}^{\tb} P|^{\g, \cO}_{\gso, \gso} 
\\&
\stackrel{\eqref{laverna}}{\lesssim_{\mathtt b}}
\sum_{q=0,\tb} \sum_{h=1}^{\nu}\Big( \|\td_{\vphi_{h}}^{q} P\|^{\g, \cO}_{\gso, \gso}
+ \|\td_x{\td}^{q}_{\vphi_h} P\|^{\g, \cO}_{\gso, \gso}
+\|{\td}_{\vphi_{h}}^{q+\beta} P\|^{\g, \cO}_{\gso, \gso}
+ \|\td_x{\td}^{q+\beta}_{\vphi_h} P\|^{\g, \cO}_{\gso, \gso}\Big)
\\&
\stackrel{\eqref{def:Pcalalpha},\eqref{panelle}}{\lesssim_{\mathtt b}}
\|\alpha\|^{\g, \cO}_{\gso+\s} 
\end{aligned}
\end{equation}
for some $\s>0$ depending on $\nu$ and $\tb$, using 
 that $N_{1}+N_{2}$ is sufficiently large.

Consider now the ultraviolet part $\jap{\td_\varphi}^\tb P^U$ 
 in \eqref{chiarodiluna16}. Recalling Definition \ref{def:bonydeco}, 
\eqref{sacrorito2}-\eqref{sacrorito}, its matrix entries are 
\begin{equation}\label{chicoe}
 (\langle D\rangle^{-N_1}(\mathcal{C}_\al-\Id)\langle D\rangle^{-N_2})_{j}^{j'}(\ell-\ell') 
 = \doublehat{(\mathtt{t}_{\alpha}-1)}(\ell-\ell',j-j',j') \langle j \rangle ^{-N_1}\langle j'\rangle ^{-N_2}
\end{equation}
with indexes $\ell,\ell'\in \Z^{\nu}$, $j,j'\in \Z$, satisfying  
$ |\ell-\ell'|+|j-j'|\geq (|\ell|+|j|)/2 $ and thus 
\begin{equation}\label{stimasempl}
|\ell-\ell'|+|j-j'|\geq \max{\{|j'|, |j|\}}/3   \, . 
\end{equation}
Under the smallness condition \eqref{buttalapasta3} 
Lemma \ref{decayultravioletto}  applies and 
therefore, in view of \eqref{chicoe}, \eqref{stimasempl}, 
the estimate \eqref{chiarodiluna11} 
implies, with $ \sigma = \eta +|N_1|+|N_2| $,   
\begin{align}\label{erice}
| (\jap{\td_\varphi}^{\tb} \langle D\rangle^{-N_1}(\mathcal{C}_\al-\Id)\langle D\rangle^{-N_2})_{j}^{j'}(\ell-\ell') |^{\g, \cO}
& \lesssim_{s} 
\|\alpha\|^{\g, \cO}_{s+\gso+\s+\tb}
\frac{\langle j\rangle ^{-N_1}\langle j'\rangle ^{-N_2}}{\langle \ell-\ell',j-j'\rangle^{s+|N_1|+
|N_2|+\gso }} \notag \\
&  \stackrel{\eqref{stimasempl}} {\lesssim_{s} } 
\|\alpha\|^{\g, \cO}_{s+\gso+\s+\tb}
\frac{1}{\langle \ell-\ell',j-j'\rangle^{s+\gso }}
\end{align}
%Now, if $N_{1}, N_{2} \geq 0$ 
This implies that the decay norm 
$|\jap{\td_\varphi}^\tb P^U|_s^{\text{dec},\gamma,\cO} \lesssim_s 
\|\alpha\|^{\g, \cO}_{s+\gso+\s+\tb} $ and then by Lemma \ref{lem:ilfreddo} we 
control $|\jap{\td_\varphi}^\tb P^U|^{\g, \cO}_{\gso, s} \lesssim_s 
\|\alpha\|^{\g, \cO}_{s+\gso+\s+\tb}
$.
%If  both $N_{1}<0, N_{2} < 0$ then we exploit \eqref{stimasempl} and 
%we use \eqref{erice} with $s\rightsquigarrow s + |N_{1}|+|N_2| $.
Combining this with \eqref{chiarodiluna16}  and \eqref{stocastico}
one gets the bound \eqref{chiarodiluna2} with  $ \mu \geq \sigma + \tb $. 
\end{proof}

\section{A quantitative Egorov theorem}\label{sec:egomichela}
In this section we prove an Egorov Theorem concerning 
 the conjugation  of a pseudo-differential
operator $\op(w)$  under  the diffeomorphism
$\mathcal{C}_{\alpha}^{\tau}$ introduced in Section \ref{sec:Calpha}.
A non trivial result of Theorem \ref{thm:egorov} are the bounds 
\eqref{zeppelin}-\eqref{zeppelinbis} % which contain only $ p $ 
and that the remainders are associated to Bony-smoothing couples 
satisfy the quantitative tame estimates \eqref{stimerestoEgorov1}
where  the sum of indices $ k_1 + k_2 + k_3 = s $.  
 In the sequel we suppose $ s, s_0 \in \mathbb N $.

\vspace{0.5em}
\noindent
{\bf Notation.} 
Given  $n\in \N$, 
we denote 
$\Sigma^{*}_{n}$ the sum over indexes $k_1,k_2,k_3\in\N_0 $ such that 
$k_1\leq n$, $k_1+k_2+k_3=n$ and $k_1+k_2\geq 1$.

\begin{thm}[{{\bf Egorov}}]\label{thm:egorov}
Fix $ m\in\R$, $M>\max\{0,-m\}$ and $\mathtt{b}\in \N$.  
Let $w\in S^m$ and consider the composition operator 
$\mathcal{C}_{\alpha}^{\tau}$
		in \eqref{ignobel}.
		There exists $\sigma:=\sigma(m, M,\tb)$ and, for all $s\geq \so$, 
		there exists $\delta:=\delta(s, m, M,\tb)$ such that, if
		\begin{equation}\label{buf}
			\| \alpha \|^{\gamma, \cO}_{\so+\s}<\delta\,,
		\end{equation}
		then, for any $\tau\in[0,1]$,
		\[
		\mathcal{C}_{\alpha}^{\tau} \circ \op(w) \circ (\mathcal{C}_{\alpha}^{\tau})^{-1}=\op(q(\tau))+R^{\tau}\,,
		\]
		where  $q(\tau)$ is a symbol in 
		$ S^m$ and the operator 
		$R^{\tau}\langle D\rangle^M$ is bounded in $H^s$. Moreover

\noindent
$(i)$ {\bf (Symbol)}
the symbol  $q(\tau) :=q(\tau;\vphi,x,\x)$ in $S^m$  has the form  
\[
q(\tau;\vphi, x, \xi)=q_m(\tau;\vphi, x, \xi)+\widetilde{q}_{m-1}(\tau;\vphi, x, \xi)
\] 
with principal symbol  at $ \tau = 1 $, 
\begin{equation}\label{fundamentalformula}
\begin{aligned} 
q_m(1;\vphi, x, \xi) & = 
w\Big( \varphi, x+\alpha(\vphi, x), \frac{\xi}{1+\alpha_x(\vphi,x)}\Big) \\
& = 
w\Big(\vphi,x+\alpha(\vphi,x), 
\x(1+\pa_{y}\breve{\alpha}(\vphi,y))\Big)_{|y=x+\alpha(\vphi,x)}
\,,
\end{aligned}
\end{equation}
and subprincipal symbol $ \widetilde{q}_{m-1} := \widetilde{q}_{m-1}(\tau;\vphi, x, \xi)\in S^{m-1}$ satisfying,
for any $p\geq0 $, $ s \geq s_0 $
\begin{align}
&\| q_{m} \|^{\gamma, \cO}_{m, s, p}\lesssim_{m, s, p, M,\tb}  
\| w \|^{\gamma, \cO}_{m, s, p}
+
\sum_{s}^{*} 
\| w \|^{\gamma, \cO}_{m, k_1, p+k_2} 
\| \alpha \|^{\gamma, \cO}_{k_3+\so+2}\,,\label{zeppelin}
\\
	&	\| \widetilde{q}_{m-1} \|^{\gamma, \cO}_{m-1, s, p}\lesssim_{m, s, p, M,\tb}  
		\sum_{s}^{*} 
		\| w \|^{\gamma, \cO}_{m, k_1, p+k_2+\sigma} 
		\| \alpha \|^{\gamma, \cO}_{k_3+\sigma}\,,\label{zeppelinbis}
\end{align}
	 uniformly in $\tau\in[0,1]$.
	
\noindent 
$(ii)$ {\bf (Remainder)} 
Assuming \eqref{buf} for all $s\leq \su$, there exists a couple
$\mathtt{R}^{\tau}\in E_{s}\equiv E_{s,\gso,\su}$ such that 
$\mathfrak{S}(\mathtt{R}^{\tau})=R^{\tau}$ and,
for any $m_1,m_2\in\R_+$ with  $m_1+m_2=M$, $\gso\le s\le \su$, one has (recall \eqref{derivata})
\begin{equation}\label{stimerestoEgorov1}
\begin{aligned}
\sup_{\tau\in [0,1]}&\bnorm{ \jap{\td_{\vphi}}^{j} \jap{D}^{m_1} \mathtt{R}^{\tau} \jap{D}^{m_2}}_{s}^{\g, \cO} 
\lesssim_{s,\su, m,M,\tb} 
\sum_{s}^{*} 
\| w \|^{\gamma, \cO}_{m, k_1+\s, k_2+\sigma} 
\| \alpha \|^{\gamma, \cO}_{k_3+\sigma}
\end{aligned}
\end{equation}
for $j=0,\tb$ and uniformly in $\tau\in[0,1]$.
\end{thm}

In the proof we use  the following lemma proved 
in the Appendix of \cite[Lemma A.7]{FGP19}, together with 
the monotonicity properties \eqref{monotoniaEs} of the norm.

\begin{lemma}\label{Lemmino}
Let $\alpha \in C^\infty (\T^{\nu+1}, \R)$ satisfy 
$\| \alpha \|^{\g, \cO}_{2\so+2}< 1$. 
Then, for any symbol $w\in S^m$, 
\begin{equation}
		A w:=w\Big(\vphi, x+\alpha(\vphi,x), \frac{ \xi}{1+\alpha_{x}(\vphi,x)} \Big)
		% \quad f(x):=x+\alpha(\vphi,x), \quad g(x):=(1+\alpha_{x}(\vphi,x))^{-1} \, , 
	\end{equation}
	is a symbol in $ S^m$ satisfying, for any $ s \geq \so $,    
	\begin{equation}\label{stima}
		\| A w \|^{\g, \cO}_{m, s, p}\lesssim_{m,s,p} \| w \|^{\g, \cO}_{m, s, p}
		+ \sum_{s}^{*}
		\| w \|^{\g, \cO}_{m, k_1, p+k_2} \| \alpha \|^{\g, \cO}_{k_3+\so+2} \, .
	\end{equation}
	% for some $ C:=C(s,p)>0$.
	 For $s=\so$  we have the rougher estimate 
$	\| A w \|^{\g, \cO}_{m, \so, p}\lesssim \| w \|^{\g, \cO}_{m, \so, p+\so} $. 
\end{lemma}

\begin{proof}[{\bf Proof of Theorem \ref{thm:egorov}}]
The $ \tau $-dependent diffeomorphism 
$\mathcal{C}_{\alpha}^{\tau}$ defined  in \eqref{ignobel}   is the flow of  the 
	% and \eqref{def:Calphagenerator}-\eqref{def:Calpha2}% one can easily check 
linear transport equation
	\begin{equation}\label{flussodiffeo}
      \pa_{\tau}\mathcal{C}_{\alpha}^{\tau}=X^{\tau}\mathcal{C}_{\alpha}^{\tau}\,,\quad 
      \mathcal{C}_{\alpha}^{0}={\rm Id}\, , 
	\end{equation}
	where $ X^\tau $ is the  $ \tau $-dependent operator % generator 
	\begin{align}
& 		{X}^{\tau}:=\mathcal{A}(\tau;\vphi,x)\pa_{x}=\op(\chi)\,,\qquad 
		\chi :=\chi(\tau;x,\x) :=\chi(\tau;\vphi,x,\x):=\ii \mathcal{A}(\tau;\vphi,x)\x\,, \label{def:Calphagenerator} \\
& 
\mathcal{A}(\tau;\vphi,x)=
\frac{\alpha(\vphi,x)}{1+\tau\alpha_x(\vphi,x)} \,. \label{def:Calpha2}
\end{align}
Since the symbol $\chi$ is linear in $ \xi $, 
\begin{equation}\label{normachichi}
\|\chi\|_{1,s,p}^{\gamma,\calO}\lesssim_{s}\|\mathcal{A}\|_{s}^{\gamma,\calO}
\lesssim_{s}\|\alpha\|_{s+1}^{\gamma,\calO}\,,\qquad \forall \,p\geq0\,.
\end{equation}
For any $\tau\in[0,1]$ the conjugated operator
$ P^{\tau}:=\mathcal{C}_{\alpha}^{\tau} \circ \op(w)\circ  (\mathcal{C}_{\alpha}^{\tau})^{-1} $
	solves  the Heisenberg equation
	\begin{equation}\label{ars}
	\partial_{\tau} P^{\tau}=[{X}^{\tau}, P^{\tau}]\,,  \quad P^{0}=\op(w) \, .
	\end{equation}
	For simplicity  we omit the dependence of the symbols on the variables
$(\omega,\vphi)\in\cO\times \T^{\nu}$.
Let us fix 
\begin{equation}\label{sceltarho}
\rho:=M+4 \big(\lfloor \nu/2 \rfloor + 4+\tb\big)+1\,.
\end{equation}
	We look for an approximate solution of \eqref{ars} 
	of the form\footnote{We suppose  that $ m + \rho \geq 2 $ 
	is an integer, otherwise we  can replace $ m + \rho $  with 
	$ [m + \rho] + 1 $.} 
	\begin{equation}\label{balconata0}
		Q^{\tau}:=\op(q(\tau;x,\x))\,,\quad 
		q=q(\tau;x,\x)=\sum_{k=0}^{m+\rho-1} q_{m-k}(\tau;x, \xi)\,,
	\end{equation}
	where $q_{m-k}$ are  symbols in $S^{m-k}$ to be determined iteratively	so that
\begin{equation}\label{approssimo}
	\partial_{\tau} Q^{\tau}=[{X}^{\tau}, Q^{\tau}] + \cM^\tau\,, \quad Q^{0}=\op(w) \, , 
\end{equation}
	where $\cM^\tau=\op(\tr_{-\rho}(\tau;x,\x))$ is a
	pseudo-differential operator of order $ -\rho $ for
	a symbol $\tr_{-\rho}\in S^{-\rho}$. 
%  the integer part . 
	Passing to the symbols in \eqref{approssimo} we obtain, 
	recalling \eqref{def:Calphagenerator}, 
	\eqref{balconata0} and that $\star $ in \eqref{espstar} 
	is  the symbol of the commutator, the equation 
	\begin{equation}\label{probapproxsimboloq}
\left\{\begin{aligned}
			&\pa_{\tau}q(\tau;x,\x)=\chi(\tau;x,\x)\star q(\tau;x,\x) + \tr_{-\rho}(\tau;x,\x)
			\\
			&q(0;x,\x)=w(x,\x) 
		\end{aligned}\right.
	\end{equation}
	where the unknowns are $q(\tau;x,\x),\tr_{-\rho}(\tau;x,\x)$.

	We  expand $\chi(\tau;x,\x)\star q(\tau;x,\x)$ into a sum of symbols with decreasing orders.
	Since $\chi$ is linear in $\x\in\R$ (see \eqref{def:Calphagenerator}), 
	\eqref{composizione troncata}-\eqref{remainder composizione} 
	and  the expansion \eqref{balconata0} 
	(with the ansatz $q_{m-k}\in S^{m-k}$)
	we have 
	\begin{align}
	\chi\star q&=\chi\#q-q\#\chi=\chi q+\frac{1}{\ii}(\pa_{\x}\chi)(\pa_{x}q)-q\#\chi \label{rearra}
	\\&
	\stackrel{\eqref{balconata0}}{=}
	\chi q+
	\sum_{k=0}^{m+\rho-1} \frac{1}{\ii}(\pa_{\x}\chi)(\pa_{x}q_{m-k})-\Big(\sum_{k=0}^{m+\rho-1} q_{m-k}(\tau;x, \xi)\Big)\#\chi \notag 
	\\&
	\stackrel{\eqref{composizione troncata}, \eqref{cancellittiEspliciti},\eqref{espstar2}}{=}
	\sum_{k=0}^{m+\rho-1} \frac{1}{\ii}\{\chi,q_{m-k}\}
	-\sum_{k=0}^{m+\rho-1}
	\sum_{n=2}^{m-k+\rho}q_{m-k}\#_{n}\chi-\sum_{k=0}^{m+\rho-1}q_{m-k}\#_{\geq m-k+\rho+1}\chi\,. \notag 
	\end{align}
By rearranging the sums in \eqref{rearra} we get 
\begin{equation}\label{civilwar}
\begin{aligned}
	\chi\star q&=\underbrace{-\ii\{\chi,q_{m}\}}_{\in S^m}
	 + \overbrace{\sum_{k=1}^{m+\rho-1}\underbrace{\big(-\ii\{\chi,q_{m-k}\}+r_{m-k}\big)}_{\in S^{m-k}} }^{{\rm orders\, from}\;-\rho+1 \;{\rm to}\;  m-1}
	-\underbrace{\mathtt{r}_{-\rho}}_{ \in S^{-\rho}}
\end{aligned}
\end{equation}
where, denoting $\tw=\tw(k,h):=k-h+1$,
\begin{equation}\label{sperobene}
\begin{aligned}
r_{m-k}&:=-\sum_{h=0}^{k-1} q_{m-h}\#_{\tw}\chi
\\&\stackrel{\eqref{cancellittiEspliciti}}{=}
-\sum_{h=0}^{k-1}
\frac{1}{ \tw! \ii^{\tw}} (\partial_{\xi}^\tw q_{m-h})(\partial_x^\tw \chi)\in S^{(m-h)+1-(k-h+1)}\equiv S^{m-k}\,,
\end{aligned}
\end{equation}
and
\begin{equation}\label{sperobeneResto}
\mathtt{r}_{-\rho} :=\sum_{k=0}^{m+\rho-1}q_{m-k}\#_{\geq \rho+m-k+1}\chi
\stackrel{\eqref{cancellittiEspliciti}}{\in}
S^{m-k+1-(m-k+1+\rho)}\equiv S^{-\rho} \,.
\end{equation}
In view of \eqref{civilwar}, to solve the equation \eqref{probapproxsimboloq} is equivalent 
to  find symbols $q_{m-k}\in S^{m-k}$, $0\leq k\leq m+\rho-1$,  which solve for $k=0$
(recall the form of $\chi$ in \eqref{def:Calphagenerator})
	\begin{equation}\label{ordm}
		\left\{\begin{aligned}
			&\pa_{\tau}q_{m}(\tau;x,\x)= \{\mathcal{A}(\tau;x)\x, q_{m}(\tau;x,\x)\}
			\\
			&q_{m}(0;x,\x)=w(x,\x)\,,
		\end{aligned}\right.
	\end{equation}
and, for any $1\le k\le m+\rho-1 $,  
	\begin{equation}\label{ordmmenok}
	\left\{\begin{aligned}
		&\pa_{\tau}q_{m-k}(\tau;x,\x)= \{\mathcal{A}(\tau;x)\x, q_{m-k}(\tau;x,\x)\}+r_{m-k}(\tau;x,\x)
		\\
		&q_{m-k}(0;x,\x)=0\,.
	\end{aligned}\right.
\end{equation}
Note 
that the symbols $r_{m-k}$ in \eqref{sperobene} with $1\leq k\leq m+\rho-1$ depend only
on $q_{m-h}$ with $0\leq h<k$, namely the 
equations \eqref{ordmmenok} can be solved iteratively.

\vspace{0.5em}
\noindent
{\bf Order $m$.} To solve \eqref{ordm}, 	we consider the solutions of the Hamiltonian system
	\begin{equation}\label{charsys}
		\left\{\begin{aligned}
			&\frac{d}{ds}x(s)=-\mathcal{A}(s,x(s))\\
			&\frac{d}{ds}\x(s)=\mathcal{A}_{x}(s,x(s))\x(s)
		\end{aligned}\right.\qquad (x(0),\x(0))=(x_0,\x_0)\in \T\times \R\,.
	\end{equation}
	If $q_m$ is a solution of \eqref{ordm}, then 
	it is constant along the flow of \eqref{charsys}. In other words
	setting
	$g(\tau):=q_{m}(\tau;x(\tau),\x(\tau))$ one has that 
	$	\frac{d}{d\tau}g(\tau)=0 $ and thus $ g(\tau)=g(0) $, for any $ \tau\in[0,1] $. 
	Let us denote by $\gamma^{\tau_0,\tau}(x,\x)$ the solution 
	of the characteristic system \eqref{charsys}
	with initial condition $\gamma^{\tau_0,\tau_0}=(x,\x)$, so that  $(x(\tau),\x(\tau))=\gamma^{0,\tau}(x_0,\x_0)$
		and the inverse flow is given by $\gamma^{\tau,0}(x,\x)=(x_0,\x_0)$.
	Then the equation \eqref{ordm} has the solution
	\begin{equation}\label{ars3}
		q_m(\tau;x, \xi)=w(\gamma^{\tau, 0}(x, \xi))
	\end{equation}
	where $\gamma^{\tau, 0}(x, \xi)$ has the explicit form (recall \eqref{def:Calpha2})
	\begin{equation*}
		\gamma^{\tau, 0}(x, \xi)=\big(f(\tau;x), \xi g(\tau;x)\big)\,, \qquad 
		f(\tau;x):=x+\tau\alpha(x)\,, \quad g(\tau;x):=\frac{1}{1+\tau\alpha_x(x)}\,.
	\end{equation*}
	This proves \eqref{fundamentalformula}. 
	Moreover Lemma \ref{Lemmino} implies that $q_{m}(\tau;x,\x)$ in \eqref{ars3} 
	satisfies  \eqref{zeppelin}.
	
	\vspace{0.5em}
	\noindent
	{\bf Order $m-k$ with $1\leq k\leq m+\rho-1$.} 
	We now assume inductively that 
	we have already found the solutions $q_{m-h}\in S^{m-h}$ 
	of \eqref{ordmmenok}
	with $0\leq h<k$, for some $k\ge 1$ satisfying 
	\begin{equation}\label{zeppelinIndut}
		\| q_{m-h} \|^{\gamma, \cO}_{m-h, s, p}\lesssim_{m, s, p, M,\tb}  
		\sum_{s}^{*} 
		\| w \|^{\gamma, \cO}_{m, k_1, k_2+\sigma_{h}+p} 
		\| \alpha \|^{\gamma, \cO}_{k_3+\sigma_{h}}\,,\quad 1\leq h<k\,,
	\end{equation}
	for some non decreasing sequence of parameters $\s_{h}$ 
	depending only on $\nu$, $|m|,M$. 
	Note that in \eqref{zeppelinIndut} the Sobolev  norm of  $\alpha$ 
	does not contain  the parameter $p$. 
	
	We now construct the solution $q_{m-k}$ of \eqref{ordmmenok} 
	satisfying estimate \eqref{zeppelinIndut} with $h=k$.
	We define $f_{m-k}(\tau):=q_{m-k}(\tau;x(\tau),\x(\tau))$
	where $x(\tau),\x(\tau)$ are the solution of the Hamiltonian system \eqref{charsys}.
	One has that, if $q_{m-k}$ solves \eqref{ordmmenok}, then
	\[
	\frac{d}{d\tau}f_{m-k}(\tau)=r_{m-k}(\tau;x(\tau),\x(\tau))\quad \Rightarrow
	\quad
	f_{m-k}(\tau)=\int_{0}^{\tau} r_{m-k}(\s;x(\s),\x(\s))d\s\,,
	\]
	where we used that $f(0)=q_{m-k}(0,x(0),\x(0))=0$.
	Therefore the solution of \eqref{ordmmenok} is
	\begin{equation}\label{gnomosi}
		q_{m-k}(\tau;x, \xi)=
		\int_0^{\tau} r_{m-k}(\gamma^{0, \s} \gamma^{\tau, 0}(x, \xi)) \,d\s \, .
	\end{equation}
	We observe also that 
	$		\gamma^{0, \s} \gamma^{\tau, 0}(x, \xi)=(\tilde{f}, \tilde{g}\,\xi) $, 
	for any $ \s,\tau\in[0,1] $,  with
	\begin{equation}\label{formasolm}
		\tilde{f}(\s, \tau,x):=x+\tau \alpha(x)+\breve{\alpha}(\s, x+\tau \alpha(x))\,, 
		\qquad 
		\tilde{g}(\s, \tau,x):=  \frac{1}{\tilde{f}_x(\s, \tau,x)} \, . 
	\end{equation}
We now have to prove that $ q_{m-k}  $ in \eqref{gnomosi}  
satisfies \eqref{zeppelinIndut} with $ h = k $.
\\[1mm]
{\sc Step 1: estimate of $r_{m-k} $.} 
In view of \eqref{sperobene}, applying  
\eqref{stimacancellettoesplicitoAlgrammo}
with $a\rightsquigarrow q_{m-h}\in S^{m-h}$, 
$b\rightsquigarrow \chi\in S^{1}$, $n\rightsquigarrow\tw=k-h+1$, we deduce
(recall that $\tw:=k-h+1$)  
\begin{align}
\|r_{m-k}\|_{m-k,s,p}^{\gamma,\calO}
& {\lesssim}
\sum_{h=0}^{k-1}
\frac{1}{  \tw! } \|(\partial_{\xi}^\tw q_{m-h})(\partial_x^\tw \chi)\|_{m-k,s,p}^{\gamma,\calO}
\label{stimaResti} 
\\&\stackrel{\eqref{normachichi}}{\lesssim_{m,s,p}}
\|q_{m}\|_{m,s,p+k+1}^{\gamma,\calO}
\|\alpha\|_{\so+k+2}^{\gamma,\calO}
+
\|q_m\|_{m,\so,p+k+1}^{\gamma,\calO}
\|\alpha\|_{s+k+2}^{\gamma,\calO} \label{stimaResti0} \\ 
&+
\sum_{h=1}^{k-1}
\|q_{m-h}\|_{m-h,s,p+\tw}^{\gamma,\calO}
\|\alpha\|_{\so+1+\tw}^{\gamma,\calO}
+
\|q_{m-h}\|_{m-h,\so,p+\tw}^{\gamma,\calO}
\|\alpha\|_{s+1+\tw}^{\gamma,\calO}\, . \label{stimaResti1}
\end{align}
Let us consider \eqref{stimaResti1}. By  the inductive assumption \eqref{zeppelinIndut} with $1 \leq h\leq k-1$, 
we deduce %(recall that $\tw=k-h+1$)
\[
\begin{aligned}
\eqref{stimaResti1} 
& \stackrel{\eqref{zeppelinIndut}}{\lesssim_{m,s,p,\rho}}
\sum_{h=1}^{k-1}
\sum_{s}^{*} 
		\| w \|^{\gamma, \cO}_{m, k_1, k_2+\sigma_{h}+p+\tw} 
		\| \alpha \|^{\gamma, \cO}_{k_3+\sigma_{h}+\tw}
\|\alpha\|_{\so+1+\tw}^{\gamma,\calO}
\\&\qquad\quad+
\sum_{h=1}^{k-1}
\sum_{\so}^{*} 
		\| w \|^{\gamma, \cO}_{m, k_1, k_2+\sigma_{h}+p+\tw} 
		\| \alpha \|^{\gamma, \cO}_{k_3+\sigma_{h}+\tw}
\|\alpha\|_{s+1+\tw}^{\gamma,\calO}
\\&\lesssim_{m, s, p,\rho} \sum_{h=1}^{k-1}
\sum_{s}^{*} 
\| w \|^{\gamma, \cO}_{m, k_1, k_2+\sigma_{h}+p+\tw} 
\| \alpha \|^{\gamma, \cO}_{k_3+\sigma_{h}+\tw}
\\ &
\qquad\quad+
\sum_{h=1}^{k-1} \sum_{1\le k_1+k_2 \le \so}
\| w \|^{\gamma, \cO}_{m, k_1, k_2+\sigma_{h}+p+\tw} 
\|\alpha\|_{s-\so+1+\tw +\so }^{\gamma,\calO} 
\\&	
\lesssim_{m, s, p, M,\tb} 
\sum_{s}^{*} 
\| w \|^{\gamma, \cO}_{m, k_1, k_2+\widehat{\s}_{k-1}+p} 
\| \alpha \|^{\gamma, \cO}_{k_3+\widehat{\s}_{k-1}} 
\end{aligned}
\]
having used  the smallness assumption
\eqref{buf} 
(assuming $\s\geq {\s}_{k-1}+1+ k %\ge \s_h + \tw +1
$), and in the last line we set $\hat\s_{k-1}:= \s_{k-1} +k +2 +\so$.
We bound \eqref{stimaResti0}  similarly using \eqref{zeppelin} instead of \eqref{zeppelinIndut}.
We conclude that 
\begin{equation}
	\label{verdone1}
\|r_{m-k}\|_{m-k,s,p}^{\gamma,\calO}
{\lesssim}_{m, s, p, M,\tb} 
\sum_{s}^{*} 
\| w \|^{\gamma, \cO}_{m, k_1, k_2+\widehat{\s}_{k-1}+p} 
\| \alpha \|^{\gamma, \cO}_{k_3+\widehat{\s}_{k-1}}\,.
\end{equation}
{\sc Step 2: estimates for $ q_{m-k} $.}  	
Let us prove that the symbol
	$q_{m-k}$ satisfies the bound \eqref{zeppelinIndut} with $h=k$ and for some 
	new $\s_{k}\geq \s_{k-1}$ depending only on $\nu$, $|m|$ and $M$.

In view of \eqref{gnomosi}, \eqref{formasolm}, 
	setting $\tilde{A} r:=r(\tilde{f}, \tilde{g}\,\xi)$, we have, for any  $\tau\in [0, 1]$, 
	$ \| q_{m-k} \|^{\gamma, \mathcal{O}}_{m-k, s, p}
			\lesssim_{s, p} 
			\sup_{\sigma \in [0,1]}\| \tilde{A} r_{m-k} \|^{\gamma, \mathcal{O}}_{m-k, s, p} $ 
			and by Lemma \ref{Lemmino} with  
	$\alpha\rightsquigarrow  \tau \alpha(x)+\breve{\alpha}(\s, x+\tau \alpha(x))$\footnote{
	The function $t(\tau,\s,x):=\tau \alpha(x)+\breve{\alpha}(\s, x+\tau \alpha(x))$
	where $y+\breve{\alpha} $ is the inverse diffeomorphism of $x+\tau\alpha$  
	satisfies (recall for instance  \eqref{stimacheckcheck}) 
	the estimate	$ \| t(\tau,\s)\|_{s}^{\g,\cO}\lesssim_s  \|  {\alpha}\|_{s+\so}^{\g,\cO} $
	uniformly in $\tau,\s\in[0,1]$.
	} we get 
	\begin{equation}\label{interpol}
		\begin{aligned}
&		\| q_{m-k} \|^{\gamma, \mathcal{O}}_{m-k, s, p}\lesssim_{s, p} 
		\| r_{m-k} \|^{\gamma, \mathcal{O}}_{m-k, s, p}
		+\sum_{s}^{*} 
		\| r_{m-k} \|^{\gamma, \mathcal{O}}_{m-k, k_1, p+k_2} \| \alpha \|^{\gamma, \mathcal{O}}_{k_3+2\so+2} \\
&	
			\| q_{m-k} \|^{\gamma, \mathcal{O}}_{m-k, \so, p}
			\lesssim
			\| r_{m-k} \|^{\gamma, \mathcal{O}}_{m-k, \so, p+\so}\, . 
		\end{aligned}
	\end{equation}
By substituting \eqref{verdone1} in  \eqref{interpol}
we get 
\begin{equation}\label{verdone2}
\begin{aligned}
\| q_{m-k} \|^{\gamma, \mathcal{O}}_{m-k, s, p}&\lesssim_{m,s, p,M,\tb} 
\sum_{s}^{*} 
		\| w \|^{\gamma, \cO}_{m, k_1, k_2+\widehat{\s}_{k-1}+p} 
		\| \alpha \|^{\gamma, \cO}_{k_3+\widehat{\s}_{k-1}}
		\\&+\sum_{s}^{*} 
		\Big(
	\sum_{k_1}^{*} 
		\| w \|^{\gamma, \cO}_{m, k_1', k_2'+\widehat{\s}_{k-1}+p+k_2} 
		\| \alpha \|^{\gamma, \cO}_{k_3'+\widehat{\s}_{k-1}}
			\Big) \| \alpha \|^{\gamma, \mathcal{O}}_{k_3+2\so+2} 
\end{aligned}
\end{equation}
where the sums run over indexes satisfying $k_1+k_2+k_3=s$ and $k_1'+k_2'+k_3'=k_1$.
By the  interpolation estimate 
\eqref{interpolotutto} and \eqref{buf}
we deduce that (if $k_3'\neq0$, otherwise we do nothing)
\begin{equation}\label{intealpa}
\| \alpha \|^{\gamma, \cO}_{k'_3+\widehat{\s}_{k-1}}\| \alpha \|^{\gamma, \mathcal{O}}_{k_3+2\so+2}
\lesssim 
\| \alpha \|^{\gamma, \mathcal{O}}_{k'_3+k_3+\widehat{\s}_{k-1}+2\so+2} \, . 
\end{equation}
Therefore  \eqref{verdone2} becomes, using \eqref{intealpa},
\[
\begin{aligned}
\| q_{m-k} \|^{\gamma, \mathcal{O}}_{m-k, s, p}&\lesssim_{m,s, p,M,\tb} 
\sum_{s}^{*} 
		\| w \|^{\gamma, \cO}_{m, k_1, k_2+\widehat{\s}_{k-1}+p} 
		\| \alpha \|^{\gamma, \cO}_{k_3+\widehat{\s}_{k-1}}
		\\&+\sum_{s}^{*} 
	\sum_{k_1}^{*} 
		\| w \|^{\gamma, \cO}_{m, k_1', k_2'+\widehat{\s}_{k-1}+p+k_2} 
		\| \alpha \|^{\gamma, \cO}_{k_3'+k_3+2\so+2+\widehat{\s}_{k-1}}
		\\&\lesssim_{m,s, p,M,\tb} 
			\sum_{s}^{*} \| w \|^{\gamma, \cO}_{m, k_1, k_2+{\s}_{k}+p} 
		\| \alpha \|^{\gamma, \cO}_{k_3+{\s}_{k}}
\end{aligned}
\]
for some $\s_{k}\geq \widehat{\s}_{k-1}+2\so+2$ depending only on $\nu$, $|m|$, $M$.
This is the estimate \eqref{zeppelinIndut} with $h=k$.
\\[1mm]
{\sc Step 3: proof of \eqref{zeppelinbis}.} 
Recalling \eqref{balconata0} and setting $\widetilde{q}_{m-1}:=
\sum_{h=1}^{m+\rho-1} q_{m-h} $  
the bounds \eqref{zeppelinIndut} with $1\leq h\leq m+\rho-1$ imply the estimate \eqref{zeppelinbis}
choosing $\s\geq \s_{\rho-1}\geq \ldots\geq \s_1$ large enough w.r.t. $\nu$, $|m|$, $M$. 
\\[1mm]
{\sc Step 4: Estimate of the remainder $\mathtt{r}_{-\rho} $.} 
Recalling \eqref{sperobeneResto} we have, using  \eqref{restocancellettoNp} 
with $a\rightsquigarrow q_{m-k}\in S^{m-k}$, $b\rightsquigarrow\chi\in S^{1}$, $N\rightsquigarrow m-k+1+\rho$ (note that $ 2<N \leq 2|m|+2\rho+1$), 
\begin{equation}\label{stimaRestiRHO}
\begin{aligned}
\|\mathtt{r}_{-\rho}\|_{-\rho,s,p}^{\gamma,\calO}
& \lesssim
\sum_{k=0}^{m+\rho-1}\|q_{m-k}\#_{\geq \rho+m-k+1}\chi \|_{-\rho,s,p}^{\gamma,\calO}
\\&
\stackrel{\eqref{restocancellettoNp}}{\lesssim_{m,s,p,M,\tb}}
\sum_{k=0}^{m+\rho-1}
\|q_{m-k}\|^{\g, \cO}_{m-k, s, 2|m|+2\rho+1+p} \|\chi\|^{\g, \cO}_{1, \so+6|m|+3+6\rho+p, p} 
\\&
\qquad\quad 
+ \|q_{m-k}\|^{\g, \cO}_{m-k, \so, 2|m|+2\rho+1+p} \|\chi\|^{\g, \cO}_{1, s+6|m|+3+6\rho+p, p}
\\&
\stackrel{\eqref{normachichi}}{\lesssim_{m,s,p,M,\tb}}
\sum_{k=0}^{m+\rho-1}
\|q_{m-k}\|^{\g, \cO}_{m-k, s, 2|m|+2\rho+1+p} \|\alpha\|^{\g, \cO}_{\so+6|m|+4+6\rho+p} 
\\&
\qquad\quad 
+ \|q_{m-k}\|^{\g, \cO}_{m-k, \so, 2|m|+2\rho+1+p} \|\alpha\|^{\g, \cO}_{s+6|m|+4+6\rho+p}\,.
\end{aligned}
\end{equation} 
Substituting the bounds \eqref{zeppelinIndut} in \eqref{stimaRestiRHO}
 we get
\begin{align}
	\|\mathtt{r}_{-\rho}\|_{-\rho,s,p}^{\gamma,\calO}
	&{\lesssim_{m,s,p,M,\tb}}
	\sum_{k=0}^{m+\rho-1}
	\big(	\sum_{s}^{*} \| w \|^{\gamma, \cO}_{m, k_1, k_2+{\s}_{k}+2|m|+1+2\rho+p} 
	\| \alpha \|^{\gamma, \cO}_{k_3+{\s}_{k}} \big) \|\alpha\|^{\g, \cO}_{\so+6|m|+4+6\rho+p} 
	\notag \\&+
	\sum_{k=0}^{m+\rho-1}
\big(	\sum_{\so}^{*} \| w \|^{\gamma, \cO}_{m, k_1, k_2+{\s}_{k}+2|m|+1+2\rho+p} 
\| \alpha \|^{\gamma, \cO}_{k_3+{\s}_{k}} \big) \|\alpha\|^{\g, \cO}_{s+6|m|+4+6\rho+p} 
\notag \\ & 
\stackrel{\eqref{interpolotutto}, \eqref{buf}} {\lesssim_{m,s,p,M,\tb}} \sum_{s}^{*} \| w \|^{\gamma, \cO}_{m, k_1, k_2+\widehat{\s}+p} 
\| \alpha \|^{\gamma, \cO}_{k_3+\widehat{\s}+p} \label{stimaerrinomenorho} 
\end{align}
for some $\widehat{\s}\geq \s_{m+\rho-1}+6|m|+4+6\rho$ 
(recall that $\s_{k}$ is a non-decreasing sequence in $k$)
and provided that \eqref{buf} holds for some $\s\geq \widehat{\s}$.
%In the last inequality we also used the interpolation estimate \eqref{interpolotutto}
%following the reasoning used for \eqref{verdone2}.
\\[1mm]
{\sc Step 5: conclusion.}
We now conclude the construction of the solution of \eqref{ars}.
We set 
$ P^{\tau} = Q^{\tau}+R^{\tau} $ where 
$ Q^\tau=\op(q)\in OPS^m $ is the  constructed 
 solution of 
(recall \eqref{balconata0}-\eqref{probapproxsimboloq}) 
	\[
	\partial_{\tau} Q^{\tau}=[{X}^{\tau}, Q^{\tau}] + \cM^\tau\,, \quad  \cM^\tau:=\op(\tr_{-\rho})  \, , \quad Q^{0} = \op(w) \, . 
	\]
Thus, since $ P^\tau $ solves the Heisenberg equation \eqref{ars},  
	\begin{equation}\label{equazioneRRtauesplicita}
			\partial_{\tau} R^{\tau}=[{X}^{\tau}, R^{\tau}]-\mathcal{M}^{\tau} \, , 
			\quad R^{0}=0\, . 
	\end{equation}
	We set $V^{\tau}:=(\mathcal{C}_{\alpha}^{\tau})^{-1}\circ R^{\tau}\circ \mathcal{C}_{\alpha}^{\tau}$
	and we note that 
	$V^{0}=(\mathcal{C}_{\alpha}^{0})^{-1}\circ R^{0}\circ\mathcal{C}_{\alpha}^{0}
	=0$.
	Moreover, by using \eqref{equazioneRRtauesplicita}, \eqref{flussodiffeo}, we have
\[
\begin{aligned}
\pa_{\tau}V^{\tau}&=
\pa_{\tau}((\mathcal{C}_{\alpha}^{\tau})^{-1})\circ R^{\tau}\circ \mathcal{C}_{\alpha}^{\tau}+
(\mathcal{C}_{\alpha}^{\tau})^{-1}\circ(\pa_{\tau}R^{\tau})\circ\mathcal{C}_{\alpha}^{\tau}
+(\mathcal{C}_{\alpha}^{\tau})^{-1}\circ R^{\tau}\circ(\pa_{\tau}\mathcal{C}_{\alpha}^{\tau})
\\&
\stackrel{\eqref{flussodiffeo},\eqref{equazioneRRtauesplicita}}{=}
-(\mathcal{C}_{\alpha}^{\tau})^{-1}\circ X^{\tau}\circ R^{\tau}\circ  \mathcal{C}_{\alpha}^{\tau}
\\&+
(\mathcal{C}_{\alpha}^{\tau})^{-1}\circ\big([{X}^{\tau}, R^{\tau}]-\mathcal{M}^{\tau}\big)\circ\mathcal{C}_{\alpha}^{\tau}
+(\mathcal{C}_{\alpha}^{\tau})^{-1}\circ R^{\tau}\circ X^{\tau}\circ \mathcal{C}_{\alpha}^{\tau}
=
-(\mathcal{C}_{\alpha}^{\tau})^{-1}\circ \mathcal{M}^{\tau}\circ\mathcal{C}_{\alpha}^{\tau}\,.
\end{aligned}
\]
Since  $ V^{0}=0$ we have 
$		V^{\tau}=-
		\int_0^{\tau} 
		(\mathcal{C}_{\alpha}^t)^{-1} \circ
		\mathcal{M}^{t}\circ \mathcal{C}_{\alpha}^{t}\,dt $ 
	and then we deduce that 
	\begin{equation*} 
		R^{\tau}=-
		\int_0^{\tau} \mathcal{C}_{\alpha}^{\tau}\circ (\mathcal{C}_{\alpha}^t)^{-1} \circ
		\mathcal{M}^t \circ \mathcal{C}_{\alpha}^{t} \circ(\mathcal{C}_{\alpha}^{\tau})^{-1}\,dt\,.
	\end{equation*}
	It remains to prove the bound \eqref{stimerestoEgorov1} for $j=0,\tb$.
	We first observe that for $m_1+m_2=M$ 
	\begin{equation}\label{losqualo15}
		\begin{aligned}
			\jap{D}^{m_1}& R^\tau \jap{D}^{m_2}= - \int_0^\tau
			\underbrace{ \jap{D}^{m_1} \mathcal{C}_{\alpha}^{\tau} \jap{D}^{-N-m_1}}_{=:A_1}
			\underbrace{\jap{D}^{N+ m_1}(\mathcal{C}_{\alpha}^t)^{-1}\jap{D}^{-2N-m_1}}_{=:A_2}\\
			&\times
			\underbrace{\jap{D}^{m_1+ 2N} \mathcal{M}^{t}  \jap{D}^{m_2+2N}}_{=:B}
			\underbrace{\jap{D}^{-2N-m_2} \mathcal{C}_{\alpha}^{t} \jap{D}^{m_2+N}}_{=:A_3}
			\underbrace{\jap{D}^{-N-m_2} (\mathcal{C}_{\alpha}^\tau)^{-1} \jap{D}^{m_2}}_{=:A_4} d \tau 
		\end{aligned}
	\end{equation}
	for $N>0$ to be fixed in terms of $\nu$ and $\tb$.
	%To shorten the notation we set (for any $m_1+m_2=M$, any $\tau,t\in [0,1]$, $N$ fixed above, and omitting the dependence on such parameters)
	%\[
	%\begin{aligned}
	%&A_1:=\jap{D}^{m_1} \mathcal{C}_{\alpha}^{\tau} \jap{D}^{-N-m_1}\,,\qquad 
	%A_2:=\jap{D}^{N+ m_1}(\mathcal{C}_{\alpha}^t)^{-1}\jap{D}^{-2N-m_1}\,,
	%\\&
%A_3:=\jap{D}^{-2N-m_2} \mathcal{C}_{\alpha}^{t} \jap{D}^{m_2+N}\,,\qquad 
%	A_4:=\jap{D}^{-N-m_2} (\mathcal{C}_{\alpha}^\tau)^{-1} \jap{D}^{m_2}\,,
%	\\&B:=\jap{D}^{m_1+ 2N} \mathcal{M}^{t}  \jap{D}^{m_2+2N}\,,
%	\end{aligned}
%	\]
	We now show that $A_1A_2BA_3A_4$ 
	and $\langle \td_{\vphi}\rangle^{\tb}A_1A_2BA_3A_4$  
	admits a representative in the space
	$E_{s}$ with estimates \eqref{stimerestoEgorov1} 
	 uniform in $t,\tau\in[0,1]$. We prove this claim for each operator
	$B,A_i$, $i=1,\ldots,4$.
	
Let us start with the operator 
$A_1=\jap{D}^{m_1} \mathcal{C}_{\alpha}^{\tau} \jap{D}^{-N-m_1}$.
We set 
\[
N_1 := -m_1\,,\ \  N_2 := N+m_1\,,  \  \ 
N_1+N_2=N:=\lfloor \nu/2 \rfloor + 4+\tb>\lfloor \nu/2 \rfloor + 3+\tb\,.
\]
Recalling the smallness condition \eqref{buf}
and that $N_1+N_2>\lfloor \nu/2 \rfloor + 3+\tb$, 
 Proposition   \ref{inclusionetotale} 
applies and guarantees that 
there is $\tA_1\in E_{s}$ with $\mathfrak{S}(\tA_1)=A_1$ and satisfying (see 
\eqref{chiarodiluna2})
\begin{equation}\label{losqualo11}
\bnorm{ \langle\td_{\vphi}\rangle^{q}\tA_1}^{\gamma,\calO}_s 
\lesssim_{s,\su,M,q} 1+
\|\al\|_{s+\s}^{\gamma,\calO}\,,\quad q=0,\tb\,,
\end{equation}
for some  $\s>0$ large depending on $M,\nu,\tb$
(note  that $|N_1|, |N_2|$ are bounded from above by a constant depending only on $\nu,M,\tb$
since  $m_1+m_2=M$).  
Reasoning similarly (using again Prop. \ref{inclusionetotale})
one  deduce that each $A_{i}$, $i=2,3,4$ admit representative $\tA_{i}\in E_{s}$
satisfying  estimates  \eqref{losqualo11}.
	
Let us now consider the operator $B$. First of all we recall that 
$\cM^t=\op(\tr_{-\rho}(t;x,\x))$ with symbol 
$\tr_{-\rho}\in S^{-\rho}\subset S^{-\rho+1}$
satisfying the estimate 
\eqref{stimaerrinomenorho}.
Then, by Proposition
	\ref{prop:immersionepseudo} 
	(applied with $m\rightsquigarrow -\rho+1$, see \eqref{sceltarho}\footnote{thanks to the choice of $\rho$ in 
	\eqref{sceltarho} and recalling that  $N:=\lfloor \nu/2 \rfloor + 4+\tb$
	we get $\rho>\rho-1=M+4N=n_1+n_2$.
	},
	$n_1\rightsquigarrow -m_1-2N$, $n_2\rightsquigarrow -m_2-2N$) 
	we have that there is $\mathtt{B}\in E_{s}$ such that $\mathfrak{S}(\tB)=B$ and 
	satisfying, for $q=0,\tb$,
\begin{equation}\label{losqualo14}
\begin{aligned}
\bnorm{\langle \td_{\vphi}\rangle^{q}\mathtt{B}}_{s}^{\gamma,\mathcal{O}}
&\lesssim_{s,\su,n_1,q} 
\|\mathtt{r}_{-\rho}\|_{-\rho+1,s+\gso+|m_1+2N|+q+1,0}^{\gamma,\mathcal{O}}
\\&
\stackrel{\eqref{stimaerrinomenorho} \text{with} \, p = 0 }{\lesssim_{s,\su,M,q}}
\sum_{s}^{*} \| w \|^{\gamma, \cO}_{m, k_1+\widehat\s, k_2+\widehat{\s}} 
\| \alpha \|^{\gamma, \cO}_{k_3+\widehat{\s}} 
%\\
%\bnorm{\mathtt{B}}_{\so}^{\gamma,\mathcal{O}}&\lesssim_{s,\su,M}
%\| w \|^{\gamma, \cO}_{m, \so, \so+\widehat{\s}} 
%\| \alpha \|^{\gamma, \cO}_{\so+\widehat{\s}}\,,
\end{aligned}
\end{equation}
for some $\widehat{\s}>0$ large (possibly larger than the one appearing in \eqref{stimaerrinomenorho})
depending only on $M,\nu,\tb$.
%We remark that to apply Proposition \ref{prop:immersionepseudo} it is enough to use estimate
%	\eqref{stimaerrinomenorho} with $p=0$.
We used that  $m_1,m_2$ are positive and sum to $M$; $N$ 
is chosen in terms of $\nu$;  $\rho$ is chosen  in terms of $\nu, M, \tb$, 
and that the constant appearing in the estimate of Prop. \ref{prop:immersionepseudo}
are non decreasing in $|n_1|$.
Clearly the estimate \eqref{losqualo14} is uniform in $t\in[0,1]$.

Combining the tame estimates provided by 
the composition Lemmata  \ref{tretameestimate}, \ref{tretameestimate3} 
with bounds \eqref{losqualo14}, \eqref{losqualo11}
to estimate $A_{1}A_2BA_3A_4$ and recalling \eqref{losqualo15}, we have that there is 
$\mathtt{R}^{\tau}\in E_{s}$ with $\mathfrak{S}(\mathtt{R}^{\tau})=R^{\tau}$ satisfying 
the bound \eqref{stimerestoEgorov1}.
This concludes the proof.
\end{proof}

\section{Straightening of a first order operator with bi-characteristics}\label{sec:szego}

The main result of this section is the straightening \eqref{straightpotente} 
of the  
first order operator
$ \omega\cdot\pa_{\vphi}-\ii(1+a)|D | $. 
In order to state the result we define the \emph{``Szeg\"o projectors"}  
\begin{equation}\label{def:szego}
\Pi_{\pm}:=\chi_{\pm}(D):=\op(\chi_{\pm} (\xi) )
\end{equation}
where 
$\chi_{\s} (\xi)  $, $\s\in\{\pm\}$,  
are cut-off functions in $ C^{\infty}(\R,[0,1])$  satisfying  $ \chi_{-}(\xi) :=\chi_{+}(-\xi) $ and 
\begin{equation}\label{def:cutoff}
\begin{aligned}
\chi_{+}(\xi)&:=\begin{cases}
0 \quad {\rm if} \quad \xi\le -\frac{1}{2}\\
\tfrac12 \quad {\rm if} \quad \xi = 0 \\
1 \quad {\rm if} \quad \xi\ge \frac{1}{2}\,,
\end{cases} \quad \partial_\xi\chi_{+}(\xi)\geq0 \quad \forall\, \xi\in\R \, , 
\quad \chi_{+}(\xi)+\chi_{-}(\xi) =1\, . 
\end{aligned}
\end{equation}
Note that $ \overline{ \Pi_+ } = \Pi_- $.

\begin{thm}{\bf (Straightening).}\label{IncredibleConjugate}
Let $a(\vphi,x)$ be a real valued function in $ C^\infty (\T^\nu \times \T, \R ) $, 
even separately in $\vphi$ and $  x $,  
depending in a Lipschitz way on 
$\omega\in \mathcal{O}\subseteq \Lambda$.
Fix $\gso,\so$ as in \eqref{costanti}.
For any $M>0$ and $\mathtt{b}\in \N_0 $ 
there is  $\s := \s(M, \tb)>0$ such that 
for any $s\geq \so$ there is $ C(s,\tb)>0$ such that, for any 
$\gamma\in(0,\tfrac{1}{2})$,
if 
\begin{equation}\label{piccino}
C(s,\tb)\g^{-1} \|a\|^{\g,\mathcal O}_{\so+\s} \leq 1 \, , 
\end{equation}
then 
there exists a Lipschitz  function 
$ \fa_+ : \Lambda \to \R $, $ \omega \mapsto \fa_+(\omega) $,  
satisfying 
\begin{equation}\label{tordo4b}
\qquad \lvert  \fa_+  \rvert^{\gamma,\Lambda}\leq 2\|a\|^{\g,\mathcal O}_{\so+\s} \, , \,
\end{equation}
such that for any $\omega $ belonging to 
\begin{equation}\label{buoneacqueb}
\Omega_{1}:= \Big\{\omega \in \cO: \; 
|\omega \cdot \ell +(1+\fa_+(\omega) ) j  |
>2\gamma \langle \ell\rangle^{-\tau}\,,
\;\forall (\ell,j)\in \Z^{\nu+1} \setminus \{0\}\Big\}
\end{equation}
the following holds.
There exist functions $\alpha_{\s} : \T^{\nu+1}\times \Omega_{1}\to \R$, $\s\in\{\pm\}$,  
satisfying  
\begin{equation}\label{atrio3} 
\alpha_{\s}(\vphi,x)=-\alpha_{\s}(-\vphi,-x)\,, \quad
 \alpha_{-}(\vphi,x):=-\alpha_{+}(\vphi,-x) \,,
 \end{equation}
 and
\begin{equation}\label{stimapm}
\|\alpha_{\s} \|^{\g,\Omega_{1}}_{s} \lesssim_s  \g^{-1}\|a\|^{\g,\cO}_{s+\s} \, ,
\quad \forall s\geq \so\,,
\end{equation}
such that the linear operator 
\begin{equation}\label{ellediffeo}
L:=\mathcal{C}_{\alpha_+}\Pi_{+}+\mathcal{C}_{\alpha_-}\Pi_{-}
\end{equation}
with $\mathcal{C}_{\alpha_{\s}}$  given in \eqref{ignobel},
maps $H^{s}$ to $H^{s}$, $s\geq \gso$, is  
invertible,  reversibility and parity preserving, and 
\begin{equation}\label{straightpotente}
L\circ\big( \omega\cdot\pa_{\vphi}-\ii(1+a)|D| \big)\circ L^{-1}=
\omega\cdot\pa_{\vphi}-\ii (1+\fa_{+})|D|+R
\end{equation}
where $R$ is a   parity preserving, reversible  bounded operator.
Moreover, assuming \eqref{piccino} for all $s\leq \su$,
then
there exists a couple $\mathtt{R}\in E_{s}\equiv E_{s,\gso,\su}$ 
such that $\mathfrak{S}(\mathtt{R})=R$,
satisfying, for any $m_1,m_2\in\R_+$ , $m_1+m_2=M$,  $\gso\le s\le \su$
(recall \eqref{derivata})
\begin{align}
\bnorm{\jap{\td_{\vphi}}^{q} \jap{D}^{m_1}\mathtt{R}\jap{D}^{m_2}}_s^{\g, \Omega_{1}} 
&\lesssim_{s,\su,q,M} 
\g^{-1}\|a\|^{\g,\cO}_{s+\s}\, , \quad q = 0, \tb \, .  
\label{raichel2}
\end{align}
\end{thm}
The rest of the section is devoted to the proof of Theorem \ref{IncredibleConjugate}. 

\begin{rmk}\label{rmk:cutoff}
(i) The function 
 $	\sign(\xi) =\chi_{+}(\xi)-\chi_{-}(\xi) $ 
is in $C^{\infty}(\R;[-1,1])$, it is odd  and 
$\sign(\xi)=1$ if $\xi\geq1/2$ and $\sign(\xi)=-1$ if $\xi\leq -1/2$.

(ii) The operator $\Pi_{+}\Pi_{-} = %\Pi_{\pm} - \Pi_{\pm}^2 
\op(\chi_{+}(\x)\chi_{-}(\x))$
is  pseudo-differential  
with symbol $\chi_{+}\chi_{-} $ in $ S^{-\infty}$ (see Def. \ref{norma pesata}).

(iii) The symbol $1-\chi\in S^{-\infty}$ where $\chi(\xi)$ is the function defined in \eqref{cutoff}. Furthermore $\Pi_+\Pi_-\op(\chi)=0$ and  $\ii |D|\Pi_\s = \s \pa_x \op(\chi)\Pi_\s$.

\end{rmk}

We first study  the commutator 
between the Szeg\"o projectors $ \Pi_\pm $ in \eqref{def:szego} and the torus diffeomorphism 
$\mathcal{C}_{\alpha} $ defined in \eqref{ignobel}, and  with a pseudo-differential operator.

\begin{lemma}\label{commutarechebello}
{\bf (Commutator with Szeg\"o projectors)}
 Let $\chi(\xi)$ be the function defined in \eqref{cutoff}.  
There is $ \eta > 0 $ and, for any $s\geq \so$,  there is $\delta :=\delta(s) >0  $ 
such that if $\| \al\|_{\so+\eta}^{\g, \cO} \le \delta $ then 
 the commutators  
$$ 
[\Pi_{\s},\mathcal{C}_{\alpha}]=\op(g_\s)  \, , \ \s\in\{\pm\} \, ,   \qquad 
[\op{(\chi)},\mathcal{C}_{\alpha}]=\op(h)  \, , 
$$
are pseudo-differential 
with symbols $ g_{\s}, h $  in $ S^{-\infty}$ satisfying, for any  
$p \geq 0$, $ N \in \N $ and  for some  $ \mu :=\mu(N,p) >0$,  
\begin{equation}\label{chebellocommutare}
\|g_{\s}\|^{\g,\calO}_{-N,s,p}\lesssim_{N,s,p}\|\alpha\|^{\g,\calO}_{s+\mu}\, ,
\quad 	\|h\|^{\g,\calO}_{-N,s,p}\lesssim_{N,s,p}\|\alpha\|^{\g,\calO}_{s+\mu}\,.
\end{equation}
For any symbol $ r \in S^m $, for any $\s\in\{\pm\}$, the  commutator  
$ [\Pi_{\s}, \op(r)]=\op(f_{\s}) $  is pseudo-differential 
with a symbol $ f_{\s} $  in $ S^{-\infty}$ satisfying, for any  
$p \geq 0$, $ N \in \N $ and for some  $ \mu :=\mu(N,m,p) >0$,  
\begin{equation}\label{chebellocommutare1}
\|f_{\s}\|^{\g,\calO}_{-N,s,p}\lesssim_{N,s,p}\|r\|^{\g,\calO}_{m,s+\mu,p+\mu}\, .
\end{equation}
\end{lemma}

\begin{proof}
We start  proving  the estimate \eqref{chebellocommutare}
for $ g_{\s} $.  We note that 
$ [\Pi_{\s},\mathcal{C}_{\alpha}] u =  \sum_{k\in\Z} g_{\s}(\vphi,x,k) u_k e^{\im k x} $
where, recalling the definition of   $\widehat{\tt}_\alpha(\vphi,k,\xi)$  in  \eqref{sacrorito}, 
\begin{equation}\label{bottiglia}
\begin{aligned}
g_{\s}(\omega;\vphi,x,\xi)\equiv g_{\s}(\vphi,x,\xi)&:= \sum_{k\in\Z} (\chi_{\s}(\x+k)- \chi_{\s}(\x))
\widehat{\tt}_\alpha(\vphi,k,\xi) e^{\im k x}
\\&
= \sum_{k\in\Z} (\chi_{\s}(\x+k)- \chi_{\s}(\x))
(\widehat{\tt}_\alpha(\vphi,k,\xi) -\delta(k,0))e^{\im k x}
\end{aligned}
\end{equation}
where  $\delta(k,0)$ is the Kroneker delta.
Recalling \eqref{normaSymbolo}, the first  bound  in \eqref{chebellocommutare}  follows  directly by proving that if $\| \al\|_{\so+\eta}^{\g, \cO} \le \delta $, with $\eta,\delta$ fixed in Lemma \ref{decayultravioletto}, then  the  following estimate  for the  Fourier coefficients in time and space $\doublehat{g}_{\s}(\ell,k,\x)$ of $g_\s$ holds.   For  any $  s\ge \so,\, N,p\ge 0$  one has
\begin{equation}
	\label{gcappuccio}
	|\pa_{\x}^{p}\,\doublehat{g}_{\s}(\ell,k,\x)|^{\g,\calO}
	\lesssim_{s,N,p}
	\frac{1}{\langle \ell,k\rangle^{s}}
	\|\alpha\|^{\g,\calO}_{s+N+p+\eta} \langle\x\rangle^{-N-p}\,,
\end{equation}
indeed \eqref{gcappuccio} implies the first bound in \eqref{chebellocommutare} with $\mu= s+N+p+\eta+\so$.
We have reduced to proving \eqref{gcappuccio}.
The key point is  that in \eqref{bottiglia} one has 
\begin{equation}\label{nome}
	\chi_\s(k+ \xi)-\chi_\s(\xi)\ne 0\quad \Rightarrow \quad |k|>|\xi|-\tfrac12\,.
\end{equation}
Indeed
\begin{align*}
	\chi_\s(k+ \xi)-\chi_\s(\xi)&=   \chi_\s(k+ \xi)(\chi_+(\xi)+ \chi_-(\xi))-\chi_\s(\xi) (\chi_+(k+\xi)+ \chi_-(k+\xi))\\
	&= - \s\chi_+(\xi)\chi_-(k+\xi)+\s\chi_-(\xi)\chi_+(k+\xi)  \,.
\end{align*} 
Now, if $|\xi|< \tfrac12$ the second inequality  in \eqref{nome}  is always satisfied. 
In the other case the condition  $\chi_+(\xi)\chi_-(k+\xi)\neq0$ implies  (by  the definition of $\chi_\pm$ in \eqref{def:cutoff}) 
that 
$ \xi>- \tfrac12$ and  $\xi +k < \tfrac12$. So if  $|\xi|\ge \tfrac12$ one gets from the first inequality  $\xi \geq \tfrac12$ and from the second one $k< \tfrac12-\xi \leq 0$.  Then, since $\xi>0$ and $k \leq 0$, the condition $\xi +k < \tfrac12$ reads $|k|
>  |\xi|-\tfrac12$. The same holds for  $\chi_-(\xi)\chi_+(k+\xi)$.
 
From \eqref{nome} the sum in \eqref{bottiglia} is restricted to  $|k|> |\xi|-\tfrac12$.  Then for any $p\in \N_0 $,
 $|k| > |\xi|-\frac12$
\begin{align*}
&| \pa_{\x}^{p}\doublehat{g}_{\s}(\ell,k,\x)|^{\g,\calO}=\\
&=\frac{1}{{(2\pi)}^{\nu}}\Big|\int_{\T^\nu}\!d\vphi\, e^{-\im \ell  \cdot \vphi}
\!\! \sum_{p_1+p_2=p}
 \binom{p}{p_1}
\pa_{\x}^{p_1}(\chi_{\s}(\x+k)- \chi_{\s}(\x))
 \pa_{\x}^{p_2}(\widehat{\tt}_\alpha(\vphi,k,\xi) -\delta(k,0))\Big|^{\g,\calO} \\
 &\lesssim_p   \sum_{p_1+p_2=p} \Big|\pa_{\x}^{p_1}(\chi_{\s}(\x+k)- \chi_{\s}(\x))\Big|
 \Big|\int_{\T^\nu} d\vphi \, e^{-\im \ell  \cdot \vphi}  \pa_{\x}^{p_2}(\widehat{\tt}_\alpha(\vphi,k,\xi) -\delta(k,0))\Big|^{\g,\calO} \\
 &\lesssim_p  \sum_{p_1+p_2=p} \Big|\int_{\T^{\nu+1}} d\vphi dx \, e^{-\im \ell  \cdot \vphi - \im k x} 
  \pa_{\x}^{p_2}( e^{\im \alpha(\vphi,x)\xi} - 1) \Big|^{\g,\calO}
  =  C_p \sum_{p_1+p_2=p} \Big| 
  \doublehat{ \pa_{\x}^{p_2} (\tt_{\alpha}  - 1)} (\ell,k,\xi) \Big|^{\g,\calO}
\end{align*}
by recalling \eqref{sacrorito}.
Now if $p_2=0$, recalling that $|k|\geq |\xi|-\tfrac12$ and applying Lemma  \ref{decayultravioletto}  with % $\s\rightsquigarrow \eta$ and 
$s\rightsquigarrow s+N+p$ , we get 
\begin{equation*}
\begin{aligned}
\Big| 
  \doublehat{ (\tt_{\alpha}  - 1)} (\ell,k,\xi) \Big|^{\g,\calO} 
&\lesssim_{s,N,p}  \frac{\jap{\xi}^{N+p}}{\langle \ell,k\rangle^{s+N+p}} \|\al\|^{\g,\calO}_{s+N+p+\eta} \jap{\xi}^{-N-p}\\
&\lesssim_{s,N,p}  \frac{1}{\langle \ell,k\rangle^{s}} \|\al\|^{\g,\calO}_{s+N+p+\eta} \jap{\xi}^{-N-p}\,.
\end{aligned}
\end{equation*}
Similarly for $p_2\ne 0$ we have 
\begin{align*}
&\Big| \doublehat{( \alpha^{p_2} \tt_{\alpha} ) } (\ell,k,\xi)  
 \Big|^{\g,\calO} 
\lesssim_{s,N}  \frac{1}{\langle \ell,k\rangle^{s}} \|\al\|^{\g,\calO}_{s+N+p+\eta}\jap{\xi}^{-N-p}\,,
\end{align*}
by using 
Lemma  \ref{decayultravioletto} with $f=\alpha^{p_2}$ and $s\rightsquigarrow s+N+p$, the tame estimates  \eqref{tameProduct} and $\|\alpha\|_{\so+\eta}^{\g,\cO}\le \delta$. 
This proves the first estimate in \eqref{chebellocommutare}.
% Regarding the bounds for $ h $ we follow the strategy 
% outlined above with the trivial modification that
%$ \chi(\xi+k)-\chi(\xi)\ne 0  \rightarrow |k| >|\xi|-1 $. 
The estimates for $ h $ and $ f_\sigma  $ % in \eqref{chebellocommutare1} 
follow similarly.
\end{proof}

The next lemma proves the invertibility of the map $ L $ defined in \eqref{ellediffeo}. 

\begin{lemma} \label{lemma: invertibilita operatore trasporto}
{\bf (Invertibility of $ L $)}
Let $\alpha_+,\alpha_{-}$ be real valued functions in $ C^{\infty}(\T^{\nu+1}, \R) $ 
and define
$ \fd(s):=\|\alpha_+\|^{\g, \cO}_{s}+\|\alpha_-\|^{\g, \cO}_{s} $.
Then, for any $N \geq 0 $ and $\mathtt{b}\in \N$ 
there is  $\s :=  \s(N,\tb)  >0$ and  $ C(s_1,N,\tb)>0$ such that
if 
\begin{equation}\label{veropiace2}
C(s_1,N,\tb)\fd(\so+ \s)<1\,,
\end{equation}
then, for any 
$  \gso \le s\le \su$, the operator 
$ L $ % and $ \overline{L} $  
in \eqref{ellediffeo} is invertible with an  inverse of  
the form 
\begin{equation}\label{formaL-1}
L^{-1} =
(\mathcal{C}_{\alpha_+}^{-1}\Pi_{+}+\mathcal{C}_{\alpha_-}^{-1}\Pi_{-})\circ(\id+\widetilde{R})\,,
\qquad \big(\overline{L}\big)^{-1}=\overline{L^{-1}} \, , 
\end{equation}
 where $\widetilde{R}$ 
is associated to a $\widetilde{\mathtt{R}}\in E_{s}$ such that $\mathfrak{S}(\widetilde{\mathtt{R}})=\widetilde{R}$ and
for any $m_1,m_2\in \R_+$, $m_1+m_2=N$
\begin{align}
\bnorm{\jap{\td_{\vphi}}^{q} \jap{D}^{m_1}\widetilde{\mathtt{R}}\jap{D}^{m_2}}_s^{\g, \cO} 
&\lesssim_{s,\su,\tb,N} 
\fd (s+ \s)\, , \quad q = 0 \, , \tb \, .
\label{raichel}
\end{align}
Under the assumption \eqref{atrio3} 
 the map  $L$  % with $\alpha_{+},\alpha_{-}$ given above 
is reversibility and parity preserving.
\end{lemma}

\begin{proof}
The operator $\Gamma:= \mathcal{C}_{\alpha_+}^{-1}\Pi_{+}+\mathcal{C}_{\alpha_-}^{-1}\Pi_{-} $
 is an \emph{approximate right } inverse of $ L $ in the following sense: 
\begin{align}
L\circ \Gamma&=(\mathcal{C}_{\alpha_+}\Pi_{+}+\mathcal{C}_{\alpha_-}\Pi_{-})\circ 
(\mathcal{C}_{\alpha_+}^{-1}\Pi_{+}+\mathcal{C}_{\alpha_-}^{-1}\Pi_{-})\nonumber
\\&
=\mathcal{C}_{\alpha_+}\Pi_{+}\mathcal{C}_{\alpha_+}^{-1}\Pi_{+}+
\mathcal{C}_{\alpha_+}\Pi_{+}\mathcal{C}_{\alpha_-}^{-1}\Pi_{-}
+\mathcal{C}_{\alpha_-}\Pi_{-}\mathcal{C}_{\alpha_+}^{-1}\Pi_{+}+
\mathcal{C}_{\alpha_-}\Pi_{-}\mathcal{C}_{\alpha_-}^{-1}\Pi_{-}\nonumber
\\&
=\Pi_{+}^2+\Pi_{-}^2 + 2\Pi_+\Pi_-\label{identitybella1}
\\&+
(\mathcal{C}_{\alpha_+}\mathcal{C}_{\alpha_-}^{-1}-\id)\Pi_{+}\Pi_{-}+
(\mathcal{C}_{\alpha_-}\mathcal{C}_{\alpha_+}^{-1}-\id)\Pi_{-}\Pi_{+}\label{identitybella2}
\\&+
\mathcal{C}_{\alpha_+}[\Pi_{+}, \mathcal{C}_{\alpha_+}^{-1} ]\Pi_{+}+
\mathcal{C}_{\alpha_+}[\Pi_{+},\mathcal{C}_{\alpha_-}^{-1}]\Pi_{-}+
\mathcal{C}_{\alpha_-}[\Pi_{-},\mathcal{C}_{\alpha_+}^{-1}]\Pi_{+}+
\mathcal{C}_{\alpha_-}[\Pi_{-},\mathcal{C}_{\alpha_-}^{-1}]\Pi_{-}\,.\label{identitybella3}
\end{align}
Since  $ \Pi_+ + \Pi_- = \id $
we have that  $\eqref{identitybella1}= \id $
and then 
\begin{equation}\label{esempioQ}
L\circ\Gamma=\id +Q \, , \qquad Q := \eqref{identitybella2}+\eqref{identitybella3} \, .
\end{equation}
We claim that the operator 
$ Q $ is  associated to a couple 
 $ \mathtt{Q} $ such that  $\mathfrak{S}(\mathtt{Q})=Q$ and  satisfies 
\eqref{raichel}. 
Consider for instance the operator 
$B:=\mathcal{C}_{\alpha_+}[\Pi_{+}, \mathcal{C}_{\alpha_+}^{-1} ]\Pi_{+}$
 in \eqref{identitybella3}.
Let  
$N':=\lfloor\nu/2\rfloor+4+\tb $. We write
\[
\begin{aligned}
\langle \td_{\vphi}\rangle^{\tb}\langle D\rangle^{m_1}B\langle D\rangle^{m_2}&=
\langle \td_{\vphi}\rangle^{\tb}\langle D\rangle^{m_1}\mathcal{C}_{\alpha_+}[\Pi_{+}, \mathcal{C}_{\alpha_+}^{-1} ]\langle D\rangle^{m_2}
\\&=
\langle \td_{\vphi}\rangle^{\tb}\langle D\rangle^{m_1}\mathcal{C}_{\alpha_+}
\langle D\rangle^{-m_1-N'}
\langle D\rangle^{m_1+N'}[\Pi_{+}, \mathcal{C}_{\alpha_+}^{-1} ]\langle D\rangle^{m_2}\,.
\end{aligned}
\]
By the choice of $N'$  and the smallness \eqref{veropiace2}
 Proposition \ref{inclusionetotale},
 taking $N_1\rightsquigarrow -m_1$, 
$N_2\rightsquigarrow m_1+N'$ so that $N_1+N_2=N'=\lfloor\nu/2\rfloor+4+\tb>\lfloor\nu/2\rfloor+3+\tb$,  \footnote{
since $m_1\in \R_{+}$, one has $|N_1|,|N_2|\lesssim_{N,\tb}1$.
Hence the condition \eqref{buttalapasta3} is implied by \eqref{veropiace2} taking $\s=\s(N,\tb)$ large enough, while
the constants appearing in \eqref{chiarodiluna2} 
depends actually only on $N$ and $\tb$, see Proposition \ref{inclusionetotale}. }
implies that there is $\mathtt{R}_1\in E_{s}$ such that 
\begin{equation}\label{peperoncino1}
\mathfrak{S}(\mathtt{R}_1)=\langle D\rangle^{m_1}(\mathcal{C}_{\alpha_+}-\id)\langle D\rangle^{-m_1-N'} \qquad 
\bnorm{\langle \td_{\vphi}\rangle^{\tb}\mathtt{R}_1
}_{s}^{\gamma,\calO}\lesssim_{s,s_1N,\tb}\fd(s+ \s)\,,
\end{equation}
for some $\s>0$. On the other hand, Lemma \ref{commutarechebello} implies that there is 
a symbol $\mathtt{g}\in S^{-N''}$ for any $N''>0$.
Taking $N''=N+N'$ by Proposition \ref{prop:immersionepseudo} (applied with $-n_1\rightsquigarrow m_1+N'$, $-n_2\rightsquigarrow m_2 $)
we get that there is $\mathtt{R}_2$ such that
\begin{equation}\label{peperoncino2}
\mathfrak{S}(\mathtt{R}_2)=\langle D\rangle^{m_1+N'}[\Pi_{+}, \mathcal{C}_{\alpha_+}^{-1} ]\langle D\rangle^{m_2}, \qquad 
\bnorm{\langle \td_{\vphi}\rangle^{\tb}\mathtt{R}_2
}_{s}^{\gamma,\calO}\lesssim_{s,s_1,N,\tb}\fd(s+ \s)\, . 
\end{equation}
By \eqref{peperoncino1}-\eqref{peperoncino2} and Lemmata \ref{tretameestimate} and \ref{tretameestimate3}
we conclude that $B$ satisfies  \eqref{raichel}.
The other terms in \eqref{identitybella3} can be treated similarly.
In order to show that the terms 
in \eqref{identitybella2} satisfy \eqref{raichel} one has to use that, 
by Remark \ref{rmk:cutoff} ($ii$), one has 
$\Pi_{-}\Pi_{+}=\op(\chi_+\chi_-)$ with $ \chi_+\chi_-\in S^{-\infty}$. 
In conclusion, by  \eqref{raichel} and Lemma \ref{inv3sbarrette} 
we deduce that
 the operator $\id +Q$ is invertible 
 and the inverse has the form  $(\id +Q)^{-1}=: \id +\widetilde{R}$ with $\widetilde{R}$ 
 satisfying \eqref{raichel}. Thus $L$ admits a right inverse.

The map $\Gamma$ is also an \emph{approximate left} inverse of $ L $
 since $ \Gamma \circ L =\id +P$ where 
$P $ is the operator obtained from $ Q $ in \eqref{esempioQ} exchanging $\mathcal{C}_{\alpha_\sigma} $ with 
$\mathcal{C}_{\alpha_\sigma}^{-1} $. 
 
As seen for $Q$,  there exists $\tP\in E_s$ which satisfies \eqref{raichel} and such that $\fS(\tP)=P$.
Thus $L$ is invertible.

To show that $L$   
is reversibility and parity preserving we write,
by \eqref{sacrorito2} and \eqref{def:szego},  
$ L = \mathcal{C}_{\alpha_{+}}\Pi_{+}+\mathcal{C}_{\alpha_{-}}\Pi_{-} 
= \op (a) $ with 
$ a(\vphi, x, \xi) := 
e^{\ii \xi \alpha_{+}(\varphi,x)}\chi_{+}(\xi )+e^{\ii \xi \alpha_{-}(\varphi,x)}\chi_{-}(\xi ) $. 
Then  by 
$ \alpha_{+}(-\vphi,x)=\alpha_{-}(\vphi,x) $ (see \eqref{atrio3})  
and $ \chi_{-}(\xi) =\chi_{+}(-\xi)  $ we check that 
$ a(\varphi,x, \xi) $ satisfies 
\eqref{reverpreserSimbolo} and \eqref{paritySimbolo}. 
\end{proof}

In  the proof of Theorem \ref{IncredibleConjugate} 
we use also the following lemma. 

\begin{lemma}[{\bf Straightening of transport}]\label{prop:straight}
Let $a(\vphi,x)\in C^\infty$ be a real valued function, 
even separately in $ \vphi $, $ x $ satisfying the smallness condition \eqref{piccino}.
Then  there exists a Lipschitz  function $\fa_+ : \Lambda \to \R, \omega \mapsto \fa_+(\omega)$,  
satisfying \eqref{tordo4b} such that 
for any $\omega\in \Omega_{1}$ defined in \eqref{buoneacqueb}
there exists a real 
function $ \beta_{+} (\omega;\varphi, x) $ 
satisfying  
\begin{equation}\label{perunpugnodidollarib}
\|\beta_{+} \|^{\g,\Omega_{1}}_{s} \lesssim_s  \g^{-1}\|a\|^{\g,\cO}_{s+\s},
\quad \forall s\geq \so\,,
\end{equation}
solving the equation 
\begin{equation}\label{veropiace}
\omega\cdot\pa_{\vphi}\beta_{+}-(1+a)(1+\pa_{x}\beta_{+})= -  { (1+\fa_{+})} \, . 
\end{equation}
Moreover  $\beta_{+}$  satisfies
\begin{equation}\label{perunpugnodidollaric}
\beta_{+}(\vphi,x)=-\beta_{+}(-\vphi,-x) 
\end{equation}
and   $\beta_{-}(\vphi,x):=-\beta_{+}(\vphi,-x)$ 
solves the equation
\begin{equation}\label{contoabotta}
\omega\cdot\pa_{\vphi}\beta_{-}+(1+a)(1+\pa_{x}\beta_-)=  { 1+\fa_{+}} \,.
\end{equation}
\end{lemma}

\begin{proof}
The existence of $\mathfrak{a}_{+}$ and  $\beta_{+}$ 
follow by Proposition $3.4$ and Corollary $3.5$ in \cite{FGMPmoser}.
The function $\beta_{-}(\vphi,x) = -\beta_{+}(\vphi,-x) $ solves 
 \eqref{contoabotta}  by  \eqref{veropiace} and since 
$a(\vphi,x)$ is even in $ (\varphi, x) $.
\end{proof}

\begin{proof}[{\bf Proof of Theorem \ref{IncredibleConjugate}}]
Under the smallness assumption \eqref{piccino} we deduce  
by Lemma \ref{prop:straight}   the existence of functions 
$ \beta_\sigma (\varphi, x) $ satisfying \eqref{perunpugnodidollarib}-\eqref{contoabotta}. 
In addition, for any $ \sigma \in \{ \pm 1 \} $ the torus diffeomorphism $ y\mapsto y+\beta_{\s}(\vphi,y)$ 
is invertible: 
there exist unique functions $\alpha_{\s} $ % : \T^{\nu+1}\times \Omega_{1}\to \R$, 
satisfying an estimate as \eqref{perunpugnodidollarib}, 
$\alpha_{+}(\vphi,x)=-\alpha_{+}(-\vphi,-x)  $ (cfr. \eqref{perunpugnodidollaric}), 
$ \alpha_{-}(\vphi,x):=-\alpha_{+}(\vphi,-x) $ (thus \eqref{atrio3} holds) 
 such that
\begin{equation}\label{albeta}
x\mapsto y=x+\alpha_{\s}(\vphi,x) 
\;\;\; 
\Leftrightarrow 
\;\;\; 
y\mapsto x=y+\beta_{\s}(\vphi,y)\,,\quad \s\in\{\pm\}\, .
\end{equation}
Moreover for any $\tau\in[0,1]$ the diffeomorphism $x\mapsto x+\tau \alpha_{\s}(\vphi,x)$
is invertible
and there exist functions
$\breve{\alpha}_{\s}: [0,1]\times \T^{\nu+1}\times \Omega_{1}\to \R$,  for any $ \s\in\{\pm\}$, 
satisfying 
\eqref{perunpugnodidollarib},   \eqref{perunpugnodidollaric}, such that 
\begin{equation}\label{speranzabella}
x\mapsto y=x+ \tau \alpha_{\s}(\vphi,x) 
\;\;\; 
\Leftrightarrow 
\;\;\; 
y\mapsto x=y+\breve \alpha_{\s}(\tau; \vphi,y) \, , \quad
\breve{\alpha}_{\s}(1;\vphi,x)=\beta_{\s}(\vphi,x) \, .
\end{equation}
We now  prove \eqref{straightpotente}-\eqref{raichel2} where
$ L $  is the operator  defined in \eqref{ellediffeo}  
with  functions $ \alpha_\sigma $ defined in \eqref{albeta}. 
Using Remark \ref{rmk:cutoff}-($iii$) we write 
\begin{equation}\label{defA+-}
\begin{aligned}
A := \omega\cdot\pa_{\vphi}-\ii(1+a)|D|  
%  \omega\cdot\pa_{\vphi}-\ii(1+a)\op(|\x|\chi(\xi))  \stackrel{\eqref{def:cutoff}} 
=A_+ \Pi_{+} + A_- \Pi_{-} \   \ 
 \text{with} \  \ A_{\s} := \oo\cdot\pa_{\vphi} -\s(1+a)\pa_{x}\op{(\chi)}\,.
\end{aligned}
\end{equation}
The key step is the following lemma.
\begin{lemma}\label{lemmaApiumeno}
\begin{equation}\label{quasifine}
L\circ A\circ L^{-1} 
= \mathcal{C}_{\alpha_+} \circ A_{+}\circ \mathcal{C}_{\alpha_+}^{-1} \circ \Pi_{+}
	 \ + \ \mathcal{C}_{\alpha_-} \circ A_{-}\circ \mathcal{C}_{\alpha_-}^{-1}  \circ \Pi_{-} + Q
\end{equation}
where $ A_\pm $ are defined in \eqref{defA+-} and $ Q $ satisfies \eqref{raichel2}.
\end{lemma}

\begin{proof}
We apply Lemma \ref{lemma: invertibilita operatore trasporto}
with $N:=M+\nu+10+2\tb$ (here $M>0$ is fixed in Thm. \ref{IncredibleConjugate})
so that $L^{-1}$ has the form \eqref{formaL-1} with $\widetilde{R}$ satisfying  
\eqref{raichel}.
\\[1mm]
{\bf Step $ 1 $.}
We first observe that 
$
L \circ  \oo\cdot \pa_{\vphi} \circ L^{-1} =\oo\cdot\pa_{\vphi} + [L, \oo\cdot\pa_{\vphi}]L^{-1}
$. Now recalling \eqref{ellediffeo}, \eqref{formaL-1}
and since 
$ [ \cC_{\alpha_{\sigma}}, \oo \cdot \pa ] = - (\oo\cdot\pa_{\vphi} \alpha_{\s}) \cC_{\alpha_{\s}} \pa_{x} $ we get 
\begin{align}
 [L, \oo\cdot\pa_{\vphi}]L^{-1}
 &  = \left[ \cC_{\alpha_{+}}\Pi_{+}+\cC_{\alpha_{-}}\Pi_{-} ,  \oo\cdot\pa_{\vphi} \right]L^{-1}
= \sum_{\s\in\{\pm\}} \left[ \cC_{\alpha_{\s}}, \oo\cdot\pa_{\vphi} \right] \Pi_{\s}L^{-1} \notag \\
& =  -\sum_{\s\in\{\pm\}} (\oo\cdot\pa_{\vphi} \alpha_{\s}) \cC_{\alpha_{\s}} \pa_{x} \Pi_{\s} (\cC^{-1}_{\alpha_{+}}\Pi_{+}+\cC^{-1}_{\alpha_{-}}\Pi_{-}) + Q_{1} \label{tempoint}
\end{align}
where
\begin{equation} \label{restoQQ1}
Q_{1} :=   -\sum_{\s\in\{\pm\}} (\oo\cdot\pa_{\vphi} \alpha_{\s}) \cC_{\alpha_{\s}} \pa_{x} \Pi_{\s} (\cC^{-1}_{\alpha_{+}}\Pi_{+}+\cC^{-1}_{\alpha_{-}}\Pi_{-}) \widetilde{R}\, . 
\end{equation}
We continue expanding \eqref{tempoint} as
\begin{align}
 [L, \oo\cdot\pa_{\vphi}]L^{-1} & = -\sum_{\s,\s'\in\{\pm\}} (\oo\cdot\pa_{\vphi} \alpha_{\s}) \cC_{\alpha_{\s}} \pa_{x}  \Pi_{\s}\cC^{-1}_{\alpha_{\s'}}\Pi_{\s'} + Q_{1} \notag \\
&= -\sum_{\s\in\{\pm\}} (\oo\cdot\pa_{\vphi} \alpha_{\s}) \cC_{\alpha_{\s}} \pa_{x}  \cC^{-1}_{\alpha_{\s}}\Pi_{\s} 
+ Q_{1} + Q_{2} \notag \\
& =  \sum_{\s\in\{\pm\}} \left[ \cC_{\alpha_{\s}}, \oo\cdot\pa_{\vphi} \right] \cC^{-1}_{\alpha_{\s}}\Pi_{\s} + Q_{1} + Q_{2} \label{q1q2l}
\end{align}
where 
\begin{align}
Q_{2}
&:=  - \!\!\!\! \sum_{\s\in\{\pm\}} \!\!\! (\oo\cdot\pa_{\vphi} \alpha_{\s}) \cC_{\alpha_{\s}} \pa_{x} \! \left(  (\cC^{-1}_{\alpha_{-\s}}- \cC^{-1}_{\alpha_{\s}}) \Pi_{\s}\Pi_{-\s} + [\Pi_{\s}, \cC^{-1}_{\alpha_{-\s}}\!- \!\cC^{-1}_{\alpha_{\s}}]\Pi_{-\s}+ \! [\Pi_{\s}, \cC^{-1}_{\alpha_{\s}}] \right).
\label{restoQQ2}
\end{align}
By  \eqref{q1q2l} 
we conclude  that 
$ L \circ  \oo\cdot \pa_{\vphi} \circ L^{-1} =\oo\cdot\pa_{\vphi} + [L, \oo\cdot\pa_{\vphi}]L^{-1} $ is equal to 
\begin{equation}\label{contime}
L \circ \oo\cdot \pa_{\vphi} \circ  L^{-1} 
= \sum_{\s\in\{\pm\}}  \cC_{\alpha_{\s}}  \circ \oo\cdot\pa_{\vphi} \circ 
 \cC^{-1}_{\alpha_{\s}}\circ \Pi_{\s} + Q_{1} + Q_{2}\,.
\end{equation}
\noindent
{\bf Step $ 2 $.}
We now conjugate $ B:=-\ii(1+a) |D | $. 
By \eqref{ellediffeo} and  \eqref{formaL-1} we have 
\begin{equation}\label{midway1}
L\circ B\circ L^{-1}=
\sum_{\s,\s'\in\{\pm\}}\mathcal{C}_{\alpha_{\s}}\Pi_{\s}B\mathcal{C}_{\alpha_{\s'}}^{-1}\Pi_{\s'}+Q_3
\end{equation}
where 
\begin{equation}\label{defQ3}
Q_3 := (\mathcal{C}_{\alpha_+}\Pi_{+}+\mathcal{C}_{\alpha_-}\Pi_{-})
  B   (\mathcal{C}_{\alpha_+}^{-1}\Pi_{+}+\mathcal{C}_{\alpha_-}^{-1}\Pi_{-})\widetilde{R} \,.
\end{equation}
Then we write  \eqref{midway1} as 
\begin{align}
& L\circ B\circ L^{-1}  = 
\nonumber
\\&
=\sum_{\s,\s'\in\{\pm\}}\mathcal{C}_{\alpha_{\s}}B\mathcal{C}_{\alpha_{\s'}}^{-1}\Pi_{\s}\Pi_{\s'} +
\sum_{\s,\s'\in\{\pm\}}\mathcal{C}_{\alpha_{\s}}B[\Pi_{\s},\mathcal{C}_{\alpha_{\s'}}^{-1}]\Pi_{\s'}+
\sum_{\s,\s'\in\{\pm\}}\mathcal{C}_{\alpha_{\s}}[\Pi_{\s},B]\mathcal{C}_{\alpha_{\s'}}^{-1}\Pi_{\s'}
+Q_3\nonumber
\\&
=\sum_{\s\in\{\pm\}}\mathcal{C}_{\alpha_{\s}}B\mathcal{C}_{\alpha_{\s}}^{-1}\Pi_{\s}
+ P_1 + P_2  +Q_3   \label{pasintt} 
\end{align}
where, using $ \Pi_\sigma + \Pi_{-\sigma} = \id $, 
\begin{align} 
P_1 & := 
 \sum_{\s\in\{\pm\}}\mathcal{C}_{\alpha_{\s}}B(\mathcal{C}_{\alpha_{-\s}}^{-1}-\mathcal{C}_{\alpha_{\s}}^{-1} )\Pi_{\s}\Pi_{-\s}\label{midway3}
\\
P_2 & := \sum_{\s,\s'\in\{\pm\}}\mathcal{C}_{\alpha_{\s}}B[\Pi_{\s},\mathcal{C}_{\alpha_{\s'}}^{-1}]\Pi_{\s'}+
\sum_{\s,\s'\in\{\pm\}}\mathcal{C}_{\alpha_{\s}}[\Pi_{\s},B]\mathcal{C}_{\alpha_{\s'}}^{-1}\Pi_{\s'} \,.
\label{midway4} 
\end{align} 
Next, using Remark \ref{rmk:cutoff}-($iii$), we decompose
\begin{equation}\label{Bsumpm}
B= B_+ \Pi_{+} + B_- \Pi_{-}  \qquad \text{where} \qquad 
B_+:=-(1+a)\pa_{x}\op{(\chi)} \,,\ \ 
B_-:=  (1+a)\pa_{x}\op{(\chi)}  \, , 
\end{equation}
we get, by \eqref{pasintt},  
\begin{align}
L \circ B \circ L^{-1}& = 
\sum_{\s\in\{\pm\}}\mathcal{C}_{\alpha_{\s}}B_{\s} \Pi_\sigma 
\mathcal{C}_{\alpha_{\s}}^{-1}\Pi_{\s}
+\sum_{\s\in\{\pm\}}\mathcal{C}_{\alpha_{\s}} B_{-\s}
\Pi_{-\s} \mathcal{C}_{\alpha_{\s}}^{-1}\Pi_{\s}
+ P_1 + P_2 +Q_3 \notag   \\
& 
= \sum_{\s\in\{\pm\}}\mathcal{C}_{\alpha_{\s}}
\circ B_{\s} \circ \mathcal{C}_{\alpha_{\s}}^{-1} \circ \Pi_{\s}
+ P_3 + P_4 + P_1 + P_2  +Q_3 \label{conspace}
\end{align}
 where
\begin{align}
P_3 &  := \sum_{\s\in\{\pm\}}\mathcal{C}_{\alpha_{\s}}B_{\s} [\Pi_\sigma, 
\mathcal{C}_{\alpha_{\s}}^{-1}]\Pi_{\s} +\sum_{\s\in\{\pm\}}\mathcal{C}_{\alpha_{\s}} B_{-\s}
[\Pi_{-\s}, \mathcal{C}_{\alpha_{\s}}^{-1}]\Pi_{\s} \label{midway6}
 \\
%R_3 &  :=  \sum_{\s\in\{\pm\}}\mathcal{C}_{\alpha_{\s}}\big[(1+a)\pa_{x}\op{(\chi)}\big]
%\mathcal{C}_{\alpha_{\s}}^{-1}\Pi_{-\s}\Pi_{\s} \label{midway6} \\
P_4 & := 
 \sum_{\s\in\{\pm\}}\mathcal{C}_{\alpha_{\s}}B_{-\s}
 \mathcal{C}_{\alpha_{\s}}^{-1}\Pi_{-\s}\Pi_{\s} +  \sum_{\s\in\{\pm\}}\mathcal{C}_{\alpha_{\s}}B_{\s} 
\mathcal{C}_{\alpha_{\s}}^{-1}(\Pi^{2}_{\s} - \Pi_{\s}) \nonumber
\\
 &= 
 \sum_{\s\in\{\pm\}}\mathcal{C}_{\alpha_{\s}}(B_{-\s} - B_\s)
\mathcal{C}_{\alpha_{\s}}^{-1}  \Pi_{-\s}\Pi_{\s} \nonumber
\\&=
\big(\mathcal{C}_{\alpha_{+}}(B_{-} - B_{+})
(\mathcal{C}_{\alpha_{+}}^{-1}-\mathcal{C}_{\alpha_{-}}^{-1})
-
(\mathcal{C}_{\alpha_{-}} - \mathcal{C}_{\alpha_{+}} )
(B_{-} - B_{+})\mathcal{C}_{\alpha_{-}}^{-1}
\big) \Pi_{-}\Pi_{+}\,.\label{midway7}
%R_4 & := 
%\sum_{\s\in\{\pm\}}\mathcal{C}_{\alpha_{\s}}\big[ (1+a)\pa_{x}\op{(\chi)}\big]
%\Big[\Pi_{-\s},\mathcal{C}_{\alpha_{\s}}^{-1}\Big]\Pi_{\s} \, .  \label{midway7}
\end{align}
By \eqref{contime}, \eqref{conspace} and \eqref{Bsumpm}, 
\eqref{defA+-} we deduce \eqref{quasifine}
with $ Q := Q_1 + Q_2 + Q_3+ P_1 +  P_2 + P_3 + P_4 $.
\\[1mm]
{\bf Step $ 3 $ - Estimate of $ Q $.}
We start with the remainders $Q_{1}$ and $Q_3$ defined in \eqref{restoQQ1} and \eqref{defQ3}.
We  have to bound operators of the form
$ \mathcal{M} \mathcal{C}_{\alpha_\s}\Pi_{\s} \cQ
\mathcal{C}_{\alpha_{\s'}}^{-1}\Pi_{\s'}\widetilde{R} $ 
where $\cQ=\pa_{x}$ or $\cQ=B$ and $\mathcal{M}=1$ or $\mathcal{M}=(\oo\cdot\pa_{\vphi} \alpha_{\s}) $,
and $\widetilde{R}$ is the operator  in \eqref{formaL-1} with $N=M+\nu+10+2\tb$.
We write, for any $\s,\s'\in\{\pm\}$, $m_1,m_2\in\R_+$ , $m_1+m_2=M$,
\[
\begin{aligned}
\langle \td_{\vphi}\rangle^{\tb}&\langle D\rangle^{m_1}\mathcal{M}\mathcal{C}_{\alpha_\s}\Pi_{\s} \cQ
\mathcal{C}_{\alpha_{\s'}}^{-1}\Pi_{\s'}\widetilde{R}\langle D\rangle^{m_2}
\\&=
\langle \td_{\vphi}\rangle^{\tb}\langle D\rangle^{m_1}
\mathcal{M}\langle D\rangle^{-m_1} \circ
\langle D\rangle^{m_1}
\mathcal{C}_{\alpha_\s}
\langle D\rangle^{-m_1-N'}\circ\langle D\rangle^{m_1+N'}
\Pi_{\s} \cQ\langle D\rangle^{-m_1-N'-1}
\\&
\circ\langle D\rangle^{m_1+N'+1}\mathcal{C}_{\alpha_{\s'}}^{-1}\Pi_{\s'}
\langle D\rangle^{-m_1-2N'-1} \circ
\langle D\rangle^{m_1+2N'+1}
\widetilde{R}\langle D\rangle^{m_2}
\end{aligned}
\]
where  $N':=\lfloor\nu/2\rfloor+4+\tb>\lfloor\nu/2\rfloor+3+\tb$.
 By Proposition \ref{prop:immersionepseudo}
(with $n_1= - n_{2}\rightsquigarrow -m_1 $) there exists 
 $\mathtt{R}_0\in E_{s}$ such that, for some $\mu:=\mu(M,\tb)>0 $,
\begin{equation}\label{pastaalforno}
\mathfrak{S}(\mathtt{R}_0)
=\langle D\rangle^{m_1}\mathcal{M}\langle D\rangle^{-m_1},
\qquad
\bnorm{\langle \td_{\vphi}\rangle^{\tb}\mathtt{R}_0
}_{s}^{\gamma,\calO}\lesssim_{s,\su,M,\tb}1+\g^{-1} \|a\|^{\g,\mathcal O}_{\so+\mu} \, . 
\end{equation}
Thanks to the choice of $N'$  and \eqref{piccino}, 
Proposition \ref{inclusionetotale}
(with $N_1\rightsquigarrow -m_1$, $N_2\rightsquigarrow m_1+N' $),
using that $ \alpha_\sigma $ satisfy \eqref{perunpugnodidollarib}, we deduce 
that there is $\mathtt{R}_1\in E_{s}$ such that 
\begin{equation}\label{peperoncino10}
\mathfrak{S}(\mathtt{R}_1)=\langle D\rangle^{m_1}(\mathcal{C}_{\alpha_\s}-\id)\langle D\rangle^{-m_1-N'}, \qquad
\bnorm{\langle \td_{\vphi}\rangle^{\tb}\mathtt{R}_1
}_{s}^{\gamma,\calO}\lesssim_{s,\su,M,\tb}\g^{-1} \|a\|^{\g,\mathcal O}_{\so+\mu}\, .
\end{equation}
Secondly we apply Proposition \ref{prop:immersionepseudo}
(with $-n_1\rightsquigarrow m_1+N' $)
to infer that there is 
 $\mathtt{R}_2$ such that
\begin{equation}\label{peperoncino20}
\mathfrak{S}(\mathtt{R}_2)=\langle D\rangle^{m_1+N'}\Pi_{\s}\cQ\langle D\rangle^{-m_1-N'-1},
\qquad 
\bnorm{\langle \td_{\vphi}\rangle^{\tb}\mathtt{R}_2
}_{s}^{\gamma,\calO}\lesssim_{s,\su,M,\tb}1+\g^{-1} \|a\|^{\g,\mathcal O}_{\so+\mu} \, . 
\end{equation}
Furthermore 
Proposition \ref{inclusionetotale}
(with $N_1\rightsquigarrow -m_1-N'-1$, 
$N_2\rightsquigarrow m_1+2N'+1 $)
implies that there is $\mathtt{R}_3\in E_{s}$ such that%, for some $\mu > 0 $,  
\begin{equation}\label{peperoncino109}
\mathfrak{S}(\mathtt{R}_3)=\langle D\rangle^{m_1+N'+1}(\mathcal{C}^{-1}_{\alpha_{\s'}}-\id)
\langle D\rangle^{-m_1-2N'-1}, \quad 
\bnorm{\langle \td_{\vphi}\rangle^{\tb}\mathtt{R}_3
}_{s}^{\gamma,\calO}\lesssim_{s,\su,M,\tb}\g^{-1} \|a\|^{\g,\mathcal O}_{\so+\mu}
\, .
\end{equation}
Finally, by 
\eqref{raichel} (with $m_1\rightsquigarrow m_1+2N'+1$, $m_2\rightsquigarrow m_2$
and recalling that $N:=M+\nu+10+2\tb$)
we deduce that there is 
 $\mathtt{R}_4$ such that
\begin{equation}\label{peperoncino30}
\mathfrak{S}(\mathtt{R}_4)=\langle D\rangle^{m_1+2N'+1}
\widetilde{R}\langle D\rangle^{m_2}, \quad 
\bnorm{\langle \td_{\vphi}\rangle^{\tb}\mathtt{R}_4
}_{s}^{\gamma,\calO}\lesssim_{s,\su,M,\tb}\g^{-1} \|a\|^{\g,\mathcal O}_{\so+\mu} 
\,.
\end{equation}
By \eqref{pastaalforno}-\eqref{peperoncino30} 
and Lemmata \ref{tretameestimate}, \ref{tretameestimate3}
we conclude that $\mathcal{M}\mathcal{C}_{\alpha_\s}\Pi_{\s} Q \mathcal{C}^{-1}_{\alpha_{\s'}}\Pi_{\s'}\widetilde{R}$ 
satisfies  \eqref{raichel2}.

\smallskip
We are left with the remainders $Q_2$ in \eqref{restoQQ2}
and 
$ P_1 , P_2 ,P_3 , P_4$ defined respectively in \eqref{midway3}, \eqref{midway4},
 \eqref{midway6}, \eqref{midway7}. 
 Note that all these operators are compositions of  
 pseudo differential operators, torus diffeomorphisms and Szeg\"o projectors, which we have already 
discussed in 
formul\ae\,\eqref{pastaalforno}-\eqref{peperoncino30}.
In particular the terms 
\eqref{midway4}, \eqref{midway6} contain commutators 
between $\Pi_{\s}$ and either $\mathcal{C}_{\alpha_{\s'}}$ or a pseudo differential operator, thus by Lemma \ref{commutarechebello} they are symbols in $S^{-\infty}$ and, by  Proposition \ref{prop:immersionepseudo}, they satisfy estimates 
% small w.r.t. $a$ 
like  \eqref{raichel}.

Similarly the terms \eqref{midway3}, \eqref{midway7}
contain the operators $\Pi_{\s}\Pi_{-\s}$ and either 
$\mathcal{C}_{\alpha_{-}}^{-1}-\mathcal{C}_{\alpha_{+}}^{-1} $
or $ \mathcal{C}_{\alpha_{-}}-\mathcal{C}_{\alpha_{+}}$.
By Remark
\ref{rmk:cutoff},
the operator $\Pi_{\s}\Pi_{-\s}$
is pseudo differential with symbol in $S^{-\infty}$, and, by  Proposition \ref{prop:immersionepseudo} 
there is $\mathtt{R}_5\in E_{s}$, such that, for any  $ m_1 + m_2 = N  $,  
\[
\mathfrak{S}(\mathtt{R}_{5})=\langle D\rangle^{m_1}\Pi_{\s}\Pi_{-\s}\langle D\rangle^{m_2}\,,
\qquad
\bnorm{\langle \td_{\vphi}\rangle^{\tb}\mathtt{R}_5
}_{s}^{\gamma,\calO}\lesssim_{s,\su,M,\tb}1\,. 
\]
Finally the terms $ \mathcal{C}_{\alpha_{-}}-\mathcal{C}_{\alpha_{+}}$, 
$\mathcal{C}_{\alpha_{-}}^{-1}-\mathcal{C}_{\alpha_{+}}^{-1}$
satisfy estimates like \eqref{peperoncino10}.

In conclusion 
$Q=Q_1+ Q_{2}+Q_3 + P_{1} + P_{2} + P_{3} + P_{4}$ satisfies \eqref{raichel2}.
\end{proof}

\noindent
{\bf Conclusion of the proof of Thm. \ref{IncredibleConjugate}}
We further expand $ L\circ A\circ L^{-1}  $ in \eqref{quasifine}. 
The function 
$ \breve \alpha_\sigma (\cdot) :=  \breve \alpha_\sigma (1; \cdot) = \beta_\sigma $ (see \eqref{speranzabella}) solve  the equations \eqref{veropiace}, \eqref{contoabotta}. 
In view of  \eqref{defA+-}, \eqref{ignobel},   it results 
\begin{align}
\mathcal{C}_{\alpha_\s} \circ A_\s\circ \mathcal{C}_{\alpha_\s}^{-1}&=
\mathcal{C}_{\alpha_\s} \circ (\omega\cdot\pa_{\vphi}-\s(1+a)\pa_{x}
\op(\chi ) )\circ \mathcal{C}_{\alpha_\s}^{-1} \notag
\\&=
\omega\cdot\pa_{\vphi}+
\mathcal{C}_{\alpha_{\s}} (\omega\cdot\pa_{\vphi}\breve{\alpha}_\s) \pa_{x}
-\s \mathcal{C}_{\alpha_{\s}} \big((1+a)(1+\pa_x\breve{\alpha}_{\s})  \big)\pa_{x} \circ  \mathcal{C}_{\alpha_{\s}}\op(\chi)\mathcal{C}_{\alpha_{\s}}^{-1} \notag \\
&=
\omega\cdot\pa_{\vphi}+
\Big(\mathcal{C}_{\alpha_{\s}} (\omega\cdot\pa_{\vphi}\breve{\alpha}_\s)
-\s \mathcal{C}_{\alpha_{\s}}\big((1+a)(1+\pa_x\breve{\alpha}_{\s})\big)\op(\chi)\Big)\pa_{x} +Q_{4,\s}
\notag \\
& \stackrel{  \breve \alpha_\sigma  = \beta_\sigma }{=}
\omega\cdot\pa_{\vphi}+
\mathcal{C}_{\alpha_{\s}} \big( \omega\cdot\pa_{\vphi} \beta_\s
-\s (1+a)(1+\pa_x \beta_{\s}) \big) \op(\chi)\pa_{x} + Q_{4,\s}  + Q_{5,\s} \notag
\\
&\stackrel{ \eqref{veropiace},\eqref{contoabotta}}{=}
\omega\cdot\pa_{\vphi}-
\s(1+\fa_+)\op(\chi)\pa_{x}  + Q_{4,\s} + Q_{5,\s} \label{Q3Q4}
\end{align}
where % $\mathcal{C}_{\alpha_{\s}}(f)(x)=f(x+\alpha_{\s}(x))$ and
\[
Q_{4,\s}:=-\s \mathcal{C}_{\alpha_{\s}} \big( (1+a)(1+\pa_x\breve{\alpha}_{\s}))\pa_{x} \circ 
\cC_{\alpha_{\s}}
[\op(\chi),\cC_{\alpha_{\s}}^{-1}]\,,\quad 
Q_{5,\s}:= \mathcal{C}_{\alpha_{\s}} (\omega\cdot\pa_{\vphi}\breve{\alpha}_\s)\op(1-\chi) \pa_x \, . 
\]
The operators $Q_{4,\s}$, $Q_{5,\s}$  satisfy \eqref{raichel2}
because $[\op(\chi),\cC_{\alpha_{\s}}^{-1}] $, $ \op(1-\chi) $ are operators with symbols in 
$ S^{-\infty} $ (cfr. Remark \ref{rmk:cutoff}
and Lemma \ref{commutarechebello})
reasoning as done in Lemma \ref{lemmaApiumeno}.
By \eqref{quasifine}, \eqref{Q3Q4} and 
setting 
$$
R := Q + (Q_{4,+}+ Q_{5,+})\Pi_+ + (Q_{4,-}+ Q_{5,-})\Pi_-
$$
we conclude that 
\begin{align}
L\circ A\circ L^{-1}&=
 \big(\omega\cdot\pa_{\vphi}-(1+\mathfrak{a}_+)\op(\chi)\pa_{x}\big) \Pi_{+}
+\big(\omega\cdot\pa_{\vphi}+(1+\mathfrak{a}_+)\op(\chi)\pa_{x}\big) \Pi_{-} + R
\notag \\&
=\big(\omega\cdot\pa_{\vphi}
- \ii (1+\mathfrak{a}_+)\op(\chi(\xi) |\xi|)\big)\big(\Pi_{+}+\Pi_{-}\big)+R \notag 
\\&
=\omega\cdot\pa_{\vphi}-\ii(1+\mathfrak{a}_+)|D|+R \label{defRfin}
\end{align}
proving \eqref{straightpotente}.
The operator $ R $ satisfies \eqref{raichel2} as $ Q $ 
(cfr. Lemma \ref{lemmaApiumeno}) and $Q_{4,\s}$, $Q_{5,\s}$ .

The operator $A$ is reversibility and parity preserving. 
By Lemma \ref{lemma: invertibilita operatore trasporto} 
and since $ \alpha_\pm $ satisfy \eqref{atrio3}, 
the operator $L$, thus $ L A L^{-1} $ is  reversibility and parity preserving as well. 
Since also  
$\omega\cdot\pa_{\vphi}-\ii(1+\mathfrak{a}_+)|D|$ (since $\mathfrak{a}_{+}\in \R$)
is reversibility and parity preserving we deduce that 
 the remainder $R$ in \eqref{defRfin}  has the same properties by difference.
\end{proof}

\part{Reducibility of the  Klein-Gordon equation}

We now start the reduction 
of the quasi-periodic Klein-Gordon operator 
$\mathcal{L} $ in \eqref{ellePhiPhi}.
In Section \ref{sec:diagBlock} we symmetrize  $\mathcal{L} $ 
up to smoothing remainders. % at the highest order.
Then in Section \ref{sec:ordine11} 
we reduce it to constant coefficients up to one-smoothing remainders. 
Here we use the quantitative Egorov theorem  of Section \ref{sec:egomichela}.
% this is the core of the paper.  
In Section \ref{sec:kam} we complete the diagonalization of $\mathcal{L} $ 
by a KAM iterative scheme. Finally Section \ref{sec:measure} contains the 
measure estimates.

\section{Symmetrization up to smoothing remainders}\label{sec:diagBlock}

In view of  \eqref{firstorderComplex}  we write the linear operator $ \cL $ in 
\eqref{ellePhiPhi} as
\begin{equation}\label{L-omega}
\cL=\oo\cdot\pa_\vphi - \ii E\op\Big( (\Id+  b_{1} (\varphi,x) \uno)\mathtt{D_{\mathtt{m}}}(\xi)
+\ii b_{0}(\varphi,x) \uno\x\mathtt{D}_{\mathtt{m}}^{-1}(\x)+b_{-1}(\varphi,x) \uno\mathtt{D}_{\mathtt{m}}^{-1}(\x)
\Big) 
\end{equation}
where $\tD_{\mathtt{m}}(\x) = \sqrt{\xi^2+ \mathtt m }$,
the matrices $E$,  $\Id $, $\uno$,   are defined in
\eqref{involutionReale}, \eqref{invoPP}, 
\eqref{matE}  
and 
\begin{equation}\label{simbolistep0}
b_{1} :=-\tfrac{1}{2}a^\2\,,  
\quad 
b_{0}
:= \tfrac{1}{2}a^\1\,,
\quad
b_{-1}
:= \tfrac{1}{2}\big(\mathtt{m}a^\2 + a^\0\big)\, . 
\end{equation}  
The functions $b_{i}$ are real valued
and, 
by the parity conditions \eqref{oddness}-\eqref{revers}, 
 $b_{1},b_{-1}$ are  even separately in $ \vphi \in  \T^\nu  $
and  $ x \in  \T $, 
 while $b_{0}$ is even in $\vphi$ and odd in $x$. Thus, by Lemma  \ref{paritasuisimboli}, each matrix of symbols appearing 
 as summand in \eqref{L-omega} is real-to-real, reversibility and parity preserving.
Hence, by Lemma \ref{equidefalgebra}-($v$), 
the operator $\mathcal{L}$ is real-to-real, reversible and parity preserving 
according to Definition \ref{giornatasolare}.

Recall 
the parameters $\gamma,\tau,\so,\su$ defined  in \eqref{costanti}-\eqref{costantiGAMMA}.
For $s>0$ we define 
\begin{equation}\label{condireiniziale}
\epsilon(s):=\|b_{1}\|_s + \|b_{0}\|_s + \|b_{-1}\|_s\, . 
\end{equation}
We will often use
the property
$\epsilon(s+\beta_1)\epsilon(s_0+\beta_2)\lesssim \epsilon(s_0)\epsilon(s+\beta_1+\beta_2)$, 
which follows by the interpolation estimate \eqref{interpolotutto}.
In the next sections we assume a   smallness condition of the form
\begin{equation}\label{ipopiccolezza}
\gamma^{-7/2}\epsilon(\so+\s_*)\leq \delta_*\,\,
\end{equation}
for some $\s_*>0$ large enough 
$0<\delta_*\ll1$ sufficiently small  (depending on $\su$). Along the following sections we shall choose 
 $\s_*$ larger and $\delta_{*}$ smaller.

\subsection{Symmetrization at order 1}
We construct a 
change of variables which symmetrizes  the operator $\mathcal{L}$ in \eqref{L-omega}
at leading order, 
by diagonalizing the matrix of symbols 
$E(\Id+b_{1}\uno)$.
\begin{prop}{\bf (Symmetrization at order 1).}\label{diagonalizzazione order 1}
%Fix $\tau, \so$ as in \eqref{costanti}. 
For $s_1$ as in \eqref{costantiGAMMA}
there is $\delta_{0} :=\delta_{0}(\su)>0$
such that, for any $\gamma\in(0,\tfrac{1}{2})$,
if  the smallness condition \eqref{ipopiccolezza} holds for $\s_{*}\geq 0$ and $\delta_{*}\le \delta_{0}$,
the following holds. 
There exists an invertible, real-to-real, reversibility and parity preserving  
multiplication operator $\mathcal{U}=\op(U(\varphi,x))$
with $U\in S^{0}\otimes\mathcal{M}_2(\C)$
such that
\begin{equation}\label{elle2}
\begin{aligned}
\cL_2&:= \cU^{-1} \cL \, \cU
\\&= 
\oo\cdot\pa_\vphi -\ii E\op \Big(\lambda (\varphi,x) \Id
\mathtt{D_{\mathtt{m}}}(\xi) +
\ii b_{0}(\varphi,x) \Id\x\mathtt{D}_{\mathtt{m}}^{-1}(\x)
+A_{0}^\2(\vphi, x, \xi) + A_{-1}^\2(\vphi, x, \xi) 
\Big)
\end{aligned}
\end{equation}
where $ b_{0} $ is the function in \eqref{simbolistep0}, 
\begin{equation}\label{nuovocoef}
 \lambda(\varphi,x) := \sqrt{1+ 2b_{1}(\varphi,x)} \,,
 \quad A_{0}^\2(\vphi, x, \xi):=\sm{0}{c^{(2)}(\vphi,x,\x)}{\overline{c^{(2)}(\vphi,x,-\x)}}{0}\,,
 \quad c^{(2)}\in S^0\,,
 \end{equation}
while $A_{-1}^{(2)} $ is a matrix of symbols in $ S^{-1}\otimes\mathcal{M}_2(\C)$. More precisely $\lambda$ is a real valued function, 
 even separately  in $ \vphi \in \T^{\nu} $ and 
 $ x\in \T$,  while
$ A_0^{(2)}, A_{-1}^{(2)} $ are  real-to-real, reversibility and  parity preserving matrices of symbols
satisfying for any $\so\le s\le \su$ and any $p\ge 0$
\begin{equation}\label{stimasulresto}
\|A_0^{(2)}\|^{\gamma,\Lambda}_{0, s, p}\,,
\|A^{(2)}_{-1}\|^{\gamma,\Lambda}_{-1, s, p} 
\lesssim_{s, p} \epsilon(s+p+5) \,, 
\quad \|\lambda-1\|^{\gamma,\Lambda}_{0, s, p}\,, \ 
	\|U-\Id \|^{\gamma,\Lambda}_{0, s, p}\lesssim_{s}\epsilon(s)\,.%\qquad s\in[\so,\su]
\end{equation}
\end{prop}

\begin{proof}
Define the matrix of functions
$ U(\vphi,x):=\sm{f}{g}{g}{f}  $
where 
\[
f:=f(\vphi, x) :=\frac{1+b_{1}+\lambda}{\sqrt{(1+b_{1}+\lambda)^2-{b_{1}}^2}}\,, 
\qquad 
g:=g(\vphi, x) :=\frac{-b_{1}}{\sqrt{(1+b_{1}+\lambda)^2-{b_{1}}^2}}
\]
and  $\lambda$ is the function \eqref{nuovocoef}. 
By Moser composition estimates 
on Sobolev spaces	one gets the estimate
\begin{equation}\label{stimelg}
\|\lambda-1\|_{s}^{\gamma,\Lambda}+\|f-1\|_{s}^{\gamma,\Lambda} 
+\|g\|_{s}^{\gamma,\Lambda} \lesssim_{s}\|b_{1}\|_{s}^{\gamma,\Lambda}\,.
\end{equation}
Therefore, recalling also \eqref{condireiniziale}, 
the functions $\lambda,U$ satisfy the estimate in  \eqref{stimasulresto}.
Since $b_{1}$ is even separately in $\vphi,x$, 
the same property holds for the functions $f$ and $g$
and $\mathcal{U}$ is real-to-real, reversibility and parity preserving.
By explicit computations we have that
\begin{equation}\label{invC}
{\rm det}(U)=f^2-g^2=1  \, , \quad 
U^{-1} =\sm{f}{-g}{-g}{f}  \, ,
\qquad U^{-1}E(\Id+b_{1}\uno)U=E \lambda \Id\,.
\end{equation}
The conjugated operator $\cL_2 = \cU^{-1} \cL \,  \cU  $ is, 
using that  $  U  $ is a function,  
\begin{equation}\label{acqua}
\cL_2 
=\cU^{-1} \oo\cdot\pa_\vphi \cU
-\ii\op\Big(
U^{-1}  E\big(  
(\Id+  b_{1} \uno)\mathtt{D_{\mathtt{m}}}(\xi)
+\ii b_{0}\uno\x\mathtt{D}_{\mathtt{m}}^{-1}(\x)+b_{-1}\uno\mathtt{D}_{\mathtt{m}}^{-1}(\x)
\big)\# U 
\Big)\,.
\end{equation}
By \eqref{invC} we deduce that 
$ f (\oo\cdot\pa_\vphi f) - g (\oo\cdot\pa_\vphi g) = 0 $ and 
\begin{equation}\label{tempo}
\cU^{-1} \oo\cdot\pa_\vphi \cU = 
 \oo\cdot\pa_\vphi  - \ii E \begin{pmatrix}
0 &&\ii ( f (\oo\cdot\pa_\vphi g) - g (\oo\cdot\pa_\vphi f))\\
-\ii (f (\oo\cdot\pa_\vphi g) - g (\oo\cdot\pa_\vphi f) )&& 0
\end{pmatrix} \, . 
\end{equation}
We now consider the highest order term in 
\eqref{acqua}. By \eqref{invC} one has 
\begin{equation}\label{histstep1}
U^{-1} E(\Id+  b_{1} \uno)\mathtt{D_{\mathtt{m}}}(\xi)  \# U =  E \lambda \Id  U^{-1}\mathtt{D_{\mathtt{m}}}(\xi)  \# U =   E \lambda \Id \mathtt{D}_{\mathtt{m}}(\xi)  +  E \lambda U^{-1} (\mathtt{D}_{\mathtt{m}}(\xi)\Id  \star U)\,.
\end{equation}
Recalling 
\eqref{mareostia} and \eqref{espstar} we can write 
\begin{equation}
	\label{fascio}
E \lambda U^{-1}(\mathtt{D}_{\mathtt{m}}(\xi)\Id  \star U) =E\sm{e_1}{h_1}{h_1}{e_1}\,,
\qquad 
\begin{aligned}
e_1&:=\lambda[ f(\tD_{\mathtt{m}}(\xi) \star f) - g (\tD_{\mathtt{m}}(\xi) \star g)]\,,
\\
h_1&:=\lambda[f (\tD_{\mathtt{m}}(\xi) \star g) - g (\tD_{\mathtt{m}}(\xi) \star f)]\,,
\end{aligned}
\end{equation}
where $h_1\in S^0$.  
On the other hand,   
since 
\begin{equation} \label{perverso}
f(\tD_{\mathtt{m}}(\xi) \star f) = 
 -\frac{\ii}2  \{\tD_\tm(\xi), f^2\}+  f \tD_{\mathtt{m}}(\xi) \#_{\ge2} f 
\end{equation}
(same for $g$) and recalling that $ \{\tD_\tm(\xi), f^2\}= \{\tD_\tm(\xi), g^2\}$, 
we deduce that the symbol $ e_1 $ is equal to 
\begin{equation}\label{defe1r}
e_1= 
\lambda[ f \tD_{\mathtt{m}}(\xi) \#_{\ge2} (f-1) -  g \tD_{\mathtt{m}}(\xi) \#_{\ge2} g]\in S^{-1}\, . 
\end{equation}
We now consider the lower  order terms in 
\eqref{acqua}.
Since $U^{-1} \uno U = \uno$ we have that 
\begin{equation}\label{pezzopiubasso}
 U^{-1}  \im b_{0} E \uno \xi\tD^{-1}_{\mathtt{m}}(\xi)\# U=
E  \big(\im b_{0} \xi\tD^{-1}_{\mathtt{m}}(\xi) \uno  + \im E  U^{-1}  b_{0} (\xi\tD^{-1}_{\mathtt{m}}(\xi) E \uno \star U)\big) \, .
\end{equation}
Then, by \eqref{acqua}, \eqref{tempo}, \eqref{histstep1}, \eqref{fascio}, 
\eqref{pezzopiubasso} and 
since $U^{-1} E = E U $, we deduce that $ \cL_2 $ has the form \eqref{elle2}-\eqref{nuovocoef}
with 
\[
A_{-1}^\2:= \sm{e_1}{0}{0}{e_1} +  \im E  U^{-1}  b_{0} (\xi\tD^{-1}_{\mathtt{m}}(\xi) E \uno \star U)+U b_{-1}\uno\tD^{-1}_{\mathtt{m}}(\xi)\# U 
\]
and
$$
c^{(2)} := 
 \ii (f (\oo\cdot\pa_\vphi g) - g (\oo\cdot\pa_\vphi f))+ h_1 +
 \ii b_{0}\xi\tD^{-1}_{\mathtt{m}}(\xi)  \, .
$$
The matrices of symbols 
$A_{0}^\2,A_{-1}^\2$
are real-to-real, reversibility and parity preserving 
since $f, g$ are even in $\vphi$ and even in $x$
(hence are reversibility and parity preserving symbols independent of $\x$)
and using Lemmata
\ref{paritasuisimboli}-\ref{paritasuisimboliBIS}.

We now prove the remaining estimates  \eqref{stimasulresto}. 
Using \eqref{tameProduct}-\eqref{interpolotutto},   \eqref{stimelg}
and the smallness condition \eqref{ipopiccolezza} with  $\s_{*}\geq 0 $, 
the term \eqref{tempo} is bounded by
\begin{equation}\label{restino1}
\|f (\oo\cdot\pa_\vphi g) - g (\oo\cdot\pa_\vphi f)\|_{s}^{\g, \Lambda}  \lesssim_{s} 
\| b_1 \|_{s+1}^{\g, \Lambda} + 
 \| b_1 \|_{s}^{\g, \Lambda}   \| b_1 \|_{s_0+1}^{\g, \Lambda}   \lesssim_{s}  
\epsilon(s+1) \,.
\end{equation}
Similarly 
% Using \eqref{tameProduct}-\eqref{interpolotutto}, 
using also Lemma \ref{lemma:Commutator} 
we get 
\begin{equation}
	\label{restino12}
\begin{aligned}
\|\lambda g (\tD_{\mathtt{m}}(\xi) \star  f)\|^{\g, \Lambda}_{0,s,p} 
&\lesssim_{s,p} \|\lambda g\|^{\g, \Lambda}_{s} \|\tD_{\mathtt{m}}(\xi) \star  f\|^{\g, \Lambda}_{0,s_0,p} +  \|\lambda g\|^{\g, \Lambda}_{s_0} \|\tD_{\mathtt{m}}(\xi) \star  f\|^{\g, \Lambda}_{0,s,p} 
\\& \lesssim_{s,p} \|\lambda g \|_s^{\g, \Lambda} \| f - 1 
\|_{\so+p+3}^{\g, \Lambda}  + \|\lambda g \|_{\so}^{\g, \Lambda}\| f - 
1 \|_{s+p+3}^{\g, \Lambda} 
\\
&  \lesssim_{s,p}  
\| b_1 \|_s^{\g, \Lambda} \| b_1 \|_{\so+p+3}^{\g, \Lambda}  +
 \| b_1 \|_{\so}^{\g, \Lambda}\| b_1 \|_{s+p+3}^{\g, \Lambda} 
% \lesssim_{s,p} \epsilon(s+p+3)(1+\epsilon(\so+3))^2 
\lesssim_{s,p} \epsilon(s+p+3) \, . 
\end{aligned}
\end{equation}
Thus the symbol $ h_1 $ in \eqref{fascio} is bounded  by 
$ \| h_1 \|^{\g, \Lambda}_{0,s,p}  \lesssim_{s,p} \epsilon(s+p+3)  $.  
%using the smallness condition \eqref{ipopiccolezza} with  $\s_{*}\geq 0 $.
Finally using Lemma \ref{stima composizione} 
%and the smallness condition \eqref{ipopiccolezza} with  $\s_{*}\geq 0$ 
we have 
\begin{equation}
	\label{restino2}
\begin{aligned}
	\|\lambda f  \tD_{\mathtt{m}}(\xi) \#_{\ge2} (f-1)\|_{-1,s,p}^{\g, \Lambda} 
	& \lesssim_{s,p} \|\lambda  f \|_s^{\g, \Lambda} \| f-1\|_{\so+p+5}^{\g, \Lambda}  + \|\lambda  f  \|_{\so}^{\g, \Lambda}\| f-1\|_{s+p+5}^{\g, \Lambda}
	 \\ &\lesssim_{s,p} \epsilon(s+p+5)\,.
\end{aligned}
\end{equation}
Thus the symbol $ e_1 $ in \eqref{fascio} is bounded  by 
$ \| e_1 \|^{\g, \Lambda}_{-1,s,p}  \lesssim_{s,p} \epsilon(s+p+5)  $.   
Following the same strategy as in \eqref{fascio} we obtain
the estimate
\begin{equation}\label{restino3}
\|U^{-1} b_{0}(\xi\tD^{-1}_{\mathtt{m}}(\xi) E \uno \star U)\|^{\g, \Lambda}_{-1, s, p}\,,
\|
U^{-1} b_{-1}\uno\tD^{-1}_{\mathtt{m}}(\xi)\# U 
\|^{\g, \Lambda}_{-1, s, p}\lesssim_{s,p}\epsilon(s+p+3)\,.
\end{equation}
By \eqref{restino1}, \eqref{restino12}, 
\eqref{restino2}, \eqref{restino3} 
we obtain  \eqref{stimasulresto}.
\end{proof}

\subsection{Symmetrization at lower orders}

We now block-diagonalize the operator $ \cL_2 $ in \eqref{elle2} 
(which is already block-diagonal at  order $ 1$)
up to symbols with very negative orders. 

\begin{prop}{\bf (Symmetrization at lower orders).}\label{blockTotale}
For any $\rho\geq1$ and $p_{*}\geq0$ there exist 
$\delta_{0} :=\delta_{0}(\su,\rho,p_{*})>0$
 and $ \mu:= \mu(\rho,p_{*})>0$,
such that if  the smallness condition \eqref{ipopiccolezza} holds  with $\s_*\geq\mu$ and $\delta_{*}\le \delta_{0}$
then the following holds.
There exists an invertible, real-to-real, reversibility and parity preserving  
map  ${\bf \Psi}\in \mathcal{L}^{\mathtt{T}}(H^{s},H^{s})\otimes \mathcal{M}_2(\C)$,
for any $\so\leq s\leq \su$,
such that
\begin{equation}\label{elle3}
\begin{aligned}
\mathcal{L}_{3}&:={\bf \Psi}
\mathcal{L}_2{\bf \Psi}^{-1}
\\&
=\omega\cdot\pa_{\vphi}
-\ii E \op \Big(\lambda (\varphi,x) \Id
\mathtt{D_{\mathtt{m}}}(\xi) +
\ii b_{0}(\varphi,x) \Id\x\mathtt{D}_{\mathtt{m}}^{-1}(\x) +A^{(3)}(\vphi,x,\x)+ R^{(3)}_{-\rho}(\vphi,x,\x)\Big)
\end{aligned}
\end{equation}
where  $b_{0}$ is the function in \eqref{simbolistep0},  $A^{(3)}$ is 
a real-to-real, reversibility and parity
preserving matrix of symbols of the form
\begin{equation}\label{matriceA3}
A^{(3)}(\vphi,x,\x):=
\left(\begin{matrix} c^{(3)}(\vphi,x,\xi) &0 \vspace{0.2em} \\
0& \ov{c^{(3)}(\vphi,x,-\xi)}  
\end{matrix}\right)\,,
\qquad c^{(3)}\in S^{-1}\,, 
\end{equation}  
with the bound
\begin{equation}\label{simboloCC33}
	\|c^{(3)}\|^{\gamma,\Lambda}_{-1, s, p}\,\lesssim_{s,p}
	\epsilon(s+ p+\mu) \,,\quad \forall\,   p\geq 0\,.
\end{equation}
The remainder $R_{-\rho}^{(3)} $  is a 
real-to-real, reversibility, parity preserving matrix of 
symbols in $ S^{-\rho}\otimes \mathcal{M}_2(\C)$
 and satisfies, for any $ \so \leq s\leq \su $, the bound
\begin{equation}\label{R3-rho}
\| R^{(3)}_{-\rho}\|^{\gamma,\Lambda}_{-\rho,s,p} 
\lesssim_{s,p,\rho} \epsilon(s+\mu)\,,\quad  0\leq p\le p_{*}\,.
\end{equation}
Finally ${\bf \Psi}$ is satisfies, for any  $ h\in H^s(\T^{\nu+1},\C^2)$,
for any  $\so\leq s\leq \su$, the tame estimate  
 \begin{equation}\label{stimaMappaPsi}
\|({\bf \Psi}^{\pm1}-\Id)h\|_{s}^{\gamma,\Lambda}
\lesssim_{s,\rho} \epsilon(\so+\mu)\|h\|_{s}^{\gamma, \Lambda}+
\epsilon(s+\mu)\|h\|^{\gamma,\Lambda}_{\so}\, . 
\end{equation}
\end{prop}

The proof of Proposition \ref{blockTotale}  
proceeds  inductively.
We first write the operator $ \cL_2 $ in \eqref{elle2} as
\begin{equation}\label{sist2 j-th}
	\begin{aligned}
		\mathcal{Y}^{(0)} := \cL_2 =
		&:=\omega\cdot\pa_{\vphi}-\ii E\op\big(
		{d}(\vphi,x,\x)+Q_0
		\big)\,,
		\qquad 
		\\
		{d}(\vphi,x,\x)&:=\lambda(\vphi,x)\Id\tD_{\mathtt{m}}(\x)
		+\ii b_{0}(\vphi,x)\Id\x\tD^{-1}_{\mathtt{m}}(\x) \, , 
	\end{aligned}
\end{equation}
where 
\[
Q_0:=  A_{0}^{(2)} + A_{-1}^{(2)}  
:= \left(
\begin{matrix}
	r_0(\vphi,x,\x)& q_0(\vphi,x,\x)\vspace{0.2em}\\
	\ov{q_0(\vphi,x,-\x)} & \ov{r_0(\vphi,x,-\x)}
\end{matrix}
\right) \, , \quad r_0 \in S^{-1} \, , \quad q_0 \in S^0   
\,.
\]
 
 \begin{lemma}\label{lemma83new}
For any $j= 0,\dots,\rho$, there exist
\begin{enumerate}
\item  linear operators 
\begin{equation}\label{mappadiagj}
	\mathcal{Y}^{(j)}  :=\omega\cdot\pa_{\vphi}-\ii E\op\big(
{d}(\vphi,x,\x)+Q_j+R_j 
\big)
\end{equation}
where  
\begin{equation}\label{bjbjbj}
	Q_j =\left(
	\begin{matrix}
		r_j(\vphi,x,\x)& q_j(\vphi,x,\x) \\
		\ov{q_j(\vphi,x,-\x)} & \ov{r_j(\vphi,x,-\x)}
	\end{matrix}
	\right), \  r_j\in S^{-1},\;  
	q_j\in S^{-j} \, ,
	\ R_{j}\in S^{-\rho}\otimes\mathcal{M}_2(\C) \, , 
\end{equation}
are real-to-real, reversibility and parity
 preserving 
 matrices of symbols  
satisfying 
 \begin{equation}\label{inductiveStima}
 \begin{aligned}
 	\|r_j\|^{\gamma,\Lambda}_{-1,s,p}\,,
 	\|q_j\|^{\gamma,\Lambda}_{-j,s,p}%  \| m_j \|_{-j-1,s , p}^{\g, \cO_0}
 	&\lesssim_{s,p,\rho,j}\epsilon(s+\mu_j+p)\,,\quad \forall\, p \geq 0  \,,
	\\
 	\|R_j\|^{\gamma,\Lambda}_{-\rho,s,p}&\lesssim_{s,p_*,\rho,j}\epsilon(s+\mu_j)\,,\quad
	\forall 0\leq p\leq p_{*} \, , 	
	\end{aligned}
 \end{equation}
where $\mu_{j} := \mu_{j}(\rho,p_{*}) > 0 $ is a non decreasing sequence; 
\item 
real-to-real, reversibility and parity preserving  invertible operators 
\begin{equation}\label{mappajesima}
{\bf \Psi}_{j} := \Id + \op(M_j (\vphi, x,\xi) ) 
\end{equation}
where
\begin{equation}\label{generatore-jth}
M_j(\vphi,x,\xi):=
\sm{0}{m_j(\vphi,x,\xi)}{\overline{m_j(\vphi,x,-\xi)}}{0}
\qquad
m_j =\frac{-q_{j}(\vphi,x,\x)}{2 {\lambda}(\vphi,x)\tD_{\mathtt{m}}(\x)}\in 
S^{-(j+1)} \, ,
\end{equation}
\end{enumerate}
 such that 
\begin{equation}\label{JJ+1}
\mathcal{Y}^{(j+1)}=
{\bf \Psi}_{j}^{-1} \mathcal{Y}^{(j)}{\bf \Psi}_{j}\,.
\end{equation}
\end{lemma}

\begin{proof}
\noindent
{\bf Inizialization.} 
The operator $ \cY^{(0)}$ in \eqref{sist2 j-th} 
has the form \eqref{mappadiagj} with $j=0$ and $R_{0}\equiv0$
and, in view of 
 \eqref{stimasulresto}, 
 the estimate \eqref{inductiveStima} with $j=0$ holds with $\mu_0\ge 5 $.
\\[1mm]
{\bf Iteration.} 
First of all notice that  if  \eqref{inductiveStima} holds up to some $0\le j < \rho$ then 
the symbol $m_j$ defined in \eqref{generatore-jth} belongs to $ S^{-(j+1)}$
and satisfies, using \eqref{stimasulresto}, \eqref{inductiveStima} and \eqref{interpolotutto}, 
\begin{equation}\label{MMJJ}
	\| m_j \|_{-j-1,s , p}^{\g, \Lambda} 
	\lesssim_{s,p,\rho,j}
	\epsilon(s+\mu_j+p)\,,\qquad \forall\, p \geq 0 \,.
\end{equation}
By  the inductive assumption on $q_j$ and since
$\lambda$ is even in $\vphi$ and $x$ separately, 
 whereas $\tD_{\mathtt{m}} (\xi) $ is even in $\x$, the symbol 
 $m_j$ is reversibility and parity preserving (see Lemma \ref{paritasuisimboli})
and so  is the map ${\bf \Psi}_{j}$.
Lemma \ref{fearofthedark} applies   since $-(j+1)\leq -1$, for any $0\leq j<\rho $,  
and the smallness condition \eqref{smalleffino} is fulfilled by \eqref{MMJJ}, \eqref{condireiniziale}. 
Hence we have
\begin{equation}\label{tissot1}
{\bf \Psi}_{j}^{-1}-\Id=\sum_{p=1}^{\infty}\big(\op(-M_j)\big)^{p}=\op(\widetilde{M}_{j})\,,
\qquad \widetilde{M}_{j}:=-M_{j}+M_{j, <\rho}+M_{j, \geq \rho}\,,
\end{equation}
for some matrices of symbols $M_{j, <\rho}\in S^{-(j+2)}\otimes\cM_2(\C)$, 
$M_{j, \geq\rho}\in S^{-\rho}\otimes\cM_2(\C)$
satisfying, in view of  \eqref{pearljam1Bis}-\eqref{pearljam1011bis},
\begin{equation}\label{stimapsij}
\begin{aligned}
\|M_{j,<\rho}\|^{\gamma,\Lambda}_{-j-2,s,p}&\lesssim_{j,s,\rho,p}
\epsilon(s+\widehat{\mu}_j+p)\quad \forall\,p\geq0\,,
\\
\|M_{j,\geq\rho}\|^{\gamma,\Lambda}_{-\rho,s,p}&\lesssim_{j,s,\rho,p_*}
\epsilon(s+\widehat{\mu}_j)\,,
\quad 0\leq p\leq p_{*} \, ,
\end{aligned}
\end{equation}
for some $\widehat{\mu}_{j}\geq \mu_{j}$ depending only on $\rho$ and $p_{*}$.

The conjugated operator $\mathcal{Y}^{(j+1)}$ in \eqref{JJ+1}
under the map  ${\bf \Psi}_{j}$ of the operator $\mathcal{Y}^{(j)}$ in \eqref{mappadiagj}
 is 
\begin{equation}\label{sis:j+1}
{\bf \Psi}_{j}^{-1} \mathcal{Y}^{(j)}{\bf \Psi}_{j}=
{\bf \Psi}_{j}^{-1}
  \big( \omega\cdot\pa_{\vphi}{\bf \Psi}_{j}\big)-\ii
{\bf \Psi}_{j}^{-1}  
 E \op\big(
{d}(\vphi,x,\x)+Q_{j}+R_j
\big) {\bf \Psi}_{j}  \,.
\end{equation}
We start by considering the time contribution in \eqref{sis:j+1}.
Recalling \eqref{mappajesima} and \eqref{tissot1} we  have that
\begin{align}\label{contributotempo1}
{\bf \Psi}_{j}^{-1}
  \big( \omega\cdot\pa_{\vphi}{\bf \Psi}_{j}\big)&=
  \omega\cdot\pa_{\vphi}+
  (\Id+\op(\widetilde{M}_j))\circ\op(\omega\cdot\pa_{\vphi}M_j)
  \\& = \omega\cdot\pa_{\vphi}+\op(\omega\cdot\pa_{\vphi} M_{j}+\widetilde{M}_{j}\#(\omega\cdot\pa_{\vphi}  M_{j})) 
  =\omega\cdot\pa_{\vphi}-\ii E\op(Q_{j+1}^{(1)}+R_{j+1}^{(1)}) \notag 
  \end{align}
where, recalling formul\ae\,\eqref{cancellittiEspliciti}, 
\[
\begin{aligned}
Q_{j+1}^{(1)}&:=\ii E\big(\omega\cdot\pa_{\vphi} M_{j}-M_{j}\#_{<\rho}(\omega\cdot\pa_{\vphi} M_{j})
+M_{j, <\rho}\#_{<\rho}(\omega\cdot\pa_{\vphi} M_{j})\big)
\\
R_{j+1}^{(1)}&:=\ii E\big(-M_{j}\#_{\geq\rho}(\omega\cdot\pa_{\vphi} M_{j})
+M_{j, <\rho}\#_{\geq\rho}(\omega\cdot\pa_{\vphi} M_{j})+
M_{j, \geq \rho}\#(\omega\cdot\pa_{\vphi} M_{j})\big)\,.
\end{aligned}
\]
By estimates \eqref{MMJJ}, \eqref{stimapsij},  the composition Lemma \ref{stima composizione} and  
\eqref{tameProduct}, \eqref{interpolotutto}
one deduces
that $Q_{j+1}^{(1)}\in S^{-j-1}\otimes\cM_2(\C)$, 
$R_{j+1}^{(1)}\in S^{-\rho}\otimes\cM_2(\C)$ satisfy
the bounds
\begin{align}
\|Q_{j+1}^{(1)}\|^{\gamma,\Lambda}_{-j-1,s,p}&\lesssim_{j,s,\rho,p}
\epsilon(s+\widehat{\mu}_j+p)\quad \forall\,p\geq0\,,
\label{stimamenojpiunesimaPP}
\\
\|R_{j+1}^{(1)}\|^{\gamma,\Lambda}_{-\rho,s,p}&\lesssim_{j,s,\rho,p_*}
\epsilon(s+\widehat{\mu}_j)\,,
\quad 0\leq p\leq p_{*} \,,
\label{stimamenojpiunesimaPPstar}
%\\
%\|Q_1\|_{-j-1,s,p}^{\gamma,\Lambda}&\lesssim_{s,p,\rho}\epsilon(s+\mu)\,,
\end{align}
for some $\widehat{\mu}_j=\widehat{\mu}_j(\rho,p_{*})>0$ (possibly larger than the one in \eqref{stimapsij}).
We now study the space contribution, 
which is the second summand in \eqref{sis:j+1}.  
First we study the contribution coming from the symbol $R_{j}$.
In view of \eqref{tissot1} we have
\begin{equation}\label{contributospazio0}
-\ii {\bf \Psi}_{j}^{-1}  
  \op\big(
E R_j
\big) {\bf \Psi}_{j} =-\ii E\op(R_{j+1}^{(2)})
\end{equation}
with 
\[
E R_{j+1}^{(2)}:= ER_{j}+\widetilde{M}_{j}\# ER_{j}+E R_{j}\#M_{j}+
\widetilde{M}_{j}\# ER_{j}\#M_{j}\,.
\]
By \eqref{inductiveStima} on $R_{j}$, \eqref{MMJJ} on $M_{j}$, 
\eqref{stimapsij} to estimate $\widetilde{M}_{j}$
and estimate \eqref{stimasharp}  to control the $\#$, we deduce that $R_{j+1}^{(2)}$
satisfies a bound like \eqref{stimamenojpiunesimaPPstar}.

Secondly we study the contribution coming from the symbol $Q_{j}$.
We have
\begin{equation}\label{contributospazio1}
\begin{aligned}
-\ii {\bf \Psi}_{j}^{-1}  
  \op\big(
E Q_j
\big) {\bf \Psi}_{j}&=-\ii \op(EQ_j+\widetilde{M}_j\#EQ_j+EQ_j\#M_j+\widetilde{M}_j\#EQ_j\#M_j)
\\&=-\ii E\op\big(Q_j+Q_{j+1}^{(3)}+R_{j+1}^{(3)}\big)
\end{aligned}
\end{equation}
where  (recall \eqref{cancellittiEspliciti})
\[
E Q_{j+1}^{(3)}:=(-M_j + M_{j,<	\rho})\#_{<\rho}EQ_j\#_{<\rho} (\Id + M_j)+EQ_j\#_{<\rho}M_j
%+(-M_j + M_{j,<	\rho}) \#_{<\rho}EQ_j\#_{<\rho}M_j
\]
and $R_{j+1}^{(3)} \in S^{-\rho}\otimes\cM_2(\C)$ is defined by difference.
Reasoning as in the previous steps (using \eqref{stimacancellettoesplicitoAlgrammo} to estimate 
$Q_{j+1}^{(3)}$ and \eqref{restocancellettoNp} to estimate $R_{j+1}^{(3)}$)
we deduce that $Q_{j+1}^{(3)}, R_{j+1}^{(3)}$ satisfy respectively estimates as 
\eqref{stimamenojpiunesimaPP}-\eqref{stimamenojpiunesimaPPstar}.
Finally we consider the highest order term   in 
\eqref{sis:j+1}.
We have
\begin{align}
-\ii {\bf \Psi}_{j}^{-1}  
 E\op\big(
{d}(\vphi,x,\x)\big) 
{\bf \Psi}_{j} &\stackrel{\eqref{sist2 j-th}}{=}
-\ii {\bf \Psi}_{j}^{-1}  E
 \op\Big(
\lambda\Id\tD_{\mathtt{m}}(\x)
+\ii b_{0}\Id\x\tD^{-1}_{\mathtt{m}}(\x)
\Big) 
{\bf \Psi}_{j} 
\nonumber
\\&\stackrel{\eqref{tissot1}}{=}-\ii E \op\big(
\lambda\Id\tD_{\mathtt{m}}(\x)+\ii b_{0}\Id\x\tD^{-1}_{\mathtt{m}}(\x)\big)\label{contributospazio21}
\\&
-\ii E\op\Big(
(E\lambda\Id\tD_{\mathtt{m}}(\x))\star (EM_{j})
\Big) \label{contributospazio22}
\\&-\ii E\op\Big(E(M_{j,<\rho}+M_{j,\geq\rho})\#E\lambda\Id\tD_{\mathtt{m}}(\x)\Big)
\label{contributospazio23}
\\&-\ii E\op\Big(E\widetilde{M}_j\#E\ii b_{0}\Id\x\tD^{-1}_{\mathtt{m}}(\x)
+\ii b_{0}\Id\x\tD^{-1}_{\mathtt{m}}(\x)\#M_j
\Big)\label{contributospazio24}
\\&-\ii E\op\Big( E\widetilde{M}_{j}\#E\big(\lambda\Id\tD_{\mathtt{m}}(\x)
+\ii b_{0}\Id\x\tD^{-1}_{\mathtt{m}}(\x)\big)\#M_{j}\Big)\,.
\label{contributospazio25}
\end{align}

\noindent
\emph{Term} \eqref{contributospazio22}. Recalling \eqref{espstar}, \eqref{resinprog} and 
\eqref{generatore-jth} we note that
\begin{equation}\label{tissot2}
\begin{aligned}
(E\lambda\Id\tD_{\mathtt{m}}(\x))\star (EM_{j})&=
(E\lambda\Id\tD_{\mathtt{m}}(\x))\# (EM_{j})
-(EM_j)\#(E\lambda\Id\tD_{\mathtt{m}}(\x))
\\&=\sm{0}{\mathtt{r}(\vphi,x,\x)}{\overline{\mathtt{r}(\vphi,x,-\x)}}{0}
\end{aligned}
\end{equation}
where
\begin{equation}\label{tissot4}
\begin{aligned}
\mathtt{r}(\vphi,x,\x)
& :=\lambda\tD_{\mathtt{m}}(\x)\#m_{j}+m_{j}\#\lambda\tD_{\mathtt{m}}(\x)
=2m_{j}\lambda\tD_{\mathtt{m}}(\x)+\mathtt{q}_{j,<\rho}+\mathtt{q}_{j,\geq \rho}\,,
\end{aligned}
\end{equation}
and
\[
\begin{aligned}
&\mathtt{q}_{j,<\rho}:=\sum_{k=1}^{\rho-1}\lambda\tD_{\mathtt{m}}(\x)\#_{k}m_{j}
+m_{j}\#_{k}\lambda\tD_{\mathtt{m}}(\x)\in S^{-(j+1)}
\\
&\mathtt{q}_{j,\geq\rho}:=\lambda\tD_{\mathtt{m}}(\x)\#_{\geq\rho}m_{j}
+m_{j}\#_{\geq\rho}\lambda\tD_{\mathtt{m}}(\x)\in S^{-\rho}\,.
\end{aligned}
\]
By estimate \eqref{MMJJ}
and using Lemma \ref{stima composizione} we get that 
$\mathtt{q}_{j,<\rho}$ satisfies a bound like \eqref{stimamenojpiunesimaPP}
whereas $\mathtt{q}_{j,\geq\rho}$ a bound like \eqref{stimamenojpiunesimaPPstar}
for some $\widehat{\mu}_{j}$ (depending on $\rho,p_{*}$ possibly larger).

\noindent
\emph{Lower order terms.} We notice that 
\begin{equation}\label{tissot3}
\eqref{contributospazio23}+\eqref{contributospazio24}+
\eqref{contributospazio25}=-\ii E\op\big(Q_{j+1}^{(4)}+R_{j+1}^{(4)}\big)
\end{equation}
where
\[
\begin{aligned}
Q_{j+1}^{(4)}&:=EM_{j,<\rho}\#_{<\rho}E\lambda\Id\tD_{\mathtt{m}}(\x)
\\&
+E(-M_j + M_{j,<	\rho})\#_{<\rho}E\ii b_{0}\Id\x\tD^{-1}_{\mathtt{m}}(\x)
+\ii b_{0}\Id\x\tD^{-1}_{\mathtt{m}}(\x)\#_{<\rho}M_j
\\&+
E(-M_j + M_{j,<	\rho})\#_{<\rho}E\big(\lambda\Id\tD_{\mathtt{m}}(\x)
+\ii b_{0}\Id\x\tD^{-1}_{\mathtt{m}}(\x)\big)\#_{<\rho}M_{j}\,,
\end{aligned}
\]
and $R_{j+1}^{(4)}$ is defined by difference.
Reasoning similarly as above we deduce that
$Q_{j+1}^{(4)},R_{j+1}^{(4)}$ satisfy respectively the estimates 
\eqref{stimamenojpiunesimaPP}-\eqref{stimamenojpiunesimaPPstar}
(recall that $\tD_{\mathtt{m}}(\x)$ has order $1$, whereas $M_{j} $ is a matrix of symbols 
in $ S^{-(j+1)} \otimes\cM_2(\C) $, $M_{j,<\rho}\in S^{-(j+2)} \otimes\cM_2(\C) $ and $M_{\geq \rho}\in S^{-\rho} \otimes\cM_2(\C) $).

Collecting together \eqref{sis:j+1}, \eqref{contributotempo1}, \eqref{contributospazio0}, \eqref{contributospazio1},
\eqref{contributospazio21}-\eqref{contributospazio25}, 
and %\eqref{tissot2},
\eqref{tissot3},
we obtain
\begin{equation}\label{tissot5}
\begin{aligned}
{\bf \Psi}_{j}^{-1} \mathcal{Y}^{(j)}{\bf \Psi}_{j}&=
-\ii E \op\big(
\lambda\Id\tD_{\mathtt{m}}(\x)+\ii b_{0}\Id\x\tD^{-1}_{\mathtt{m}}(\x)\big)
\\&
-\ii E\op\big(Q_{j+1}^{(1)}+Q_{j+1}^{(3)}+Q_{j+1}^{(4)}+R_{j+1}^{(1)}
+R_{j+1}^{(2)}+R_{j+1}^{(3)}+R_{j+1}^{(4)}\big)
\\&
-\ii E\op\big( Q_{j}+(E\lambda\Id\tD_{\mathtt{m}}(\x))\star (EM_{j})\big)\,.
\end{aligned}
\end{equation}
By \eqref{bjbjbj}, \eqref{tissot2} and \eqref{tissot4} we have
\begin{equation}\label{pezzqjn}
Q_{j}+(E\lambda\Id\tD_{\mathtt{m}}(\x))\star (EM{j})=
\left(
	\begin{matrix}
		r_j(\vphi,x,\x)& \widetilde{q}_j(\vphi,x,\x)\vspace{0.2em}\\
		\ov{\widetilde{q}_j(\vphi,x,-\x)} & \ov{r_j(\vphi,x,-\x)}
	\end{matrix}
	\right)
\end{equation}
where $r_{j}$ is the same symbol in \eqref{bjbjbj} and, using 
\eqref{generatore-jth},  
\begin{equation}\label{cancellazkey}
\widetilde{q}_j % (\vphi,x,\x)
:={q}_j % (\vphi,x,\x)
+
2m_{j}\lambda\tD_{\mathtt{m}}(\x)+\mathtt{q}_{j,<\rho}+\mathtt{q}_{j,\geq \rho}
=\mathtt{q}_{j,<\rho}+\mathtt{q}_{j,\geq \rho}\in S^{-(j+1)}\, .
\end{equation}
In conclusion  \eqref{tissot5}, \eqref{pezzqjn}, \eqref{cancellazkey} imply that 
$  \mathcal{Y}^{(j+1)} $ in \eqref{JJ+1} 
% the conjugation  \eqref{JJ+1} 
has the form \eqref{mappadiagj} with 
 matrices of symbols
\[
\begin{aligned}
Q_{j+1}&:=Q_{j+1}^{(1)}+Q_{j+1}^{(3)}+Q_{j+1}^{(4)}
+\left(\begin{matrix} 
r_{j}(\vphi,x,\x) & \mathtt{q}_{j,<\rho}(\vphi,x,\x)\\ \overline{\mathtt{q}_{j, <\rho}(\vphi,x,-\x)} & \overline{r_{j}(\vphi,x,-\x)}
\end{matrix}\right)
\\
R_{j+1}&:=
R_{j+1}^{(1)}
+R_{j+1}^{(2)}+R_{j+1}^{(3)}+R_{j+1}^{(4)}
+\left(\begin{matrix} 
0 & \mathtt{q}_{j, \ge\rho}(\vphi,x,\x)\\ \overline{\mathtt{q}_{j,\ge\rho}(\vphi,x,-\x)} & 0
\end{matrix}\right) 
\end{aligned}
\]
that satisfy the bounds \eqref{inductiveStima} 
with $j\rightsquigarrow j+1 $, 
by the fact that all the summand satisfy \eqref{stimamenojpiunesimaPP}-\eqref{stimamenojpiunesimaPPstar}
and 
taking $\mu_{j+1}\geq \mu_{j}$ large enough and depending only on $\rho,p_{*}$.
Moreover 
by Lemma \ref{paritasuisimboliBIS}, the inductive assumptions and the explicit construction
above, we have that the matrix of symbols $Q_{j+1}$ is real-to-real, reversibility and parity preserving.
Since the maps ${\bf \Psi}^{\pm1}_{j}$ are real-to-real, reversibility  and parity preserving, we
deduce, by difference that $R_{j+1}$ satisfies the same properties.
\end{proof}

\begin{proof}[{\bf Proof of Proposition \ref{blockTotale} concluded}]
We define the invertible, real-to-real, reversibility and parity preserving map 
$ {\bf \Psi}:=
{\bf \Psi}_{\rho-1}^{-1}\circ\cdots \circ{\bf \Psi}_0^{-1} $. 
Recalling % \eqref{elle2}, \eqref{nuovocoef}, \eqref{L-omega}, 
\eqref{sist2 j-th} and Lemma \ref{lemma83new} 
we have that $\mathcal{L}_{3} =\mathcal{Y}^{(\rho)}$
has the form \eqref{elle3}-\eqref{matriceA3} with symbols 
satisfying \eqref{simboloCC33} and \eqref{R3-rho}.
The estimate \eqref{stimaMappaPsi} follows by composition
using the estimates of the generators in \eqref{stimapsij}
and Lemma \ref{sobaction}.
\end{proof}

\section{Reduction to constant coefficients up to one-smoothing operators}\label{sec:ordine11}
 
\subsection{Reduction  at order  1 }

In the sequel we consider the  operator 
$\mathcal{L}_{3}$ 
 in Proposition \ref{blockTotale}  with $p^{*}=0 $ (on the smoothing remainder 
 $ R_{-\rho}^{(3)} $ we do not perform symbolic calculus any more).
The goal of this section is to 
eliminate the dependence on  $(\vphi, x)$ from the first order symbol 
$\lambda(\vphi, x)\tD_{\mathtt{m}} (\xi) $ in \eqref{elle3}.

\begin{prop}{\bf (Straightening of the first order operator).}\label{prop:ridord1} 
For any $\mathtt{b}\geq0$
there exist $\delta_{0} :=\delta_{0}(\su,\tb)>0$,
$\rho :=\rho(\tb) \geq 1 $, 
$ \mu :=\mu(\tb) > 0 $, 
such that, for any $\gamma\in(0,\tfrac{1}{2})$, if \eqref{ipopiccolezza} holds with 
$\s_{*}\geq \mu$ and $\delta_{*}\le \delta_{0}$
 the following holds.
Consider the linear operator 
$\mathcal{L}_{3}$ obtained in Proposition \ref{blockTotale}  with $\rho = \rho(\tb) $ 
 and $p^{*}=0$.
There exists a Lipschitz function $\mathfrak{c} : \Lambda \to \R$ satisfying 
\begin{equation}\label{stimaMgotico1} 
|\mathfrak{c}|^{\gamma,\Lambda}\lesssim \epsilon(\so+\mu)\,,
\end{equation}
such that for any $\omega $ in % \in \Omega_{1}$ defined as
\begin{equation}\label{calOinfty1sec6}
\Omega_1:=\left\{\omega\in \Lambda\,:\, 
|\omega\cdot\ell-(1+\mathfrak{c})j|\geq 2\gamma\langle\ell\rangle^{-\tau}\,,
\;\; j\in\Z\,,\; \ell\in \Z^{\nu}\,,\;
(\ell,j)\neq(0,0)
\right\}
\end{equation}
with  $\tau$ fixed in  \eqref{costanti}, 
there exists a real-to-real, reversibility and parity preserving,  invertible map 
$\Theta_1\in \mathcal{L}^{\mathtt{T}}(H^{s},H^{s})\otimes \mathcal{M}_2(\C)$,
for any $\gso\leq s\leq \su$,
such that
\begin{equation}\label{elle4}
\cL_4:= \Theta_1 \cL_3 \Theta_1^{-1}
= 
\oo\cdot\pa_\vphi 
-\ii E \op \Big(
(1+\mathfrak{c})\tD_{\mathtt{m}}(\x) \Id +
\sm{ \ii c^{(4)} (\vphi,x,\x) }{0}{0}{{-\ii  c^{(4)}(\vphi,x,-\x)}}
\Big)
-\ii E \cR^{(4)}
\end{equation}
where $\tD_{\mathtt{m}}(\x)$ is the Fourier multiplier 
in \eqref{TDM}, $c^{(4)} $ is a real valued symbol in $ S^0 $,
and 
 $\cR^{(4)} $ is a smoothing remainder.

\noindent
{\bf (Symbol).}
The symbol $c^{(4)}$ is reversible, parity preserving and  has the form
\begin{equation}\label{simbolCC44}
	c^{(4)}(\vphi,x,\x):=c^{(4)}_{+}(\vphi,x)\chi_{+}(\x)-
	c^{(4)}_{-}(\vphi,x)\chi_{-}(\x)\,,
\end{equation}
(see \eqref{def:cutoff}) where 
$c^{(4)}_{\s}$, $\s\in \{\pm\}$, are  real valued functions 
% $c^{(4)}_{\s}\in S^{0}$, $\s\in  \{\pm\}$ independent of $\x\in \R$
satisfying
\begin{equation}\label{algebrella}
c^{(4)}_{-}(\vphi,x)=-c^{(4)}_{+}(\vphi,-x)\,,\quad c^{(4)}_{\s}(\vphi,x)=-c^{(4)}_{\s}(-\vphi,-x)\,,\;\;\s\in \{\pm\}\,,
\end{equation}
and for any $\so\le s\le \su$ 
\begin{equation}\label{estranged4}
\|c^{(4)}\|_{0,s,p}^{\g,\Omega_1} \lesssim_{s,\tb,p} \gamma^{-1}
\epsilon(s+\mu) \,,\quad \forall\, p\ge 0\,.
\end{equation}

\noindent
\noindent
{\bf (Remainder).}
The remainder $\cR^{(4)}$ is real-to-real, reversibility  and parity preserving and
there is 
 $\mathtt{R}^{(4)}\in E_{s}$, %$\so\le s\le \su$,  
such that $\mathfrak{S}(\mathtt{R}^{(4)})=\mathcal{R}^{(4)}$  and  for any $\gso\le s\le \su$
 \begin{align}
  \bnorm{\langle\td_{\vphi}\rangle^{q}%\langle D\rangle^{m_1}
  \mathtt{R}^{(4)}\langle D\rangle}_{s}^{\gamma,\Omega_1}
  &\lesssim_{s,\su,\mathtt{b}}\gamma^{-1}\epsilon(s+\mu)\,,
  \qquad q=0\,,\tb\,.
  \label{restosmooth1}
 \end{align}
 
\noindent
{\bf (Transformation).} For any $\so\le s\le \su $, for any $ h\in H^{s}(\T^{\nu+1},\C^{2})$, 
\begin{equation}\label{stimasuPhiuno}
\|(\Theta_1)^{\pm}h\|_s^{\g, \Omega_1}\lesssim_s 
\|h\|_s^{\g, \Omega_1} 
+ \g^{-1} \epsilon(s+\mu)
\|h\|_{\so}^{\g, \Omega_1} \,.
\end{equation}
\end{prop}

The rest of this section is devoted to the proof of Proposition \ref{prop:ridord1}.

We apply  the straightening  Theorem \ref{IncredibleConjugate}  with 
$ a\rightsquigarrow \lambda-1$,  $\tb \in \N_0 $  and $M=1$.
The real valued function $\lambda(\vphi,x)$
defined in \eqref{nuovocoef} is  even separately  in $ \vphi $, $ x$ and, since  
it satisfies  \eqref{stimasulresto}, 
the smallness condition \eqref{piccino} of 
Theorem \ref{IncredibleConjugate}  holds  % applies 
thanks to the smallness assumption \eqref{ipopiccolezza} with 
$\s_*> \mu$, taking $\mu$  large w.r.to $\s(1,\tb)$ in \eqref{piccino}.
Thus Theorem \ref{IncredibleConjugate} provides the existence  of a 
 Lipschitz function $\mathfrak{c} (\omega) := \fa_{+} (\omega) $ which 
satisfies the estimate \eqref{stimaMgotico1} by 
\eqref{tordo4b} and using \eqref{stimasulresto}. %  to estimate $\lambda-1$.
Moreover Theorem \ref{IncredibleConjugate}  guarantees
the existence, for any $ \omega \in \Omega_1 $ in \eqref{calOinfty1sec6}, of
functions $\alpha_\s$, $\breve{\alpha}_{\s}$, $\s\in \{\pm\}$, such that
\begin{align}
&\alpha_{-}(\vphi,x):=-\alpha_{+}(\vphi,-x)\,,\qquad \alpha_{+}(\vphi,x)=-\alpha_{+}(-\vphi,-x)\,,
\qquad \forall \, (\vphi,x)\in \T^{\nu+1}\,,\label{parityalfette}
\\&
y=x+\alpha_{\s}(\vphi,x) 
\;\;\; 
\Leftrightarrow 
\;\;\; 
x=y+\breve{\alpha}_{\s}(\vphi,y)\,,\quad \s\in\{\pm\}\,,\qquad \forall\, x,y\in \T\,,\; \vphi\in \T^{\nu}\,,\nonumber
\end{align}
such that \eqref{straightpotente} holds
and, by \eqref{stimapm}, \eqref{stimasulresto}, for any $s\geq \so$, 
\begin{equation}\label{superstimaalpha}
\| \breve \alpha_{\s} \|^{\g,\Omega_{1}}_{s} \, , \
\|\alpha_{\s} \|^{\g,\Omega_{1}}_{s} \lesssim_s  \g^{-1}\epsilon(s+\mu)\,,\;\;\;\s\in \{\pm\}\, . 
\end{equation}
We now define the real-to-real map (recall also \eqref{ellediffeo})
\begin{equation}\label{mappaPhidef}
\Theta_1:=\left(\begin{matrix} L & 0 \\0 &\overline{L}\end{matrix}\right)\,,\quad
\begin{aligned}
L&:=\mathcal{C}_{\alpha_+}\Pi_+ + \mathcal{C}_{\alpha_-}\Pi_- \,,
\\
\overline{L}&:=\mathcal{C}_{\alpha_+}\Pi_{-} + \mathcal{C}_{\alpha_-}\Pi_{+}\,,
\end{aligned}
\end{equation}
where $\mathcal{C}_{\alpha_\s}$ is defined as in \eqref{ignobel}
and $\Pi_{\s}$ are in \eqref{def:szego} for $\s\in \{\pm\}$. 
The map $\Theta_1$ is also 
reversibility and parity preserving since, by Theorem \ref{IncredibleConjugate}, $L$ is so.
The estimate \eqref{stimasuPhiuno} on $\Theta_1$ follows by Lemma \ref{bastalapasta}
and \eqref{superstimaalpha}.
We apply Lemma \ref{lemma: invertibilita operatore trasporto} 
with $N := \nu+11+2\tb $, thanks to 
\eqref{superstimaalpha} and  
the smallness condition
 \eqref{ipopiccolezza} with $ \s_* > \mu $ large enough,   
deducing that 
\begin{equation}\label{inversaPhiuno}
\Theta_1^{-1}=\left(\begin{matrix} L^{-1} & 0 \\0 &\overline{L^{-1}}\end{matrix}\right)\,,
\qquad
L^{-1}=(\mathcal{C}_{\alpha_+}^{-1}\Pi_{+}+\mathcal{C}_{\alpha_-}^{-1}\Pi_{-})\circ(\id+\widetilde{R})\, , 
\end{equation}
and there exists  $\widetilde{\mathtt{R}}\in E_{s}$ such that $\mathfrak{S}(\widetilde{\mathtt{R}})=\widetilde{R}$ and
for any $m_1,m_2\in\R_{+}$  such that  $m_1+m_2=N$,
one has
\begin{equation}\label{commedia}
\bnorm{\jap{\td_{\vphi}}^{q} \jap{D}^{m_1}\widetilde{\mathtt{R}}\jap{D}^{m_2}}_s^{\g, \Omega_1} 
\lesssim_{s,\su,\tb} {\g^{-1}}
\epsilon(s+\mu)\,,\;\;q=0\,,\tb\,.
\end{equation}
Hence estimate \eqref{stimasuPhiuno} on $\Theta_1^{-1}$ follows
by composition 
recalling Lemma \ref{rmk:tametresbarre}.
Furthermore, possibly taking a larger $ \mu $ in the smallness condition, 
 Proposition \ref{inclusionetotale} guarantees that 
there are $ \mathtt{L}^{\pm} $ such that 
$ \mathfrak{S}(\mathtt{L^\pm}) = L^{\pm 1} $ and,
for any  $ N_1 + N_2 = \lfloor\nu/2\rfloor+4+\tb $, 
$ N_1,N_2\in \R$, $|N_1|, |N_2|\lesssim_{\tb}1$, the bound
\begin{equation}\label{psitresbarre}
\bnorm{\langle \td_{\vphi}\rangle^{q}\langle D\rangle^{-N_1}(\mathtt{L}^\pm -\id)
\langle D\rangle^{-N_2}}_{s}^{\gamma, \Omega_1}
\lesssim_{s,\su,\tb}\g^{-1}\epsilon(s+\mu)\,,\;\; q=0\,,\tb\, .
\end{equation}
Recalling \eqref{elle3}-\eqref{matriceA3} and since $ \Theta_1 E = 
E \Theta_1 $, 
 the conjugate operator
has the form
\begin{equation}\label{marechiaro1000}
\mathcal{L}_{4} = \Theta_1\mathcal{L}_{3}\Theta_1^{-1}=
\left(\begin{matrix} B_1& 0\\0 & \overline{B}_1\end{matrix}\right)
-\ii E\left(\begin{matrix}
B_2  & 0 \\ 0 & \overline{B}_2
\end{matrix}\right)
-\ii E F
\end{equation}
where
\begin{align}
B_1&:=L\Big(\omega\cdot\pa_{\vphi}-\ii\op(\lambda {\mathtt{D}}_{\mathtt{m}}(\xi))\Big)L^{-1}\,, \label{diagpezzoord1}
\\
B_2&:=L\op\Big(
\ii b_{0}\x\mathtt{D}_{\mathtt{m}}^{-1}(\x)+c^{(3)}(\vphi,x,\x)\Big)L^{-1}\,, \label{diagpezzoord0}
\\
%C&:=L\op\big(f_{-\rho}\big)\overline{L^{-1}} \label{offdiagpezzo}\,.
F&:=\Theta_1\op\big(R^{(3)}_{-\rho}(\vphi,x,\x)\big)\Theta_1^{-1}\,.
\label{pezzobrutto}
\end{align}
In the following we analyze these terms  separately.
We first consider the highest order operator  \eqref{diagpezzoord1}.

\begin{lemma}\label{lemma1} % {\bf (Term $B_1$).} 
The operator $B_{1}$ in \eqref{diagpezzoord1} has the form
\begin{equation}\label{odioprofondo3stat}
B_{1} = \omega\cdot\pa_{\vphi}-\ii \op((1+\mathfrak{c})\tD_{\mathtt{m}}(\x))+Q 
\end{equation}
where $Q$ is a reversible and parity preserving operator
satisfying \eqref{restosmooth1}. 
\end{lemma}

\begin{proof}
We first decompose, recalling  \eqref{TDM} and \eqref{cutoff}, 
\begin{equation}\label{billythekid}
\tD_{\mathtt{m}}(\x)=|\x|\chi(\x)+\widetilde{\tD}_{\mathtt{m}}(\x)\,,
\qquad
\widetilde{\tD}_{\mathtt{m}}(\x):=\sqrt{|\x|^2+\mathtt{m}}-|\x|\chi(\x) \in S^{-1}\, , 
\end{equation}
and so we write the operator $B_{1}$ in  
 \eqref{diagpezzoord1} as (recall that $|D|=\op(\chi(\x)|\x|)$ as in \eqref{notazioneNota})
\begin{align}
B_1&=L\big(\omega\cdot\pa_{\vphi}-\ii \lambda|D|\big)L^{-1}\label{odioprofondo}
\\
&
-\ii L \op(\lambda\widetilde{\tD}_{\mathtt{m}}(\x))L^{-1} \, . \label{odioprofondo2}
\end{align}
The operator  \eqref{odioprofondo} is studied by
Theorem \ref{IncredibleConjugate}
with $a\rightsquigarrow \lambda-1$, $ M = 1 $, %$L\rightsquigarrow\Psi$
obtaining  (see \eqref{straightpotente}) 
\begin{equation}\label{tempopezzo2}
\eqref{odioprofondo} =
\omega\cdot\pa_{\vphi}-\ii (1+\mathfrak{c})|D| + Q_1
\end{equation}
where  $ \mathfrak{c}  \rightsquigarrow  \fa_+  $
and $ Q_1 $  (denoted by $ R $ 
in Theorem \ref{IncredibleConjugate}) is a smoothing operator satisfying \eqref{restosmooth1}
(this is a consequence of estimates \eqref{raichel2} with $M=1$, $m_1=0,m_2=1$).

We now study the conjugate \eqref{odioprofondo2} of the operator  $ A :=\op(\lambda\widetilde{\tD}_{\mathtt{m}}(\x))$. The treatment is extremely similar to the one of the first order term \eqref{straightpotente} in Theorem \ref{IncredibleConjugate}. 
In view of \eqref{mappaPhidef}-\eqref{inversaPhiuno} we have that
\begin{equation}\label{cheodio}
L\circ A\circ L^{-1}=(\mathcal{C}_{\alpha_+}\Pi_{+}+\mathcal{C}_{\alpha_-}\Pi_{-})\circ A
\circ(\mathcal{C}_{\alpha_+}^{-1}\Pi_{+}+\mathcal{C}_{\alpha_-}^{-1}\Pi_{-})+Q_2
\end{equation}
where 
\begin{equation}\label{qu1}
Q_2 := (\mathcal{C}_{\alpha_+}\Pi_{+}+\mathcal{C}_{\alpha_-}\Pi_{-})
 \circ A \circ  (\mathcal{C}_{\alpha_+}^{-1}\Pi_{+}+\mathcal{C}_{\alpha_-}^{-1}\Pi_{-})\circ\widetilde{R} \,.
\end{equation}
We claim that $Q_2$ satisfies \eqref{restosmooth1}.
We follow the reasoning used for the remainder term $Q$ in \eqref{esempioQ}.
First we write 
\begin{equation}\label{decopezz}
\begin{aligned}
\langle \td_{\vphi}\rangle^{\tb}\mathcal{C}_{\alpha_\s}\Pi_{\s} A
\mathcal{C}_{\alpha_{\s'}}^{-1}\Pi_{\s'}\widetilde{R}\langle D\rangle
& =
\langle \td_{\vphi}\rangle^{\tb}\mathcal{C}_{\alpha_\s}
\langle D\rangle^{-N'}\circ\langle D\rangle^{N'}
\Pi_{\s} A\langle D\rangle^{-N'+1}
\\&
\quad \circ\langle D\rangle^{N'-1}\mathcal{C}_{\alpha_{\s'}}^{-1}\Pi_{\s'}
\langle D\rangle^{-2N'+1}
\circ \langle D\rangle^{2N'-1}
\widetilde{R}\langle D\rangle
\end{aligned}
\end{equation}
where  $N':=\lfloor\nu/2\rfloor+4+\tb$.
Proposition \ref{inclusionetotale}
(with $N_1\rightsquigarrow 0$, $N_2\rightsquigarrow N' $)
implies the existence of $\mu := \mu (\tb ) >0 $, $\delta := \delta (\tb ) >0 $ such that,  
if \eqref{ipopiccolezza} holds with 
$\s_{*}\geq \mu$ and $\delta_{*}\le \delta_{0}$, then,  
by \eqref{superstimaalpha}, 
there is $\mathtt{R}_1\in E_{s}$ such that, 
\begin{equation}\label{agosto0}
\bnorm{\langle \td_{\vphi}\rangle^{\tb}\mathtt{R}_1
}_{s}^{\gamma,\calO}\lesssim_{s,\su,\tb}\gamma^{-1}\epsilon(s+\mu)\,,
\qquad 
\mathfrak{S}(\mathtt{R}_1)=(\mathcal{C}_{\alpha_\s}-\id)\langle D\rangle^{-N'}\, .
\end{equation}
Secondly we apply Proposition \ref{prop:immersionepseudo}
(with $-n_1\rightsquigarrow N'$ and $-n_2\rightsquigarrow -N' + 1$ )
to infer that there is 
 $\mathtt{R}_2$ such that, using also \eqref{stimasulresto}
 (and with a possibly larger $ \mu(\tb)$)
\begin{equation}\label{agosto1}
\bnorm{\langle \td_{\vphi}\rangle^{\tb}\mathtt{R}_2
}_{s}^{\gamma,\calO}\lesssim_{s,\su,\tb}1+\gamma^{-1}\epsilon(s+\mu)\,,
\qquad 
\mathfrak{S}(\mathtt{R}_2)=\langle D\rangle^{N'}\Pi_{\s}A\langle D\rangle^{-N'+1}\,.
\end{equation}
Furthermore 
we apply Proposition \ref{inclusionetotale}
(with $N_1\rightsquigarrow -N'+1$, 
$N_2\rightsquigarrow 2N'-1$)
to obtain that there is $\mathtt{R}_3\in E_{s}$ such that 
\begin{equation}\label{agosto2}
\bnorm{\langle \td_{\vphi}\rangle^{\tb}\mathtt{R}_3
}_{s}^{\gamma,\calO}\lesssim_{s,\su,\tb}\gamma^{-1}\epsilon(s+\mu)\,,
\qquad 
\mathfrak{S}(\mathtt{R}_3)=\langle D\rangle^{N'-1}(\mathcal{C}^{-1}_{\alpha_{\s'}}-\id)
\langle D\rangle^{-2N'+1}\, . 
\end{equation}
Finally, by \eqref{commedia} 
(with $m_1\rightsquigarrow 2N'-1$, $m_2\rightsquigarrow 1$, and noting that $2N'< N$)
we deduce that there is 
 $\mathtt{R}_4$ such that
\begin{equation}\label{agosto3}
\bnorm{\langle \td_{\vphi}\rangle^{\tb}\mathtt{R}_4
}_{s}^{\gamma,\calO}\lesssim_{s,\su,\tb}\gamma^{-1}\epsilon(s+\mu)\,,
\qquad 
\mathfrak{S}(\mathtt{R}_4)=\langle D\rangle^{2N'-1}
\widetilde{R}\langle D\rangle\,.
\end{equation}
By \eqref{decopezz}, \eqref{agosto0}, 
\eqref{agosto1}, \eqref{agosto2}, \eqref{agosto3} 
and Lemmata \ref{tretameestimate} and \ref{tretameestimate3}
we conclude that $\mathcal{C}_{\alpha_\s}\Pi_{\s} A \mathcal{C}_{\alpha_{\s'}}^{-1}\Pi_{\s'}\widetilde{R}$ 
is an operator satisfying estimates like 
\eqref{restosmooth1}.
This proves the claim that  $Q_2$ satisfies \eqref{restosmooth1}. 

We now write % conjugate 
the operator \eqref{cheodio} as
\begin{align}
L\circ A\circ L^{-1}& = 
\sum_{\s,\s'\in\{\pm\}}\mathcal{C}_{\alpha_{\s}}\Pi_{\s}A\mathcal{C}_{\alpha_{\s'}}^{-1}\Pi_{\s'}+Q_2
\nonumber
\\&
=\sum_{\s\in\{\pm\}}\mathcal{C}_{\alpha_{\s}}A\mathcal{C}_{\alpha_{\s}}^{-1}\Pi_{\s}
\label{cheodio2}
\\&
+\sum_{\s\in\{\pm\}}\mathcal{C}_{\alpha_{\s}}A\mathcal{C}_{\alpha_{\s}}^{-1}
\big( \Pi_{\s}^{2}-\Pi_{\s}\big)
+\sum_{\s\in\{\pm\}}\mathcal{C}_{\alpha_{\s}}A\mathcal{C}_{\alpha_{-\s}}^{-1}\Pi_{\s}\Pi_{-\s}
\label{cheodio3}
\\&
+\sum_{\s,\s'\in\{\pm\}}\mathcal{C}_{\alpha_{\s}}A[\Pi_{\s},\mathcal{C}_{\alpha_{\s'}}^{-1}]\Pi_{\s'}+
\sum_{\s,\s'\in\{\pm\}}\mathcal{C}_{\alpha_{\s}}[\Pi_{\s},A]\mathcal{C}_{\alpha_{\s'}}^{-1}\Pi_{\s'}
+Q_2  \label{cheodio4}
\end{align}
so that 
\begin{equation}\label{bastaa}
L\circ A\circ L^{-1}=  \mathcal{C}_{\alpha_+} \, A \, 
\mathcal{C}_{\alpha_+}^{-1} \Pi_{+}
+ \mathcal{C}_{\alpha_-} \, A \, \mathcal{C}_{\alpha_-}^{-1}  \Pi_{-} 
+ Q_3
\end{equation}
where $ Q_{3}:=\eqref{cheodio3}+\eqref{cheodio4}$.

We write, using Remark
\ref{rmk:cutoff}-($ii$),  % \eqref{cheodio3} we have
$$
\eqref{cheodio3} = 
\sum_{\s\in\{\pm\}}
\mathcal{C}_{\alpha_{\s}}A \Big( \id - \, \mathcal{C}_{\alpha_{\s}}^{-1}
+ \mathcal{C}_{\alpha_{-\s}}^{-1} - \id
\Big) \Pi_{\s}\Pi_{-\s}
$$
where
 $\Pi_{\s}\Pi_{-\s}$ is a pseudo-differential operator 
with symbol in $ S^{-\infty}$. Thus arguing as above 
\eqref{cheodio3}  satisfies \eqref{restosmooth1}.
The terms 
\eqref{cheodio4} 
 contain either $[\Pi_{\s}, \mathcal{C}_{\alpha_{\s'}}^{-1} ]  $ 
 or $[\Pi_{\s}, A] = [\Pi_{\s},  \op ( (\lambda-1) \widetilde{\tD}_{\mathtt{m}}(\x) ) ] $ 
which
are pseudo-differential operators with symbols in $S^{-\infty}$
 that we estimate  by Lemma \ref{commutarechebello}. 
Reasoning as done for the remainder $Q_{2}$ in \eqref{qu1},
we get that  \eqref{cheodio4}, and hence  $Q_3$ in \eqref{bastaa},  satisfies \eqref{restosmooth1}.

We now claim that
\begin{equation}\label{claimo2}
\mathcal{C}_{\alpha_\s} A % \op(\lambda \widetilde{\tD}_{\mathtt{m}}(\x))
\mathcal{C}_{\alpha_{\s}}^{-1}=
\op((1+\mathfrak{c})\widetilde{\tD}_{\mathtt{m}}(\x))+Q_{5,\s}
\end{equation}
(where $ A = \op(\lambda \widetilde{\tD}_{\mathtt{m}}(\x)) $) 
for some  operator $Q_{5,\s}$ satisfying \eqref{restosmooth1}.
Indeed, applying Theorem \ref{thm:egorov} with 
$w\rightsquigarrow\lambda \widetilde{\tD}_{\mathtt{m}}(\x)$, $m=-1$ and $M=2$
(the smallness condition \eqref{buf} is implied by \eqref{superstimaalpha} and 
\eqref{ipopiccolezza}) we get 
\begin{equation}
	\label{rosacroce}
\begin{aligned}
\mathcal{C}_{\alpha_\s}\op(\lambda \widetilde{\tD}_{\mathtt{m}}(\x))\mathcal{C}_{\alpha_{\s}}^{-1}=
\op\Big( \mathtt{r}^{(1)}_{\s}(\vphi,x,\x)+\mathtt{r}^{(2)}_{\s}(\vphi,x,\x) \Big)+\widetilde{Q}_{4,\s}\,,
\end{aligned}
\end{equation}
with (recalling formula \eqref{fundamentalformula})
\begin{equation}\label{defr1s}
\mathtt{r}^{(1)}_{\s}(\vphi,x,\x) :=\lambda\big(\vphi,x+\alpha_\s(\vphi,x)\big)
\widetilde{\tD}_{\mathtt{m}}\big(\x(1+\pa_{y}\breve{\alpha}_{\s}(\vphi,y))\big)_{|y=x+\alpha_{\s}(\vphi,x)}\,,
\end{equation}
a symbol $\mathtt{r}^{(2)}_{\s}(\vphi,x,\x) $ in $ S^{-2}$ 
and  a remainder  $\widetilde Q_{5,\s}$  admitting a representative in $E_{s}$
which satisfies \eqref{restosmooth1}, by estimate \eqref{stimerestoEgorov1}
with $w\rightsquigarrow\lambda \widetilde{\tD}_{\mathtt{m}}(\x)$ and $\alpha\rightsquigarrow\alpha_\s$,
using \eqref{superstimaalpha}, \eqref{stimasulresto},
and the interpolation estimate \eqref{interpolotutto}. 
Similarly by  estimate \eqref{zeppelinbis}, \eqref{superstimaalpha}, \eqref{stimasulresto}
and \eqref{interpolotutto} we deduce, for some $\mu=\mu(\tb)$,
$$ 
\|\mathtt{r}^{(2)}_{\s}\|^{\gamma,\calO}_{-2,s,p}\lesssim_{s,p}\g^{-1}\epsilon(s+\mu+p) \, . 
$$ 
 Therefore, by Proposition
\ref{prop:immersionepseudo} with $m=-1$ we get that $\op(\mathtt{r}^{(2)}_{\s})$
is a smoothing remainder satisfying  \eqref{restosmooth1}.

We now study the structure of the symbol $\mathtt{r}^{(1)}_{\s}$ in \eqref{defr1s}.
We first note that, by 
the estimate  \eqref{stimaMgotico1} on $\mathfrak{c}$, 
the estimate \eqref{stimasulresto} on $\lambda$, 
\eqref{superstimaalpha}, 
and \eqref{bastalapasta}
\begin{equation}\label{sottomarino2}
\|\lambda\big(\vphi,x+\alpha_\s(\vphi,x)\big)
-(1+\mathfrak{c})\|_{s}^{\gamma,\Omega_1}
\lesssim_{s}
\gamma^{-1}\epsilon(s+\mu)\,.
\end{equation}
Secondly (recalling the expansion \eqref{billythekid}) we note that, for any $(\vphi,y)\in \T^{\nu+1}$,
\begin{equation*}
\begin{aligned}
\widetilde{\tD}_{\mathtt{m}}&\big(\x(1+\pa_{y}\breve{\alpha}_{\s})\big)-
(1+\pa_{y}\breve{\alpha}_{\s})\widetilde{\tD}_{\mathtt{m}}(\x)
\\&
=(1+\pa_{y}\breve{\alpha}_\s)\Big[
\underbrace{\Big(\sqrt{\x^2+\frac{\mathtt{m}}{(1+\pa_{y}\breve{\alpha}_\s)^2}} -\sqrt{\x^2+\mathtt{m}}\Big)}_{=: (1)}
+ \underbrace{|\x|\big(\chi(\x)-\chi(\xi(1+\pa_y\breve\alpha_\s))\big)}_{=: (2)}\Big]\,.
\end{aligned}
\end{equation*}
The first summand $(1)$ is a symbol satisfying 
$\|(1)\|_{-1,s,p}^{\gamma,\Omega_1}\lesssim_{s,p}\g^{-1}\epsilon(s+\mu)$.
Moreover
$ \chi(\xi(1+\pa_y\breve\alpha_\s))-\chi(\x)=\int_{0}^{1}\chi'(\x(1+ \tau 
 \pa_y\breve\alpha_\s))\x\pa_y\breve\alpha_\s d \tau $
from which we deduce that recalling \eqref{cutoff} and \eqref{superstimaalpha}, 
the term $(2)$ is a symbol satisfying 
$\|(2)\|_{-1,s,p}^{\gamma,\Omega_1}\lesssim_{s,p}\g^{-1}\epsilon(s+\mu)$.
The discussion above implies that, for some $ \mu = \mu (\tb) $,  
\begin{equation}\label{sottomarino3}
\|
\widetilde{\tD}_{\mathtt{m}}\big(\x(1+\pa_{y}\breve{\alpha}_{\s})\big)-
(1+\pa_{y}\breve{\alpha}_{\s})\widetilde{\tD}_{\mathtt{m}}(\x)
\|_{-1,s,p}^{\gamma,\Omega_1}
\lesssim_{s,p}
\g^{-1}\epsilon(s+\mu)\,.
\end{equation}
By \eqref{defr1s}, \eqref{sottomarino2}, \eqref{sottomarino3} we can write
\[
\begin{aligned}
\mathtt{r}^{(1)}_{\s}(\vphi,x,\x)&=(1+\mathfrak{c})
%(1+\pa_{y}\breve{\alpha}_{\s}(\vphi,y))_{|y=x+\alpha_{\s}(\vphi,x)}
\widetilde{\tD}_{\mathtt{m}}(\x)+q_{\s}
\end{aligned}
\]
for some symbol $q_\s\in S^{-1}$ satisfying
$\|q_{\s}\|_{-1,s,p}^{\gamma, \Omega_1}\lesssim_{s,p} \g^{-1}\epsilon(s+\mu)$, $\s\in\{\pm\}$
by Lemma \ref{Lemmino}  
and \eqref{interpolotutto},  
for some $\mu (\tb) $.
Using again Proposition \ref{prop:immersionepseudo}
the operator $\op(q_{\s})$ can be absorbed into a  remainder satisfying \eqref{restosmooth1}.
This proves  \eqref{claimo2} with $Q_{5,\s}= \widetilde Q_{5,\s}+\op(q_\s)$
with $ Q_{5,\s} $ satisfying \eqref{restosmooth1}.

Summarizing, by \eqref{odioprofondo}-\eqref{odioprofondo2} and 
\eqref{tempopezzo2}, \eqref{bastaa}-\eqref{claimo2}, $ \Pi_+ + \Pi_- = {\rm Id} $, 
the operator $B_{1}$ in \eqref{diagpezzoord1} is
$$
\begin{aligned}
B_{1}&=\omega\cdot\pa_{\vphi}-\ii (1+\mathfrak{c})\op(|\x|\chi(\x))-\ii\sum_{\s\in\{\pm\}}
\Big((1+\mathfrak{c})\op(\widetilde{\tD}_{\mathtt{m}}(\x))+Q_{5,\s}\Big)
\Pi_{\s}
+ Q_1{-\ii}{Q}_3
\\&
\stackrel{\eqref{billythekid}}{=}\omega\cdot\pa_{\vphi}-\ii \op((1+\mathfrak{c})\tD_{\mathtt{m}}(\x))+Q\,,
\end{aligned}
$$
for some remainder $Q$ satisfying \eqref{restosmooth1}. 
We also have that $Q$ is reversible and parity preserving by difference, 
because $B_1$ in \eqref{diagpezzoord1} is reversible and parity preserving
as well as $ - \ii \op ( \lambda {\mathtt D}_{\mathtt m} ) $, since 
$L$ is reversibility and parity preserving.
\end{proof}

Let us now consider the operator in \eqref{diagpezzoord0}. 
\begin{lemma}\label{lemma2}
The operator $ B_2 $ in \eqref{diagpezzoord0} has the form 
$$ 
B_2 = \op\big(\ii c^{(4)}(\vphi,x,\x)\big)+R 
$$ 
where $c^{(4)}(\vphi,x,\x)$ is a zero order real valued symbol of the form 
\eqref{simbolCC44}-\eqref{algebrella} satisfying \eqref{estranged4} 
an $R$ is a reversibility  and parity preserving 
operator  satisfying \eqref{restosmooth1}.
\end{lemma}

\begin{proof}
Reasoning as done in \eqref{cheodio}-\eqref{bastaa} (with $A= \op(\im b_0 \xi \tD_\tm^{-1}(\xi)+ c^{(3)})$), using  that $ b_0, c^{(3)} $ satisfy \eqref{condireiniziale}, \eqref{simboloCC33}, we have that
\begin{equation}\label{odioprofondo4}
\eqref{diagpezzoord0}=
\sum_{\s\in\{\pm\}}
\mathcal{C}_{\alpha_\s}
\op\big(\ii b_{0}\x\mathtt{D}_{\mathtt{m}}^{-1}(\x)+c^{(3)}\big)
\mathcal{C}_{\alpha_{\s}}^{-1} \Pi_{\s}
+R_1
\end{equation}
for some $R_1$ satisfying \eqref{restosmooth1}. 

Egorov  
Theorem \ref{thm:egorov}, Lemma \ref{Lemmino} and Proposition 
\ref{prop:immersionepseudo}, \eqref{simboloCC33} imply 
(as we did in \eqref{rosacroce}) that, for any 
$ \sigma = \pm $, 
\begin{equation}\label{bancodisabbia2}
\mathcal{C}_{\alpha_\s}
\op\big(\ii b_{0}\x\mathtt{D}_{\mathtt{m}}^{-1}(\x)+c^{(3)}\big)
\mathcal{C}_{\alpha_{\s}}^{-1}=
\op\big( \mathtt{g}_{\s} \big)  
%\mathcal{C}_{\alpha_\s} \op\big(c^{(3)}\big) \mathcal{C}_{\alpha_{\s}}^{-1}
+R_{2,\s} \, , 
\end{equation}
where $R_{2,\s}$ is an operator satisfying \eqref{restosmooth1}
and (recall formula \eqref{fundamentalformula}), 
\begin{equation}\label{sottomarino}
\mathtt{g}_{\s}(\vphi,x,\x)=\ii b_{0}(\vphi,x+\alpha_{\s}(\vphi,x))
\frac{\x(1+\pa_{y}\breve{\alpha}_{\s}(\vphi,y))}{\sqrt{|\x(1+\pa_{y}\breve{\alpha}_{\s}(\vphi,y))|^2+\mathtt{m}}}_{|y=x+\alpha_{\s}(\vphi,x)} \,.
\end{equation}
Then, setting $ c^{(4)}_{\s}(\vphi,x):=b_{0}(\vphi,x+\alpha_{\s}(\vphi,x)) $,
 we decompose % we claim that  
(recall \eqref{cutoff})
\begin{equation}\label{GGtildino}
\begin{aligned}
& \mathtt{g}_{\s}(\vphi,x,\x)=  \ii  \, c^{(4)}_{\s}(\vphi,x)
% b_{0}(\vphi,x+\alpha_{\s}(\vphi,x))
\frac{\x}{|\x|}\chi(\x)+\widetilde{\mathtt{g}}_{\s}(\vphi,x,\x) \\
& 
\widetilde \tg_\s (\vphi,x,\x) := 
(1-\chi(\x))\tg_\s (\vphi,x,\x) 
+\im b_0(\varphi, x+\al_\s)\xi \tf_\s( \varphi, x+\al_\s,\xi) \\
& 
\mathtt{f}_{\s}(\vphi,y,\x):= \chi(\x)\Big(\frac{1+\pa_{y}\breve{\alpha}_{\s}
(\varphi, y)}{\sqrt{|\x(1+\pa_{y}\breve{\alpha}_{\s}(\varphi, y))|^2
+\mathtt{m}}}-\frac{1}{|\x|}\Big) \, . 
\end{aligned}
\end{equation}
The  functions $ c^{(4)}_{\s}(\vphi,x) =b_{0}(\vphi,x+\alpha_{\s}(\vphi,x)) $,
 are odd in the couple $(\vphi,x)$ 
since  $b_{0}$ is odd  in $x\in \T$ and even in $\varphi$ (see \eqref{simbolistep0}). 
Then using \eqref{parityalfette}
the condition \eqref{algebrella} holds true.  Moreover, by the smallness condition 
\eqref{condireiniziale} and 
\eqref{superstimaalpha}, we have
\begin{equation}\label{marechiaro1}
\|c^{(4)}_{\s}\|_{s}^{\g,\Omega_1} \lesssim_{s} \gamma^{-1}
\epsilon(s+\mu)\,.
\end{equation}
and the estimate \eqref{estranged4} for the symbol $c^{(4)} (\vphi,x,\x) $ 
defined in \eqref{simbolCC44} follows. 
%by  \eqref{marechiaro1}.

We claim that the 
symbol $\widetilde{\mathtt{g}}_{\s} $ in \eqref{GGtildino} is in $ S^{-2} $
and satisfies, for some $\mu > 0 $,   
\begin{equation}\label{stimaGGtildino}
\|\widetilde{\mathtt{g}}_{\s}\|_{-2,s,p}^{\gamma,\Omega_1}\lesssim_{s,p}\epsilon(s+\mu )\,,\qquad \s\in \{\pm\}\, .
\end{equation}
Indeed, 
the symbol $(1-\chi(\x))\mathtt{g}_{\s}(\vphi,x,\x) \in S^{-\infty}$ satisfies 
\eqref{stimaGGtildino} 
by \eqref{condireiniziale}, using \eqref{ipopiccolezza} to estimate $b_{0}$, and 
using \eqref{superstimaalpha} to estimate $\alpha_{\s}$.
Then we remark that $ \mathtt{f}_{\s}(\vphi,y,\x) $ in \eqref{GGtildino} is equal to 
\[
\begin{aligned}
	\mathtt{f}_{\s}(\vphi,y,\x)& 
= \chi(\x)\frac{|\x|(1+\pa_{y}\breve{\alpha}_{\s})-
\sqrt{|\x(1+\pa_{y}\breve{\alpha}_{\s})|^2+\mathtt{m}}
}{|\x|\sqrt{|\x(1+\pa_{y}\breve{\alpha}_{\s})|^2+\mathtt{m}}}
\\&
=  \frac{- \mathtt{m}\chi(\x)}
{(|\x|(1+\pa_{y}\breve{\alpha}_{\s})+
	\sqrt{|\x(1+\pa_{y}\breve{\alpha}_{\s})|^2+\mathtt{m}})|\x|\sqrt{|\x(1+\pa_{y}\breve{\alpha}_{\s})|^2+\mathtt{m}}}\, , 
\end{aligned}
\]
and therefore, since $\chi(\x)\neq0$ implies that $|\x|>1/2$, 
for some $\mu>0$, 
$$
\| \mathtt{f}_{\s} (\vphi,x+\alpha_{\s} (\vphi,x),\xi ) \|^{\gamma,\Omega_1}_{-3,s,p}
\lesssim_{s,p}1+\gamma^{-1}\epsilon(s+\mu)\, ,
$$
and, using also \eqref{condireiniziale}, we deduce % \eqref{GGtildino}-
\eqref{stimaGGtildino}.
By Proposition
\ref{prop:immersionepseudo} with $m=-2$ we conclude 
that $\op(\widetilde\tg_{\s})$
satisfies  \eqref{restosmooth1}.

In conclusion, by \eqref{odioprofondo4}, \eqref{bancodisabbia2}, \eqref{GGtildino},
\eqref{def:cutoff} we have obtained
\begin{equation}\label{marechiaro3}
\eqref{diagpezzoord0} =
\sum_{\s\in\{\pm\}}\op\big(\ii c^{(4)}_{\sigma} \frac{\x}{|\x|}\chi(\x) \chi_\sigma (\xi) \big) 
+ R_4 = \op\big( \ii c^{(4)}(\vphi,x,\x) \big) + \op(\tr) +R_{4}
\end{equation}
where 
$c^{(4)}(\vphi,x,\x)$ is defined in \eqref{simbolCC44}  and 
%$R_{4} := R_1 +\sum_{\s\in\{\pm\}} (R_{2,\s}+\op(\widetilde\tg_\s))\Pi_{\s} $ satisfies \eqref{restosmooth1} and
\[
\begin{aligned}
& R_{4} := R_1
+\sum_{\s\in\{\pm\}} (R_{2,\s}+\op(\widetilde\tg_\s))\Pi_{\s}  	\\
& \tr := \sum_{\s\in\{\pm\}} \ii c^{(4)}_{\s}(\x|\x|^{-1}\chi(\x)\chi_{\s}(\x)-\s\chi_{\s}(\x))\,.
\end{aligned}
\]
The operator $R_{4} $ satisfies \eqref{restosmooth1}. 
Furthermore, recalling \eqref{cutoff} and \eqref{def:cutoff}
the symbols $	\x|\x|^{-1}\chi(\x)\chi_{\s}(\x)-\s\chi_{\s}(\x) $ 
are in $ S^{-\infty} $, and 
$\op(\tr)$ satisfies \eqref{restosmooth1}
by \eqref{marechiaro1} and Proposition \ref{prop:immersionepseudo}. 
In conclusion, by \eqref{marechiaro3} we deduce that 
$ B_2  =
\op\big(\ii c^{(4)}(\vphi,x,\x)\big)+R $
for some $R$ satisfying \eqref{restosmooth1}.
The operator $ R $ is reversibility and parity preserving by difference, 
because $B_2 $ and $ \ii c^{(4)}(\vphi,x,\x) $ are reversibility and parity preserving.
\end{proof}

It remains to study the term   
\eqref{pezzobrutto}.

\begin{lemma}\label{lemma3}
The  operator $F$ in \eqref{pezzobrutto} is reversibility and parity preserving 
and satisfies \eqref{restosmooth1}.
\end{lemma}

\begin{proof}
We recall that $ R_{-\rho}^{(3)}$ is a 
real-to-real 
matrix of reversibility and parity preserving symbols 
\[
R_{-\rho}^{(3)} =  \begin{pmatrix}
r_{-\rho}(\vphi, x, \xi) && q_{-\rho}(\vphi, x, \xi)\\ 
\overline{q_{-\rho}(\vphi, x, -\xi)} && \overline{r_{-\rho}(\vphi, x, -\xi)}\end{pmatrix}
\,,\quad r_{-\rho}, q_{-\rho} \in S^{-\rho}\, , 
\]
satisfying \eqref{R3-rho} with $\rho>0$ arbitrary large, so that 
\begin{equation}\label{coraggio}
\eqref{pezzobrutto}= \begin{pmatrix}
L \op( r_{-\rho}) L^{-1} && L \op( q_{-\rho}) \overline{L}^{-1}\\
\overline{L \op( q_{-\rho}) \overline{L}^{-1}} && \overline{L \op( r_{-\rho}) L^{-1}} 
\end{pmatrix} \, . 
\end{equation}
Reasoning exactly as done in \eqref{cheodio}-\eqref{bastaa} (with $A= \op(r_{-\rho} )$) and applying the Egorov Theorem \ref{thm:egorov} 
 we obtain that the diagonal term 
\[
L \op( r_{-\rho}) L^{-1} =  \sum_{\s\in\{\pm\}}
\op \Big( r_{-\rho} (\vphi, y, \xi (1+\pa_{y}\breve{\al}_{\s}) )_{|y= x + \alpha_\sigma} 
\chi_\sigma (\xi) \Big)   + T_{1}
\]
with $T_{1}$ satisfying \eqref{restosmooth1}. 
By Proposition \ref{prop:immersionepseudo},  for any $\rho \ge 1 $,  
also the first summands in the right hand side above satisfy  \eqref{restosmooth1}
and then  the whole diagonal term in \eqref{coraggio} satisfies \eqref{restosmooth1}.

We now prove that also the off-diagonal term in \eqref{coraggio} satisfies 
 \eqref{restosmooth1}. 
We write 
\begin{equation}\label{pezforuidia}
\jap{\td_{\vphi}}^{\tb} L \op( q_{-\rho}) \overline{L}^{-1}  \jap{D}=  \jap{\td_{\vphi}}^{\tb} L\jap{D}^{-N'}  \jap{D}^{N'}\op( q_{-\rho}) \jap{D}^{N'+1}  \jap{D}^{-N'-1}\overline{L}^{-1}  \jap{D}
\end{equation}
with $N' = \lfloor\nu/2\rfloor+4+\tb $.  
Since $L$, $\overline{L}^{-1}$ satisfy 
\eqref{psitresbarre}, 
 taking $\rho := \rho ( \tb ) 
:= \nu+10+2\tb$ (in Prop. \ref{blockTotale}), applying 
 Proposition \ref{prop:immersionepseudo} and using \eqref{R3-rho},
 we dedude that the operator in \eqref{pezforuidia} satisfies \eqref{restosmooth1}.
\end{proof}

Formulas 
\eqref{marechiaro1000}-\eqref{pezzobrutto}
together with Lemmata \ref{lemma1}, \ref{lemma2} and \ref{lemma3}
imply \eqref{elle4} with a symbol $ c^{(4)}$ 
of the form \eqref{simbolCC44}-\eqref{algebrella}
satisfying \eqref{estranged4} and 
an operator 
 $\cR^{(4)}$ satisfying \eqref{restosmooth1}. 
This concludes the proof of Proposition \ref{prop:ridord1}.

\subsection{Reduction at order  0 }
We now eliminate the dependence on $(\vphi, x)$ from the zero order 
symbol  $c^{(4)}(\vphi, x,\x)$ of the operator $ \cL_4 $ defined in \eqref{elle4}.

\begin{prop}\label{prop:ridord0} 
For any $\mathtt{b}\geq0$ there exist $\mu:=\mu(\tb)>0$, 
$\delta_{0}:=\delta_{0}(\su,\tb)>0$
such that, if \eqref{ipopiccolezza} holds with 
$\s_*\geq \mu$ and $\delta_{*}\le \delta_{0} $, then
for any $\omega\in \Omega_{1}$ (defined in \eqref{calOinfty1sec6})
the following holds.
There exists a real-to-real, reversibility and parity preserving,  invertible map 
$\Theta_2\in \mathcal{L}^{\mathtt{T}}(H^{s},H^{s})\otimes \mathcal{M}_2(\C)$,
for any $\so\leq s\leq \su$,
such that 
\begin{equation}\label{elle5}
\cL_5:= \Theta_2 \cL_4 \Theta_2^{-1}= \oo\cdot\pa_\vphi 
-\ii(1+\mathfrak{c})E \, \tD_{\mathtt{m}} - \ii E \cR^{(5)}
\end{equation}
where $\mathfrak{c}$ is given in Proposition \ref{prop:ridord1},
the operator  $ \tD_{\mathtt{m}}  $ is defined in
\eqref{TDM} and
the remainder $\cR^{(5)}$ 
is a real-to-real, reversibility  and parity preserving operator. Moreover
there is 
 $\mathtt{R}^{(5)}\in E_{s}\otimes \cM_2(\C)$ for any $\gso\le s\le \su$,  
such that $\mathfrak{S}(\mathtt{R}^{(5)})=\mathcal{R}^{(5)}$, satisfying 
 %for any $m_1,m_2\in \R_{+}$, 
%$m_1+m_2=1$ one has
 \begin{align}
  \bnorm{\langle\td_{\vphi}\rangle^{\tb}
  \mathtt{R}^{(5)}\langle D\rangle}_{s}^{\gamma,\Omega_1}
  &\lesssim_{s,\su,\mathtt{b}}\gamma^{-2}\epsilon(s+\mu)\,.
 % \qquad q=0\,,\tb\,.
  \label{restosmooth1bis}
 \end{align}
Finally for any $\oo\in \Omega_1$ and $\so\leq s \leq \su$, 
for any $ h\in H^{s}(\T^{\nu+1},\C^{2})  $, one has
\begin{equation}\label{stimasuPhidue}
\|\Theta_2^{\pm}h\|_s^{\g, \Omega_1}\lesssim_s 
\|h\|_s^{\g, \Omega_1} 
+ \g^{-2} \epsilon(s+\mu)
\|h\|_{\so}^{\g, \Omega_1} \,.
\end{equation}
\end{prop}

In the rest of the section we prove Proposition \ref{prop:ridord0}.
\begin{lemma}\label{eqordine0}
For any $\om\in \Omega_{1}$ there exists a real valued, reversibility and parity preserving  symbol $d :=d(\vphi,x,\x)\in S^{0}$ and a real valued reversible, parity preserving symbol  $r_{-2} := r_{-2}(\vphi,x,\x)\in S^{-2}$ satisfying  for  any $p\ge 0$, $s\geq \so$  the estimate
\begin{equation}\label{stimaGGplus}
\|d\|_{0,s,p}^{\gamma,\Omega_1}, \|r_{-2}\|_{-2,s,p}^{\gamma,\Omega_1}\lesssim_{s,p}\gamma^{-2}\epsilon(s+\mu)\,,
\end{equation}
such that  
\begin{equation}\label{equaztotale0}
\omega\cdot\pa_{\vphi}d(\vphi,x,\x)-(1+\mathfrak{c})\pa_{x}d(\vphi,x,\x)\x\tD_{\mathtt{m}}^{-1}(\x)
=c^{(4)}(\vphi,x,\x) 
+ r_{-2}(\vphi,x,\x) 
\end{equation}
where $c^{(4)}(\vphi,x,\x)$ is the symbol defined in
\eqref{simbolCC44}-\eqref{algebrella}.
\end{lemma}

\begin{proof}
We look for a solution $ d(\vphi,x,\x) $ of \eqref{equaztotale0} of the form
\begin{equation}\label{geneGG}
d(\vphi,x,\x):=d_{+}(\vphi,x)\chi_{+}(\x)+d_{-}(\vphi,x)\chi_{-}(\x) \, , \quad 
d_{-}(\vphi,x) := d_{+}(\vphi,-x)  \, .
\end{equation}
Let 
$d_{+}(\vphi,x)$ be the solution of the equation
\begin{equation}\label{equazgeneplus}
(\omega\cdot\pa_{\vphi}-(1+\mathfrak{c})\pa_{x})d_{+}(\vphi,x)
=c_{+}^{(4)}(\vphi,x) 
\end{equation}
where $ c_{+}^{(4)} $ is the function   in \eqref{simbolCC44}, which  
in Fourier expansion (recall \eqref{Fou1}) amounts to 
\[
\ii (\omega\cdot\ell-(1+\mathfrak{c})j )(d_{+})_{\ell,j}=(c_{+}^{(4)})_{\ell,j}\,,
\quad \forall \, (\ell,j)\in \Z^{\nu+1}  \, .
\]
By \eqref{algebrella} we deduce that $(c_{+}^{(4)})_{0,0}\equiv0$, 
and so, for any $\omega\in \Omega_1$ defined in \eqref{calOinfty1sec6},
\begin{equation}\label{estranged}
(d_{+})_{\ell,j}=\frac{(c_{+}^{(4)})_{\ell,j}}{\ii (\omega\cdot\ell-(1+\mathfrak{c})j )}\,,
\quad \forall\, (\ell,j)\in \Z^{\nu+1} \setminus \{(0,0)\} \, , \quad 
(d_{+})_{0,0} := 0 
\, . 
\end{equation}
Since $ c_{+}^{(4)} (\varphi,x)  $ is odd in $ (\varphi,x) $ by \eqref{algebrella} then 
\begin{equation}\label{algebrageneplus}
d_{+}(-\vphi,-x)=d_{+}(\vphi,x)\,.
\end{equation}
The function $d_{-}(\vphi,x) =d_{+}(\vphi,-x) $ satisfies,
using 
\eqref{equazgeneplus} and 
\eqref{algebrella},  the equation
\begin{equation}\label{equazgeneminus}
(\omega\cdot\pa_{\vphi}+(1+\mathfrak{c})\pa_{x})d_{-}(\vphi,x)
=-c_{-}^{(4)}(\vphi,x) \,.
\end{equation}
The symbol $ d  $ defined in 
in \eqref{geneGG} \eqref{estranged} 
 is real, reversibility and parity preserving according to 
\eqref{reverpreserSimbolo}-\eqref{paritySimbolo},
using \eqref{algebrageneplus}, \eqref{def:cutoff} and \eqref{algebrella}. 
Furthermore   \eqref{stimaGGplus} holds 
by \eqref{estranged}, the bounds on the small divisors in  \eqref{calOinfty1sec6}, \eqref{stimaMgotico1},
the estimate \eqref{estranged4} on $c^{(4)}(\vphi,x,\x)$. 

In order to prove that the symbol $ d $
in \eqref{geneGG}, \eqref{estranged} solves \eqref{equaztotale0} 
 we expand  (recall \eqref{def:cutoff} and \eqref{TDM})
\begin{equation}\label{patgarret}
\chi_{\s}(\x)\x\tD_{\mathtt{m}}^{-1}(\x)=\s \chi_{\s}(\x)+\tr_{\s}(\x)\,,\quad \s\in\{\pm\} \, , 
\end{equation}
where $ \mathtt{r}_- (\xi ) = - \mathtt{r}_+ (- \x)  $,  
\begin{equation}\label{estranged2}
\mathtt{r}_{\s}(\x) := \chi_{\s}(\x)\big(\x\tD_{\mathtt{m}}^{-1}(\x)-\s 1\big)
=-\mathtt{m}\chi_{\s}(\x) \sigma 
\tD_{\mathtt{m}}^{-1}(\x)\big( \sigma \x+ \sqrt{\x^{2}+\mathtt{m}}\big)^{-1}
\end{equation}
is a symbol in  $ S^{- 2} $, because  
$\tm>0 $ and,  by the 
definition of the cut-off functions $\chi_{\s}(\x)$ in \eqref{def:cutoff}, 
we have that $\chi_{\s}(\x)\neq0$ implies $\s\x\geq -1/2 $.
The symbol  $d(\vphi,x,\x)$ in 
\eqref{geneGG} satisfies, by \eqref{simbolCC44},\eqref{patgarret}, % the equation \eqref{equaztotale0}, indeed
$$
\begin{aligned}
& \omega\cdot\pa_{\vphi}d(\vphi,x,\x)
-(1+\mathfrak{c})\pa_{x}d(\vphi,x,\x)\x\tD_{\mathtt{m}}^{-1}(\x)-c^{(4)}(\vphi,x,\x) 
\\ &
= 
\big(\omega\cdot\pa_{\vphi}d_{+}(\vphi,x)
-(1+\mathfrak{c})\pa_{x}d_+(\vphi,x)-c^{(4)}_{+}(\vphi,x)\big)\chi_{+}(\x)
\\&
\   +
\big(\omega\cdot\pa_{\vphi}d_{-}(\vphi,x)
+(1+\mathfrak{c})\pa_{x}d_{-}(\vphi,x)+c^{(4)}_{-}(\vphi,x)\big)\chi_{-}(\x)
\\ &
\  -(1+\mathfrak{c})\pa_{x}d_+(\vphi,x)\mathtt{r}_{+}(\x)
-(1+\mathfrak{c})\pa_{x}d_{-}(\vphi,x)\mathtt{r}_{-}(\x)
\\&
\stackrel{\eqref{equazgeneplus}, \eqref{equazgeneminus}}{=}
-(1+\mathfrak{c})\pa_{x}d_+(\vphi,x)\mathtt{r}_{+}(\x)
-(1+\mathfrak{c})\pa_{x}d_{-}(\vphi,x)\mathtt{r}_{-}(\x) =: r_{-2}(\vphi,x,\x)  \, .
\end{aligned}
$$
Note that $ r_{-2} $
is a symbol in $ S^{-2}$, as $ \mathtt{r}_{\sigma}(\x)  $ 
in \eqref{estranged2}, and it satisfies \eqref{stimaGGplus} as 
well as $ d$.
The explicit expression of $ r_{-2}(\vphi,x,\x) $ shows that 
 $r_{-2} $ is real valued, reversible and parity preserving.
\end{proof}

\begin{proof}[{\bf Proof of Proposition \ref{prop:ridord0}}]
We conjugate $ \cL_4 $ via 
the real-to-real, reversibility and parity preserving invertible  map 
\begin{equation}\label{mappapsi2}
\Theta_{2}:=\sm{\Psi_2}{0}{0}{\overline{\Psi_2}}  
\qquad \text{where} \qquad 
\Psi_{2}:= \Psi_{2}^1 \, , \quad 
 \Psi_{2}^\tau := \exp\big\{\tau \op(d(\vphi,x,\x))\big\} \, , 
\end{equation}
and $d(\vphi,x,\x)$   is  the zero order symbol
defined  in Lemma \ref{eqordine0}.
Thanks to the algebraic properties of $d$ in Lemma \ref{eqordine0}, 
the map $\Psi_{2}$
is reversibility and parity preserving.
 %  follows by \eqref{stimasuPhiduedue}.
% recalling \eqref{mappapsi2}, 
Lemma 
% \ref{potenza pseudo}, 
\ref{ecponential pseudo diff}, \eqref{stimaGGplus}, and 
\eqref{ipopiccolezza} imply that % hat, for any $ \tau \in [0,1] $,  
\begin{equation}\label{psi2tau}
\| \Psi_2^\tau - \id \|_{0,s,0}^{\gamma,\cO} \lesssim_{s}  \| d \|_{0,s,0}^{\gamma,\cO}
\lesssim_{s} {\g^{-2}}\epsilon(s+\mu)
\end{equation}
 uniformly in  $ \tau \in [0,1] $ and
 thus, by Lemma \ref{sobaction}, 
we deduce \eqref{stimasuPhidue}.
Furthermore, by Proposition \ref{prop:immersionepseudo}, 
there are $\mathtt{B}^\pm 
\in E_{s}\otimes\cM_2(\C)$, $\gso\leq s\leq \su$, such that
$\mathfrak{S}(\mathtt{B}^\pm )= \Theta_2^{\pm 1}$
for any  $m_1,m_2\in \R$, $m_1+m_2=0$ with   $|m_i|\le 1$ one has
for some $\mu := \mu (\tb) $ (possibly larger that the one in \eqref{stimaGGplus})
\begin{equation}\label{psi22tresbarre}
\bnorm{\langle \td_{\vphi}\rangle^{q}\langle D\rangle^{-m_1}(\mathtt{B}^{\pm}-\id)\langle D\rangle^{-m_2}}_{s}
\lesssim_{s,\su,\tb}{\g^{-2}}\epsilon(s+\mu)\,,\;\; q=0\,,\tb\, .
\end{equation}
In view of  \eqref{elle4} the conjugate operator is 
\begin{align}
\Theta_2 \cL_4 \Theta_2^{-1} \! &= \! \Theta_2 \oo\cdot\pa_\vphi \Theta_2^{-1}\notag
\\&-\ii \Theta_2E \op\big((1+\mathfrak{c})\tD_{\mathtt{m}}(\x) 
+\sm{\ii  c^{(4)}(\vphi,x,\x)}{0}{0}{\! \! \!\! \overline{\ii  c^{(4)}(\vphi,x,-\x)}}
\big)\Theta_2^{-1}
-\Theta_2 \ii E\cR^{(4)}\Theta_2^{-1} \notag 
\\
&=\sm{F}{0}{0}{\overline{F}}-\ii E\sm{G_1}{0}{0}{\overline{G_1}}
-\ii E\sm{G_2}{0}{0}{\overline{G_2}}-\ii E \Theta_2\cR^{(4)}\Theta_2^{-1}  \label{pezcon4}
\end{align}
where
\begin{equation}\label{cipresso}
F:=\Psi_2\omega\cdot\pa_{\vphi}\Psi_2^{-1}, \,  G_1:= \Psi_2\mathtt{X}\Psi_2^{-1},
\,  
\mathtt{X}:=\op\big((1+\mathfrak{c})\tD_{\mathtt{m}}(\x)\big), \, 
G_2:= \Psi_2\mathtt{Y}\Psi_2^{-1}, 
\,  \mathtt{Y}:=\op\big(\ii  c^{(4)}\big)\,.
\end{equation}
By  \eqref{psi22tresbarre} and \eqref{restosmooth1}
the operator
$\Theta_2\mathcal{R}^{(4)}\Theta_2^{-1}$ satisfies 
\eqref{restosmooth1bis}.
We now study the other terms.
A Lie expansion  
gives  that
\begin{equation}\label{cipresso0}
F = \Psi_2 \ \oo\cdot\pa_\vphi \Psi_2^{-1} 
=\omega\cdot\pa_{\vphi}-\op\big(\omega\cdot\pa_{\vphi}d\big)+\mathcal{Q}_1
\end{equation}
where
\[
\mathcal{Q}_1:= - \int_0^1 (1-\tau) \Psi_2^\tau \, 
{\rm ad}_{\op(d)} [  \op (\omega\cdot\pa_{\vphi}d)] \Psi_2^{-\tau} d \tau \, .
\]
By Lemma \ref{lemma:Commutator} the symbol
 $d\star (\omega \cdot\pa_{\vphi}d) $ 
 is in $ S^{-1}$.
The
pseudo-differential norm $\| \mathcal{Q}_1 \|_{-1, s, 0}^{\g, \cO}$ can be estimated by 
\eqref{psi2tau}, \eqref{stimasharp} and, 
by Proposition \ref{prop:immersionepseudo} 
we deduce that the operator $\mathcal{Q}_1$ 
satisfies \eqref{restosmooth1bis}.

Similarly, recalling \eqref{cipresso},  by a Lie expansion we get 
\begin{equation}\label{cipresso2}
G_1  = \Psi_2\mathtt{X}\Psi_2^{-1} =\mathtt{X}+[\op(d),\mathtt{X}]+\mathcal{Q}_2\,, 
\qquad 
G_2 = \Psi_2\mathtt{Y}\Psi_2^{-1}=\mathtt{Y}+\mathcal{Q}_3\,,
\end{equation}
where
\[
\mathcal{Q}_2:=
\int_0^1 (1-\tau) \Psi_2^\tau \, 
{\rm ad}^2_{\op(d)} [ \mathtt{X}  ] \Psi_2^{-\tau} d \tau  \, , 
\qquad 
\mathcal{Q}_3:= 
\int_0^1  \Psi_2^\tau \, 
{\rm ad}_{\op(d)} [ \mathtt{Y}  ] \Psi_2^{-\tau} d \tau  \, . 
\]
By \eqref{psi2tau}, \eqref{stimasharp} and, 
by Proposition \ref{prop:immersionepseudo} we deduce that $\mathcal{Q}_2$ and $\mathcal{Q}_3$
satisfy  \eqref{restosmooth1bis}.

Moreover, recalling \eqref{cipresso},  
Lemma \ref{lemma:Commutator} (see also \eqref{espstar}-\eqref{espstar2})
and \eqref{stimaGGplus} imply that 
\begin{equation}\label{cipresso3}
{[}\op(d),\mathtt{X}{]} = \op ( d \star ((1+\mathfrak{c})\tD_{\mathtt{m}}(\x))) 
=\op\big(\ii (1+\mathfrak{c})(\pa_{x}d)(\vphi,x,\x)\x\tD_{\mathtt{m}}^{-1}(\x)+\tr_{2} \big)
\end{equation}
for some $\tr_{2}\in S^{-1}$ satisfying
$ \|\tr_{2}\|_{-1,s,p}^{\gamma,\Omega_1}\lesssim_{s,p}\gamma^{-2}\epsilon(s+\mu) $.
Then by Proposition \ref{prop:immersionepseudo}, 
the operator $\op(\tr_{2})$ satisfies  
\eqref{restosmooth1bis}. In conclusion, 
by  \eqref{cipresso}, \eqref{cipresso0}, \eqref{cipresso2}
and \eqref{cipresso3} we deduce that 
\begin{equation}\label{FiG1G2}
F-\ii( G_{1 }+G_{2}) =\omega\cdot\pa_{\vphi}
-\ii \op\big((1+\mathfrak{c})\tD_{\mathtt{m}}(\x)\big)
+\op\big( r\big) + \cQ_{4}
\end{equation}
where $\cQ_{4}$ satisfies \eqref{restosmooth1bis}
and  
$ r :=
-\omega\cdot\pa_{\vphi}d+(1+\mathfrak{c})\pa_{x}d \, 
\x\tD_{\mathtt{m}}^{-1}(\x)+c^{(4)} \,  =  -r_{-2} $ 
by Lemma \ref{eqordine0} (see  \eqref{equaztotale0}). 
By \eqref{stimaGGplus} and Proposition \ref{prop:immersionepseudo}
the operator $\op(r_{-2})$ 
satisfies \eqref{restosmooth1bis}. 
Finally  $\Theta_2\cR^{(4)}\Theta_2^{-1}$ in \eqref{pezcon4}
satisfies \eqref{restosmooth1bis} by 
\eqref{psi22tresbarre} and \eqref{restosmooth1} and we 
conclude by \eqref{pezcon4}, \eqref{FiG1G2} that $ \cL_5 $ has the form  \eqref{elle5}
with an operator $\cR^{(5)} $ satisfying \eqref{restosmooth1bis}. 

The operator $\cL_{5}$ is real-to-real, reversible and parity preserving 
as well as $\cL_{4}$ since  $\Theta_{2}$ is real-to-real, reversibility and parity preserving and, 
by difference,  $\ii E \cR_{5}$ is so.
\end{proof}

\section{KAM reducibility}\label{sec:kam}

We complete the  reducibility of the operator $\mathcal{L}_{5}$
in \eqref{elle5} in Section \ref{sec:measure}  by applying 
the  abstract reducibility Theorem \ref{thm:reduKG} below.
It applies to a family of 
real-to-real, reversible and parity preserving operators 
$\widehat{\mathcal{O}}\ni\omega\to \mathcal{M}_0 :=\mathcal{M}_0(\omega)$,
defined for any $ \omega $ in 
a compact subset $\widehat{\mathcal{O}}\subseteq\Lambda =[-1/2,1/2]^{\nu}$,   
of the form
\begin{equation}\label{emme0}
\cM_0 := \oo\cdot \pa_\vphi -\ii E  \cD_0 -\ii E \cP_0\,,
\end{equation}
satisfying the   assumptions below. We use the block matrix representation of operators
of Section  \ref{sec:matrici22},
the parameters $\gamma,\tau,s_*,\su$   in \eqref{costanti}-\eqref{costantiGAMMA}
and we set
\begin{equation}\label{def:bbb}
\tb :=6\tau +6\,.
\end{equation}

\smallskip
\noindent
$\bullet$ {\bf (Unperturbed Normal form).} The operator $\mathcal{D}_0 :=((\cD_0)_{\s}^{\s'})_{\s,\s'\in \{\pm\}}\in \cL^{\tT}(H^s,H^{s-1})\otimes \cM_2(\C)$ 
is defined as 
\begin{equation}\label{D00D}
\begin{aligned}
& (\cD_0)_{-}^{-} = (\cD_0)_{+}^{+}:= \diag_{j\in \N_{0}} d_{\vec{\jmath}}^{(0)} \,, \quad (\cD_0)_{+}^{-}=(\cD_0)_{-}^{+}\equiv0\,,  
\\
&d_{\vec{\jmath}}^{(0)}:= 
(1 + \fc) \Id\tD_\mathtt{m}(j)\,,\; j\in\N\,, \quad d_{\vec{0}}^{(0)}:= 
(1 + \fc) \sqrt{\tm} \, , 
\end{aligned}
\end{equation}
where $\tD_\mathtt{m}(j)$ is  in \eqref{TDM}
and $\mathfrak{c} : \Lambda \to \R $ is a Lipschitz function satisfying 
\begin{equation}\label{stimamasterdeltino}
 \gamma^{-1}|\mathfrak{c}|^{\gamma,\Lambda}\le \upsilon_{0}\,.
\end{equation}
Thus it is in normal form, real-to-real, reversibility and parity preserving 
according to Definition \ref{def:normalform} and Lemma \ref{rmk:algebra2}.

\smallskip
\noindent
$\bullet$ {\bf (Perturbation).}
For any $ s_* \le s\le \su$
the operator $\mathcal{P}_0\in \cL^{\mathtt{T}}(H^{s}, H^{s})\otimes \cM_2(\C)$ is real-to-real, reversibility and parity preserving
and $1$-smoothing in the sense that
there exists
 $\mathtt{P}_0\in E_{s}\otimes \cM_2(\C)$ 
such that $\mathfrak{S}(\mathtt{P}_0)=\cP_0$  and
\begin{equation}\label{def:epsilon0}
	\begin{aligned}
		\eps_0(s)&:=
		\g^{-3/2}
		\bnorm{ 
			\tP_0 \jap{D}}^{\g^{3/2},\widehat{\mathcal{O}}}_{s}<+\infty \,,\quad 
		\eps_0(s, \tb):=\g^{-3/2}
		\bnorm{
			\jap{\td_{\vphi}}^\tb  \tP_0 \jap{D}}^{\g^{3/2}, \widehat{\mathcal{O}}
		}_{s}<+\infty\,.
	\end{aligned}
\end{equation}

\begin{rmk}\label{allerta}
By Definition \ref{defn:japphi}, formula \eqref{rmk:commuES}  and  \eqref{od} it results 
$\eps_0(s)\leq \eps_0(s,\tb)$.
\end{rmk}

The main result of this section is the following.

\begin{thm}{\bf (KAM Reducibility).}\label{thm:reduKG}
Fix $\tau, \gso$ as in \eqref{costanti},  $\mathtt{b}>0$ as in \eqref{def:bbb}.
%and 
%consider $\mathcal{M}_0$ in \eqref{emme0}.
There
exist $\upsilon_{0} \in (0, \tfrac{1}{2})$,  $C_0\geq1$, such that, 
for any $\gamma,\su$ as in \eqref{costantiGAMMA},
if  $\mathcal{M}_0$ in \eqref{emme0} satisfies \eqref{stimamasterdeltino}, \eqref{def:epsilon0} and 
\begin{equation}\label{PiccolezzaperKamredDP}
{C_0}\eps_0(\gso,\tb)\le {\upsilon}_0\,, 
\end{equation}
then for any $\gso\leq s\leq \su$ the following holds: 
\begin{itemize}
\item[(i)] \textbf{(Normal form).} For any $\omega\in \Lambda$ there exists 
a real-to-real, reversibility and parity preserving  operator $\mathcal{D}_{\infty}=\mathcal{D}_{\infty}(\omega)$
in normal form %(see Def. \ref{def:normalform})
of the form
\begin{equation}\label{Finalnormalform}
	\begin{aligned}
	&	(\cD_\infty)_{-}^{-}=(\cD_\infty)_{+}^{+}:= \diag_{j\in \N_{0}} d_{\vec{\jmath}}^{(\infty)}\,,
		\quad (\cD_\infty)_{+}^{-}=(\cD_\infty)_{-}^{+}\equiv0\, , 	\\
	&	d_{\vec{\jmath}}^{(\infty)}:= d_{\vec{\jmath}}^{(0)}+\mathtt{r}_{\vec{\jmath}}^{(\infty)}\,,\quad 
		\mathtt{r}_{\vec{\jmath}}^{(\infty)}:=  (\mathtt{r}_{\infty})_{j}^{j} \Id + (\mathtt{r}_{\infty})_{j}^{-j} S
		\, , \   j\in \mathbb{N}\,,
	\quad \mathtt{r}_{\vec{0}}^{(\infty)}:=  (\mathtt{r}_{\infty})_{0}^{0} \,,
	\end{aligned}
\end{equation}
where $d_{\vec{\jmath}}^{(0)}$ is in  \eqref{D00D}, 
the matrix $S$ is in \eqref{involutionCompl}, %{bolena1},  
and  
\begin{equation}\label{stimeautovalfinali}
	\begin{aligned}
		&(\tr_{\infty})_{j}^{\s j}\;: \Lambda\to \R\,,
		\quad 
		\sup_{j\in\N_{0}, \s\in \{\pm\}} \langle j\rangle |(\tr_{\infty})_{j}^{\s j}|^{\gamma^{3/2},\Lambda}
		\lesssim \gamma^{3/2}\eps_0(\gso)\,.
	\end{aligned}
\end{equation}
We denote by 
\begin{equation}\label{autovalorifinali}
\lambda^{(\infty)}_{j,\pm} :=(1 + \fc) \tD_\mathtt{m}(j)
+(\tr_{\infty})_{j}^{j}\pm(\tr_{\infty})_{j}^{-j}, \  j\in\N\,,\ \ 
\lambda^{(\infty)}_{0,\pm} : =(1 + \fc) \sqrt{\tm}
+(\tr_{\infty})_{0}^{0}  \, ,
\end{equation}
the eigenvalues of the matrix $ d_{\vec{\jmath}}^{(\infty)} $.
 \item[(ii)] \textbf{(Conjugacy)}. Define the sets
\begin{equation}\label{calOinfty2sec7}
\begin{aligned}
\calO_{\infty}:=&\Omega_{\infty}^{+}\cap\Omega_{\infty}^{-}\,,\\
\Omega_{\infty}^{+}:=&\Big\{\omega\in \widehat{\calO}\,:\, 
|\omega\cdot\ell+\lambda_{j,\eta}^{(\infty)}+\lambda_{k,\eta}^{(\infty)}|\geq 
\frac{2\gamma}{\langle\ell\rangle^{\tau}}\,,
\;\; j,k\in\N_{0}\,,\; \ell\in \Z^{\nu}\,,\eta \in \{\pm\} 
\Big\}
\\
\Omega_{\infty}^{-}:=&\Big\{\omega\in \widehat{\calO}\,:\, 
|\omega\cdot\ell+\lambda_{j,\eta}^{(\infty)}-\lambda_{k,\eta}^{(\infty)}|\geq 
\frac{2\gamma^{3/2}}{\langle\ell\rangle^{\tau}}\,,\; \eta\in \{\pm\} \, , 
\\&\qquad\qquad\qquad\qquad \qquad\qquad \qquad  
\;\; j,k\in\N_{0}\,,\; \ell\in \Z^{\nu}\,,
\;\;(\ell, j,k)\neq(0, j, j)
\Big\} \, . 
\end{aligned}
\end{equation}
For any 
$\omega\in\calO_{\infty}$ 
there is a real-to-real, reversibility and parity preserving, bounded, 
invertible, linear operator 
$\Phi_{\infty}(\omega)\in\mathcal{L}^{\mathtt{T}}(H^s, H^{s})\otimes \cM_2(\C)$, 
with bounded inverse $\Phi_{\infty}^{-1}(\omega)$, that conjugates 
the operator $\mathcal{M}_0$ 
in \eqref{emme0} to % constant coefficients, namely 
\begin{equation}\label{Linfinito}
\begin{aligned}
&\mathcal{M}_{\infty}(\omega):=\Phi_{\infty} (\omega) \circ \mathcal{M}_0 \circ \Phi_{\infty}^{-1}(\omega)=\omega\cdot \partial_{\varphi}-\ii E \mathcal{D}_{\infty}\,.
\end{aligned}
\end{equation}
Moreover, there exists $\mathtt{\Phi}_{\infty}\in E_{s}\otimes \cM_2(\C)$ 
with $\fS(\mathtt{\Phi}_{\infty})=\Phi_{\infty}$ such that
\begin{equation}\label{grano}
\bnorm{ \mathtt{\Phi}_{\infty}^{\pm1} - \id }^{\gamma^{3/2}, \calO_{\infty}}_{s}\lesssim
\eps_0(s,\tb)\,,\qquad \forall \, \gso\leq s\leq \su\,.
\end{equation}
For any $ \varphi \in {\mathbb T}^\nu $, for any $ \gso\leq s\leq \su $,  
\begin{equation}\label{granophi}
\| (\mathtt{\Phi}_{\infty}^{\pm1} (\varphi) - \id ) u \|_{H^s_x} \lesssim
\eps_0(s_0,\tb) \| u \|_{H^s_x} + \eps_0(s,\tb) \| u \|_{H^{s_0}_x} \,.
\end{equation}
\end{itemize}

\end{thm}
In the rest of this section we prove Theorem \ref{thm:reduKG}.

\begin{rmk}
The constants $C_0,{\upsilon}_0$ and the ones appearing in the estimates 
\eqref{stimeautovalfinali}-\eqref{grano} only depend on $\nu,\tau,\gso$ fixed in \eqref{costanti}
and not on the Sobolev index $s,s_1$, nor on the diophantine constant $\gamma$.
As is common in KAM schemes the dependence on $\gamma$ appears only in the definition of the
 small parameter $\eps_0$ in \eqref{def:epsilon0}.
The independence on $s,s_1$ is instead a feature of the norm $\bnorm{ \cdot}_{s}$
which satisfies strong algebra/interpolation estimates (see Lemma \ref{tretameestimate}).
\end{rmk}

\subsection{KAM step}
Let $ N \geq 2 $ and consider a 
real-to-real, reversible and parity  preserving operator
\begin{equation}\label{Msteppino}
\mathcal{M}=\omega\cdot\pa_{\vphi}-\ii E \mathcal{D}-\ii E\mathcal{P}\,,
\end{equation}
where $\mathcal{D}$
is a real-to-real, reversibility and parity  preserving operator 
in normal form  
\begin{equation}\label{def:Dold}
(\mathcal{D})_{+,\vec{\jmath}}^{+,\vec{\jmath}} \, (0)=
d_{\vec{\jmath}}^{(0)}+
\tr_{\vec{\jmath}}\,,\quad
\mathtt{r}_{\vec{\jmath}}:= \mathtt{r}_{j}^{j} \Id + \mathtt{r}_{j}^{-j} S\,,
 \quad j\in \N\,,\qquad   \mathtt{r}_{\vec{0}}:=  (\mathtt{r}_{\infty})_{0}^{0}\,,
\end{equation}
where  $\tr_{j}^{\eta j} : \Lambda\to \R$, $\eta\in \{\pm\}$ are  Lipschitz functions
satisfying   (recall \eqref{def:epsilon0})
%$\tr_{j}^{\s j}=\tr_{-j}^{-\eta j}$ and
\begin{equation}\label{kamass1}
%\tr_{j}^{\s j}=\tr_{-j}^{-\s j}\qquad{\rm and}\qquad 
\sup_{j\in\N_{0},\eta\in \{\pm\}}\langle j\rangle |\tr_{j}^{\eta j}|^{\gamma^{3/2},\Lambda}
\lesssim \g^{\frac{3}{2}} \eps_{0}(\gso)\,.
\end{equation}
Furthermore the operator $\mathcal{P}$ is defined
in a compact subset $\cO\subset\widehat{\calO}\subseteq\Lambda$,  
it is a real-to-real, reversibility   and parity preserving  and
there exists $\tP\in E_s\otimes \cM_2(\C)$ such that $\fS(\tP)=\cP$ and
\begin{equation}\label{stimePPstepBIS}
	\begin{aligned}
		\eps(s):=\gamma^{-3/2}
		\bnorm{  \tP \jap{D}}^{\gamma^{3/2} ,\cO}_{s} < + \infty \,,\quad
		\eps(s,\tb):=\gamma^{-3/2}
		\bnorm{\jap{\td_{\vphi}}^\tb  \tP \jap{D}}^{\gamma^{3/2} ,\cO}_{s} < + \infty \,  . 
	\end{aligned}
\end{equation} 
We  define 
\begin{equation}\label{calOhomoeq1}
	\begin{aligned}
		\mathcal{O}_{+}&:=\Omega_{+}^{+}\cap\Omega_{+}^{-}\,,
		\\
		\Omega_{+}^{+}&:=\left\{\omega\in \mathcal{O}\,:\, 
		|\omega\cdot\ell+\lambda_{j,\eta}+\lambda_{k,\eta}|\geq 
		\frac{\gamma}{\langle\ell\rangle^{\tau}}\,, \quad \eta\in \{\pm\},
		\;\; |\ell|\leq N, j,k\in\N_{0}
		\right\}
		\\
		\Omega_{+}^{-}&:=\Big\{\omega\in \mathcal{O}\,:\, 
		|\omega\cdot\ell+\lambda_{j,\eta}-\lambda_{k,\eta}|\geq 
		\frac{\gamma^{3/2}}{\langle\ell\rangle^{\tau}}\,,\quad \eta\in \{\pm\},
		\\&\qquad \qquad\qquad \qquad \qquad\quad
		\;\; |\ell|\leq N, j,k\in\N_{0}\,,\; \ell\in \Z^{\nu}\,,
		\;\;(\ell, j,k)\neq(0, j, j)
		\Big\}
	\end{aligned}
\end{equation}
where (recall \eqref{D00D})
\begin{equation}\label{calOhomoeq2}
	\begin{aligned}
		&\lambda_{j,\pm}:=(1 + \fc) \tD_\mathtt{m}(j)+\tr_{j}^{j} \pm \tr_{j}^{-j}\,,
		\quad \forall\, j\in\N\,, \quad 
		\lambda_{0,\pm} :=(1 + \fc) \sqrt{\tm}
		+\tr_{0}^{0}\,,
	\end{aligned}
\end{equation}
are the eigenvalues of $\cD$.

\begin{lemma}{\bf (Homological equation).}\label{lem:homoeq}
Consider 
$\cP$ satisfying \eqref{stimePPstepBIS}.
For any $\omega\in \mathcal{O}_{+}$ defined in \eqref{calOhomoeq1}
there exists 
\[
\mathcal{S}=(\mathcal{S}_{\s}^{\s'}(\vphi))_{\s,\s'\in \{\pm\}}
\in \mathcal{L}^{\mathtt{T}}(H^{s}, H^{s})\otimes\mathcal{M}_{2}(\C)\,,\;\;\;
\forall\gso\leq s\leq \su \, , 
\]
such that
\begin{equation}\label{omoeq1}
-\omega\cdot\pa_{\vphi}\mathcal{S}_{\s}^{\s'}(\vphi)
+ \ii  \s \mathcal{D}_{\s}^{\s}\mathcal{S}_{\s}^{\s'}(\vphi)- \ii 
\s' \mathcal{S}_{\s}^{\s'}(\vphi)\mathcal{D}_{\s'}^{\s'}
=\ii \s\Big( (\Pi_{N}\mathcal{P})_{\s}^{\s'}(\vphi)- [\mathcal{P}]_{\s}^{\s'}\Big)\,,
\end{equation}
for any $\s,\s'\in \{\pm\}$, 
where $\Pi_{N}$ is defined in  \eqref{def:proj}, and 
$[\cdot ]$ denotes the projection on normal forms (see Definition \ref{def:normalform}).
Moreover
there exists $\mathtt{S}\in E_s\otimes \cM_2(\C)$ such that $\mathfrak{S}(\mathtt{S})=\mathcal{S}$
and 
\begin{equation}\label{solenapoli}
\begin{aligned}
\bnorm{  \mathtt{S} 
\jap{D}}^{\g^{3/2} ,\cO_{+}}_{s}\lesssim N^{2\tau+1}\eps(s)\,,\quad 
\bnorm{ \langle\td_{\vphi}\rangle^{\tb}  \mathtt{S} 
\jap{D}}^{\g^{3/2} ,\cO_{+}}_{s}
\lesssim N^{2\tau+1}\eps(s,\tb)\,.
\end{aligned}
\end{equation}
Finally one has $\Pi_{N}\mathtt{S}=\mathtt{S}$ so that
$\bnorm{ \langle\td_{\vphi}\rangle^{b}  \mathtt{S} 
\jap{D}}^{\g^{3/2} ,\cO_{+}}_{s}
\lesssim N^{2\tau+1+b}\eps(s)$ for any $b>\nu/2$.
The operator  $\mathcal{S}$ is real-to-real, reversibility and parity preserving.
\end{lemma}

\begin{proof}
Recalling  the notation \eqref{supermatrice}-\eqref{supermatrice2}, 
in order to find a solution of \eqref{omoeq1}, we have to solve 
\begin{equation}\label{omoeq2}
-\ii\omega\cdot\ell \mathcal{S}_{\s,\vec{\jmath}}^{\s',\vec{k}}(\ell)
+
\ii \s\mathcal{D}_{\s,\vec{\jmath}}^{\s,\vec{\jmath}}\,
\mathcal{S}_{\s,\vec{\jmath}}^{\s',\vec{k}}(\ell)
-
\ii\s' \mathcal{S}_{\s,\vec{\jmath}}^{\s',\vec{k}}(\ell)\,
\mathcal{D}_{\s',\vec{k}}^{\s',\vec{k}}
=
\ii \s \mathcal{P}_{\s,\vec{\jmath}}^{\s',\vec{k}}(\ell) 
\end{equation}
for any $ | \ell | \leq N $, and, 
 if $\s\neq\s'$  for any $ j,k\in\N_{0}$, 
or, if $\s=\s'$, for any  $(\ell,j,k)\neq(0,j,j)$. 
Otherwise we set $  \mathcal{S}_{\s,\vec{\jmath}}^{\s',\vec{k}}(\ell) := 0 $. 
In order to solve \eqref{omoeq2} we diagonalize it 
by conjugating with the matrix $\fU$ introduced in Remark \ref{algebra2}. 
We define $\widetilde{\mathcal{P}}_{\s,\vec{\jmath}}^{\s',\vec{k}}(\ell)$, $ \widetilde{\mathcal{S}}_{\s,\vec{\jmath}}^{\s',\vec{k}}(\ell)$  following formula \eqref{contilde} as well as 
\[
\widetilde{\mathcal{D}}_{-,\vec{\jmath}}^{-,\vec{\jmath}}(0):= \mathfrak{U}^{-1}\mathcal{D}_{-,\vec{\jmath}}^{-,\vec{\jmath}}\,(0)\mathfrak{U}\
=
\widetilde{\mathcal{D}}_{+,\vec{\jmath}}^{+,\vec{\jmath}}(0)
=
\left(
\begin{matrix}
	\lambda_{j,+} & 0 \\ 0 & \lambda_{j,-}
\end{matrix}
\right)\,, \  j\in\N\,, \quad
\widetilde{\mathcal{D}}_{-,\vec{\jmath}}^{-,\vec{\jmath}}(0) :=
{\mathcal{D}}_{-,\vec{\jmath}}^{-,\vec{\jmath}}(0) = \lambda_{0,\pm}  \, , 
\]
where $\lambda_{j,\pm}$ are defined as in \eqref{calOhomoeq2}.  
Then   equation \eqref{omoeq2} is equivalent to 
\begin{equation}\label{eqhodopou}
-\ii\omega\cdot\ell \widetilde{\mathcal{S}}_{\s,\eta j}^{\s', \eta' k}(\ell)
+\Big(\ii \s \lambda_{j,\eta}-\ii\s' \lambda_{k,\eta'}\Big)
\widetilde{\mathcal{S}}_{\s,\eta j}^{\s', \eta' k}(\ell)
=
\ii \s \widetilde{\mathcal{P}}_{\s,\eta j}^{\s', \eta' k}(\ell)\, , \quad  \forall \eta, \eta' \in \{\pm\} \, .
\end{equation}
Since
$ \widetilde{\mathcal{P}}_{\s,\eta j}^{\s', \eta' k}(\ell) = 0 $ if $ \eta \eta' = - 1 $,
 %$ \widetilde \cP_{\s,\vec{\jmath}}^{\s',\vec{k}}$ is diagonal) that 
 the solution of \eqref{eqhodopou} is
\begin{equation}\label{ipogeo}
\widetilde{\mathcal{S}}_{\s,\eta j}^{\s', \eta k}(\ell)=
\frac{ \s \widetilde{\mathcal{P}}_{\s,\eta j}^{\s', \eta k}(\ell)}{-\omega\cdot\ell
+ \s \lambda_{j,\eta}-\s' \lambda_{k,\eta}}\,, \quad  \widetilde{\mathcal{S}}_{\s,\eta j}^{\s', -\eta k}(\ell)=0 \, , 
\end{equation}
for any $ | \ell | \leq N $, and, 
 if $\s\neq\s'$   for any $ j,k\in\N_{0}$, 
or, if $\s=\s'$,  for any  $(\ell,j,k)\neq(0,j,j)$.

 By construction one has the equivalence 
\begin{equation}\label{ciprovo}
	\|\widetilde{\mathcal{P}}_{\s,\vec{\jmath}}^{\s', \vec{k}}(\ell)\|_{\infty} 
	% = \| \mathfrak{U}^{-1}   {\mathcal{P}}_{\s,\vec{\jmath}}^{\s', \vec{k}}(\ell)
	%\mathfrak{U}  \|_{\infty} 
	 \sim \|{\mathcal{P}}_{\s,\vec{\jmath}}^{\s', \vec{k}}(\ell)\|_{\infty}   
\end{equation}
where
$\|\cdot\|_{\infty}$ is the sup operator norm on (at most)  $2\times2$ matrices.

Setting for convenience 
\begin{equation}\label{ipogeo3}
	\psi_{\ell,j,k}^{\s,\s',\eta}:=\psi_{\ell,j,k}^{\s,\s',\eta}(\omega):=- \omega\cdot\ell
	+ \s \lambda_{j,\eta}(\omega)-\s' \lambda_{k,\eta}(\omega)\,,
\end{equation}
we note that  $\psi_{\ell,j,k}^{\s,\s',\eta}$ is real valued and
 $	\psi_{\ell,j,k}^{\s,\s',\eta} = - \psi_{-\ell,j,k}^{-\s,-\s',\eta}$.
 By construction $\widetilde{\mathcal{S}}_{\s,\vec \jmath}^{\s', \vec{k}}(\ell)$ is diagonal 
 and parity preserving
 (see \eqref{algebra} and \eqref{algebratilde}).
Since $\widetilde{\cP}$ 
is reversibility preserving and real-to-real, we have that
\begin{equation}
	\label{realo}
	\widetilde{\mathcal{S}}_{\s,\vec \jmath}^{\s', \vec{k}}(\ell) 
	= \widetilde{\mathcal{S}}_{-\s,\vec \jmath}^{-\s', \vec{k}}(-\ell)
\end{equation}
which means, recalling \eqref{bolena1quatuor}-\eqref{bolena3quatuor}, that $\widetilde{\mathcal{S}}$ is real-to-real and reversibility preserving.
 It follows  that $\cS$ is real-to-real, reversibility and parity preserving.

By \eqref{ipogeo}, \eqref{ciprovo}, for any $\omega\in\calO_{+}$ defined in  \eqref{calOhomoeq1}, we have 
\begin{equation}
	\label{stimaS}
\|{\mathcal{S}}_{\s,\vec{\jmath}}^{\s', \vec{k}}(\ell)\|_{\infty}\lesssim \gamma^{-3/2}N^{\tau}
\|{\mathcal{P}}_{\s,\vec{\jmath}}^{\s', \vec{k}}(\ell)\|_{\infty}\,,
\end{equation}
uniformly in $\omega\in \mathcal{O}_{+}$.
To estimate the Lipschitz variation we  note that, by \eqref{ipogeo3}, \eqref{calOhomoeq2} 
\begin{equation}
	\label{fame}
\begin{aligned}
	\Big|\frac{1}{\psi_{\ell,j,k}^{\s,\s',\eta}(\omega_1)}&-\frac{1}{\psi_{\ell,j,k}^{\s,\s',\eta}(\omega_2)}\Big| = \Big|\frac{\psi_{\ell,j,k}^{\s,\s',\eta}(\omega_1)-\psi_{\ell,j,k}^{\s,\s',\eta}(\omega_2)}{\psi_{\ell,j,k}^{\s,\s',\eta}(\omega_1)\psi_{\ell,j,k}^{\s,\s',\eta}(\omega_2)}\Big|
	\\&\lesssim
	\frac{|\omega_{1}-\omega_2|}{|\psi_{\ell,j,k}^{\s,\s',\eta}(\omega_1)\psi_{\ell,j,k}^{\s,\s',\eta}(\omega_2)|}\Big(|\ell|+
	\sup_{\omega_1\ne \omega_2}\frac{|\fc(\omega_1)
		-\fc(\omega_2)|}{|\omega_1-\omega_2|}
	|\s\mathtt{D}_{\mathtt{m}}(j)-\s'\mathtt{D}_{\mathtt{m}}(k)|
	\\&
	+
	\sup_{\substack{j\in\N_{0},\eta\in \{\pm\}\\ \omega_1\neq\omega_2}}\frac{|\mathtt{r}_{j}^{\eta j}(\omega_1)
		-\mathtt{r}_{j}^{\eta j}(\omega_2)|}{|\omega_1-\omega_2|}\Big)
	\\&
	\stackrel{\eqref{stimamasterdeltino}, \eqref{kamass1}}{\lesssim}|\omega_{1}-\omega_2|\frac{ |\ell|+\langle\s|j|-\s'|k|\rangle}{|\psi_{\ell,j,k}^{\s,\s',\eta}(\omega_1)\psi_{\ell,j,k}^{\s,\s',\eta}(\omega_2)|}
	\lesssim
	 \gamma^{-2 - \frac{\s+\s'}{2}}N^{2\tau +1}|\omega_1-\omega_2|\,.
	%\\&\stackrel{\eqref{calOhomoeq1}}{\lesssim}
\end{aligned}
\end{equation}
To prove the last bound, we distinguish two cases:  
\\
If  $\langle\s|j|-\s'|k|\rangle \le 4 |\omega| |\ell|\lesssim N$ then the 
 bound follows by \eqref{calOhomoeq1}.
If  $\langle\s|j|-\s'|k|\rangle > 4 |\omega| |\ell|,$ using 
 \eqref{calOhomoeq2}, \eqref{kamass1}, \eqref{stimamasterdeltino},
 \eqref{TDM} and  the smallness condition \eqref{PiccolezzaperKamredDP},
 we deduce  that 
 $ 	|\psi_{\ell,j,k}^{\s,\s',\eta}(\omega)|\gtrsim  \langle\s|j|-\s'|k|\rangle $. 
Therefore, recalling \eqref{ipogeo}, the fact that the conjugation matrix $\fU$ does not depend on $\oo$,
using the smallness condition 
in \eqref{PiccolezzaperKamredDP},
and the bounds \eqref{calOhomoeq1}, we get, 
for any $\omega_1\neq\omega_2$ in $\mathcal{O}_{+}$,
\[
\begin{aligned}
\|{\mathcal{S}}_{\s,\vec{\jmath}}^{\s', \vec{k}}(\omega_1;\ell)
-{\mathcal{S}}_{\s,\vec{\jmath}}^{\s', \vec{k}}(\omega_2;\ell)\|_{\infty}
&\lesssim
\gamma^{-\frac{3}{2}}N^{\tau}
\|{\mathcal{P}}_{\s,\vec{\jmath}}^{\s', \vec{k}}(\omega_1;\ell)
-{\mathcal{P}}_{\s,\vec{\jmath}}^{\s', \vec{k}}(\omega_2;\ell)\|_{\infty}
\\&+
 \gamma^{-3}N^{2\tau+1}\|{\mathcal{P}}_{\s,\vec{\jmath}}^{\s', \vec{k}}(\omega_1;\ell)\|_{\infty}
|\omega_1-\omega_2|\,.
\end{aligned}
\]
Recalling the notation in Def. \ref{majOp}, \eqref{rmk:propmajNorm}
and passing to the corresponding operators $\widecheck{\mathcal{S}},\widecheck{\mathcal{P}}$ 
(according to Lemma \ref{rmk:maggiorantevsSupblocchi}),
we have that the bound above 
imply that  
\[
\begin{aligned}
\widecheck{\mathcal{S}}(\omega)&\preceq\gamma^{3/2} C N^{\tau}\widecheck{\mathcal{P}}(\omega)\,, \quad 
\widecheck{ \Delta_{12} \mathcal{S}}
&\preceq
\gamma^{3/2} C N^{\tau}\big( \widecheck{\Delta_{12}  \mathcal{P}} \big)
+ \gamma^{-3} CN^{2\tau+1}\widecheck{\mathcal{P}}(\omega_1) \,,
%&\preceq
%\gamma^{3/2} C N^{\tau}\big(\rwhat{\mathcal{P}(\omega_1;\vphi)-
%\mathcal{P}(\omega_2;\vphi)}\big)
%+ \gamma^{-3} CN^{2\tau+1}\widehat{\mathcal{P}}(\omega_1;\vphi)
%|\omega_1-\omega_2|\,,
\end{aligned}
\]
for some pure constant $C>0$.
The same relations as above hold verbatim for 
$\langle \td_{\vphi}\rangle^{q}\widecheck{\mathcal{S}}(\omega)\langle D\rangle$, $q=0,\tb$,
recalling that
$\jap{D}, \langle \td_{\vphi}\rangle$ commute with $\cD$ .

In order to prove the bounds \eqref{solenapoli} we recall that 
$\cP= \fS(\tP)= M_\tP + R_{\tP}$ , for some couple
$(M_{\mathtt P}, R_{\mathtt P}) \in E_s\otimes \cM_2(\C)$. 
Without loss of generality (using Lemma \ref{lemma:generalitaparity})
we can assume that both $M_\tP$, $ R_{\tP}$ are real-to-real, reversibility and parity preserving.
Then $\cS= M_{\mathtt S} + R_{\mathtt S}$ where  
$M_{\mathtt S}$ (resp. $R_{\mathtt S}$) solves the homological equation \eqref{omoeq1} with $\cP$ replaced by $M_{\tP}$ (resp. $R_{\mathtt P}$), see formula \eqref{ipogeo},\eqref{stimaS}. Thus 
\[
\begin{aligned}
	(\widecheck{M}_\mathtt S)_{\s}^{\s'}(\omega)&\preceq\gamma^{3/2} 
	C N^{\tau}(\widecheck{M}_{\tP})_{\s}^{\s'}(\omega)\,,
\, \quad 
	(\widecheck{R}_\mathtt S)_{\s}^{\s'}(\omega)&\preceq\gamma^{3/2} 
	CN^{\tau}(\widecheck{R}_{\tP})_{\s}^{\s'}(\omega)\,,
\end{aligned}
\]
same if we apply the Lipschitz variation in $\omega$ or the operators $\langle \td_{\vphi}\rangle^{\tb}$ and $\langle D\rangle$ (note that the action of these operators is proportional to the identity on each block matrix ${\mathcal{S}}_{\s,\vec{\jmath}}^{\s', \vec{k}}(\ell)$). 
By %Definition  \ref{Es}, 
the bounds above and recalling 
Lemma \ref{rmk:maggiorantevsSupblocchi} we have that 
$\mathtt S:=(M_{\mathtt S}, R_{\mathtt S}) \in E_s\otimes \cM_2(\C)$ and satisfies  \eqref{solenapoli}.
\end{proof}

We now study how the operator $\mathcal{M}$ changes under conjugation through 
the map $\exp(\mathcal{S})$.
\begin{lemma}{\bf (KAM step).}\label{lem:kamstep}
Consider the operator $\mathcal{M}$ in \eqref{Msteppino}
satisfying \eqref{kamass1} and \eqref{stimePPstepBIS}.
There is $C>1$ such that, 
if 
\begin{equation}\label{smallcondKAMstep}
C N^{2\tau+1}\eps(\gso)
\leq 1\,,
\end{equation}
the following holds.
 For any $\omega\in \calO_+$ defined in \eqref{calOhomoeq1}
there exists an invertible operator $\mathtt{\Phi}\in E_{s}\otimes \cM_2(\C)$
such that $ \fS(\mathtt{\Phi}) = \Phi := \exp(\mathcal{S}) $ where $\mathcal{S}$ is the operator 
in $\mathcal{L}^{\mathtt{T}}(H^{s},H^{s})\otimes \cM_2(\C)$ 
defined  in Lemma \ref{lem:homoeq}, 
satisfying 
\begin{equation}\label{stimamappaPhi+}
\begin{aligned}
\bnorm{ \mathtt{\Phi}^{\pm1} - \id }^{\gamma^{3/2}, \calO_{+}}_{s}&\lesssim
N^{2\tau+1}\eps(s)\,,\qquad \forall \, \gso\leq s\leq \su\,,
\\
\bnorm{ \langle\td_{\vphi}\rangle^{b} \mathtt{\Phi}^{\pm1} - \id
}^{\g^{3/2} ,\cO_{+}}_{s}
&\lesssim N^{2\tau+1+b}\eps(s)\,,\quad \forall b>\nu/2\,,
\end{aligned}
\end{equation}
and 
\begin{equation}\label{Msteppinoplus}
\mathcal{M}_{+}:=\Phi \circ \mathcal{M}\circ\Phi^{-1}
=\omega\cdot\pa_{\vphi}-\ii E \mathcal{D}_{+}-\ii E\mathcal{P}_{+}\,,
\end{equation}
where $\mathcal{D}_{+}=[\mathcal{D}_{+}]$ is in normal form and 
(recall \eqref{pioggiamilano1}-\eqref{pioggiamilano2}
and \eqref{D00D})
\begin{equation}\label{def:Dold+}
(\mathcal{D}_{+})_{+,\vec{\jmath}}^{+,\vec{\jmath}}=
d_{\vec{\jmath}}^{(0)}+
\tr^{+}_{\vec{\jmath}}
\qquad
\mathtt{r}^{+}_{\vec{\jmath}}:= (\mathtt{r}_{+})_{j}^{j}\Id + (\mathtt{r}_{+})_{j}^{-j}S 
\,,\quad j\in \mathbb{N}\,,\quad \mathtt{r}^{+}_{\vec{0}}:=  (\mathtt{r}_{+})_{0}^{0} \,,
\end{equation}
for some 
some Lipschitz functions $(\tr^{+})_{j}^{\eta  j} : \Lambda\to \R$, $\eta \in \{\pm\}$,
\begin{equation}\label{kamass1plus}
\sup_{j\in\Z,\eta\in \{\pm\}}\langle j\rangle |(\tr_{+})_{j}^{\eta j}-\tr_{j}^{\eta j}|^{\gamma^{3/2},\Lambda}
\lesssim \gamma^{3/2}\eps(\gso)\,.
\end{equation}
The operator $\mathcal{P}_+$ is defined in $\calO_+$,
it is  real-to-real, reversibility  and parity preserving, and
there exists $\tP_{+}\in E_s\otimes \cM_2(\C)$ such that $\fS(\tP_{+})=\cP_{+}$ 
and 
\[
\begin{aligned}
\eps_{+}(s)&:=\gamma^{-3/2}
\bnorm{  \mathtt{P}_{+} 
\jap{D}}^{\g^{3/2} ,\cO_{+}}_{s}\,, \qquad 
\eps_{+}(s,\tb) :=\gamma^{-3/2}
\bnorm{\jap{\td_{\vphi}}^\tb \tP_{+} \jap{D}}^{\gamma^{3/2} ,\cO_{+}}_{s}\,,
\end{aligned}
\]
satisfy, for some  constant  $\tc := \tc (\tau,\nu) $,   
 % (recall \eqref{stimePPstepBIS})
\begin{equation}\label{stimePPstepPLUS}
\begin{aligned}
\eps_{+}(s)&\leq
\tc N^{-\tb}\eps(s,\tb)+
\tc N^{2\tau+1}\eps(s)\eps(\gso)\,,
\\
\eps_{+}(s,\tb)
&\leq
\eps(s,\tb)\big(1+\tc N^{2\tau+1}\eps(\gso)\big)
+\tc N^{2\tau+1}\eps(\gso,\tb)\eps(s)\, . 
\end{aligned}
\end{equation} 
\end{lemma}

\begin{proof}
%Let
%$\mathcal{S}\in\mathcal{L}^{\mathtt{T}}(H^{s},H^{s})\otimes \cM_2(\C)$ be given in  Lemma \ref{lem:homoeq} 
%and consider  the exponential  map $\Phi=\exp(\mathcal{S})$.
Recalling that $\cS=\fS(\mathtt{S})$, and setting $\mathtt{\Phi}=e^{\mathtt{S}}$, then $\Phi=\fS(\mathtt{\Phi})$. 
Of course $\mathtt{\Phi^{-1}}=e^{-\mathtt{S}}$ and $\Phi^{-1}=\fS(\mathtt{\Phi}^{-1})$ by the homomorphism property.
By estimates \eqref{solenapoli},
the smallness condition \eqref{smallcondKAMstep}
and Lemma \ref{inv3sbarrette} we deduce the estimates in \eqref{stimamappaPhi+}. 
%Analogously, we deduce the bound for the inverse by Lemma \ref{inv3sbarrette}.
Moreover, since $\mathcal{S}$  is real-to-real, reversibility and parity preserving
so  is $\Phi$ and  the conjugate operator 
$\mathcal{M}_{+}$ in \eqref{Msteppinoplus}
is real-to-real, reversible and parity preserving.

By classical Lie expansion series we get 
\begin{align}
\Phi\circ\mathcal{M}\circ\Phi^{-1}&=
\omega\cdot\pa_{\vphi}-\ii E\cD\nonumber
\\&
-\omega\cdot\pa_{\vphi}\cS+\big[\cS,-\ii E\cD\big]-\ii E\cP
\label{omoomoeq}
\\&
+\big[\cS,-\ii E\cP\big]
+\sum_{k\geq2}\frac{1}{k!}{\rm ad}_{\cS}^{k}\big(-\ii E\cD-\ii E\cP\big)
-\sum_{k\geq2}\frac{1}{k!}{\rm ad}_{\cS}^{k-1}
\big( \omega\cdot\pa_{\vphi}\cS\big)
\label{restoresto}\,.
\end{align}
Recalling Definition \ref{def:normalform}, \eqref{def:proj}
and since $ \cS $ solves  the homological equation 
$ -\omega\cdot\pa_{\vphi}\cS+\big[\cS,-\ii E\cD\big]=\ii E\big(\Pi_{N}\cP-[\cP]\big) $  
we have that
$ \eqref{omoomoeq} 
%-\omega\cdot\pa_{\vphi}\cS+\big[\cS,-\ii E\cD\big]-\ii E\big(\Pi_{N}\cP-[\cP]\big)
%-\ii E[\cP]-\ii E\Pi_{N}^{\perp}\cP
=-\ii E[\cP]-\ii E\Pi_{N}^{\perp}\cP $.

By  Lemma \ref{rmk:algebra2}, 
since $\mathcal{P}$ is real-to-real, reversibility and parity preserving then 
\begin{equation}
	\label{nuovadiag}
\mathcal{P}_{+,\vec{\jmath}}^{+,\vec{\jmath}}(0) =  \mathcal{P}_{-,\vec{\jmath}}^{-,\vec{\jmath}}(0) =\mathcal{\cP}_{+,j}^{+, j}(0)\Id + \mathcal{\cP}_{+,j}^{+, - j}(0)S\,,\quad \mathcal{\cP}_{+,j}^{+, \eta j}(0)\in \R\,.
\end{equation}
By Kirszbraun theorem we extend  the coefficients 
$ \mathcal{\cP}_{+,j}^{+, \eta j}(0) $ to Lipschitz functions
$ \breve{\cP}_{+,j}^{+,\eta j}(0) : 
\Lambda\to \R $
defined on the whole set $ \Lambda $, satisfying the Lipschitz bound
\begin{equation}\label{pioggia32}
|\breve{\mathcal{\cP}}_{+,j}^{+,\eta j}(0)|^{\gamma,\Lambda}
\leq |{\mathcal{\cP}}_{+,j}^{+,\eta j}(0)|^{\gamma,\calO} \, , \quad \forall  j\in \N_{0},\eta \in \{\pm\} \, . 
\end{equation}
Consequently we define the new normal form 
$$
\mathcal{D}_{+}:=\cD+{\rm diag}_{\s\in \{\pm\}, j\in \N_{0}}
\breve{\mathcal{P}}_{\s,\vec{\jmath}}^{\s,\vec{\jmath}}(0)\,,\quad \breve{\mathcal{P}}_{\s,\vec{\jmath}}^{\s,\vec{\jmath}}(0)=\breve{\cP}_{+,j}^{+, j}(0)\Id + \breve{\cP}_{+,,j}^{+, - j}(0)S \,,\quad \s\in \{\pm\}\,.
$$
The  estimate \eqref{kamass1plus} 
for $ \tr^{+}_{\vec{\jmath}} = \tr_{\vec{\jmath}} + \breve{\mathcal{P}}_{+,\vec{\jmath}}^{+,\vec{\jmath}}(0) $
follows by 
\eqref{pioggia32} and \eqref{stimePPstepBIS}.

We now consider the last term \eqref{restoresto}. % We obtain \eqref{Msteppinoplus}. 
Now recalling that $\mathcal{S}=\fS(\mathtt{S}), \mathcal{P}=\fS(\mathtt{P})$, and 
that $ \cS  $ solves the homological equation \eqref{omoeq1}, 
we set 
\begin{equation}\label{gattamora}
\begin{aligned}
-\ii E\mathtt{P}_{+}&:=-\ii E \Pi_{N}^{\perp}\mathtt{P}+
\sum_{k\geq1}\frac{1}{k!}{\rm ad}_{\mathtt{S}}^{k}\big(-\ii E\mathtt{P}\big)
+\sum_{k\geq2}\frac{1}{k!}{\rm ad}_{\mathtt{S}}^{k-1}\big(
\ii E\big(\Pi_{N}\mathtt{P}-[\mathtt{P}]\big)
\big)\,,
\end{aligned}
\end{equation}
where $[\mathtt{P}]=[(M_{\mathtt{P}}, R_{\mathtt{P}})]:=([M_{\mathtt{P}}], [R_{\mathtt{P}}])$.
We then obtain  formula \eqref{Msteppinoplus} % Since $\fS$ is a linear homomorphism, 
with  $\mathcal{P}_{+}:=\fS(\mathtt{P}_{+}) $.
 Multiplying by $\jap{D}$ to the right  of \eqref{gattamora} 
 we obtain
estimates \eqref{stimePPstepPLUS}
by applying Lemmata \ref{dito}, \ref{dito2}, %\ref{conjugo}, 
formulas \eqref{solenapoli}, \eqref{stimePPstepBIS}
and the smallness condition \eqref{smallcondKAMstep}.
This concludes the proof.
\end{proof}

\subsection{The iterative scheme}
The  reducibility Theorem \ref{thm:reduKG}
is deduced by the next proposition.

\begin{prop}{\bf (Iterative reduction).}\label{IterativeKAM}
Let $ \tb $ defined 
in \eqref{def:bbb} 
and fix
\begin{equation}\label{def:aaa}
\ta:=6\tau+4\,.
\end{equation}
 There is  $N_0:=N_0(\tau,\nu)>0$ such that, 
 if
\begin{equation}\label{conditeration}
 N_0^{2\tau+2} \eps_0(\gso,\tb)\leq 1\,,
\end{equation}
then, setting $N_{-1}:=1$ and   $N_{n}:=N_0^{(3/2)^n}$ for any $ n \geq 1 $,  the following holds for any $n\ge 0$. 

$({\rm \textbf{S1}})_n$. There exists a sequence of operators
\begin{equation}\label{MsteppinoPasson}
\mathcal{M}_{n}
=\omega\cdot\pa_{\vphi}-\ii E \mathcal{D}_{n}-\ii E\mathcal{P}_{n} 
\end{equation}
where $\mathcal{D}_{n}$ is in normal form and 
\begin{equation}\label{FinalnormalformstepN}
\begin{aligned}
(\mathcal{D}_{n})_{-,\vec{\jmath}}^{-,\vec{\jmath}}=(\mathcal{D}_{n})_{+,\vec{\jmath}}^{+,\vec{\jmath}}&=
d_{\vec{\jmath}}^{(0)}+
\tr^{(n)}_{\vec{\jmath}}\,,
\qquad j\in\N_{0}\,,\\
&\mathtt{r}^{(n)}_{\vec{\jmath}}:= (\mathtt{r}_{n})_{j}^{j}  \Id + (\mathtt{r}_{n})_{j}^{-j} S 
\,,\quad j\in \mathbb{N}\,, \qquad\mathtt{r}^{(n)}_{\vec{0}}:=(\tr_{n})^{0}_{0}\,,
\end{aligned}
\end{equation}
with 
Lipschitz functions $(\tr_{n})_{j}^{\eta j} : \Lambda\to \R$, $\eta\in \{\pm\}$,
satisfying $\mathtt{r}^{(0)}_{\vec{\jmath}}\equiv0$ and for $n\geq1$,
\begin{align}\label{algalgstepnn}
& \qquad (\tr_{n})_{j}^{\eta j}=(\tr_{n})_{-j}^{-\eta j}\,, \\
& \label{kamass1stepn}
\sup_{j\in\N_{0},\eta\in \{\pm\}}\langle j\rangle |(\tr_{n})_{j}^{\eta j}-(\tr_{n-1})_{j}^{\eta j}|^{\gamma^{3/2},\Lambda}
\leq \gamma^{3/2}
\eps_{0}(\gso,\tb) N_{n-2}^{-\ta}\,.
\end{align}
The operators $\mathcal{P}_{n}$ are defined for any $ \omega $ 
in the sets $\cO_n$ defined as follows:  
$\cO_0:=\widehat{\calO}$, while for  $n\geq 1$,  
$\cO_n:=\Omega_{n}^{+}\cap\Omega_{n}^{-}$ 
where
\begin{equation}\label{omegan}
	\begin{aligned}
		\Omega_{n}^{+}:=\Big\{\omega\in \mathcal{O}_{n-1}\,:\,  
		|\omega\cdot\ell+\lambda^{(n-1)}_{j,\eta}+\lambda^{(n-1)}_{k,\eta}|\geq 
		\frac{\gamma}{\langle\ell\rangle^{\tau}}\,, \quad &\eta\in \{\pm\},
		\\&
		\;\; |\ell|\leq N_{n-1}, j,k\in\N_{0}
		\Big\}
		\\
		\Omega_{n}^{-}:=\Big\{\omega\in \mathcal{O}_{n-1}\,:\, 
		|\omega\cdot\ell+\lambda^{(n-1)}_{j,\eta}-\lambda^{(n-1)}_{k,\eta}|\geq 
		\frac{\gamma^{3/2}}{\langle\ell\rangle^{\tau}}\,,\quad &\eta\in \{\pm\},
		\\|\ell|\leq N_{n-1},\quad  j,k\in\N_{0}\,,\; \ell\in &\Z^{\nu}\,,
		\;\;(\ell, j,k)\neq(0, j, j)
		\Big\}
	\end{aligned}
\end{equation}
and 
\begin{equation}\label{lambdinostepn}
\begin{aligned}
&\lambda^{(n)}_{j,\pm}:=(1 + \fc) \tD_\mathtt{m}(j)+(\tr_n)_{j}^{j} \pm (\tr_{n})_{j}^{-j}\,,
\;\, j\in\N\,,\quad  
\lambda^{(n)}_{0,\pm}:=(1 + \fc) \sqrt{\tm}+(\tr_n)_{0}^{0} \,.
\end{aligned}
\end{equation}
Moreover, 
the operator $\cP_n$ is  real-to-real, reversibility  and parity preserving and
there exists $\tP_n\in E_s\otimes\cM_2(\C)$ such that $\fS(\tP_n)=\cP_n$ and 
\begin{equation}\label{stimePn}
\begin{aligned}
\eps_{n}(s):=
\g^{-3/2}
\bnorm{ \tP_n \jap{D}}^{\g^{3/2} ,\cO_{n}}_{s}
&\le \eps_0(s,\tb) N_{n-1}^{-\ta}\,, 
\\
\eps_{n}(s,\tb):=
\g^{-3/2} 
\bnorm{\jap{\td_{\vphi}}^\tb  \tP_n \jap{D}}^{\g^{3/2} ,\cO_{n}}_{s}
&\le\eps_0(s, \tb)N_{n-1}\,.
\end{aligned}
\end{equation}

$({\rm \textbf{S2}})_n$ 
For any $n\geq 1$ and  any $\oo\in \cO_n$ there exists an invertible map 
$\Phi_{n-1}\in \mathcal{L}^{\mathtt{T}}(H^s\times H^s)\otimes\mathcal{M}_2(\C) $ 
such that
\begin{equation}\label{favoledigloria}
\mathcal{M}_{n}=\Phi_{n-1}\circ\mathcal{M}_{n-1}\circ\Phi_{n-1}^{-1}\,.
\end{equation}
 Moreover the maps $\Phi_{n-1}, \Phi_{n-1}^{-1}$ are real-to-real, reversibility and parity preserving.
 Finally, there exists an invertible $\mathtt{\Phi}_{n-1}\in E_{s}\otimes \cM_2(\C)$ such that $\fS(\mathtt{\Phi}_{n-1})= \Phi_{n-1}$  
 satisfying the bound
 \begin{equation}\label{stimacciasuphiN}
\begin{aligned}
 \bnorm{\mathtt{\Phi}^{\pm 1}_{n-1}-\id}_s^{\g^{3/2}, \cO_{n}} &\leq 
N_{n-1}^{2\tau+2}N_{n-2}^{-\ta} \eps_0(s,\tb) \\
 \bnorm{ \langle \mathtt d_\varphi \rangle^b  \mathtt{\Phi}^{\pm 1}_{n-1}-\id}_s
 &\leq 
N_{n-1}^{2\tau+2+b}N_{n-2}^{-\ta} \eps_0(s,\tb) \, , \ \forall\,
 b > \tfrac{\nu}{2} \, . 
\end{aligned}
\end{equation} 
$({\rm \textbf{S3}})_n$ Defining for any $\oo\in\cO_n$, $n\geq1$
\begin{equation}\label{trasformazionestepN}
U_n:= \Phi_0\circ\Phi_1\circ\dots \circ\Phi_{n-1}	\,,
\end{equation}
one has that there is  $\tU_n\in E_s \otimes\cM_2(\C)$ with $\fS(\tU_n)=U_n$ such that, setting $\tU_{0}=\id$,
\begin{equation}\label{stimacompositionn}
\bnorm{\tU_n - \tU_{n-1}}_s^{\g^{3/2}, \cO_{n}}, \bnorm{\tU^{-1}_n - \tU_{n-1}^{-1}}_s^{\g^{3/2}, \cO_{n}} 
\leq 2^{-n} \eps_0(s,\tb)\,.
\end{equation}
\end{prop}

\begin{proof}
We reason by induction on $n\geq0$.

\noindent
{\bf Inizialitazion.}
$({\rm \textbf{S1}})_0$. The operator $\mathcal{M}_0$ in \eqref{emme0}
has the form \eqref{MsteppinoPasson} and satisfies the bounds \eqref{stimePn} with $n=0$.
$({\rm \textbf{S2}})_0$ and $({\rm \textbf{S3}})_0$ are empty.

\noindent
{\bf Inductive step.}
Assume that for $n\geq0$ statements $({\rm \textbf{Sj}})_n$, $  {\bf j}=1,2,3$ hold true. 

\vspace{0.3em}
\noindent
\emph{Proof of} $({\rm \textbf{S1}})_{n+1}$, $({\rm \textbf{S2}})_{n+1}$.
We apply Lemmata \ref{lem:homoeq}-\ref{lem:kamstep}
to the operator $\mathcal{M}_{n}$
by setting $\cD\rightsquigarrow\cD_n $ and $\cP\rightsquigarrow\cP_n $ in \eqref{MsteppinoPasson}, 
$\calO_{+}\rightsquigarrow \calO_{n+1}$ in \eqref{omegan},
$N\rightsquigarrow N_{n}$. 
We start by verifying the smallness hypothesis \eqref{smallcondKAMstep}. If $n=0$ this follows from  \eqref{conditeration} provided that $N_0$ is sufficiently large. Otherwise if $n\ge 1$, by the inductive estimate \eqref{stimePn} and the definition of $\ta$ in \eqref{def:aaa}
\[
\begin{aligned}
 N_{n}^{2\tau+1}
\eps_{n}(\gso)
\leq 
 N_{n}^{2\tau+1}
\eps_0(\gso, \tb)N_{n-1}^{-\ta}=
\eps_0(\gso, \tb) N_{n}^{2\tau+1-\tfrac{2}{3}\ta}\leq \eps_0(\gso, \tb) N_{0}^{-2\tau-\tfrac{5}{3}}\,,
\end{aligned}
\]
 so again \eqref{smallcondKAMstep} follows by 
 taking $N_0$ in \eqref{conditeration} large enough.
Then we set
$\Phi_{n}:=\exp\{\cS_{n}\}$ with $\cS_n$ the solution of the equation
\eqref{omoeq1}.
Notice that by \eqref{solenapoli}, \eqref{stimePPstepBIS} and the inductive estimates 
\eqref{stimePn} we deduce that there exists $\mathtt S_n\in E_{s}\otimes\cM_2(\C)$ such that $\cS_n= \fS(\mathtt S_n)$ with
\begin{equation*}%\label{vesuvio1}
\begin{aligned}
\bnorm{  \mathtt{S}_n
\jap{D}}^{\g^{3/2} ,\cO_{n+1}}_{s}
&
\leq \tC N_{n}^{2\tau+1} N_{n-1}^{-\ta}\eps_0(s,\tb)\,,
%\\
%\bnorm{ \langle\td_{\vphi}\rangle^{\tb}  \mathtt{S}_{n}
%\jap{D}}^{\g^{3/2} ,\cO_{n+1}}_{s}
%&
%\leq \tC N_{n}^{2\tau+1}N_{n-1}\eps_0(s, \tb)\,,
\end{aligned}
\end{equation*}
for some constant $\tC:=\tC(\tau, \nu)$. 
By Lemma \ref{lem:kamstep}, there exists $\mathtt{\Phi}_n\in E_{s}\otimes\cM_2(\C)$ such that 
$\Phi_n= \fS(\mathtt{\Phi}_n)$. 
Moreover, 
the inductive hypothesis
%by the latter estimates 
together with  \eqref{stimamappaPhi+} %and \eqref{vesuvio1} 
implies \eqref{stimacciasuphiN} with $n-1\rightsquigarrow n$.
%one has that $\mathtt{\Phi}_n$ satisfies
%\begin{equation*}%\label{stimacciasuphiNbis}
%\bnorm{\mathtt{\Phi}^{\pm1}_{n}-\id}_s^{\g^{3/2}, \cO_{n+1}} \leq 2\tC 
%N_{n}^{2\tau+1}N_{n-1}^{-\ta} \eps_0(s,\tb)\,.
%\end{equation*}
%The bound \eqref{stimacciasuphiN}
%follows by the estimate above %\eqref{stimacciasuphiNbis}
%and taking $N_0$ in \eqref{conditeration} large enough and since
%$ \mathtt S_n $ is Fourier supported on $ |\ell | \leq N_{n-1} $.
Lemma \ref{lem:kamstep} implies also that the conjugate 
operator (see \eqref{Msteppinoplus})
$ \mathcal{M}_{n+1}:=\Phi_{n}\circ\mathcal{M}_n\circ\Phi_{n}^{-1} $
has the form \eqref{MsteppinoPasson} with $n\rightsquigarrow n+1$.
The conditions \eqref{algalgstepnn}-\eqref{kamass1stepn}  with $n\rightsquigarrow n+1$
follow by 
\eqref{kamass1plus}
and the inductive hypothesis $({\rm \textbf{S1}})_n$.
It remains to prove estimates \eqref{stimePn} with $n\rightsquigarrow n+1$.
By the first bound in 
\eqref{stimePPstepPLUS} and $({\rm \textbf{S1}})_n$ we get
\[
\begin{aligned}
\eps_{n+1}(s)&\leq
\tc  N_{n}^{-\tb}\eps_{n}(s,\tb)+
\tc  N_{n}^{2\tau+1}\eps_{n}(s)\eps_{n}(\gso)
\\&\leq 
\tc  N_{n}^{-\tb}N_{n-1}\eps_0(s,\tb) 
+\tc  N_{n}^{2\tau+1}N_{n-1}^{-2\ta}
\eps_0(s,\tb)\eps_{0}(\gso,\tb)
\leq \eps_0(s,\tb)N_{n}^{-\ta}\,,
\end{aligned}
\]
since  $\tb > \ta+1$, 
$2\tau+1-\tfrac{1}{3}\ta < 0$ 
 (which holds thanks to \eqref{def:bbb}, \eqref{def:aaa}),  provided that
$N_0$ in \eqref{conditeration} is large enough.
%and using also Remark \ref{allerta}.
This implies the first bound in \eqref{stimePn}.
By the second estimate in \eqref{stimePPstepPLUS} and $({\rm \textbf{S1}})_n$ 
we get 
\[
\begin{aligned}
\eps_{n+1}(s,\tb)
&\leq
\eps_n(s,\tb)\big(1+\tc  N_{n}^{2\tau+1}\eps_n(\gso)\big)
+\tc  N_{n}^{2\tau+1}\eps_n(\gso,\tb)\eps_n(s)
\\&\leq
\eps_0(s, \tb)N_{n-1}(1+\tc N_{n}^{2\tau+1}N_{n-1}^{-\ta}\eps_0(\gso,\tb))
\\&\qquad+\eps_0(s, \tb)\eps_0(\gso, \tb)\tc  N_{n}^{2\tau+1}N^{1-\ta}_{n-1}
\leq\eps_0(s, \tb)N_{n} 
\end{aligned}
\]
where in the last inequality, we used \eqref{conditeration} 
with $N_0$ large enough,
and  $2\tau+ \frac23 -\tfrac{2}{3}\ta < 0$.
This proves  \eqref{stimePn} at the step $n+1$.
Finally the operator $\mathcal{P}_{n+1}$
is real-to-real, reversibility  and parity preserving by Lemma \ref{lem:kamstep}
and the inductive assumption.

\vspace{0.3em}
\noindent
\emph{Proof of} $({\rm \textbf{S3}})_{n+1}$.
 We define $\tU_{n+1}:=\tU_{n}\circ \mathtt{\Phi}_{n}$,
%We write $U_{n+1}-U_{n}=U_{n}\circ(\Phi_{n}-\id)$
where $\Phi_{n}=\fS(\mathtt{\Phi}_{n})$ is given by $({\rm \textbf{S2}})_{n+1}$.Then estimate \eqref{stimacompositionn} with $n\rightsquigarrow n+1$
follows by \eqref{stimacciasuphiN} with $n\rightsquigarrow n+1$
and the inductive assumption.
The same holds for the inverse $\tU_{n+1}^{-1}$.
\end{proof}

\begin{proof}[{\bf Proof of Theorem \ref{thm:reduKG}}]
Consider the operator $\mathcal{M}_0$ in \eqref{emme0}.
Fix $N_{0}$ as in Proposition \ref{IterativeKAM}. Choosing $C_0$ large and $\upsilon_{0} $ small enough, we have 
 that \eqref{PiccolezzaperKamredDP} 
 and \eqref{def:epsilon0}
imply the smallness condition \eqref{conditeration}.
Hence Proposition \ref{IterativeKAM} applies.
Let us now define
\begin{equation}\label{limitrinfty}
	(\mathtt{r}_{\infty})_{j}^{\eta j}:=\lim_{n\to\infty}(\tr_{n})_{j}^{\eta j}\,,\qquad \forall\, \eta\in \{\pm\}\,,\; j\in\N_{0}\,.
\end{equation}
First of all we claim that 
\begin{equation}\label{claimIterat}
\mathcal{O}_{\infty}\subseteq \cap_{n\geq0}\mathcal{O}_{n}\,,
\end{equation}
where $\mathcal{O}_{\infty}$ is defined in \eqref{calOinfty2sec7}
and $\mathcal{O}_{n}$ is in  \eqref{omegan}.
If \eqref{claimIterat} holds then  for any $\omega \in \mathcal{O}_{\infty}$ 
we define the final transformation $\Phi_{\infty}$
as 
\[
\Phi_{\infty} :=\lim_{n\to\infty}U_{n}\stackrel{\eqref{trasformazionestepN}}{=}
\lim_{n\to \infty}\Phi_0\circ\Phi_1\circ\dots \circ\Phi_{n-1}\,.
\]
The definition is well-posed since, by \eqref{stimacompositionn}, 
we have that $U_{n}=\fS( \tU_n)$, where $\tU_n$ is 
 a Cauchy sequence of couples in the norm 
$\bnorm{\cdot}_{s}^{\gamma^{3/2}, \calO_{\infty}}$ 
converging to $\mathtt\Phi_\infty\in E_s\otimes\cM_2(\C)$ such that $\Phi_{\infty}=\fS( \mathtt{\Phi}_{\infty})$.
In particular 
we have
$ \bnorm{\mathtt{\Phi}^{\pm 1}_{\infty} - \id}_s^{\g^{3/2}, \cO_{\infty}}\leq  2\eps_{0}(s,\tb) $, 
which implies 
estimate \eqref{grano}. The \eqref{granophi} follows by the second estimate in 
\eqref{stimacciasuphiN} and Remark \ref{rmk:sezphi}, 
taking $b=  \tfrac{\nu}{2} +1< \frac23 \ta - 2\tau-2 $.

The parity conditions and estimates in \eqref{stimeautovalfinali} holds true
by \eqref{algalgstepnn}-\eqref{kamass1stepn}.
Passing to the limit,
using \eqref{MsteppinoPasson}, \eqref{favoledigloria},
\eqref{FinalnormalformstepN}
and recalling the bounds \eqref{stimePn}
one gets the conjugation result
\eqref{Linfinito} with $\mathcal{D}_{\infty}$ 
as in \eqref{Finalnormalform}.

\vspace{0.3em}
\noindent
{\bf Proof of the claim} \eqref{claimIterat}.
We show by induction that if $\omega\in\calO_{\infty}$, defined  in \eqref{calOinfty2sec7}, then
$\omega\in \cO_{n}$  for  any $n \ge 0$. The basis of the induction is trivial 
since by definition we have 
$\calO_{\infty}\subseteq\calO_0:=\widehat{\calO}$.
Assume now, for $n\ge 1$,  that $\calO_{\infty}\subseteq \calO_{n-1}$. 
 We note that for any $\oo\in \calO_{\infty}$ and for any fixed $j,k\in \N_0$, $\ell\in \Z^{\nu}$, $\eta\in \{\pm\}$ one has
 (recall \eqref{lambdinostepn}, \eqref{autovalorifinali})
\[
\begin{aligned}
\omega\cdot\ell+\lambda^{(n-1)}_{j,\eta}&-\lambda^{(n-1)}_{k,\eta}=
\omega\cdot\ell+\lambda^{(\infty)}_{j,\eta}-\lambda^{(\infty)}_{k,\eta}
\\&+
(\tr_{n-1})_{j}^{j} -(\tr_{\infty})_{j}^{j} +\widehat\eta(j) \big((\tr_{n-1})_{j}^{-j}-(\tr_{\infty})_{j}^{-j} \big) 
\\&-(\tr_{n-1})_{k}^{k} +(\tr_{\infty})_{k}^{k} -\widehat\eta(k) \big((\tr_{n-1})_{k}^{-k}+(\tr_{\infty})_{k}^{-k} \big)
\end{aligned}
\]where $\widehat\eta(j)= \eta$ if $j\ne 0$ and $\widehat\eta(j)=0$ if $j=0$.
%\shu{shula: manca caso $j$ o $k$ nulli}
Moreover, by \eqref{kamass1stepn} and \eqref{limitrinfty}, we deduce
\[
\sup_{j\in\N_{0}, \s\in \{\pm\}}\langle j\rangle |(\tr_{\infty})_{j}^{\s j}-(\tr_{n-1})_{j}^{\s j}|^{\gamma^{3/2},\Lambda}
\leq \tC \eps(\gso,\tb) N_{n-2}^{-\ta}\,.
\]
Therefore, using that $\omega\in \mathcal{O}_{\infty}$ and taking 
$|\ell|\leq N_{n-1}$,
we get
\begin{equation}\label{terramare}
|\omega\cdot\ell+\lambda^{(n-1)}_{j,\eta}-\lambda^{(n-1)}_{k,\eta}|
\stackrel{\eqref{calOinfty2sec7}}{\geq}
2\gamma^{3/2}\langle\ell\rangle^{-\tau}-4\tC \eps(\gso,\tb) N_{n-2}^{-\ta}
\geq 
\gamma^{3/2}\langle\ell\rangle^{-\tau}\,,
\end{equation}
using the smallness condition \eqref{conditeration} and the definition of $\ta$.
The bound \eqref{terramare} together with the inductive hypothesis $\om\in \cO_\infty\subseteq \cO_{n-1}$ implies  that $\omega\in \Omega_{n}^{-}$
defined in \eqref{omegan}.
The proof that  $\omega\in \Omega_{n}^{+}$ is similar. 
Then $\omega\in \calO_{n}$ and   \eqref{claimIterat}  follows.
The proof of Theorem \ref{thm:reduKG} is complete.
\end{proof}

\section{Measure estimates and proof of Theorems \ref{main:thm} and \ref{thm:cantorset}}\label{sec:measure}

Let us fix $\tb$ as in \eqref{def:bbb}.
By the assumptions \eqref{oddness}-\eqref{revers},
 the operator $\mathcal{L}$ in \eqref{L-omega}
is real-to-real, reversible and parity preserving
according to Definition \ref{giornatasolare}.
The smallness condition \eqref{smallCondCoeff}
on the coefficients $a^{(i)}$ in \eqref{NLW}
guarantees 
that the smallness condition \eqref{ipopiccolezza} required in Proposition \ref{prop:ridord0}  
holds true.  
 Therefore  
Propositions \ref{diagonalizzazione order 1}, \ref{blockTotale}, \ref{prop:ridord1}, \ref{prop:ridord0}
apply and we conjugate $\cL$
in \eqref{L-omega} to $\mathcal{L}_5$ in \eqref{elle5}
which is defined in the set of parameters $\Omega_1$ in \eqref{calOinfty1sec6}.
The operator $\mathcal{L}_{5}$ % in \eqref{elle5}  % of Proposition \ref{prop:ridord0}
has the form \eqref{emme0}, \eqref{D00D}, where
$\mathfrak{c}$ is the one in  Proposition  \ref{prop:ridord1}, 
$\widehat{\mathcal{O}}\equiv\Omega_1$ given in \eqref{calOinfty1sec6}, 
and $\mathcal{P}_0=\mathcal{R}^{(5)}$. 
Moreover $\mathcal{R}^{(5)}$
is real-to-real, reversibility and parity preserving, and hence 
$\mathcal{P}_0$ is so. Our aim is now to apply Theorem \ref{thm:reduKG}.
By  \eqref{restosmooth1bis} 
we have that, for any $\gso\leq s\leq \su$ 
(recall \eqref{def:epsilon0}) %  \eqref{condireiniziale})
$$
\eps_0(s) \le \eps_0(s,\tb)=\gamma^{-3/2}\ \bnorm{\langle\td_{\vphi}\rangle^{\tb}
  \mathtt{P}_0\langle D\rangle}_{s}^{\gamma,\Omega_1}
  \lesssim_{\su}\gamma^{-7/2}\epsilon(s+\mu)\,.
$$
The smallness condition   \eqref{PiccolezzaperKamredDP} follows by the above 
estimate 
 by taking in \eqref{smallCondCoeff} the constant $\delta_{0} (s_1, \nu) $ small enough. 
Furthermore  \eqref{stimamasterdeltino} holds 
by \eqref{stimaMgotico1}.   
Then Theorem \ref{thm:reduKG} applies and  provides a map $\Phi_{\infty}$, defined on the set 
 $\calO_{\infty}$ in \eqref{calOinfty2sec7},
 which conjugates $\mathcal{L}_{5}$
 to the operator $\mathcal{M}_{\infty} = \omega \cdot \pa_\vphi 
 - \ii E \cD_{\infty} $ 
 in \eqref{Linfinito}. The operator $ \ii E \cD_{\infty} $
 has  the form \eqref{ellediag} with $\fD^{+}_{\infty} = (\cD_{\infty})_{+}^{+} $ 
 in \eqref{Dinfinito}, with 
$ \mathfrak{r}_{j}^{\s j}:=\langle j \rangle(\mathtt{r}_{\infty})_{j}^{\s j} $, for any $ j\in \N_{0}$, $ \s\in \{\pm\} $, 
and  $(\mathtt{r}_{\infty})_{j}^{\s j}$ % are given by Theorem \ref{thm:reduKG} 
defined in \eqref{Finalnormalform}. Formulas 
\eqref{lipfinal} and \eqref{pressione} are implied by \eqref{Finalnormalform} and \eqref{stimeautovalfinali}.  
Since the set  $\Omega_{1}$ defined in
\eqref{calOinfty1sec6} 
coincides with $\Lambda_0\cap\Lambda_{1}$ (cf. \eqref{calOzero}-\eqref{calOinfty1}), 
%Remark \ref{rmk:includozeroeprime}
 it follows that
\begin{equation}\label{setfinalefinale}
\mathcal{O}_{\infty}\equiv\Lambda_0\cap\Lambda_1\cap\Lambda_2^{+}\cap\Lambda_2^{-}
\,.
\end{equation}
For any $\omega\in \mathcal{O}_{\infty}$ we define the T\"oplitz in time operator 
\[
\mathfrak{F} :=\Phi_{\infty}\circ\Theta_2\circ\Theta_1\circ{\bf \Psi}\circ\cU^{-1} 
\]
where $  \cU^{-1}, {\bf \Psi}, \Theta_1, \Theta_2 $ are defined in 
Propositions \ref{diagonalizzazione order 1}, \ref{blockTotale}, \ref{prop:ridord1}, \ref{prop:ridord0}.
By Lemma \ref{equidefalgebra}-$(vi)$ 
the map $\mathfrak{F}$ is real-to-real, reversibility and parity preserving and hence items $(1), (2), (3)$ of Theorem \ref{main:thm} follow.

The first estimate \eqref{stimemappa} is proven by composition using the corresponging estimates on  
each operator $  \cU^{-1}, {\bf \Psi}, \Theta_1, \Theta_2 $ and $\Phi_{\infty}$.
The first four maps satisfy the bound by direct inspection due to their pseudo-differential structure.
For the map $\Phi_{\infty}$
recall \eqref{granophi}. 
The estimate \eqref{nuovatamediff} follows by composition using
\eqref{stimasulresto} (together with Lemma \ref{sobaction}), 
\eqref{stimaMappaPsi}, \eqref{stimasuPhiuno}, \eqref{stimasuPhidue}
and \eqref{grano} (together with Lemma \ref{rmk:tametresbarre}).

The quasi-periodic in time map $\mathfrak{F}(\omega t)$  conjugates 
the dynamical system \eqref{coordcomp}
to \eqref{firstorderBis} (see formul\ae\, \eqref{complexZZ2}-\eqref{ellediag}). 
Then for any $\omega$ in the set $\calO_\infty$ 
in \eqref{setfinalefinale} the thesis of Theorem \ref{main:thm} holds,
namely we have proved the inclusion in \eqref{includoAutovalfinali} 
in Theorem \ref{thm:cantorset}.

To conclude the proof of Theorems \ref{main:thm} and \ref{thm:cantorset}
it remains to prove the measure estimate \eqref{includoAutovalfinali}. 

\begin{thm}{\bf (Measure estimates).}\label{stimedimisura}
Let $\cO_{\infty}$ be the set in 
\eqref{setfinalefinale}, see also \eqref{calOzero}-\eqref{calOinfty2}.
Then 
\begin{equation}\label{stimedimisuraTeo}
|\Lambda\setminus\cO_{\infty}|\leq C\gamma \,,
\end{equation}
for some constant $C := C(\tau, \nu, {\mathtt m})>0$ and where $\Lambda=[-1/2,1/2]^{\nu}$.
\end{thm}

The rest of this section is devoted to prove Theorem \ref{stimedimisura}. 

Recalling \eqref{calOzero}, \eqref{calOinfty1}, \eqref{calOinfty2} 
we define the ``resonant sets"
\begin{equation}\label{BadSets}
\begin{aligned}
Q_{\ell  }^{(0)}& := Q_{\ell  }^{(0)} (\gamma, \nu) :=\big\{ 
\omega\in \Lambda : \lvert \omega\cdot \ell \rvert
< 2\gamma \langle \ell \rangle^{-\nu}
\big\}\,, 
\\
Q_{\ell j }^{(1)}&  := Q_{\ell j }^{(1)} (\gamma, \tau)
:=\big\{ 
\omega\in \Lambda: \lvert \omega\cdot \ell+(1+\mathfrak{c}) j \rvert
< 2\gamma \langle \ell \rangle^{-\tau} \big\}\,, 
\\
R^{(+)}_{\ell j k \eta}& := R^{(+)}_{\ell j k \eta} (\gamma, \tau ) :=\big\{ 
\omega\in \Lambda : \lvert  \omega\cdot \ell+
\lambda_{j,\eta}^{(\infty)}+\lambda_{k,\eta}^{(\infty)}\rvert
< 2\,\gamma\, \langle \ell \rangle^{-\tau} \big\},
\\
R^{(-)}_{\ell j k \eta}& := R^{(-)}_{\ell j k \eta} (\gamma^{3/2}, \tau ) :=\big\{ 
\omega\in \Lambda : \lvert  \omega\cdot \ell+
\lambda_{j,\eta}^{(\infty)}-\lambda_{k,\eta }^{(\infty)}\rvert
< 2\,\gamma^{3/2}\, \langle \ell \rangle^{-\tau}  \big\}\,,
\end{aligned}
\end{equation}
where $\lambda_{j,\eta}^{(\infty)}$ are given in \eqref{finaleigenv}
and $\mathfrak{c}$ satisfies \eqref{pressione}.
In this way
\begin{align}\label{11.4}
\Lambda\setminus\Lambda_{0}&= \bigcup_{\ell\in\mathbb{Z}^{\nu}\setminus\{0\}} Q_{\ell  }^{(0)}\,, \qquad
&&\Lambda\setminus\Lambda_{1}= \bigcup_{\ell\in\mathbb{Z}^{\nu}, j\in \N}Q_{\ell j }^{(1)}\,,\\
\Lambda\setminus\Lambda_{2}^{+}&=\bigcup_{\substack{\ell\in\mathbb{Z}^{\nu}, j, k\in \N_{0} \\ \eta \in \{\pm\}}}
R^{(+)}_{\ell j k \eta} \,,\qquad
&&\Lambda\setminus\Lambda_{2}^{-}
= \bigcup_{\substack{\ell\in\mathbb{Z}^{\nu}, j, k\in \N_{0}, \eta \in \{\pm\} \\ (\ell, j, k)\neq (0, j, j)} }
R^{(-)}_{\ell j k \eta}\,. \label{11.5}
\end{align}

\begin{rmk}\label{rmk.inclusione.diofantei}
Recalling  \eqref{calOzero} $\tau>\nu$ and $\gamma>\gamma^{3/2}$,
we note that  $R^{(-)}_{\ell j j \eta} \subseteq Q^{(0)}_{\ell}$.
\end{rmk}

The next lemma
gives  relations among $\ell,j,k$
 in order to have non empty resonant sets.

\begin{lemma}\label{bip} 
There exists  $C>0$ such that,  for any $\ell \in\Z^{\nu}, j,k\in\N_{0}$, $ \eta \in \{ \pm \} $, 

\noindent
$(i)$ if $Q^{(1)}_{\ell j}\neq \emptyset$ then $j \le C\jap{\ell}$;

\noindent
$(ii)$ if $R^{(+)}_{\ell j k \eta}\neq \emptyset$,  
then $j,k \le C \jap{\ell}$;

\noindent
$(iii)$ if $R^{(-)}_{\ell j k \eta}\neq \emptyset$ 
 then 
$\big|\tD_{\mathtt{m}}(j)-\tD_{\mathtt{m}}(k) \big| \le C \jap{\ell}$
where $\tD_{\mathtt{m}}(j)$ is in \eqref{TDM}.
\end{lemma}

\begin{proof}
We prove the most difficult case item $ (iii) $. 
Recalling \eqref{TDM}, we note that, for $\mathtt{m}>0$ and $j\in\N$, 
\begin{equation}\label{asintotica}
\tD_{\mathtt{m}}(j)=\sqrt{j^{2}+\mathtt{m}} 
=j+\frac{\mathtt{m}}{2j}+\mathfrak{n}(j)\,,
\qquad |\mathfrak{n}(j)|\leq\frac{\mathtt{m}^{2}}{8j^{3}}\,.
\end{equation}
We prove item $(iii)$, which is trivially true for $j=k$.
If $R^{(-)}_{\ell j k \eta}\neq \emptyset$, then there exists $\omega$ such that 
\begin{equation}\label{marechiaro}
\lvert \lambda^{(\infty)}_{j, \eta}(\omega)-\lambda^{(\infty)}_{k, \eta}(\omega) \rvert< 
2 \gamma^{3/2}\langle \ell \rangle^{-\tau}+  \lvert {\omega}\cdot \ell \rvert\,.
\end{equation}
Moreover, by \eqref{finaleigenv}, for any $j\neq k$
\begin{equation*}
\begin{aligned}
\lvert \lambda^{(\infty)}_{j, \eta}(\omega) - \lambda^{(\infty)}_{k, \eta}(\omega) \rvert 
&\geq 
\lvert 1+ \mathfrak{c} \rvert \lvert \tD_{\mathtt{m}}(j)-\tD_{\mathtt{m}}(k) \rvert
- 2 \langle j\rangle^{-1}\sup_{\s\in \{\pm\}}|\mathfrak{r}_{j}^{\s j}|
- 2 \langle k\rangle^{-1}\sup_{\s\in \{\pm\}}|\mathfrak{r}_{k}^{\s k}\rvert  
\\&
\geq \tfrac{1}{3} \lvert \tD_{\mathtt{m}}(j)-\tD_{\mathtt{m}}(k)  \rvert\,,
\end{aligned}
\end{equation*}
by  \eqref{pressione} and taking $\delta_{0}$ in \eqref{smallCondCoeff} small enough.
Therefore by \eqref{marechiaro} 
we deduce 
$ \tfrac{1}{3} \lvert \tD_{\mathtt{m}}(j)-\tD_{\mathtt{m}}(k)  \rvert
\leq 2 \gamma^{3/2}\langle \ell \rangle^{-\tau}+|\omega||\ell|\le 2(1+|\oo|) \jap{\ell} $ 
which proves Item $(iii)$  taking $C$ large enough. 
\end{proof}

We estimate the measure of a  resonant set
in \eqref{BadSets} with constants $ (\mu, \tau_0) $ instead of $ (\gamma, \tau)$.

\begin{lemma}\label{singolo}
For any $\mu\in (0, 1/2)$ and any $\tau_0\in\mathbb{N}$ 
the sets in \eqref{BadSets} satisfy 
$\lvert R^{(-)}_{\ell j k \eta}(\mu, \tau_{0}) \rvert\lesssim \mu \langle \ell \rangle^{-\tau_0-1}$ for $(\ell, j, k)\neq (0, j, j)$. 
The same estimate holds for 
$R^{(+)}_{\ell j k \eta}(\mu, \tau_{0})$,  for 
 $Q^{(1)}_{\ell j}(\mu,\tau_{0})$ if $j\neq 0$ and for $Q^{(0)}_{\ell }(\mu,\tau_{0})$ if $\ell\neq 0$.
\end{lemma}

\begin{proof}
We prove the bound for the set $R^{(-)}_{\ell j k\eta}$ % (with $\ell\neq 0$)
which is the most difficult case. Let us assume that $R^{(-)}_{\ell j k\eta}\neq \emptyset$ so that the Lemma \ref{bip} holds.
The other cases are similar.
We define the function
\begin{equation}\label{perfubini}
\begin{aligned}
\phi_{R^{-}}(\omega)&:= \omega\cdot \ell
+\lambda^{(\infty)}_{j, \eta}(\omega)-\lambda^{(\infty)}_{k, \eta}(\omega)
\\&\stackrel{\eqref{finaleigenv}}{=}
\omega\cdot \ell+(1+ \mathfrak{c}) (\tD_{\mathtt{m}}(j)-\tD_{\mathtt{m}}(k) )
+
\frac{\mathfrak{r}_{j}^{j}+\widehat\eta(j) \mathfrak{r}_{j}^{-j}}{\langle j\rangle}
-\frac{\mathfrak{r}_{k}^{k}+\widehat\eta(k) \mathfrak{r}_{k}^{-k}}{\langle k\rangle} \, .
\end{aligned}
\end{equation}
where $\widehat\eta(j)= \eta$ if $j\ne 0$ and $\widehat\eta(j)=0$ if $j=0$ and similarly for $\widehat\eta(k)$.
For any $ \ell \neq 0 $
we write  $\omega=s \hat{\ell}+v$ where $\hat{\ell}:=\ell /\lvert \ell \rvert$ and $v\cdot \ell=0$ 
and we set $\Psi_{R^{-}}(s):=\phi_{R^{-}}(s \hat{\ell}+v)$.
By  \eqref{pressione}
and  Lemma \ref{bip} 
we have  (recall \eqref{smallCondCoeff})
\begin{equation*}
\begin{aligned}
\lvert \Psi_R(s_{1})-\Psi_R(s_2) \rvert&\geq \lvert s_{1}-s_2 \rvert \big( 
|\ell|-
C|\ell||\mathfrak{c}|^{{\rm lip}}
-4\sup_{j\in\Z,\s\in \{\pm\}}
|\fr_{j}^{\s j}|^{{\rm lip}} \big)
%\\&
\geq\frac{\lvert\ell\rvert}{2} \lvert s_{1}-s_2 \rvert \,,
\end{aligned}
\end{equation*}
by  \eqref{smallCondCoeff}.
As a consequence  for $\ell\neq 0$ 
the set $\Delta_{\ell j k \eta}:=\{ s: s \hat{\ell}+v\in R^{(-)}_{\ell j k \eta}\}$ 
has  measure
$ \lvert \Delta_{\ell j k \eta} \rvert\le {2}{\,\lvert \ell\rvert^{-1}} \,{4\,\mu\,}{| \ell |^{-\tau_{0}}} $
and the bounds of the lemma follows by Fubini's theorem. 
For $ \ell = 0 $, there is $ c := c (\tm)> 0 $ such that for any $ j \neq k $, $ j, k \in \N_0 $, 
the function  in  \eqref{perfubini} satisfies $ |\phi_{R^{-}}(\omega) | \geq c $. Hence if $\mu <c$ then the set $R^{(-)}_{0 j k\eta}= \emptyset$, otherwise $|R^{(-)}_{0 j k\eta}| \le1 \le c^{-1}\mu$.  In both case the desired estimate holds.
\end{proof}

\begin{lemma}\label{inclusionenelleprime} Fix $\gamma, \tau$ as in \eqref{costanti} and \eqref{costantiGAMMA}.
There exists $\mathtt C>0$ such that 
% taking $\tau$ as in \eqref{costanti} 
for any $ \tau_1 \in [1, \tau] $ then, for any $ \ell \in \Z^{\nu} $, $ \eta \in \{ \pm \}$, 
\begin{equation}\label{incluwave}
R^{(-)}_{\ell j k \eta}(\gamma^{3/2}, \tau)\subseteq Q^{(1)}_{\ell, j-k}(\gamma, \tau_1)\,, \qquad {\rm for} \, j, k\geq \mathtt{C} \langle \ell \rangle^{\tau_1}\gamma^{-1/2}\,.
\end{equation}
\end{lemma}

\begin{proof} If $ \ell = 0 $, $ j = k $, or  if $ R^{(-)}_{\ell j k \eta}(\gamma^{3/2}, \tau) $ is empty   
then \eqref{incluwave} is trivial.
% Let  $R^{(-)}_{\ell j k \eta}(\gamma^{3/2}, \tau)\neq \emptyset$ and 
 Consider $\omega\not\in Q^{(1)}_{\ell, j-k}(\gamma, \tau_1)$.  
 For any  $ j, k\geq \mathtt{C} \langle \ell \rangle^{\tau_1}\gamma^{-1/2}  $,  
  by \eqref{finaleigenv}, \eqref{asintotica}, 
 \eqref{pressione}
we have
\begin{align}
\lvert \omega\cdot \ell +\lambda^{(\infty)}_{j, \eta}-\lambda^{(\infty)}_{k, \eta}\rvert 
&\geq 
\lvert \omega\cdot \ell + (1+\fc) (j - k) \rvert -  
\frac{\mathtt{m}}{2} \Big| \frac{1}{j} - \frac{1}{k} \Big| \notag  \\
&- \frac{ \tm^2}{8\min\{ j, k \}^3}
 -\frac{4}{\min\{ j, k \}}\sup_{j\in\Z,\s\in \{\pm\}}
|\fr_{j}^{\s j}| \notag 
\\
& \stackrel{\eqref{BadSets},\eqref{smallCondCoeff}} \geq 
\frac{2\gamma}{\langle \ell \rangle^{\tau_1}}
- \lvert j-k \rvert \frac{\tm \g }{2 \tC^2 \jap{\ell}^{2\tau_1}}
- \frac{\tm^2 \g^{3/2}}{8 \tC^3 \jap{\ell}^{3\tau_1}} 
-\frac{ C \g^{2}\delta_0}{  \tC \jap{\ell}^{\tau_1}} \, . \label{lastdisa}
\end{align}
Since $j$ and $k$ are large,  
then by Lemma \ref{bip} we have $C\jap{\ell} \geq |\tD_\tm (j) - \tD_\tm (k) | \geq |j-k| / 2  $ and so, by   \eqref{lastdisa}
$$
\lvert \omega\cdot \ell +\lambda^{(\infty)}_{j, \eta}-\lambda^{(\infty)}_{k, \eta}\rvert 
\geq \frac{2\gamma}{\langle \ell \rangle^{\tau_1}}
- \tc_1 \Big(\frac{\gamma}{\langle \ell \rangle^{2\tau_1-1}}
+ \frac{\g^{3/2}}{ \jap{\ell}^{3\tau_1}}
+\frac{  \g^{2}}{  \jap{\ell}^{\tau_1}}\Big)
\geq \frac{2\gamma^{3/2}}{\langle \ell \rangle^{\tau}}\,,
$$
taking  $\tc_1$ small enough by taking $\tC$ large and $\delta_0$ small.  Then 
 $ \omega \not \in R^{(-)}_{\ell j k \eta}(\gamma^{3/2}, \tau) $. 
\end{proof}

\begin{proof}[{\bf Proof of Theorem \ref{stimedimisura}}]
To prove the theorem we estimate separately the sets in \eqref{11.4}-\eqref{11.5}. 
We  prove the most difficult estimate $|\Lambda\setminus\Lambda_{2}^{-}| \lesssim \gamma $.  
We first observe 
that it is sufficient to prove the claim for $\gamma\in (0,\gamma_0)$ for $\gamma_0$ fixed but arbitrarily small.
Recall that  $\tau>2\nu+4$ as in \eqref{costanti} and take $ \tau_1 := \nu + 2 $. 
By Remark \ref{rmk.inclusione.diofantei}, we only need to consider $j\neq k$. Without loss of generality we can assume that $j>k$. Defining, for any $\ell\in \Z^{\nu}$, the sets  
$$
\begin{aligned}
& \mathcal{A}(\ell) :=\{(j,k)\in\N^2\,:\,
j > k\geq \mathtt{C}
\langle \ell \rangle^{\tau_1}\gamma^{-1/2}
 \}   \, ,  \\
&  \mathcal{B}(\ell) :=\{(j,k)\in\N_0^2\setminus\cA(\ell)\,:\,
 k<j\leq 2\mathtt{C}
 \langle \ell \rangle^{\tau_1}\gamma^{-1/2}
 \} \, , \\
& 
\mathcal{C}(\ell) :=\{(j,k)\in\N_0^2\setminus\cA(\ell)\,:\,
j\geq  2\mathtt{C}
\langle \ell \rangle^{\tau_1}\gamma^{-1/2}
\}  \, , 
\end{aligned}
$$ 
 where $\mathtt{C}$ is given by Lemma \ref{inclusionenelleprime}, the measure of the set
$ \Lambda\setminus\Lambda_{2}^{-} $  in \eqref{11.5} is bounded by   
 \begin{align*}
 	\left\lvert \bigcup_{\substack{\ell\in\mathbb{Z}^{\nu},j, k\in \N_{0}, \eta\in \{\pm\}\\ (\ell, j, k)\neq (0, j, j)}} 
 	R^{(-)}_{\ell j k \eta} (\g^{3/2}, \tau)  \right\rvert
 &	\le  \sum_{{\ell\in \Z^{\nu}\setminus \{0\}} }  \left| Q^{(0)}_{\ell} \right | +
 		\left\lvert \bigcup_{\substack{\ell\in \Z^{\nu}, \eta\in \{\pm\}\\(j,k)\in\mathcal{A}(\ell)}}R^{(-)}_{\ell j k \eta} (\g^{3/2}, \tau) 	\right\lvert  
 \\
 &	+
 	\sum_{\substack{\ell\in \Z^{\nu}, \eta\in \{\pm\}\\(j,k)\in\mathcal{B}(\ell)}}|R^{(-)}_{\ell j k \eta} (\g^{3/2}, \tau) |
 	+\sum_{\substack{\ell\in \Z^{\nu}, \eta\in \{\pm\}\\(j,k)\in\mathcal{C}(\ell)}}|R^{(-)}_{\ell j k \eta} (\g^{3/2}, \tau) | \, .
 \end{align*}
Standard estimates on diophantine vectors imply that the first summand is bounded proportionally to $\g$.
By Lemmata \ref{inclusionenelleprime},   \ref{singolo},  and  Lemma \ref{bip}$(i)$ we have
\begin{equation*}
\begin{aligned} 
	\left\lvert \bigcup_{\substack{\ell\in \Z^{\nu}, \eta\in \{\pm\}\\(j,k)\in\mathcal{A}(\ell)}}R^{(-)}_{\ell j k \eta} (\g^{3/2}, \tau) 	\right\lvert  \lesssim\sum_{\substack{\ell\in \Z^{\nu}, |h|\lesssim \jap{\ell}}}
  \!\!\! |Q^\1_{\ell h } (\g, \tau_{1}) | &\lesssim   
\sum_{\ell\in \Z^{\nu}, |h|\lesssim \jap{\ell}} 
\gamma \langle \ell \rangle^{-\tau_1-1}
\lesssim
 \gamma\sum_{\ell\in\Z^{\nu}}  \langle \ell \rangle^{-\tau_1}
 \lesssim \g \,.
\end{aligned}
\end{equation*}

We now consider the case  $(j, k)\in \cB(\ell)$. By Lemma  \ref{bip}$-(iii)$ we can restrict to $\big|\tD_{\mathtt{m}}(j)-\tD_{\mathtt{m}}(k) \big| \le C \jap{\ell}$.
If $j - k > 2\mathtt{m}$ then one has $\tD_{\mathtt{m}}(j)-\tD_{\mathtt{m}}(k) > j - k - \mathtt{m} > \frac{j-k}{2}$ and therefore $j - k < 2 C \jap{\ell}$. Hence
by using Lemma \ref{singolo} 
one has
\begin{equation*}
\begin{aligned}
\sum_{\substack{\ell\in \Z^{\nu}, \eta\in \{\pm\}\\(j,k)\in\mathcal{B}(\ell)}}|R^{(-)}_{\ell j k \eta} (\g^{3/2}, \tau) | &\lesssim 
\sum_{\substack{\ell\in \Z^{\nu}, \eta\in \{\pm\}\\(j,k)\in\mathcal{B}(\ell)\\ 2 \mathtt{m}< j - k < 2 C \jap{\ell}}}|R^{(-)}_{\ell j k \eta} (\g^{3/2}, \tau) | +
\sum_{\substack{\ell\in \Z^{\nu}, \eta\in \{\pm\}\\(j,k)\in\mathcal{B}(\ell)\\ j - k \le 2 \tm}}|R^{(-)}_{\ell j k \eta} (\g^{3/2}, \tau) | \\
& \lesssim \gamma^{3/2} 
\sum_{\ell\in\Z^{\nu}} \frac{\jap{\ell} \langle \ell \rangle^{\tau_1}}{\sqrt{\gamma}\langle \ell \rangle^{\tau+1}}  +
\gamma^{3/2}  \sum_{\ell\in\Z^{\nu}} \frac{ \langle \ell \rangle^{\tau_1}}{\sqrt{\gamma}\langle \ell \rangle^{\tau+1}} 
%&\lesssim \gamma 
%\sum_{\ell\in\Z^{\nu}}  \langle \ell \rangle^{-(\tau-\tau_1)} +  \gamma 
%\sum_{\ell\in\Z^{\nu}}  \langle \ell \rangle^{-(\tau-\tau_1+1)}
\lesssim \gamma \,,
\end{aligned}
\end{equation*}
using that $\tau-\tau_1=\tau-\nu-2>\nu+1$. 

Finally we study the case $(j, k)\in \cC(\ell)$. We claim that for any 
$(j, k)\in\cC(\ell)$ the set $R^{(-)}_{\ell j k \eta} (\g^{3/2}, \tau) $ is  empty.
Indeed, if $(j, k)\in\cC(\ell)$ then, by difference,  $ j - k  \ge \tC \langle \ell \rangle^{\tau_1}\gamma^{-1/2}$. Let us suppose that $R^{(-)}_{\ell j k \eta} (\g^{3/2}, \tau) \neq \emptyset$. By using Lemma \ref{bip}$(iii)$ we have
\[
C\jap{\ell} \ge \tD_{\mathtt{m}}(j)-\tD_{\mathtt{m}}(k) > j - k - \tm \ge \tC \langle \ell \rangle^{\tau_1}\gamma^{-1/2} -\tm  \ge \frac{\tC \langle \ell \rangle^{\tau_1}\gamma^{-1/2}}{2} 
\]
provided  $\gamma \le \left( \frac{\tC}{2m}\right)^{2}$.
The last inequality is not possible if
$\g_{0} < \min{  \left\{ ( \frac{\tC}{2m})^{2} , (\frac{\tC}{2C} \right)^{2}  \}}$. This contradiction 
proves the claim.

The estimates for 
$ \Lambda_{2}^{+}, \Lambda_{1}, \Lambda_0 $ follow similarly.
\end{proof}
\appendix

% \bibliographystyle{plain}

%\bibliography{biblioBFPTwave}

\end{document}